\documentclass[12pt]{amsart}

\usepackage{amsmath,amssymb,amsthm}
\usepackage{mathrsfs,bm}
\usepackage[pdftex]{graphicx}
\usepackage{hyperref}
\usepackage{float}
\usepackage{xcolor}
\usepackage{microtype}
\usepackage[shortlabels]{enumitem}
\usepackage{cite}
\usepackage{fullpage}
\usepackage{multirow}
\usepackage{caption}

\newtheorem{theorem}{Theorem}[section]

\newtheorem{lemma}[theorem]{Lemma}
\theoremstyle{definition}

\theoremstyle{remark}
\newtheorem{remark}[theorem]{Remark}

\def\O{\Omega}
\def\RR{\mathbb{R}}

\def\CC{\mathbb{C}}
\def\cH{\mathcal{H}}

\def\curl{\mbox{curl}}

\newcommand{\ii}{\mathfrak i}
\newcommand{\mb}[1]{\ensuremath{\mathbf{#1}}}

\newcommand{\change}[1]{\textcolor{black}{#1}}

\def\bfA{\mb{A}}
\def\bfF{\mb{F}}

\author[J. Ovall]{Jeffrey S. Ovall}
\address{Jeffrey S. Ovall,
  Fariborz Maseeh Department of Mathematics and Statistics,
  Portland State University,
  Portland, OR 97201}
\email{jovall@pdx.edu}

\author[L. Zhu]{Li Zhu}
\address{Li (Julie) Zhu,
	Fariborz Maseeh Department of Mathematics and Statistics,
	Portland State University,
	Portland, OR 97201}
\email{lizhu@pdx.edu}

\thanks{The work of J. Ovall was partially supported by the NSF grant
  2208056 ``Computational Tools for Exploring Eigenvector
  Localization''.  The work of L. Zhu was supported by the NSF grant
  2136228 ``RTG: Program in Computation- and Data-Enabled Science''.
  The authors thank Dr. Jay Gopalakrishnan for useful conversations
  and coaching on the use of NGSolve and pyeigfeast in this context.}

\title[Magnetic Schr\"odinger Canonical Gauge]{A Canonical Gauge for
  Computing of Eigenpairs of the Magnetic Schr\"odinger Operator}
\date{\today}

\begin{document}
\begin{abstract}
  We consider the eigenvalue problem for the magnetic Schr\"odinger
  operator and take advantage of a property called gauge invariance to
  transform the given problem into an equivalent problem that is more
  amenable to numerical approximation.  More specifically, we propose
  a canonical magnetic gauge that can be computed by solving a Poisson problem,
  that yields a new operator having the same spectrum but eigenvectors
  that are less oscillatory.  Extensive numerical tests demonstrate
  that accurate computation of eigenpairs can be done more efficiently
  and stably with the canonical magnetic gauge.
\end{abstract}

\maketitle

\section{Introduction}\label{Introduction}
The magnetic Schr\"odinger operator,
\begin{align}\label{HAV}
    H(\mb{A}, V) = (-\ii \nabla - \mb{A}) \cdot (-\ii \nabla - \mb{A}) + V ~,
\end{align}
is used in models of the motion of a charged particle in an
electromagnetic field.  We will refer to $H(\mb{A},V)$ as the
\textit{Hamiltonian} associated with $\mb{A}$ and $V$.  The curl of
the \textit{magnetic vector potential} $\mb{A}$ is, up to scaling, the
underlying magnetic field, and the gradient of the scalar potential
$V$ is similarly associated with the electric field (cf.~\cite[Section
4.5]{Griffiths2018}).
In this context, the potentials are typically referred to as
\textit{gauges}, and it is clear that adding a gradient to $\mb{A}$,
or a constant scalar to $V$, cannot change the underlying dynamics, as
it does not change the magnetic and electric fields.  Thus, two gauges
for the magnetic field are deemed \textit{equivalent} if they differ
by a gradient.  From the earliest days of electrodynamic theory,
researchers in the field have sought to fix a ``canonical'' gauge by
applying additional constraints on $\mb{A}$ that are motivated by a
variety of physical or mathematical considerations~\cite{Jackson2001}.
Popular gauges include those of Lorenz and Coulomb, though scientists
continue to debate the merits of different gauges in different
contexts as well as the methodology for evaluating gauge choices
(cf.~\cite{Maudlin2018,Gomes2022}).

In the present work, we consider the eigenvalue problem associated
with $H(\mb{A},V)$, and take advantage of this \textit{gauge
  invariance} to select a canonical magnetic potential that will be
seen to improve computational efficiency and accuracy for associated eigenvalue problems.  More
specifically, we consider the eigenvalue problem
\begin{align}\label{EP}
H(\mb{A},V)\psi =\lambda\psi\mbox{ in }\Omega~,
\end{align}
where $\Omega\subset\RR^d$ is a bounded open set with sufficiently
regular boundary (e.g. Lipschitz regularity).
For this problem to be well-posed, we must impose suitable boundary
conditions.  Dirichlet conditions, $\psi = 0$ on $\partial\Omega$, or
Neumann conditions,  $(-\ii \nabla -\mb{A})\psi\cdot\mb{n}=0$ on
$\partial\Omega$, both  ensure that $H=H(\mb{A},V)$ is selfadjoint,
\begin{align}\label{BilinearForm}
\int_\Omega (H
  v)\overline{w}\,dx=\underbrace{\int_\Omega(-\ii\nabla-\mb{A})v\cdot\overline{(-\ii\nabla-\mb{A})w}+Vv\overline{w}\,dx}_{B(v,w)}\mbox{
  for }v,w\in\cH~,
\end{align}
in the weak sense.  Here, $\cH=H_0^1(\Omega)=H_0^1(\Omega;\CC)$ in the
case of Dirichlet conditions, and $\cH=H^1(\Omega)=H^1(\Omega;\CC)$ in
the case of Neumann conditions.  Mixed Dirichlet and Neumann
conditions also give rise to the same bilinear form $B$ and
variational problem~\eqref{BilinearForm}, with appropriate adjustments
to the function space $\cH$ that incorporate the (essential) Dirichlet
conditions.  It follows that the eigenvalues of $H(\mb{A},V)$ for any
of these boundary conditions are real, though the eigenvectors are
typically complex.  Contributions that consider how the vector and
scalar fields influence the asymptotic distribution of eigenvalues, the
asymptotic behavior of low-lying eigenvalues with respect to field
strength, and the spatial localization of eigenvectors for~\eqref{EP}
include~\cite{Helffer1987,Shen1996,Poggi2024,Hoskins2024,Ovall2024}.

We here assume that $\mb{A}$ is continuously differentiable in each of
its components, up to the boundary, i.e.
$\mb{A}\in [C^1(\overline\Omega;\RR)]^d$.  Although we might relax the
regularity requirements on $\mb{A}$, typical examples that are deemed
``physically meaningful'' have at least this regularity. We assume
that $V\in L^\infty(\Omega;\RR)$ is non-negative almost everywhere.
Because we are primarily concerned with determining a vector potential
that is canonical in a sense motivated by computational
considerations, we will primarily focus on the case $V=0$ in our
numerical experiments.  

When $V=0$, the resulting operator,
\begin{align}\label{MagneticLaplacian}
H(\mb{A})=H(\mb{A},0)=(-\ii \nabla - \mb{A}) \cdot (-\ii \nabla - \mb{A})~,
\end{align}
is called the \textit{magnetic Laplacian}, and it has received
significant attention in its own
right~\cite{BonnaillieNoel2006,Bonnaillie-Noel2007,BonnaillieNoel2016,Helffer1996,Fournais2006a,Pino2000,Fournais2019,Fournais2023,Baur2024}.
Much of this work has employed the tools of semiclassical analysis
(cf.~\cite{Zworski2012}), and considers the behavior of eigenvalues
and vectors of the operator
$H_h=(-\ii h\nabla - \mb{A}) \cdot (-\ii h\nabla - \mb{A})$ as the
\textit{semiclassical parameter} $h\to 0$.  We note that this $h$ is
not related to a discretization (mesh/grid) parameter.  Very little
has been published involving numerical methods for $H(\mb{A},V)$,
$H(\mb{A})$ or $H_h$.  We
highlight~\cite{Bonnaillie-Noel2007} as an exception in this regard,
which considers a variant of the operator $H_h$ on planar domains with
corners, together with the constant-curl vector field
$\mb{A}=(-y/2,x/2)$, which seems to be a popular choice in the
literature.

 \change{In brief, our approach is to
  replace the given gauge $\mb{A}$ with a ``canonical'' gauge $\mb{F}$
  such that $H(\mb{A},V)$ and $H(\mb{F},V)$ have the same spectrum,
  but for which the eigenvectors of $H(\mb{F},V)$ can be more readily
  resolved using coarser discretizations than would be needed to
  resolve the associated eigenvectors of $H(\mb{A},V)$.  The vector
  field $\mb{F}$ is obtained from $\mb{A}$ by solving a Neumann
  Poisson problem.  On the level of discretization, we must first
  compute a discrete approximation of $\mb{F}$ before solving the
  associated eigenvalue, and this approximation is done at the
  same level of discretization as is used for solving the
  eigenvalue problem.}

The rest of this paper is organized as follows: \change{In
  Section~\ref{Theory}, we provide heuristic motivation and initial
  empirical validation for our proposed approach via a few theoretical
  results and preliminary examples. 
  Section~\ref{Experiments} contains a more detailed comparison
  between computing eigenpairs of $H(\mb{A},V)$ and $H(\mb{F},V)$ for
  a variety of operators and domains.  These comparisons strongly
  support the claims made in
  Section~\ref{Theory}. A few concluding
  remarks are given in Section~\ref{Conclusions}.}

\section{Gauge Invariance and a Practical Helmholtz
  Decomposition}\label{Theory}
Gauge invariance for the magnetic potential is reflected
mathematically in the following lemma, which uses conjugation of
$H(\mb{A},V)$ to affect a gauge transform.  Although the result has
been known in the mathematical physics literature for many years
(cf.~\cite{Fock1926,Weyl1929,Jackson2001}), we
include a brief proof for completeness. 
\begin{lemma}[Gauge Invariance] \label{lemma1} Suppose that
  $\bfA=\nabla a + \bfF$ in $\O$ \change{for some scalar field $a$ and vector
  field $\mb{F}$}. Then
  $e^{-\ii a}H(\bfA, V)e^{\ii a}=H(\bfF, V)$. Furthermore,
  $(\lambda, \psi)$ is an eigenpair of $H(\bfA, V)$ if and only if
  $(\lambda, e^{-\ii a}\psi)$ is an eigenpair of $H(\bfF, V)$.
\end{lemma}
\begin{proof}
  For any differentiable scalar field $u$ and vector field $\mb{G}$,
  it holds that
\begin{align*}
(-\ii \nabla - \mb{A})e^{\ii a}u &=e^{\ii a}(-\ii \nabla - \mb{F})u~,\\
(-\ii \nabla - \mb{A})\cdot(e^{\ii a}\mb{G})&=e^{\ii a}(-\ii \nabla - \mb{F})\cdot\mb{G}~.
\end{align*}
Taking $\mb{G}=(-\ii \nabla - \mb{F})u$ in the second identity, we see
that $e^{-\ii a}H(\mb{A}, V)e^{\ii a}u=H(\mb{F}, V)u$, as claimed.
The assertion about eigenpairs of $H(\mb{A}, V)$ and $H(\mb{F}, V)$
follows immediately from this.
\end{proof}

Stated less formally, the eigenvalues of Hamiltonians associated with
equivalent magnetic potentials are identical, and their corresponding
eigenvectors can only differ in phase.  The \textit{phase} of $\psi$
is $\theta:\Omega\to(-\pi,\pi]$ such that
$\psi=|\psi|e^{\ii\,\theta}$.  Since $\phi=e^{-\ii a}\psi$ satisfies
$|\phi|=|\psi|$, they clearly differ only in phase.  \change{We also
  see that, for smooth $a$ (which will be the case in our examples),
  corresponding eigenvectors $\psi$ and $\phi$ have precisely the same
  regularity.  As such, asymptotic convergence rates with respect to a
  discretization parameter, such as a mesh parameter $h$ in a finite
  element or finite difference discretization, will be the same
  regardless of which equivalent(!) magnetic potentials are used.  The
  point of our proposed approach is that, through proper choice of
  $\mb{F}$, as described in Theorem~\ref{CanonicalGauge}, the
  eigenvectors of $H(\mb{F},V)$ can be more readily resolved on
  coarser discretization than can their counterparts for
  $H(\mb{A},V)$.  Alternatively, one might say that the asymptotic
  convergence regime is reached earlier, e.g. for larger $h$, with
  $H(\mb{F},V)$ than it is for $H(\mb{A},V)$.}

We now return to the objective mentioned in the introduction, given in
more detail:
\begin{equation}
  \tag{T}\label{KeyTask}
  \parbox{\dimexpr\linewidth-4em}{%
    \strut
 \textit{Given a vector potential $\mb{A}$, determine an equivalent field
  $\mb{F}$ for which the accurate(!) numerical approximation of
  eigenpairs of $H(\mb{F},V)$ can be done more efficiently than for
  $H(\mb{A},V)$.}%
    \strut
  }
\end{equation}
More specifically, given a suitably small error tolerance for the
approximation of a collection of eigenvalues of $H(\mb{A},V)$ via (finite element)
discretization, achieving this tolerance by discretizing $H(\mb{F},V)$
should be cheaper than by discretizing $H(\mb{A},V)$.  \change{This is
  what we mean by use of ``more efficiently'' in this paper---that one reaches a reasonable
  error tolerance for eigenvalue approximations using $H(\mb{F},V)$ on
  (much) coarser
  discretizations than using $H(\mb{A},V)$.  We will explain our use of
  ``more stably'', first used in the abstract, in
  Section~\ref{Example1}, where we compare the computed
  eigenvalues on a few discretization levels for both $\mb{A}$ and $\mb{F}$.} 
We propose the
following choice of $\mb{F}$. 
\begin{theorem}[Canonical Gauge]\label{CanonicalGauge}
Let $\mb{A}$ be a given magnetic potential.  The solution $a\in H^1(\Omega;\RR)$ of the minimization problem
\begin{align}\label{GaugeMinimizer}
\|\mb{A}-\nabla{a}\|_{L^2(\Omega;\RR)}=\min_{v\in H^1(\Omega;\RR)}\|\mb{A}-\nabla{v}\|_{L^2(\Omega;\RR)}~,
\end{align}
which is unique up to an additive constant, is also a weak solution of the Neumann problem
\begin{align}\label{EquivNeumann}
\Delta a = \nabla\cdot\mb{A}\mbox{ in }\Omega\quad,\quad \nabla a\cdot\mb{n}=\mb{A}\cdot\mb{n}\mbox{ on }\partial\Omega~.
\end{align}
Setting  $\mb{F}=\mb{A}-\nabla a$, we have 
$\nabla\cdot \mb{F}=0$. 
\end{theorem}
\begin{proof}
The Neumann problem~\eqref{EquivNeumann} is guaranteed to have a weak
solution $a\in H^1(\Omega)$ that is unique up to an additive constant,
and the definition of $\mb{F}$ ensures that
$\nabla\cdot\mb{F}=\nabla\cdot\mb{A}-\Delta a=0$.  The solution $a$ satisfies the variational equation
\begin{align*}
\int_\Omega \nabla a\cdot\nabla v\,dx=
\int_{\partial\Omega}\change{(\mb{A}\cdot\mb{n})v}\,ds-\int_\Omega\nabla\cdot\mb{A}\,v\,dx=
\int_\Omega\mb{A}\cdot\nabla v\,dx \mbox{ for all }
v\in H^1(\Omega)~.
\end{align*}
Since $\int_\Omega(\mb{A}-\nabla a)\cdot\nabla v\,dx=0$ for all $v\in H^1(\Omega;\RR)$, it follows that $a$ satisfies the minimization problem~\eqref{GaugeMinimizer}, which completes the proof.
\end{proof}

We make a few simple (overlapping) observations concerning our candidate $\mb{F}$ for the
canonical gauge equivalent to $\mb{A}$:
\begin{itemize}
  \item Among all vector fields equivalent to $\mb{A}$, $\mb{F}$ has
    the smallest $L^2$-norm on $\Omega$.
  \item The scalar and vector fields $a$ and $\mb{F}$ provide a 
    Helmoltz-Hodge decomposition of $\mb{A}$ (cf.~\cite{Chorin1993}),
    $\mb{A}=\nabla a+\mb{F}$, with $\int_\Omega\mb{F}\cdot\nabla a\,dx=0$.
  \item The vector field $\mb{F}$ satisfies Coulomb's constraint,
    $\nabla\cdot\mb{F}=0$ in $\Omega$.
  \item The vector field $\mb{F}$ has vanishing boundary flux,
    $\mb{F}\cdot\mb{n}=0$ on $\partial\Omega$.
\end{itemize}
We note that Coulomb's constraint does not, by itself, fix $\mb{F}$.
This fact is illustrated in Section~\ref{SimpleExample}.  Our choice
of $\mb{F}$ is uniquely determined by the minimal norm condition,
which we achieve in practice via solving the Neumann
problem~\eqref{EquivNeumann}.  In relation to the fourth observation,
we note that any divergence-free vector field $\mb{A}$ on $\Omega$
will satisfy $\int_{\partial\Omega}\mb{A}\cdot\mb{n}\,ds=0$, but the
constraint that $\mb{F}\cdot\mb{n}=0$ pointwise on $\partial\Omega$ is
much stronger.

\begin{remark}
  We are not the first to consider the type of gauge described in
  Theorem~\ref{CanonicalGauge}.  For example, Fournais and Helffer consider this
  gauge in\cite{Fournais2010} (cf. Section 1.5 and Appendix D) as part
  of their theoretical development, and they revisit it
  in~\cite{Fournais2019} in the derivation of asymptotics of ground
  state eigenvalue in the special case of constant magnetic fields
  $\mb{A}=(B/2)\,(-y,x)$.
  We are, however, the first to consider the use of the canonical
  gauge as a means of computing eigenpairs more efficiently and stably.
\end{remark}

Although the discussion above provides some indication that the
canonical gauge $\mb{F}$ for $\mb{A}$ on $\Omega$ is special in
certain respects, it may not be apparent why accurately(!) computing
eigenpairs might be easier for $H(\mb{F},V)$ than for
$H(\mb{A},V)$.  Some heuristic motivation for this expectation
follows.
\change{Suppose that $(\lambda,\psi)$ is an eigenpair of
$H(\mb{A},V)$, with $\|\psi\|_{L^2(\Omega)}=1$, and that
$(\lambda,\phi)$ is the corresponding eigenpair of $H(\mb{F},V)$, so
$\psi=e^{\ii\, a}\phi$.  We have
\begin{align}\label{BasicIdentity}
  \nabla\psi-\ii\,\mb{A}\psi=e^{\ii\,a}\left(\nabla\phi-\ii\,\mb{F}\phi\right)\quad,\quad
  |\nabla\psi-\ii\,\mb{A}\psi|=|\nabla\phi-\ii\,\mb{F}\phi|~.
\end{align}
Here and afterward, we use $|\mb{B}|$ to denote the Euclidean norm of
a vector field $\mb{B}$.  We also note that
$\lambda=\|\nabla\psi-\ii\,\mb{A}\psi\|_{L^2(\Omega)}^2+\|V^{1/2}\psi\|_{L^2(\Omega)}^2
=\|\nabla\phi-\ii\,\mb{F}\phi\|_{L^2(\Omega)}^2+\|V^{1/2}\phi\|_{L^2(\Omega)}^2$,
and that
$\|V^{1/2}\psi\|_{L^2(\Omega)}=\|V^{1/2}\phi\|_{L^2(\Omega)}$.  This
provides a second way of seeing that
$\|\nabla\psi-\ii\,\mb{A}\psi\|_{L^2(\Omega)}=\|\nabla\phi-\ii\,\mb{F}\phi\|_{L^2(\Omega)}$
which, though not as direct as the first, provides a value for this
common quantity in terms of $\lambda$ and $V$.
Recalling that
$\|\mb{F}\|_{L^2(\Omega)}<\|\mb{A}\|_{L^2(\Omega)}$ unless
$\mb{A}=\mb{F}$, it is reasonable to conjecture that
$|\mb{F}|<|\mb{A}|$ for much of the domain $\Omega$.  We clearly have
$|\mb{F}\phi|=|\mb{F}\psi|$ and $|\mb{A}\phi|=|\mb{A}\psi|$.  This
discussion leads to our first heuristic:
\begin{align}\label{Heuristic1}\tag{H1}
  \|\mb{F}\phi\|_{L^2(\Omega)}<\|\mb{A}\psi\|_{L^2(\Omega)}, \mbox{ and
  } |\mb{F}\phi|<|\mb{A}\psi| \mbox{ for much of }\Omega~.
\end{align}
Combining~\eqref{Heuristic1} with~\eqref{BasicIdentity} leads to our
second heuristic:
\begin{align}\label{Heuristic2}\tag{H2}
  \|\nabla\phi\|_{L^2(\Omega)}<\|\nabla\psi\|_{L^2(\Omega)},\mbox{ and
  } |\nabla\phi|<|\nabla\psi| \mbox{ for much of }\Omega~.
\end{align}
The heuristic~\eqref{Heuristic2} provides some indication as to why
it might be easier to approximate $\phi$ than $\psi$ using finite
elements.  Although having a large(r) gradient does not, by itself,
indicate that a function is more difficult to approximate using
piecewise polynomials, we expect larger gradients to be indicative of
higher frequency oscillation in the context of eigenvectors, and
highly oscillatory eigenvectors are a greater challenge to resolve.
The heuristics~\eqref{Heuristic1} and~\eqref{Heuristic2}, as well
as the statement about larger gradients being associated with
higher-frequency oscillation, are first demonstrated in
Sections~\ref{ConstantAExample} and~\ref{Ex2FirstLook}.  The experiments in
Section~\ref{Experiments} provide more extensive validation. In these
experiments, we also see that it is common for a third heuristic to
hold:
\begin{align}\label{Heuristic3}\tag{H3}
  \|\nabla\phi\|_{L^2(\Omega)}\approx \|\mb{F}\phi\|_{L^2(\Omega)}\mbox{ and
  } \|\nabla\psi\|_{L^2(\Omega)}\approx \|\mb{A}\psi\|_{L^2(\Omega)}~,
\end{align}
at least lower in the spectrum.}

\change{
  \subsection{Constant vector fields}\label{ConstantAExample}
  Although a constant vector field $\mb{A}$ is not interesting from
  the point of applications---it is associated with having a zero
  magnetic field---it provides a simple illustration of why one might expect
  a well-chosen gauge to reduce oscillation in associated eigenvectors.
  For constant $\mb{A}$, one sees directly that $a=\mb{A}\cdot x$ (up
  to additive constant) and $\mb{F}=\mb{0}$ in
  Theorem~\ref{CanonicalGauge}.  If $(\lambda,\phi)$ is an eigenpair
  of $H(\mb{F})=-\Delta$, then $(\lambda,\psi)$ is the associated
  eigenpair of $H(\mb{A})=-\Delta+2\ii\,\mb{A}\cdot\nabla
  +|\mb{A}|^2$, where $\psi=e^{\ii\,a}\phi=e^{\ii\,\mb{A}\cdot
    x}\phi$.  The heuristic~\eqref{Heuristic1} holds trivially in this case.
In the 1D case with $\Omega=(0,1)$, constant $A$, and homogeneous Dirichlet
conditions, the eigenvalues and
eigenvectors are $\lambda_n=(n\pi)^2$, $\phi_n=\sqrt{2}\sin(n\pi x)$ for
$H(\mb{F})$, and $\psi_n=e^{\ii A x}\sqrt{2}\sin(n\pi x)$ for $H(\mb{A})$.
We see that, for any fixed $n$, $\psi_n$ generically oscillates more
rapidly in both its real and imaginary components than does $\phi_n$,
particularly for larger $A$.  In other words,~\eqref{Heuristic2}
holds.  For this 1D problem, applying~\eqref{Heuristic3} to $\phi_n$
and $\psi_n$, it reads $n\pi\approx 0$ and $\sqrt{(n\pi)^2+A^2}\approx
A$.  Depending on how far one is willing to push the interpretation of
``$\approx$'', ~\eqref{Heuristic3} can be said to hold for smaller $n$.
}

\subsection{Constant-curl vector field comparisons}\label{SimpleExample}
  Consider the vector fields $\mb{A}_1=(-y/2,x/2)$, $\mb{A}_2=(-y,0)$ and
$\mb{A}_3=(0,x)$.  These vector fields have the same (scalar) curl,
$\curl\,\mb{A}_j=1$, and they all satisfy the Coulomb constraint,
$\nabla\cdot\mb{A}_j=0$.  It is also clear that
\begin{align*}
  \mb{A}_1-\mb{A}_2=\mb{A}_3- \mb{A}_1=\nabla(xy/2)~.
\end{align*}
Suppose that $\Omega\subset\RR^2$ is the unit disk.  Since
$\nabla\cdot\mb{A}_1=0$ in $\Omega$ and
$\mb{A}_1\cdot\mb{n}=0$ on $\partial\Omega$, we determine from
Theorem~\ref{CanonicalGauge} that $\mb{F}=\mb{A}_1$ is the canonical gauge
in its equivalence class.  Here, we have
$\|\mb{A}_1\|_{L^2(\Omega)}=\sqrt{\pi/8}$, and
$\|\mb{A}_j\|_{L^2(\Omega)}=\sqrt{\pi/4}$ for $j=2,3$.

Now suppose that $\Omega=(-1,1)\times(-1,1)$.  In this
case, $\|\mb{A}_1\|_{L^2(\Omega)}=\sqrt{2/3}\approx 0.816497$ and
$\|\mb{A}_j\|_{L^2(\Omega)}=\sqrt{4/3}\approx 1.1547$ for $j=2,3$, but none of these
minimize the norm.  The norm-minimizer $\mb{F}$ from this equivalence
class, approximated using a finite element discretization, satisfies
$\|\mb{F}\|_2\approx 0.749872$

\subsection{A First Look at Example 2}\label{Ex2FirstLook}
Let $\Omega=(-1,1)\times(-1,1)$ and
\begin{align*}
  \bfA = -100(\cos( f_1)\sin(f_2), \sin(f_1)\cos(f_2)))\;,\; f_1 =
  5\pi \sin(x^2 + y^2) \;,\; f_2 =5\pi \cos(x^2 + y^2)~.
\end{align*}
We give a stream plot of $\mb{A}$ overlaid on
$\|\mb{A}\|$, together a plot of $|\curl\mb{A}|$ in Figure~\ref{VectorFields}.  The latter of these
is shown because it was argued in~\cite{Ovall2024} that low-lying
eigenvalues of $H(\mb{A})$ tend to be spatially concentrated in areas
where $\curl\mb{A}$ is small.

In Figure~\ref{Ex2FirstEigenvector}, we show plots of the groundstate
eigenvectors of $H(\mb{A})$ and $H(\mb{F})$, $\psi_1$ and $\phi_1$,
under Dirichlet boundary conditions. We show the modulus, real and
imaginary parts, and phase of both vectors.  To compute the phase of a
complex number $z$, we use $\arcsin(\frac{\Im z}{|z|})$, as was done
in~\cite{Bonnaillie-Noel2007}.  These were computed using a finite
element discretization as described in Section~\ref{Experiments}; both
eigenvectors were computed in precisely the same finite element space.
The computed groundstate eigenvalue for $H(\mb{A})$ is
$\lambda_1 \approx 104.0694$, and the computed eigenvalue for
$H(\mb{F})$ is $\lambda_1 \approx 104.0568$; \change{after enriching
  the finite element space by uniformly increasing the local
  polynomial degree, the computed eigenvalue for both $H(\mb{A})$ and
  $H(\mb{F})$ agree to seven digits, $104.0568$, which is the value
  achieved on the coarser discretization for $H(\mb{F})$}.  The plots of $|\psi_1|$ and $|\phi_1|$ are also
visually nearly indistinguishable.  These facts are in line with
Lemma~\ref{lemma1}, which indicates that the discretization is fine
enough to accurately reflect the behavior of the true $\psi_1$ and
$\phi_1$.  We highlight the real and imaginary parts of $\psi_1$ and
$\phi_1$, noting that those for $\phi_1$ are significantly less
oscillatory.  This suggests that $\phi_1$ can be well-resolved in a
coarser finite element space than $\psi_1$ can, at much less cost!  This
claim will be validated through extensive experiments in
Section~\ref{Experiments}.

We finally consider the \change{heuristics~\eqref{Heuristic1}-\eqref{Heuristic3}}.
In this example, we calculated
\begin{align*}
  \|\mb{A}\|_{L^2(\O)} &= 130.643590 &
  \|\nabla\psi_1\|_{L^2(\O)} &= 71.000721 &
  \|\mb{A}\psi_1\|_{L^2(\O)} &= 70.758747\\
  \|\mb{F}\|_{L^2(\O)} &= 89.861420 &
  \|\nabla\phi_1\|_{L^2(\O)} &= 31.032205 &                
\|\mb{F}\phi_1\|_{L^2(\O)} &= 30.473089
\end{align*}
We first note that $\|\mb{F}\|_{L^2(\Omega)}< \|\mb{A}\|_{L^2(\O)}$, as
predicted by Theorem~\ref{CanonicalGauge}.  We also observe that
$\|\mb{A}\psi_1\|_{L^2(\O)} \approx
\|\nabla\psi_1\|_{L^2(\O)}>\|\nabla\change{\phi_1}\|_{L^2(\O)}\approx
\|\mb{F}\phi_1\|_{L^2(\O)}$, \change{as suggested
  in~\eqref{Heuristic2} and~\eqref{Heuristic3}}.
Finally, in Figure~\ref{Ex2FirstEigenvectorB}, we see a validation of
the \change{heuristic~\eqref{Heuristic1}} for the ground states.

\begin{figure}
  \centering
  \includegraphics[width=0.24\textwidth]{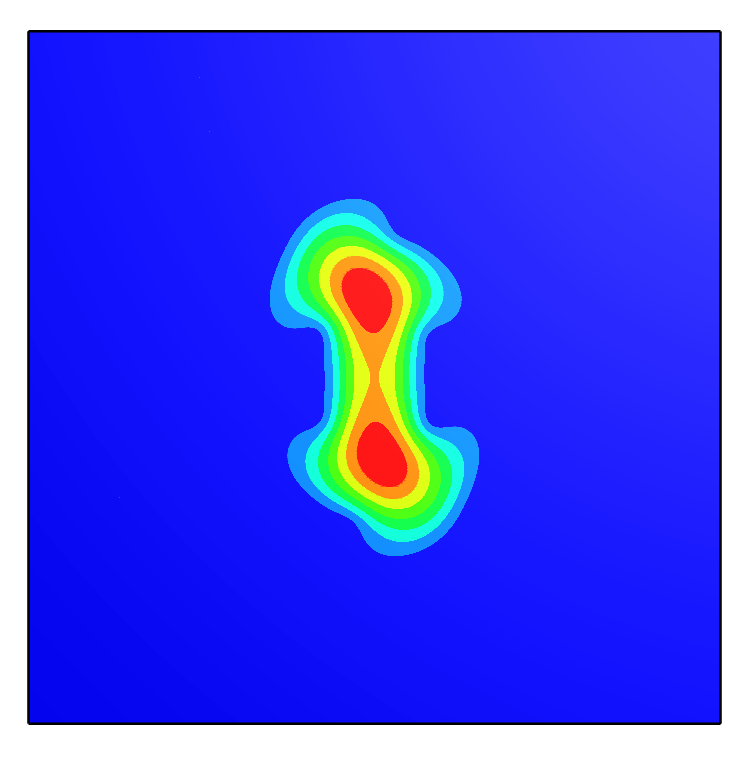}
  \includegraphics[width=0.245\textwidth]{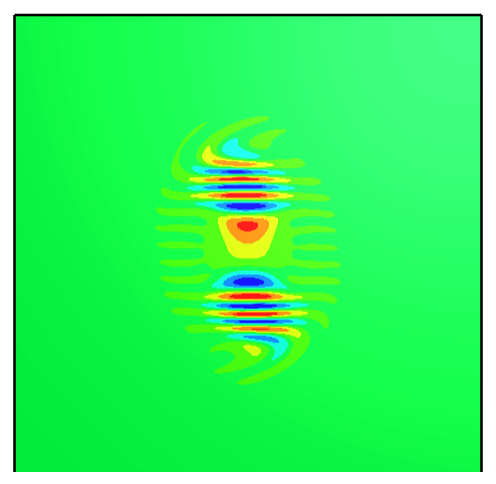}
  \includegraphics[width=0.24\textwidth]{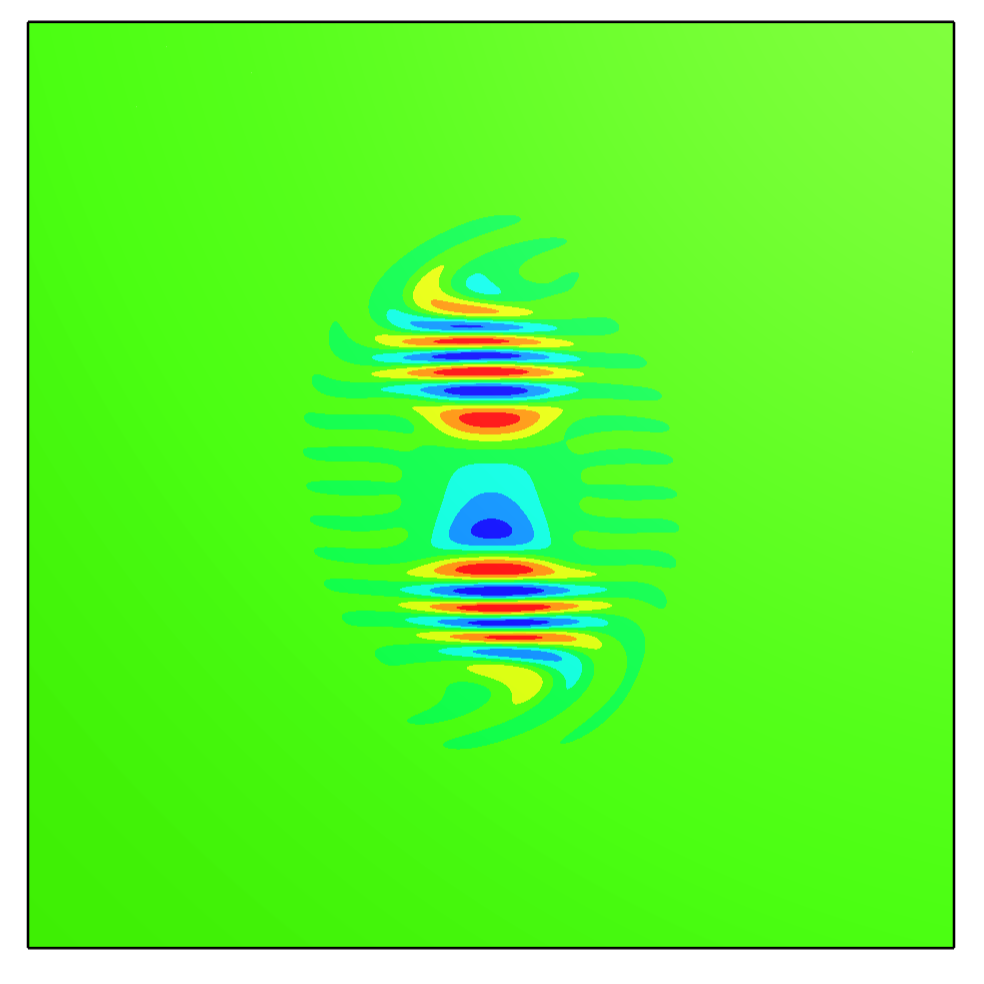}
  \includegraphics[width=0.24\textwidth]{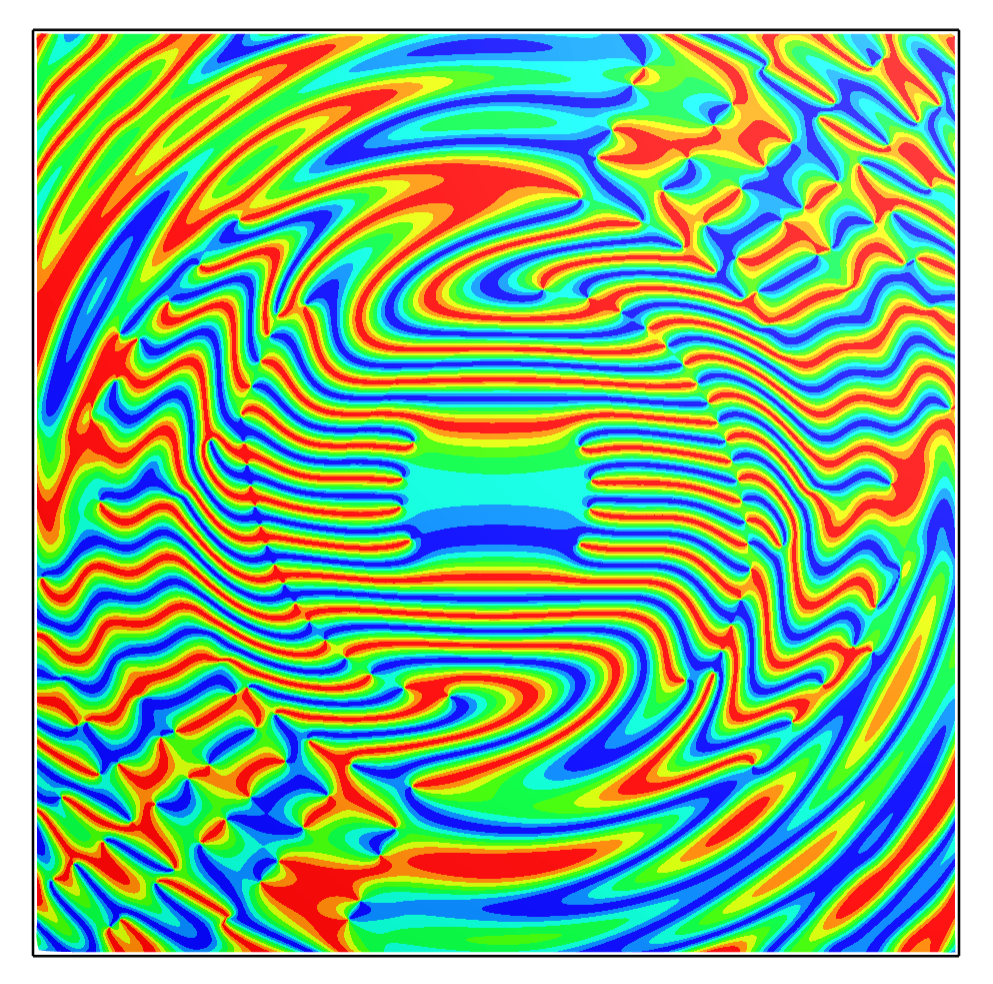}
  \includegraphics[width=0.24\textwidth]{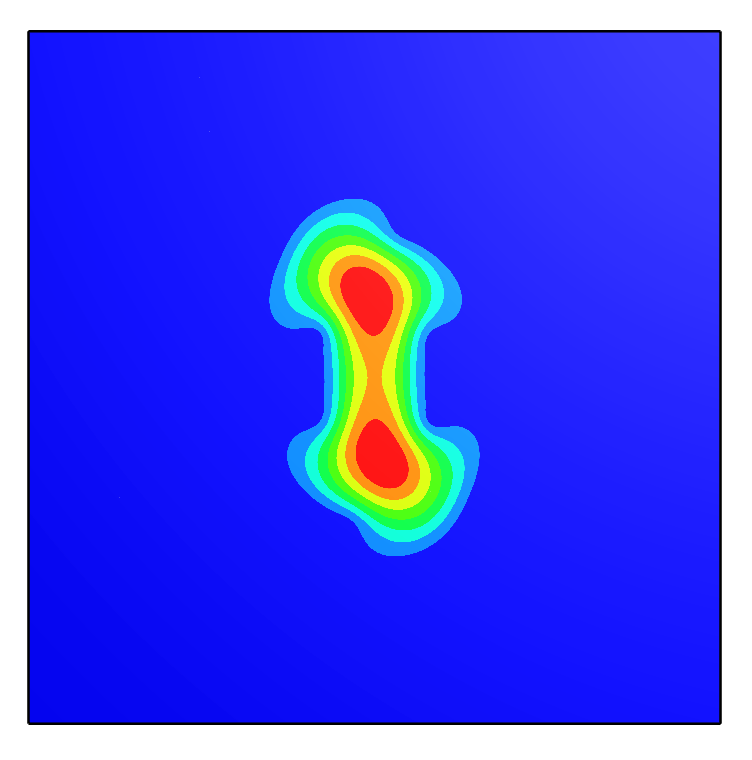}
  \includegraphics[width=0.25\textwidth]{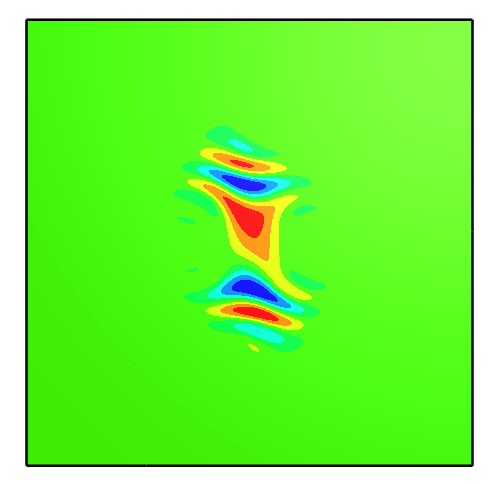}
  \includegraphics[width=0.24\textwidth]{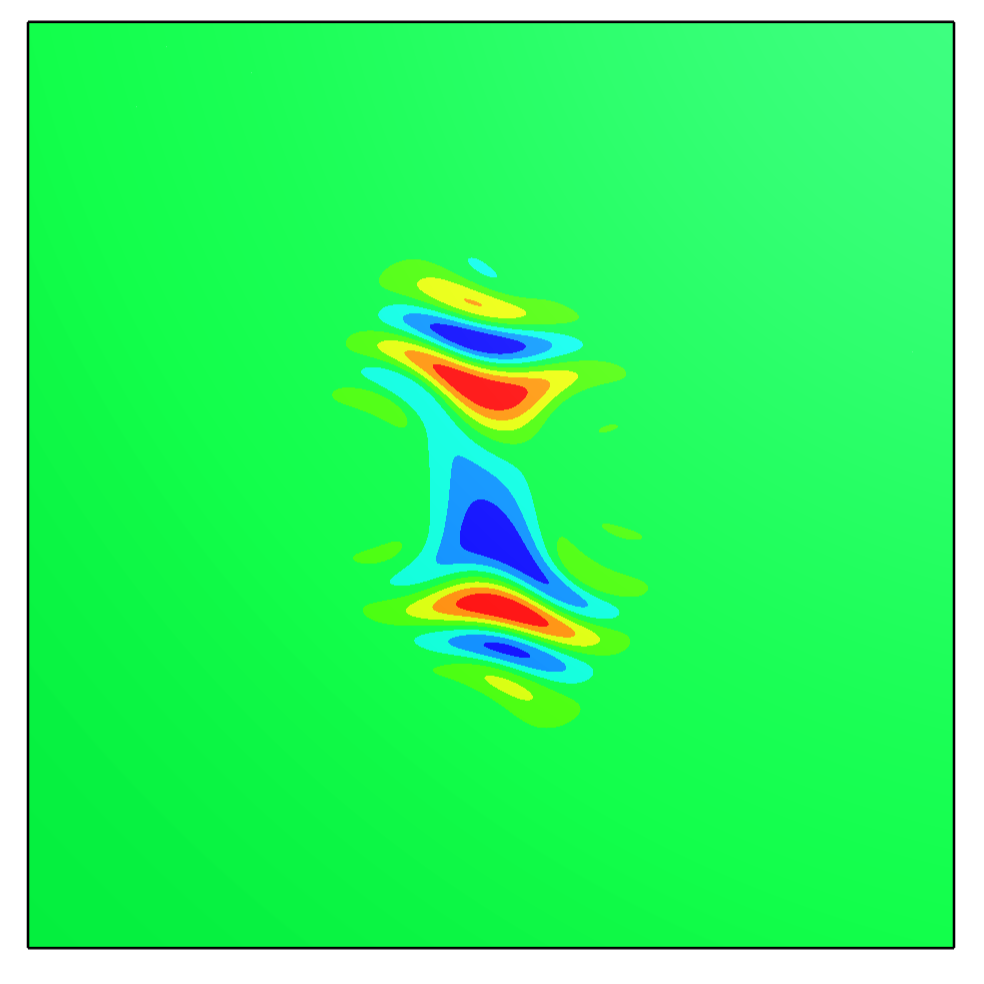}
  \includegraphics[width=0.24\textwidth]{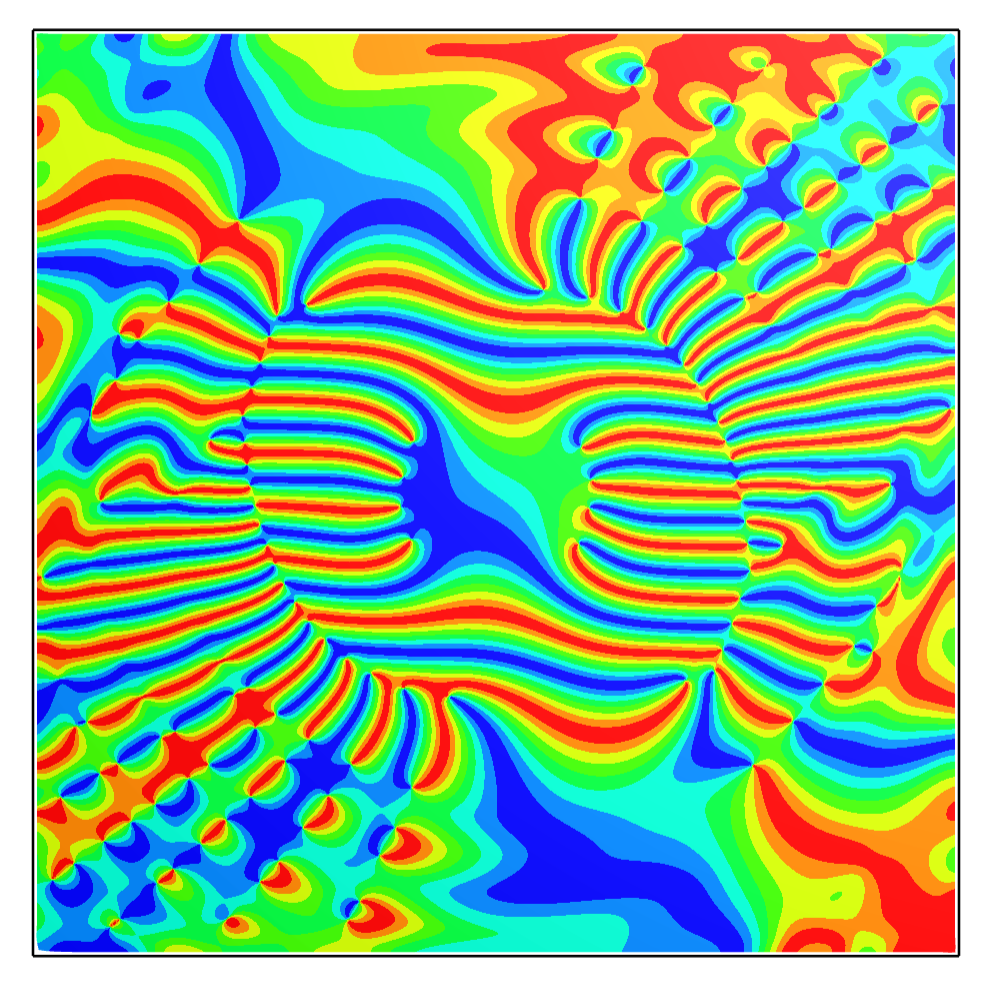}
  \caption{The first eigenvector $\psi_1$ of $H(\mb{A})$ (top row), and
    the  first eigenvector $\phi_1$ of $H(\mb{F})$, for
    Example 2. From left to right in the top row, we see
  $|\psi_1|$, $\Re\psi_1$, $\Im\psi_1$ and the phase of $\psi_1$; the bottom
  row is analogous for $\phi_1$.}\label{Ex2FirstEigenvector}
\end{figure}

 \begin{figure}
   \centering
   \begin{tabular}{cccc}
     \includegraphics[width=0.23\textwidth]{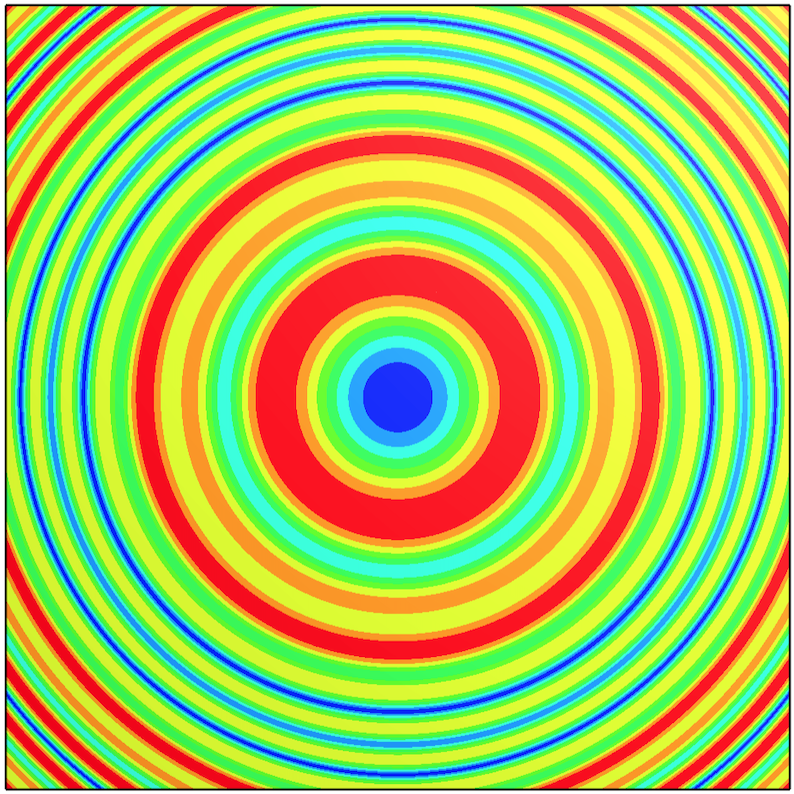}&
     \includegraphics[width=0.23\textwidth]{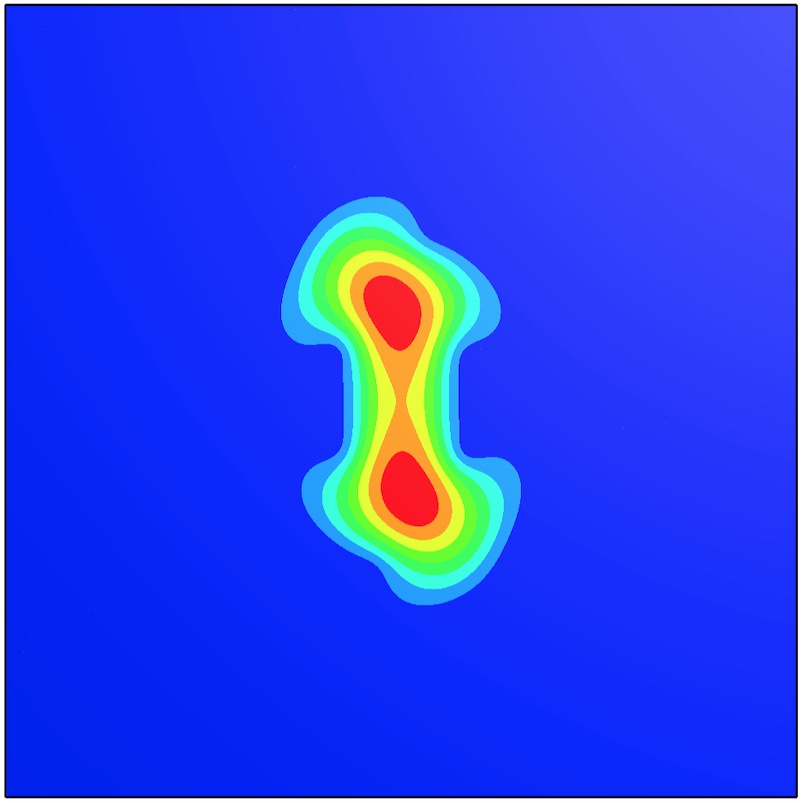}&
     \includegraphics[width=0.23\textwidth]{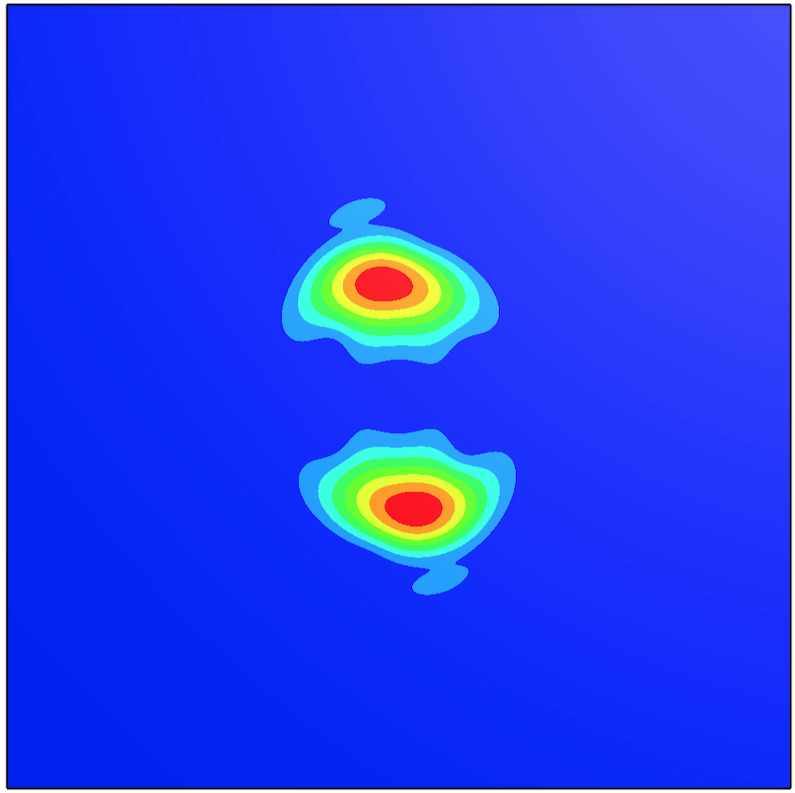}&
    \includegraphics[width=0.23\textwidth]{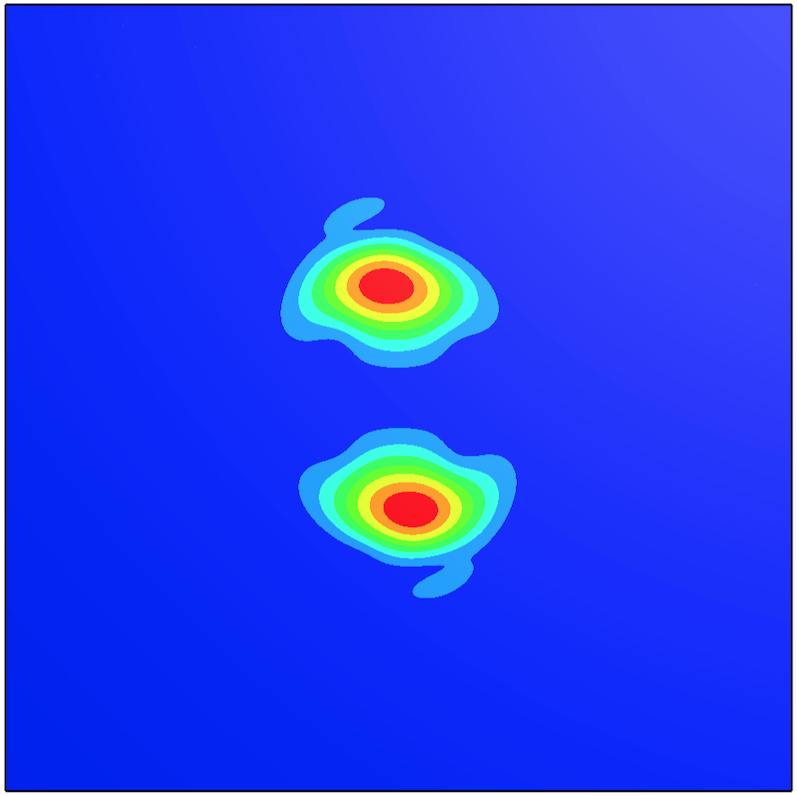}\\
     \includegraphics[width=0.23\textwidth]{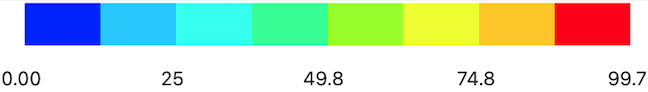}&
      \includegraphics[width=0.23\textwidth]{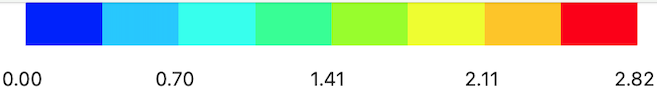}&
      \includegraphics[width=0.23\textwidth]{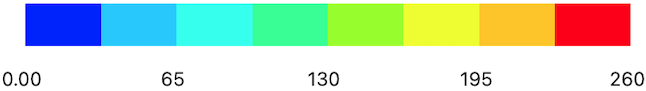}&
     \includegraphics[width=0.23\textwidth]{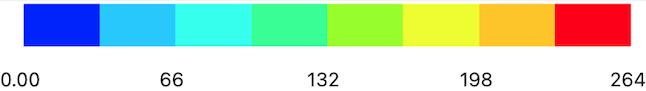}\\
     \includegraphics[width=0.23\textwidth]{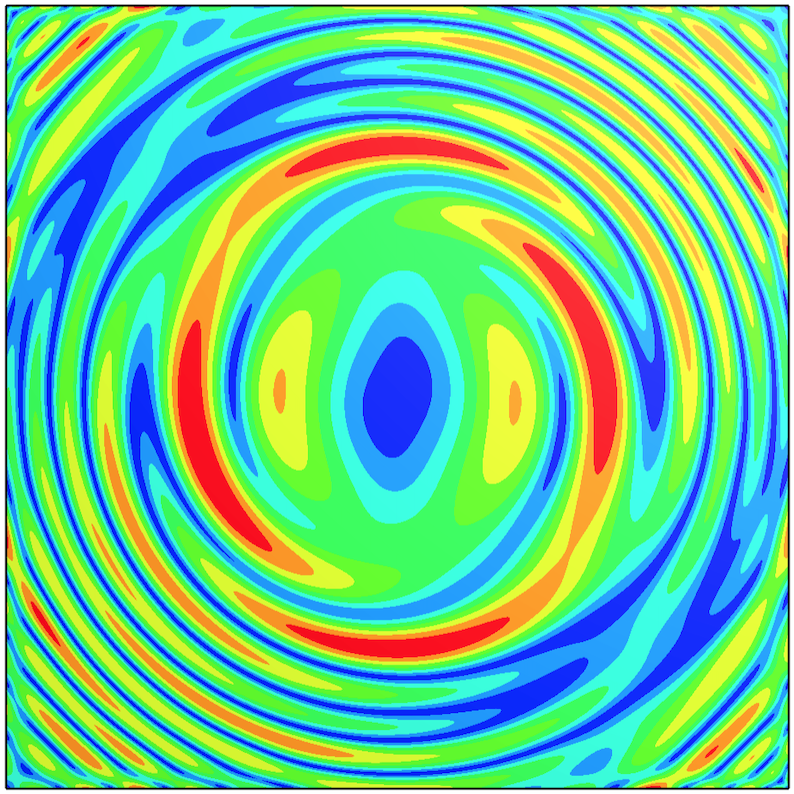}&
     \includegraphics[width=0.23\textwidth]{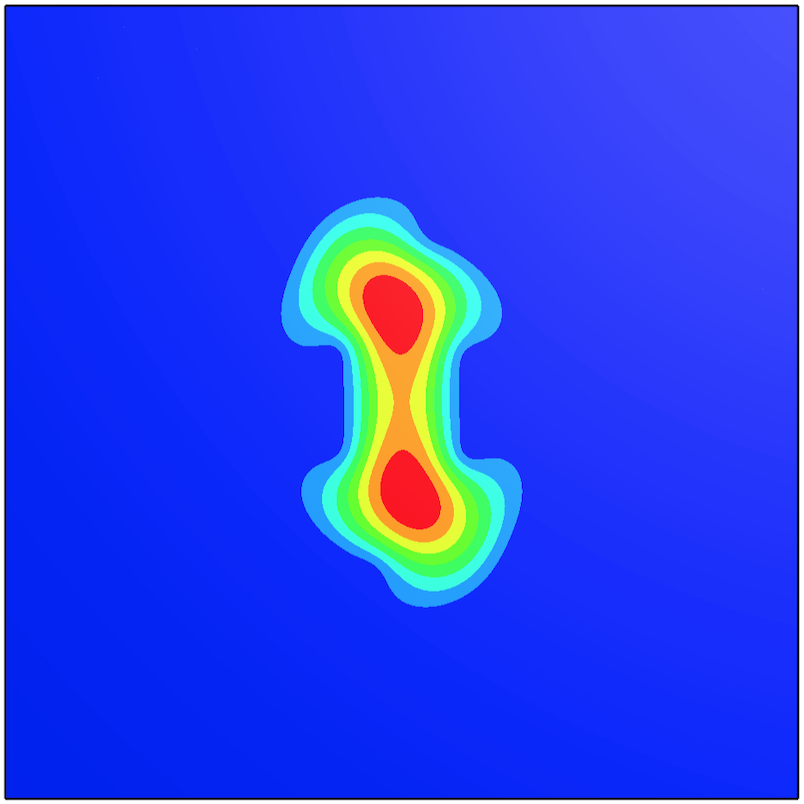}&
     \includegraphics[width=0.23\textwidth]{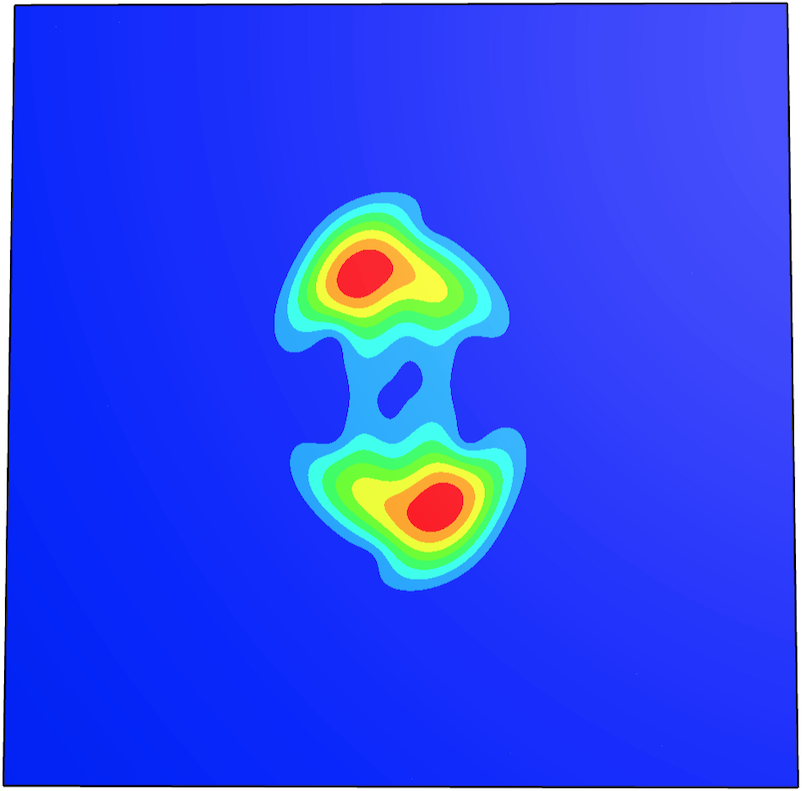}&
     \includegraphics[width=0.23\textwidth]{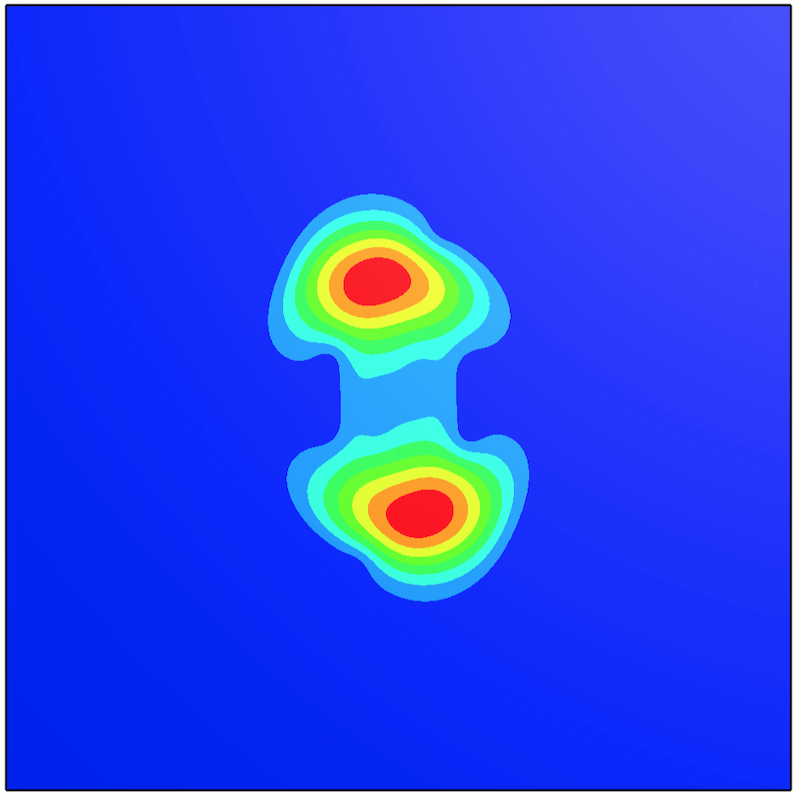}\\
     \includegraphics[width=0.23\textwidth]{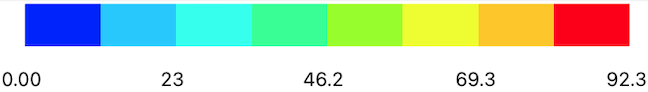}&
     \includegraphics[width=0.23\textwidth]{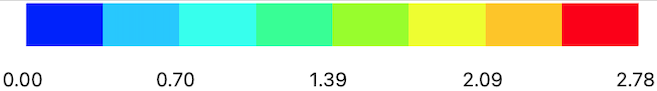}&
     \includegraphics[width=0.23\textwidth]{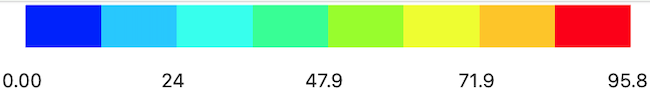}&
     \includegraphics[width=0.23\textwidth]{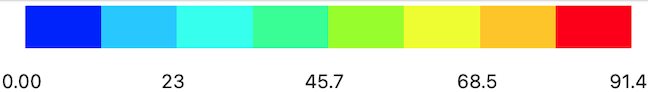}  
   \end{tabular}
  \caption{From left to right in the top row, we see $|\mb{A}|$,
    $|\psi_1|$,$|\mb{A}\psi_1|$, and $|\nabla \psi_1|$; the bottom row
    is analogous for $\mb{F}$ and $\phi_1$.}\label{Ex2FirstEigenvectorB}
 \end{figure}

\section{Numerical Results}\label{Experiments}
For our numerical experiments, we use the finite element software
package NGSolve~\cite{Schoeberl2014,Schoeberl2021} for discretization,
solution of source problems, visualization and several other
computational aspects.  For our eigenvalue solver, we use the add-on package
pyeigfeast~\cite{Gopalakrishnan2017}, which is a Python implementation
of the FEAST algorithm~\cite{Polizzi2009,Tang2014}.  The finite
element spaces consist of globally continuous piecewise polynomials on
quasi-uniform triangulations/meshes of $\Omega$.  We fixed the local
polynomial degree at $p=3$, and varied the meshsize parameter $h$,
which here represents the longest allowed edge length of any triangle
in the mesh.

In each of the experiments, we consider the first six eigenpairs of
$H(\mb{A})$ and $H(\mb{F})$, or of $H(\mb{A},V)$ and $H(\mb{F},V)$,
and we approximate them when $h=0.01$, $h=0.03$ and $h=0.05$.  For
Examples 1, 2a and 3 on $\Omega=(-1,1)\times (-1,1)$, the number of
degrees of freedom for the associated matrices are 417451, 46996, and
17017 for the three different mesh parameters. For Example 2b, which
requires that the meshes also conform to a cartesian grid partition of
$\Omega$, we have 417340, 36043, and 21025.  For the variants in
Example 3, the degrees of freedom in $\Omega_1$ are 393474,
44259, and 16221; in $\Omega_2$, they are 386399, 43286, and 15893;
and in $\Omega_3$, they are 372829, 41698, and 15268.
For the L-shaped domain in Example 4,
they are 517177, 57427, and 20845.
In each case, the computation of $a$ via~\eqref{EquivNeumann},
which is needed to determine $\mb{F}$, requires one additional degree
of freedom to fix the additive constant.

On the finest mesh, we illustrate~\change{\eqref{Heuristic2} and~\eqref{Heuristic3}}.  For each mesh,
we provide the computed eigenvalues and timing information for
essential components of the computation, i.e. those that do not
involve plotting solutions or computing of numbers such as
$\|\nabla\psi\|_{L^2(\Omega)}$ and $\|\mb{A}\psi\|_{L^2(\Omega)}$ that
are used in validation of~\change{\eqref{Heuristic2} and~\eqref{Heuristic3}}.  As FEAST is the most
computationally intensive part of the process, and one might employ a
different (cheaper) eigensolver, we provide both the total time and
the time spent on FEAST.  We note that computations involving $\mb{F}$
first require the solution of a Poisson problem, as indicated in
Theorem~\ref{CanonicalGauge}.  The cost of this solve is negligible in
comparison with that of the eigenvalue solver, and we will see that it
is typically well worth the investment.

\begin{figure}
  \centering
  \begin{tabular}{llll}
    \multicolumn{1}{c}{Example 1}&\multicolumn{1}{c}{Example 2}&\multicolumn{1}{c}{Example 3}&\multicolumn{1}{c}{Example 4}\\
    \includegraphics[width=0.23\textwidth]{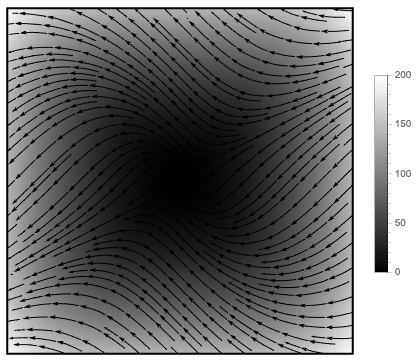}&
    \includegraphics[width=0.23\textwidth]{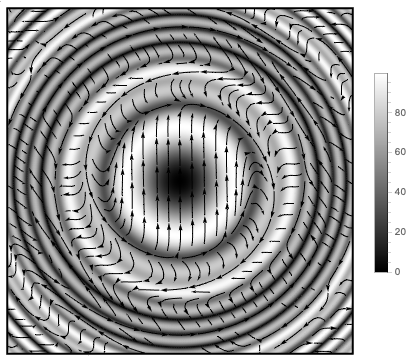}&
   \includegraphics[width=0.2025\textwidth]{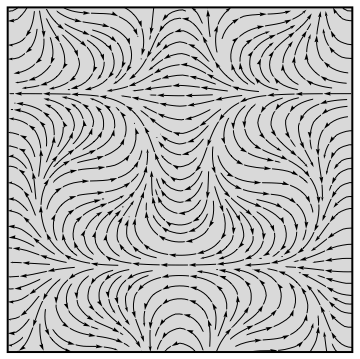}&
    \includegraphics[width=0.23\textwidth]{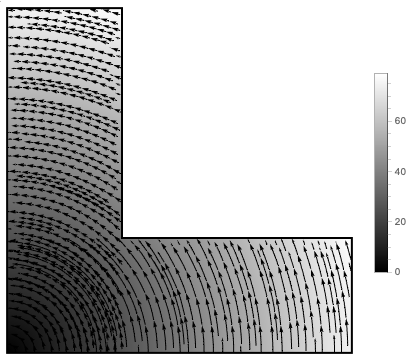}\\                          
    \includegraphics[width=0.23\textwidth]{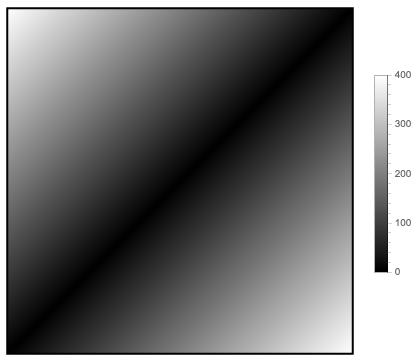}&
   \includegraphics[width=0.23\textwidth]{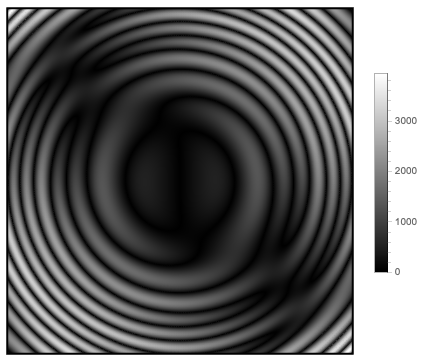}&
   \includegraphics[width=0.23\textwidth]{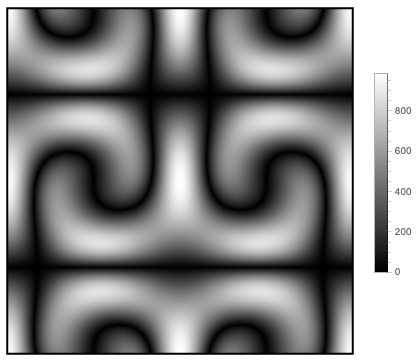}&
    \includegraphics[width=0.2025\textwidth]{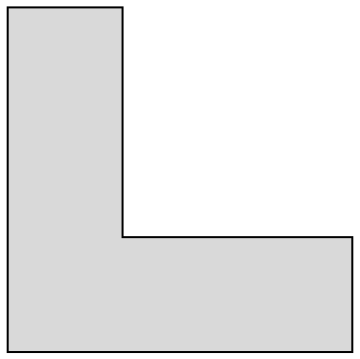}
  \end{tabular}
  \caption{Streamplot of $\mb{A}$ overlaid on plot of $\|\mb{A}\|$ (top) and
    plot of $|\curl\mb{A}|$ (bottom) for the
    vector fields used in the experiments.\label{VectorFields}}
\end{figure}

\subsection{Example 1}\label{Example1}
We take the vector field
\begin{align}\label{Ex1A}
  \bfA = -100(x^2+y^2, x^2-y^2)~,
\end{align}
on the domain $\Omega=(-1,1)\times(-1,1)$, see
Figure~\ref{VectorFields}.  We impose homogeneous Dirichlet boundary
conditions for the eigenvalue problems.
We have $\curl\mb{A}=200(y-x)$, which
suggests that low-lying eigenvectors will be concentrated near the
line $y=x$, and this is what we observe.   For this problem, we have
\begin{align*}
  \|\mb{A}\|_{L^2(\Omega)}=178.8854\quad,\quad \|\mb{F}\|_{L^2(\Omega)}=70.3584
\end{align*}
Table~\ref{Ex1Table}
provides the computed eigenvalues and timing information on each of
three meshes, as well as the validation of~\change{\eqref{Heuristic2} and~\eqref{Heuristic3}}. 

The computations on the finest mesh, $h=0.01$, provide a baseline for
comparison.  We see very clear validation of~\change{\eqref{Heuristic2} and~\eqref{Heuristic3}} in Table~\ref{Ex1Table}, with
one exception---$\|\nabla \phi_1\|_{L^2(\Omega)}$ is slightly larger
than $\|\nabla \psi_1\|_{L^2(\Omega)}$.  Looking at the real parts of
$\psi_1$ and $\phi_1$ in Figure~\ref{Ex1Fig1} indicates why this might
be the case.  Although $\psi_1$ is much more oscillatory than $\phi_1$
away from the origin along the central channel, it varies slowly near
the origin, whereas $\psi_1$ is more uniformly oscillatory along the
channel.  Similar statements can be made for the other eigenvectors,
but $\|\nabla \phi_j\|_{L^2(\Omega)}<\|\nabla \psi_j\|_{L^2(\Omega)}$
for those, as typically expected.

The strong oscillation of $\psi_j$ nearer the corners of the channel,
in contrast to the mild oscillation of $\phi_j$, suggests that the
latter can be more readily resolved in coarser spaces, and this is
what we observe both in Table~\ref{Ex1Table} and Figure~\ref{Ex1Fig2}.
More specifically, we see greater variation in the computed
eigenvalues of $H(\mb{A})$ than of $H(\mb{F})$ in
Table~\ref{Ex1Table}.  Perhaps even more striking in this case is the
significant drop in computational cost for $H(\mb{F})$ in comparison
to $H(\mb{A})$.  Not only do we see a clear deterioration of the
computed eigenvalues of $H(\mb{A})$ as $h$ increases, we also observe
non-trivial differences in the computed eigenvectors---in some cases,
eigenvectors that belong higher in the spectrum are being (not very
well) approximated instead of those that should appear among the first
six!  We highlight that the quality of the computed eigenpairs of
$H(\mb{F})$ when $h=0.03$ is at least as good as that of the computed
eigenpairs of $H(\mb{A})$ when $h=0.01$, but the corresponding cost is
significantly smaller---13.5 seconds versus 422.43 seconds!  Although
we do not show the corresponding plots for $H(\mb{F})$, we state that
they remain essentially the same over each level of discretization,
and look like those seen for $H(\mb{A})$ on the finest mesh.
\change{We also highlight how little the computed eigenvalues change for
  $H(\mb{F})$ over
  this range of mesh sizes, in comparison with those of
  $H(\mb{A})$.  These relatively small changes in the computed
  eigenvalues and eigenvectors over a range of discretizations are what me mean
when we assert that the computations are ``more stable'' when using $\mb{F}$ than when using $\mb{A}$.  The differences
are generally even more stark for subsequent examples.}
\begin{table}
\caption{Computed eigenvalues and timing information, and validation
  of ~\eqref{Heuristic2} when $h=0.01$, for Example 1.\label{Ex1Table}}
\begin{center}
\begin{tabular}{|c|cll|cccccc|}\hline
&$h$ & Total & FEAST& $\lambda_1$ &  $\lambda_2$  & $\lambda_3$ &  $\lambda_4$ & $\lambda_5$ &  $\lambda_6$\\\hline
\multirow{3}{*}{$H(\bfA)$} &0.01  &422.43s& 375.25s    &25.8482&   29.6950 &35.9656&   44.4081 &54.5735&   65.7054\\
&0.03   &67.81s&  64.64s   &26.3270&   31.5307 &39.8076&   50.5411 &62.7682&   72.4820\\
&0.05  &12.78s&11.85s&27.0502&   34.3334 &45.7499&   60.0641 &73.0750&   75.2048\\
\hline
\multirow{3}{*}{$H(\bfF)$} & 0.01  &491.92s& 413.76s    &25.8453&   29.6843 &35.9438&   44.3743&54.5290&   65.6540\\
& 0.03   &13.6s&9.99s&25.8557&   29.6875 &35.9519&   44.3811 &54.5375&   65.6622\\
& 0.05  &4.62s&3.69s&26.0736&   29.7487 &36.1300&   44.5101 &54.7189&   65.8232\\
\hline
\end{tabular}

\vspace*{2mm}
 \begin{tabular}{ |c|cccccc|}
   \hline
   &$j=1$&$j=2$&$j=3$&$j=4$&$j=5$&$j=6$\\\hline
   $\|\nabla \psi_j\|_{L^2(\Omega)}$  & 33.9056& 55.1878 & 64.0270  & 67.2759 & 68.5995 & 68.7150 \\  
   $\|\mb{A} \psi_j\|_{L^2(\Omega)}$  & 33.7639 & 55.0632 & 63.8676  & 67.0595 & 68.3161 & 68.3653 \\  
   \hline
  $\|\nabla \phi_j\|_{L^2(\Omega)}$  & 34.1399 & 26.6008 & 27.4577  & 27.5653 & 27.6055 & 27.5628 \\  
   $\|\mb{F} \phi_j\|_{L^2(\Omega)}$  &  33.9989 & 26.3400 & 27.0811  & 27.0285 & 26.8875 & 26.6720 \\\hline
 \end{tabular}
\end{center}
\end{table}

\begin{figure}
\includegraphics[width=.16\textwidth]{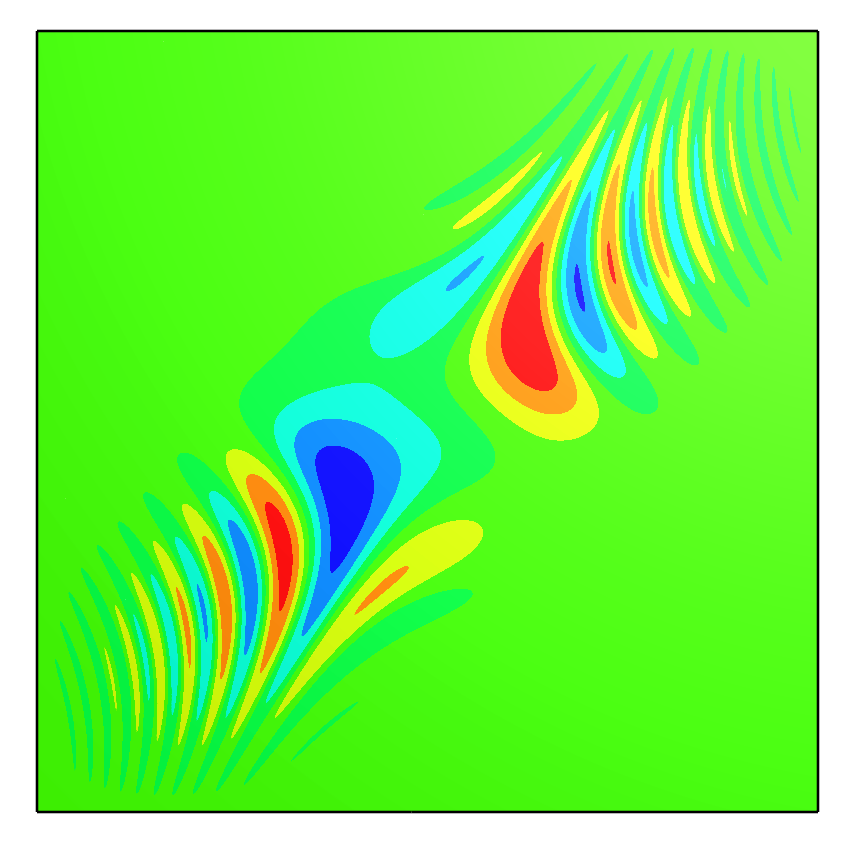}
\includegraphics[width=.16\textwidth]{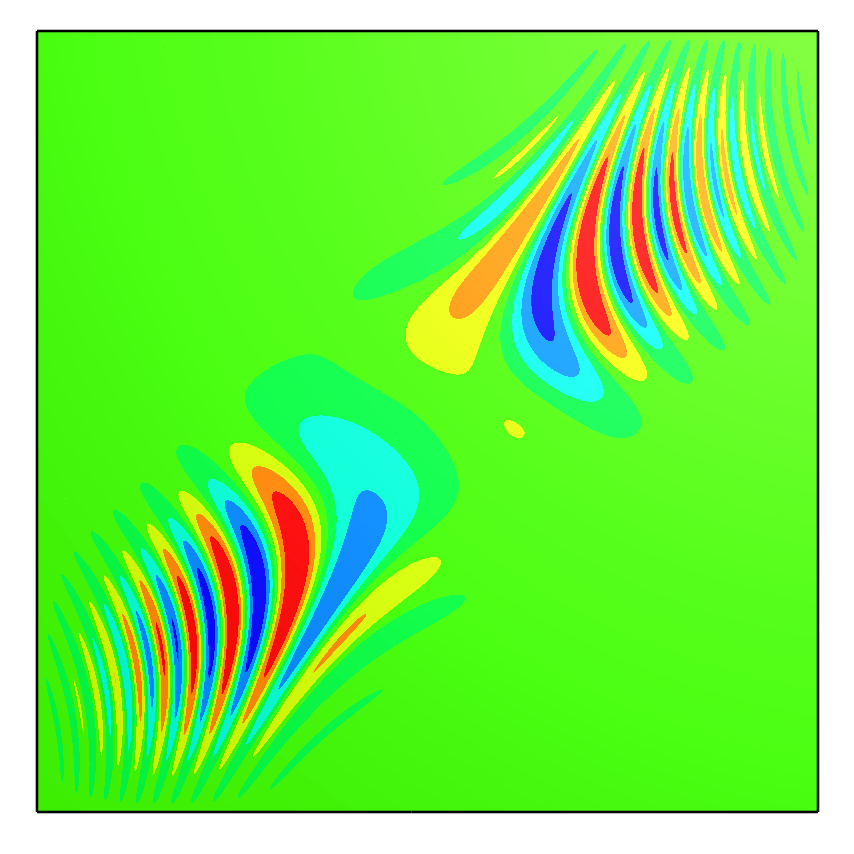}
\includegraphics[width=.16\textwidth]{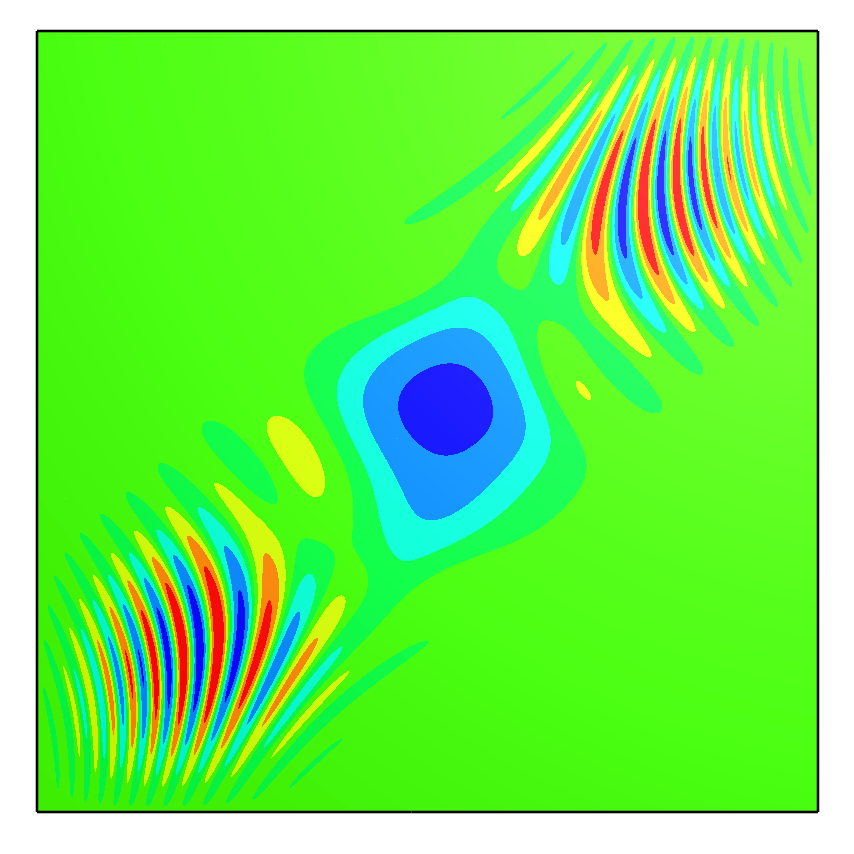}
\includegraphics[width=.16\textwidth]{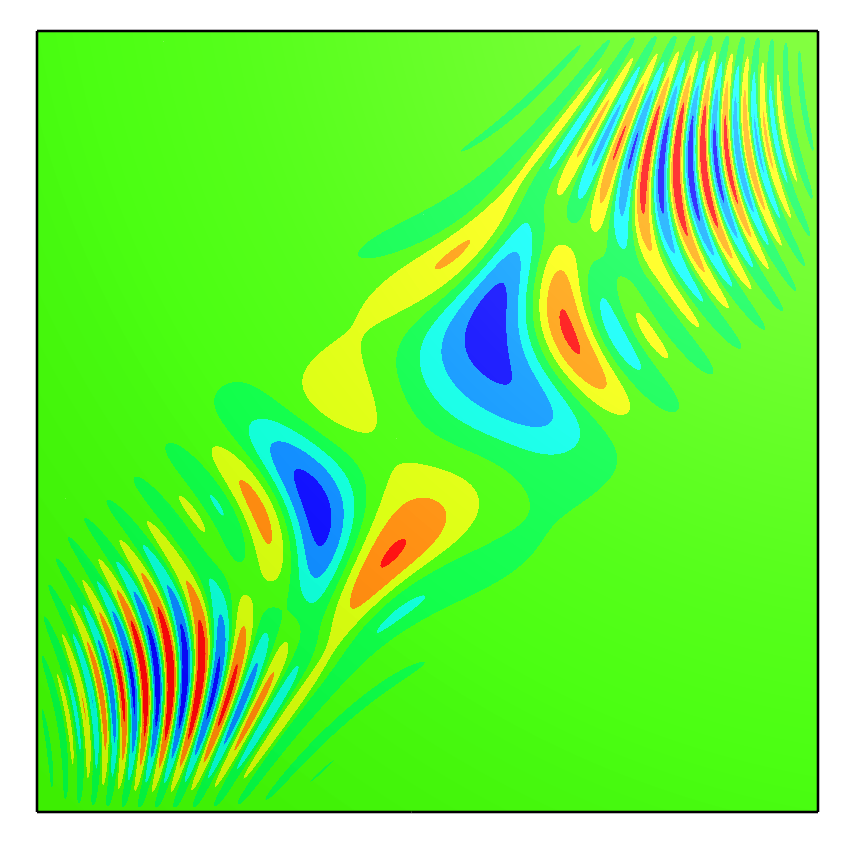}
\includegraphics[width=.16\textwidth]{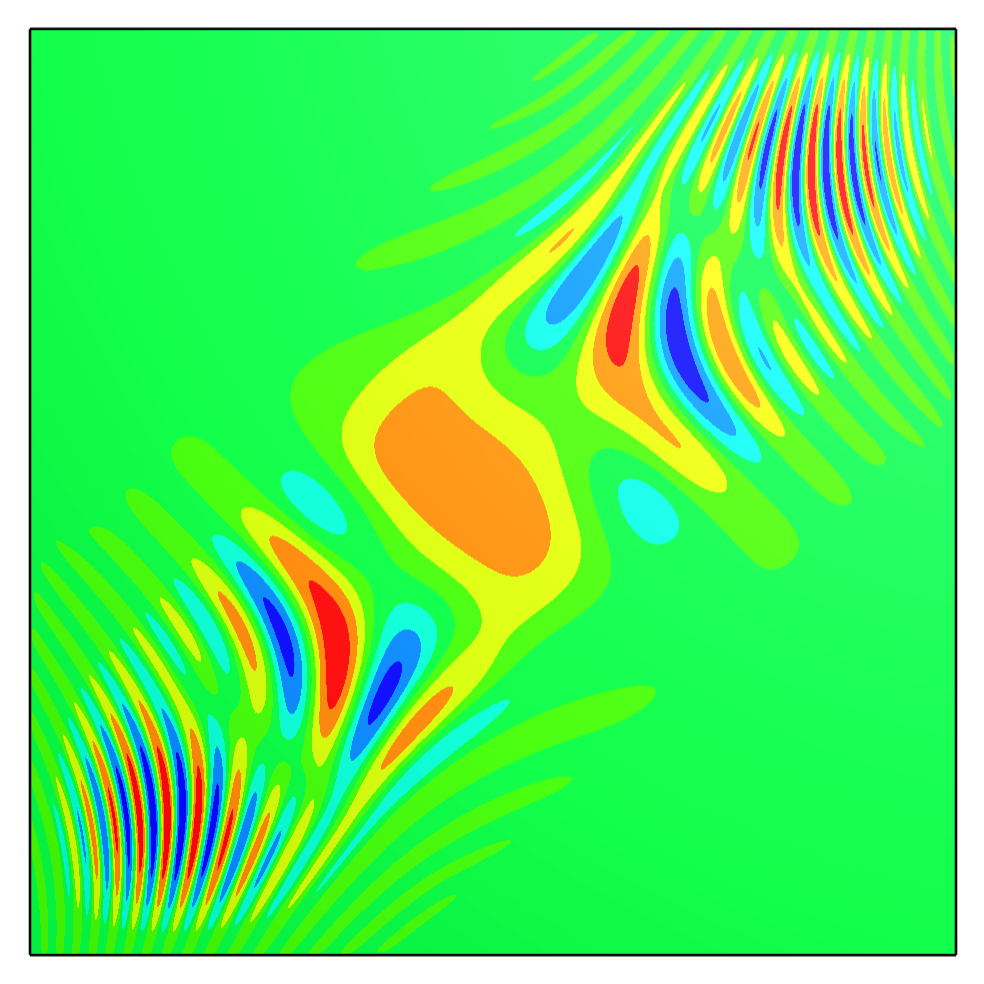}
\includegraphics[width=.16\textwidth]{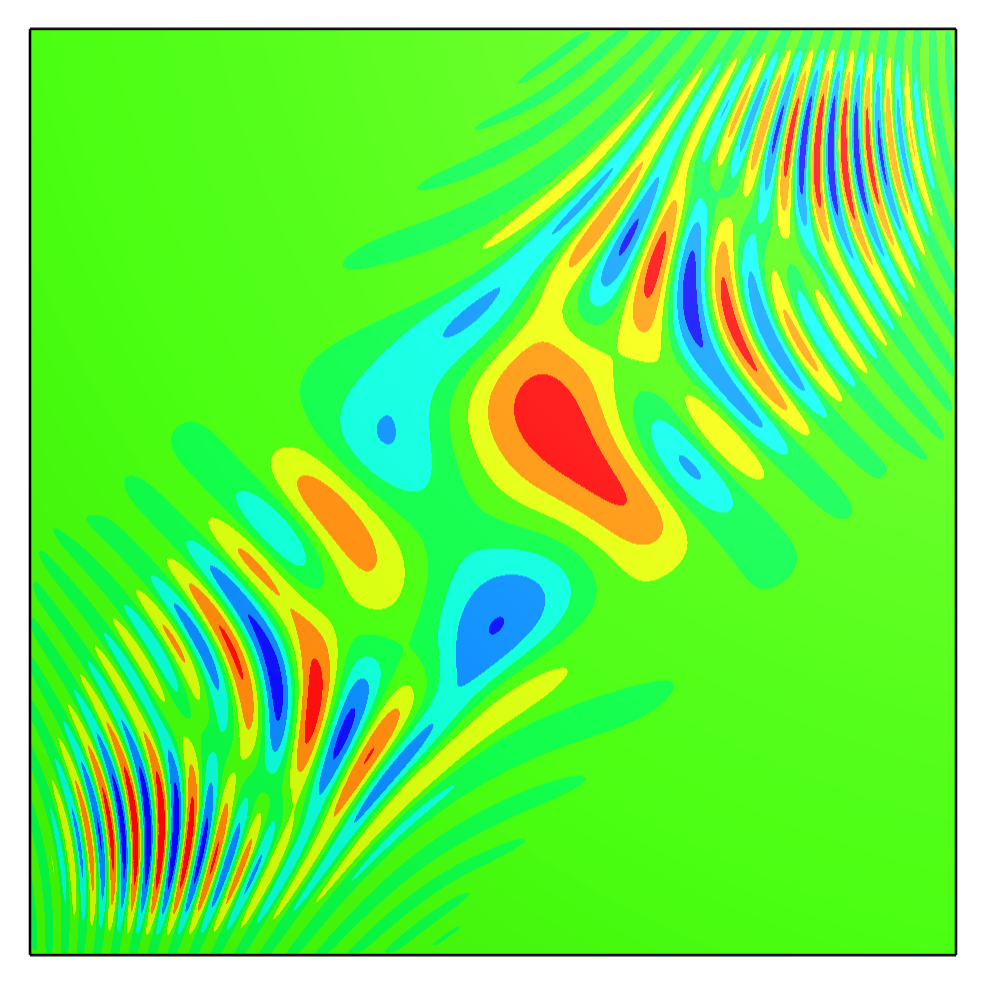}
\includegraphics[width=.16\textwidth]{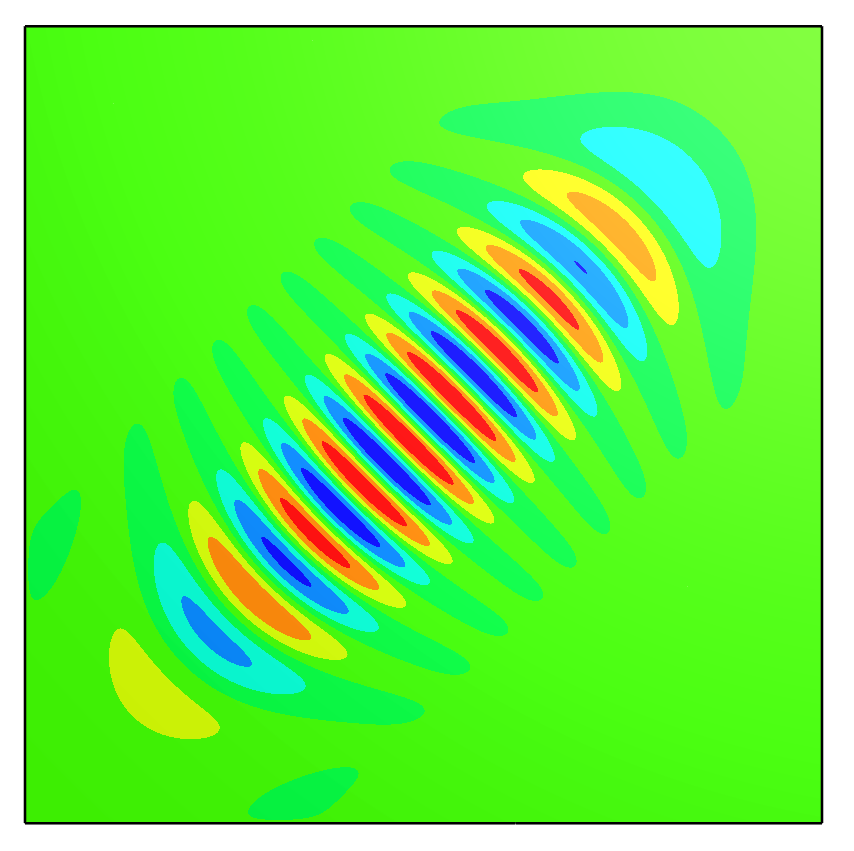}
\includegraphics[width=.16\textwidth]{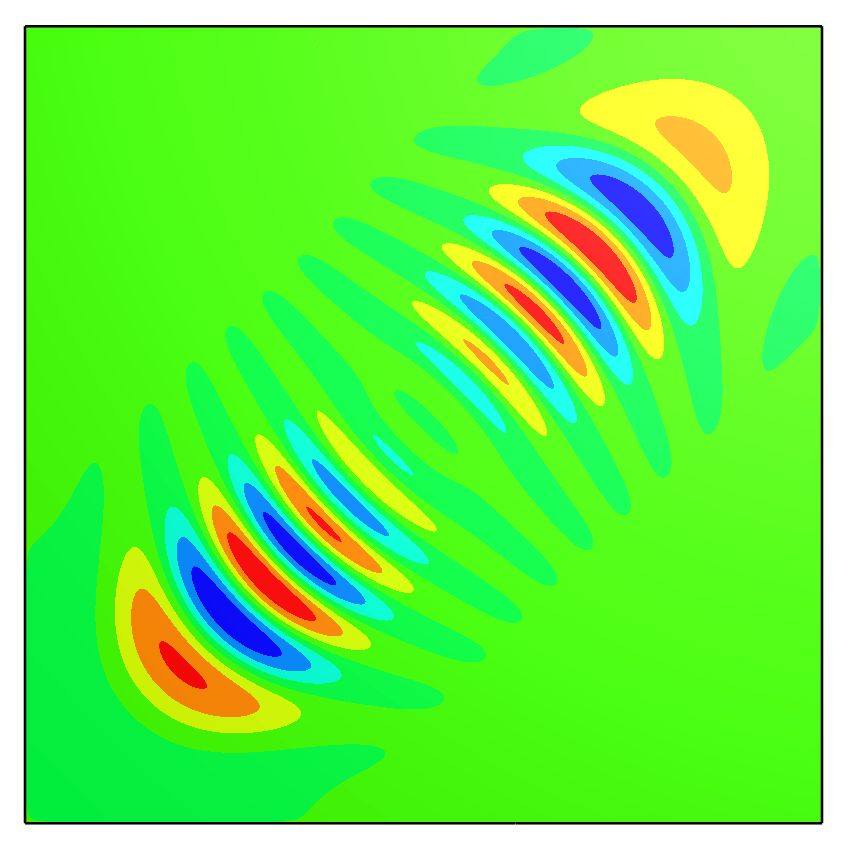}
\includegraphics[width=.16\textwidth]{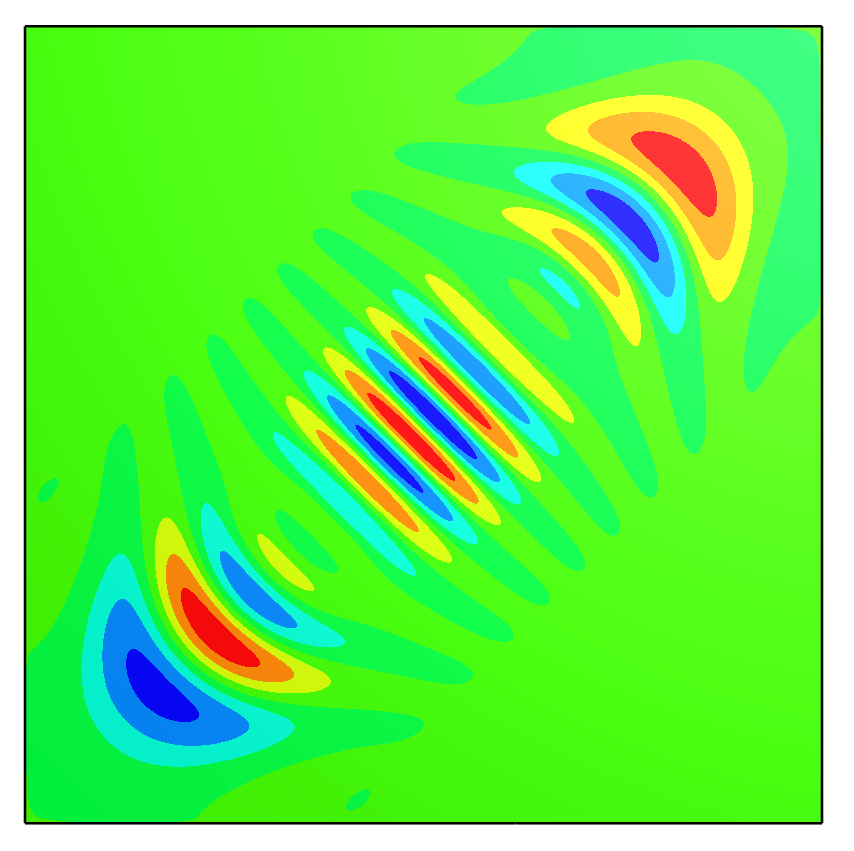}
\includegraphics[width=.16\textwidth]{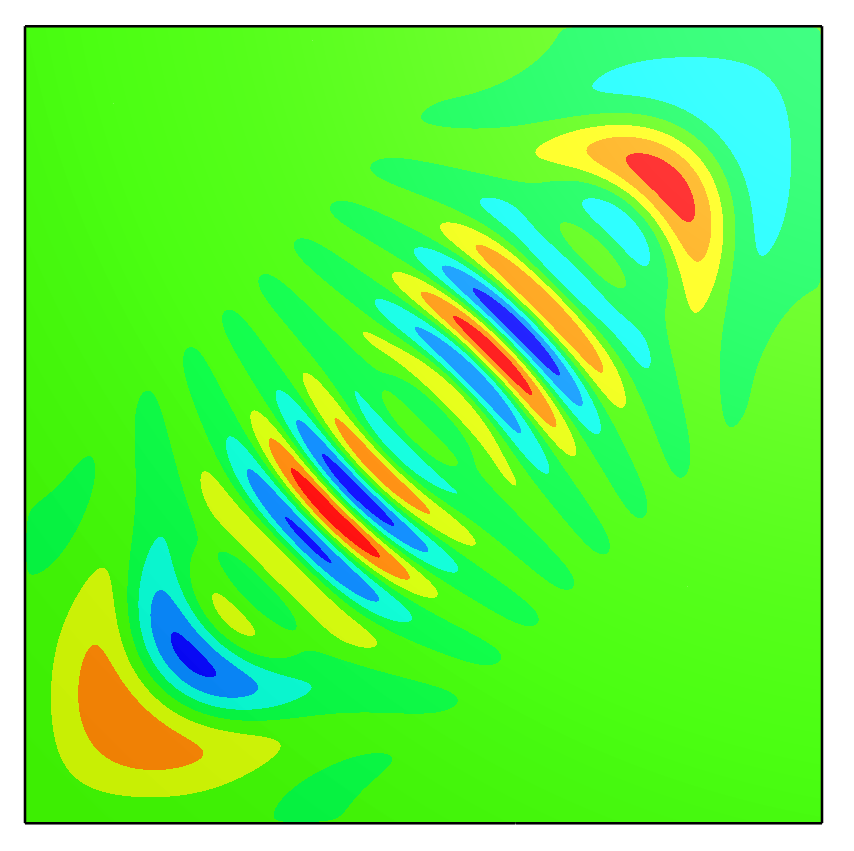}
\includegraphics[width=.16\textwidth]{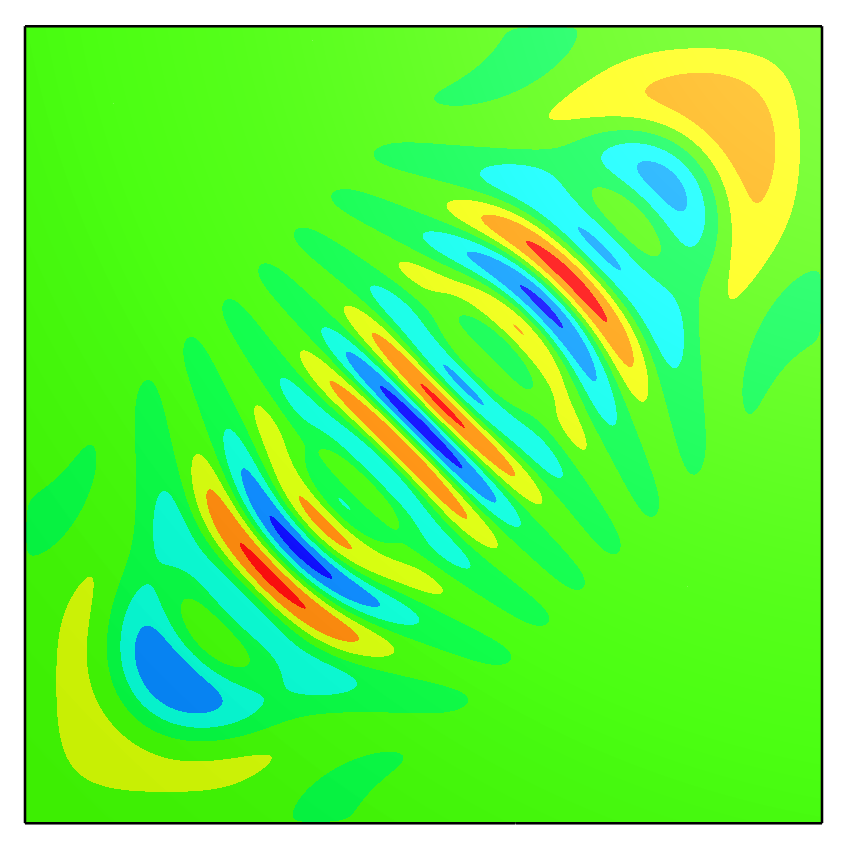}
\includegraphics[width=.16\textwidth]{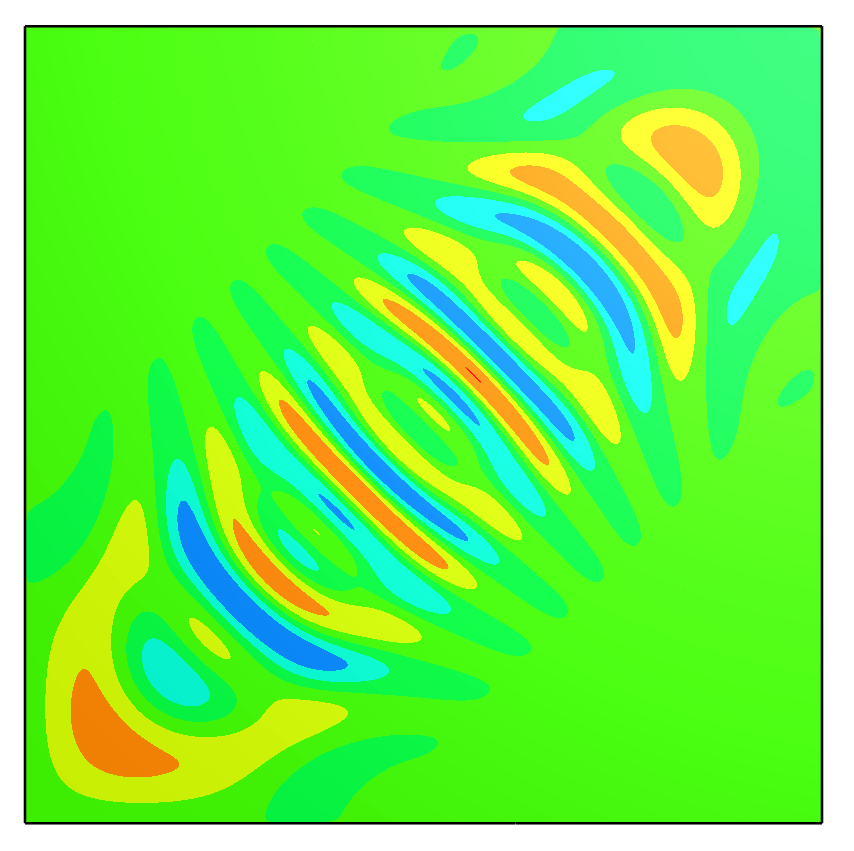}
\caption{$\Re(\psi_j)$ for $H(\bfA)$ (top row) and $\Re(\phi_j)$ for
  $H(\bfF)$ (bottom row) when $h=0.01$, $1\leq j\leq 6$, for
  Example 1.}
\label{Ex1Fig1}
\end{figure}

\begin{figure}
\includegraphics[width=.16\textwidth]{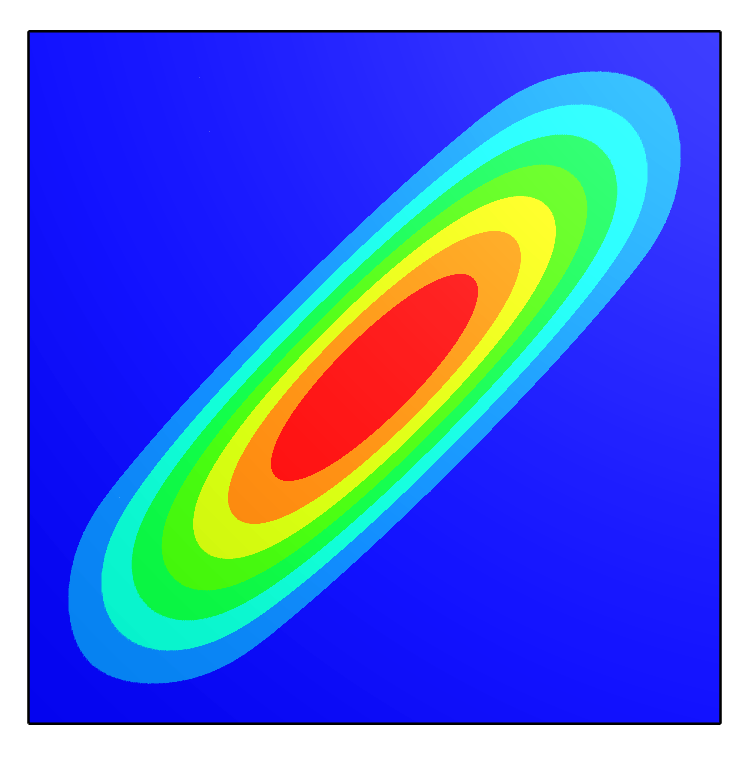}
\includegraphics[width=.16\textwidth]{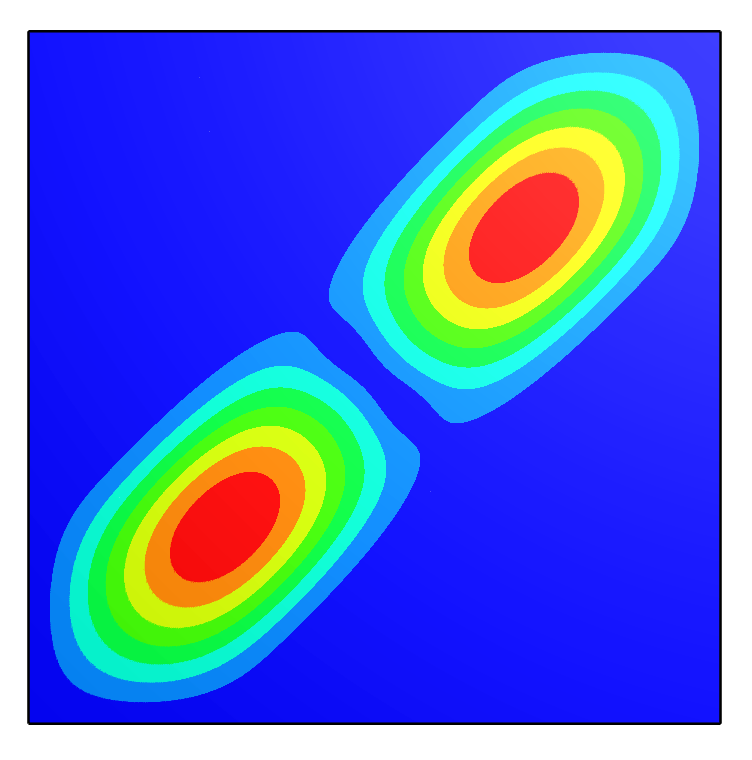}
\includegraphics[width=.16\textwidth]{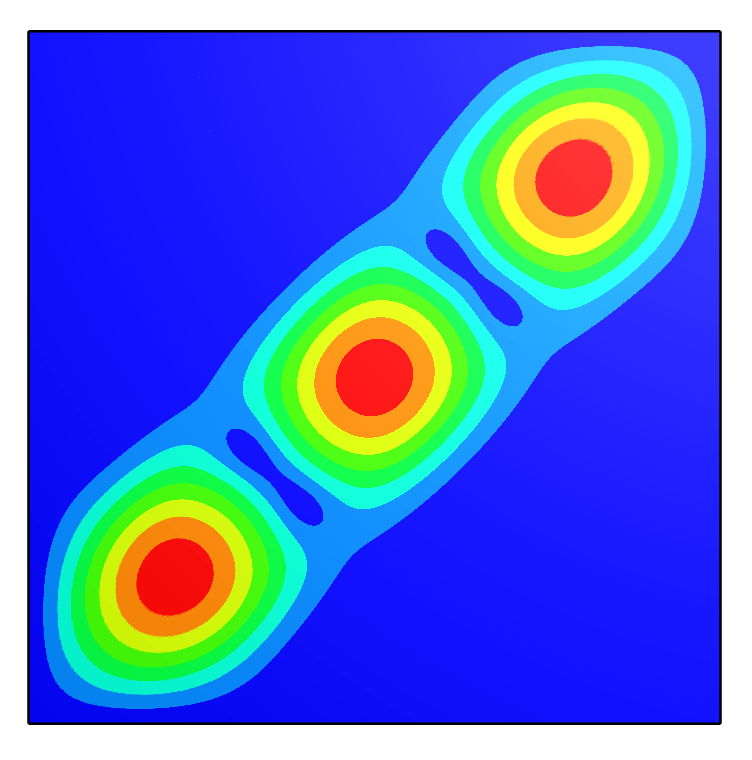}
\includegraphics[width=.16\textwidth]{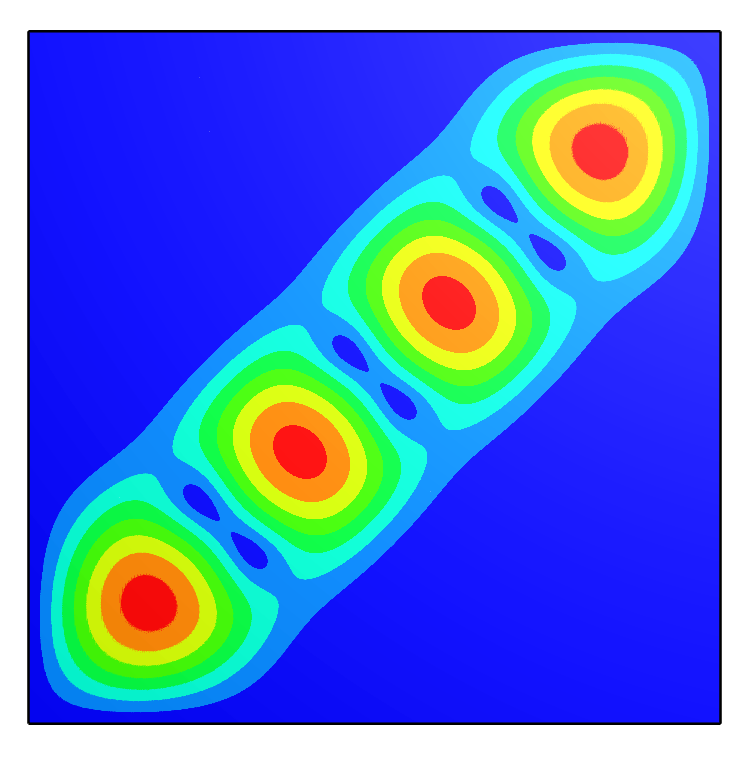}
\includegraphics[width=.16\textwidth]{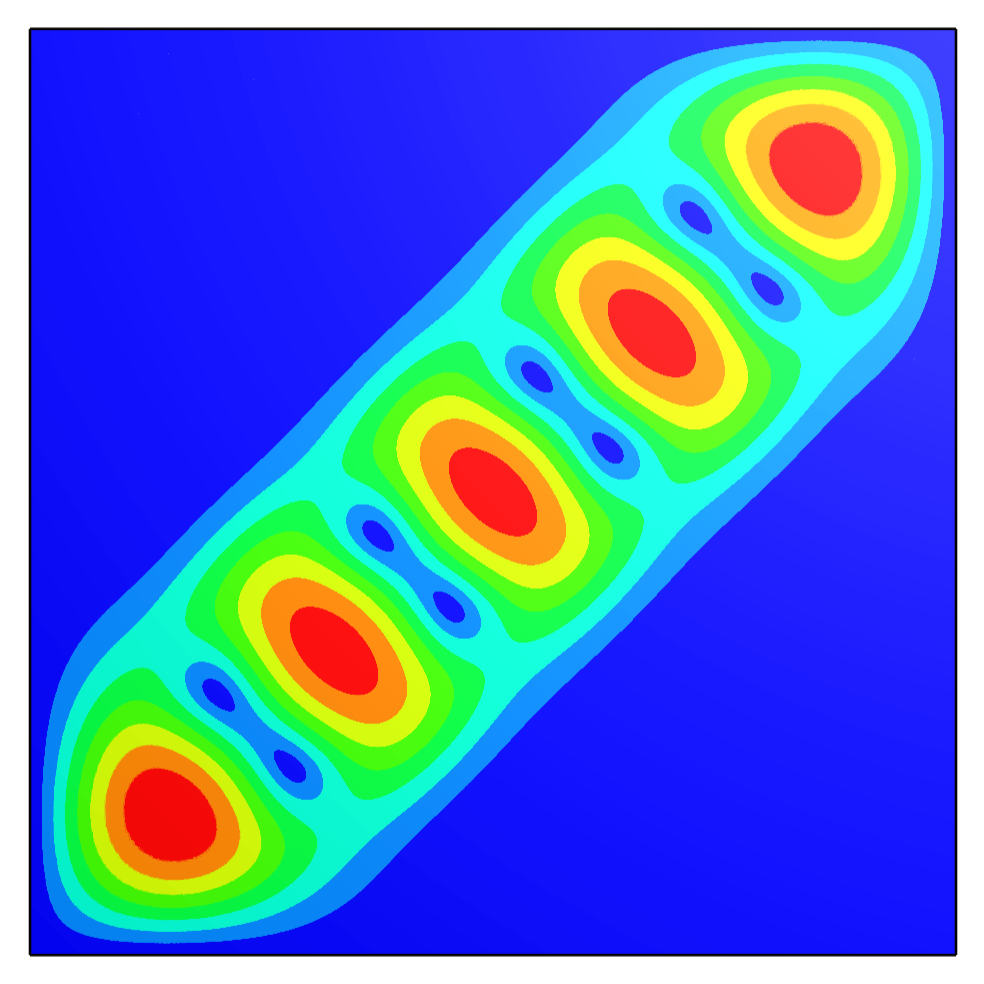}
\includegraphics[width=.16\textwidth]{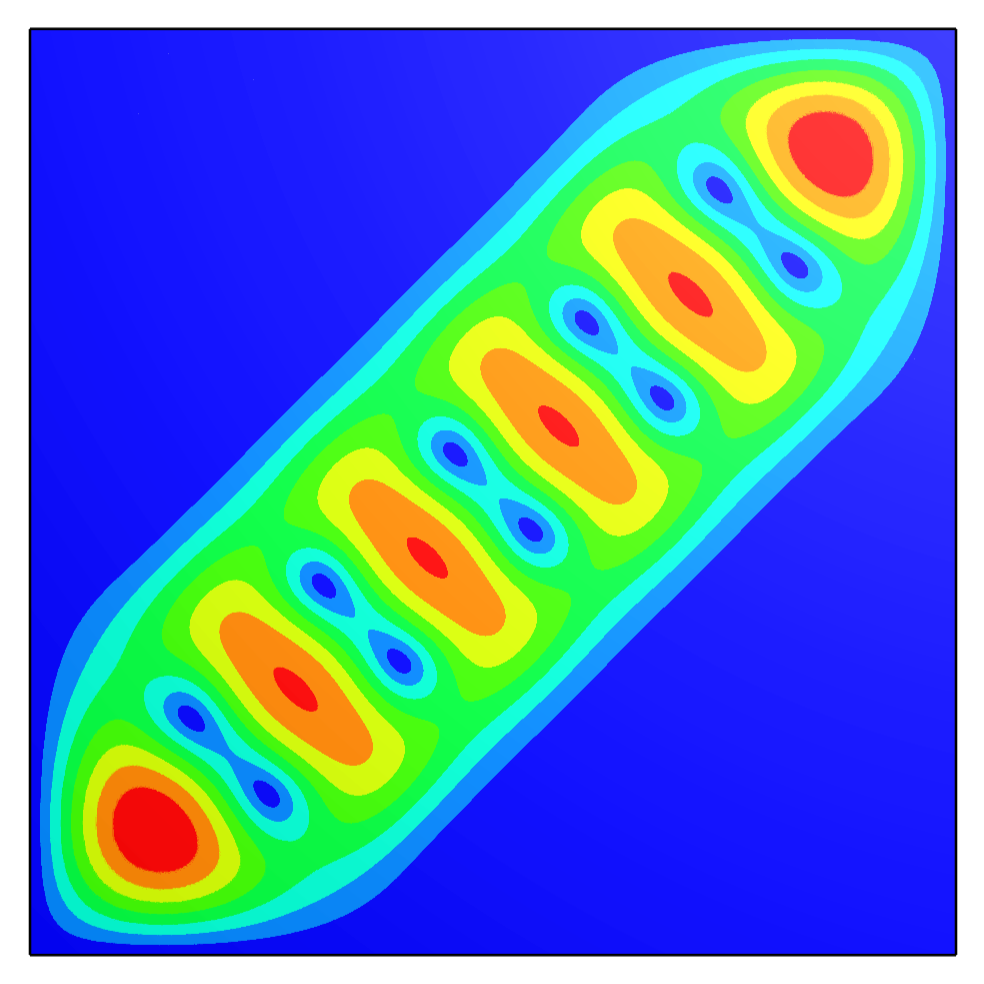}
\includegraphics[width=.16\textwidth]{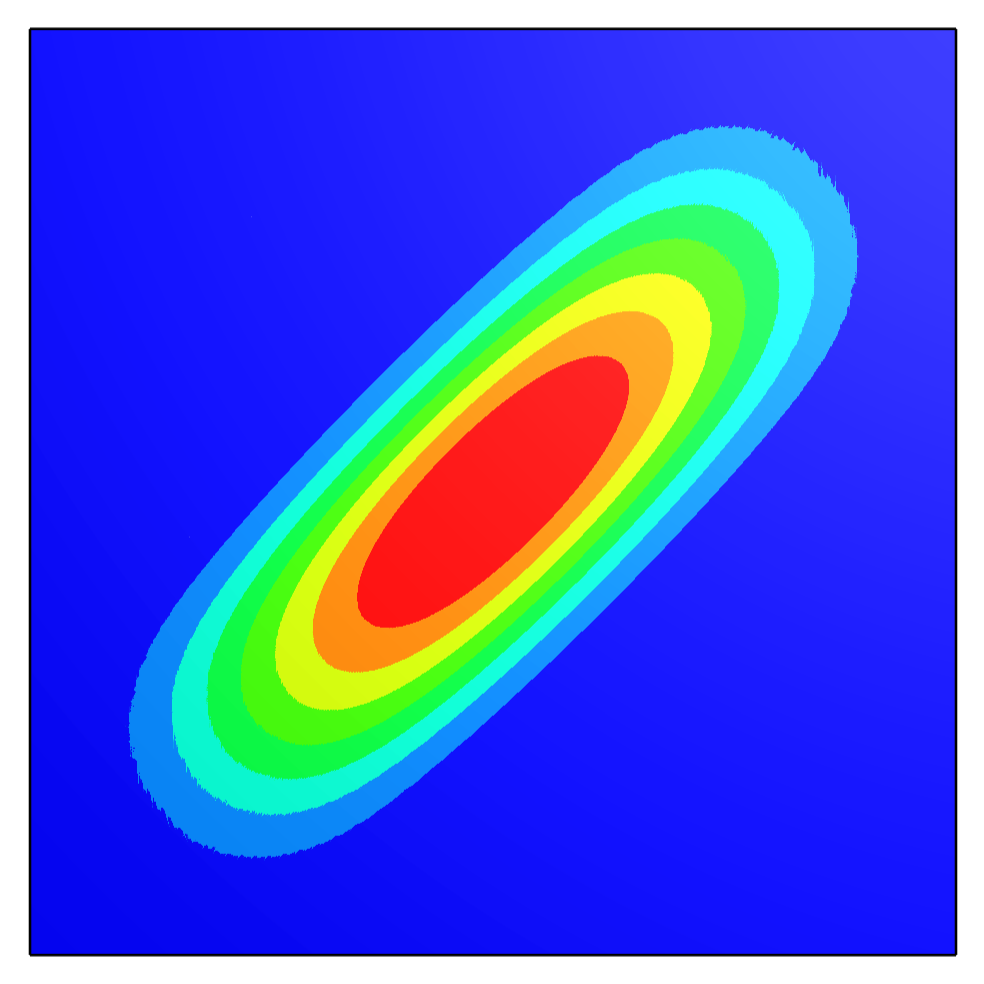}
\includegraphics[width=.16\textwidth]{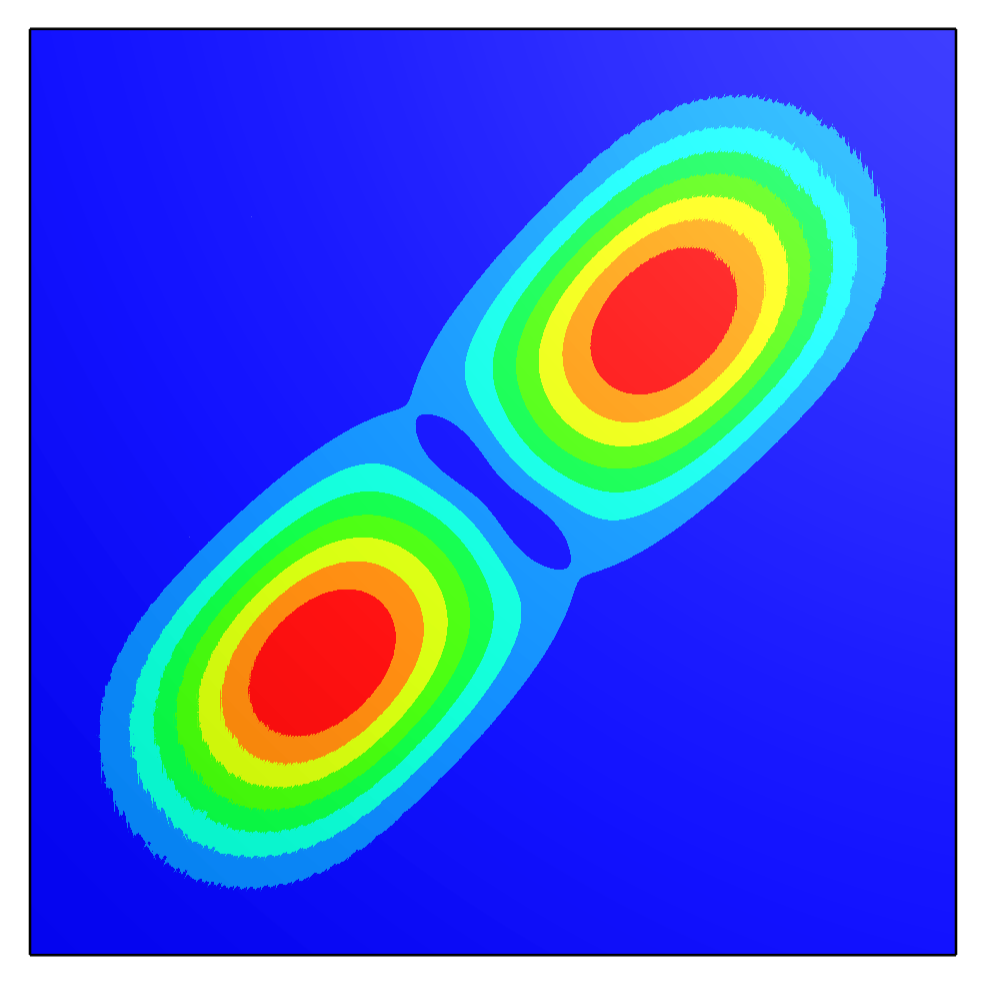}
\includegraphics[width=.16\textwidth]{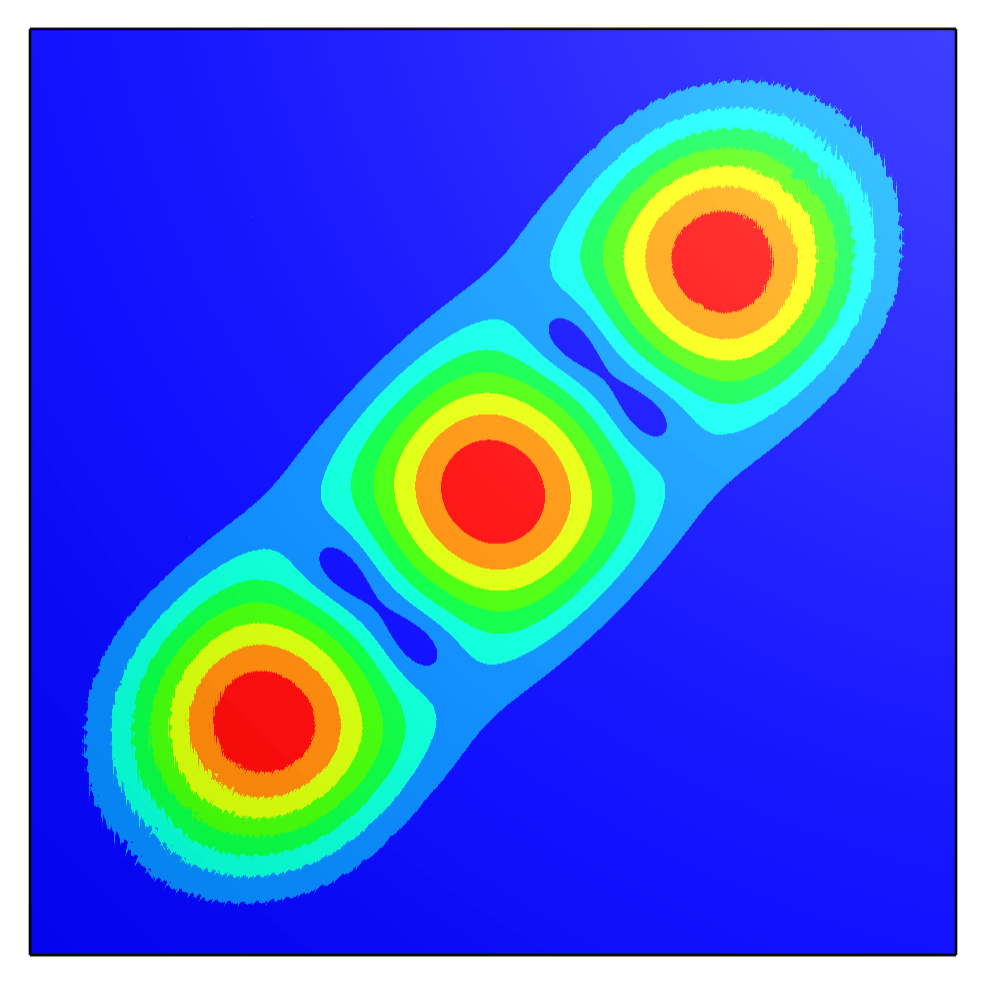}
\includegraphics[width=.16\textwidth]{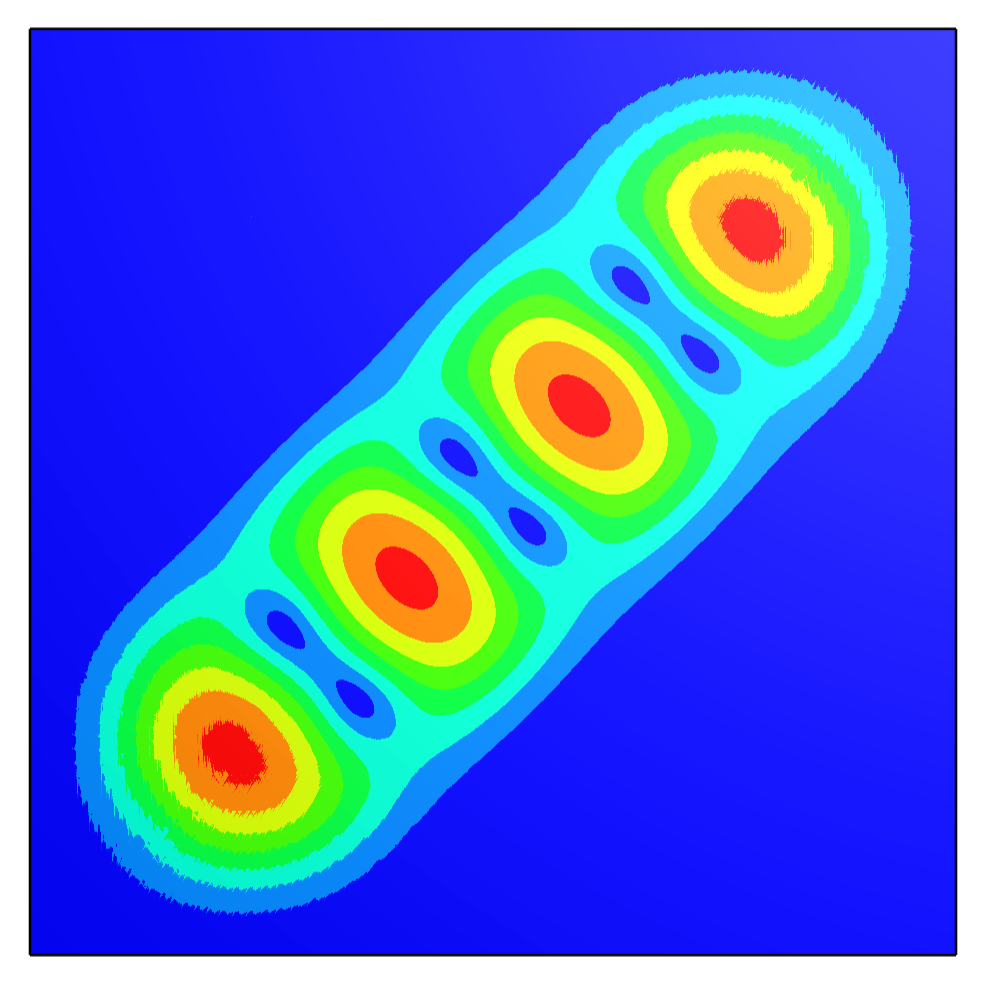}
\includegraphics[width=.16\textwidth]{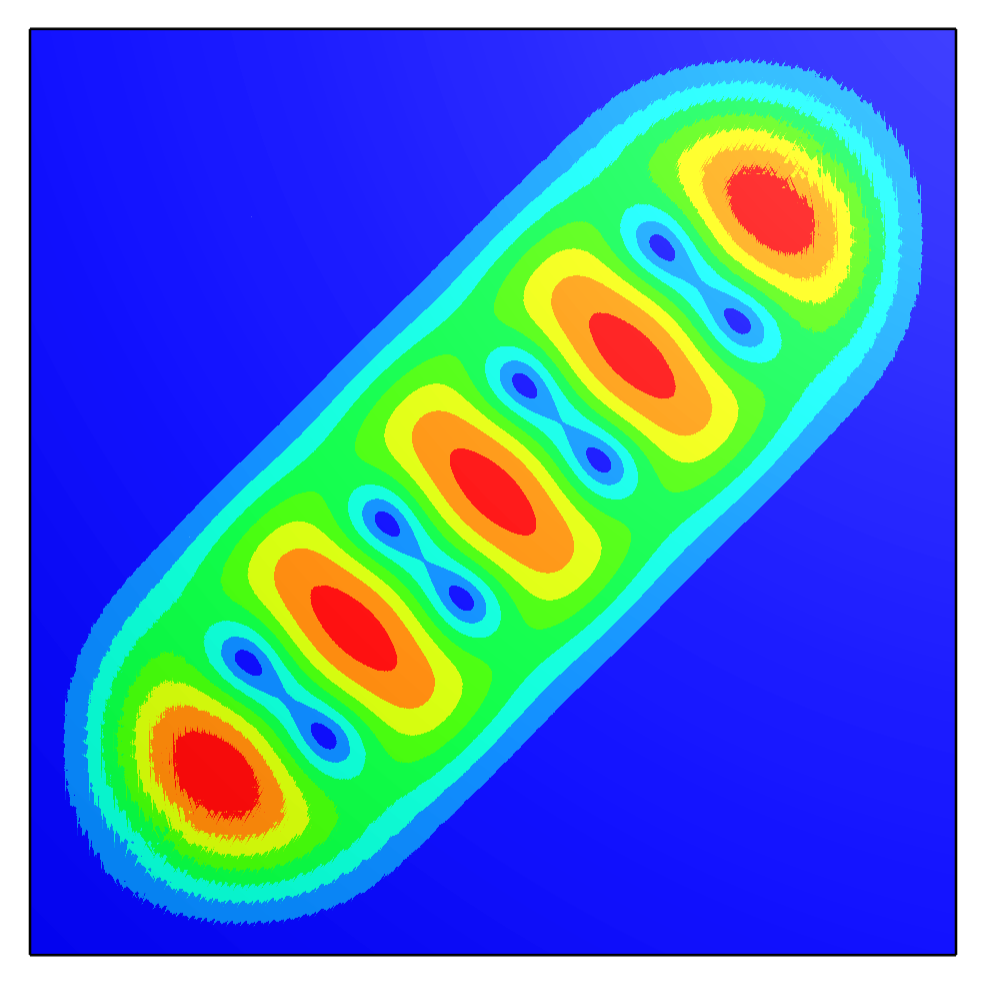}
\includegraphics[width=.16\textwidth]{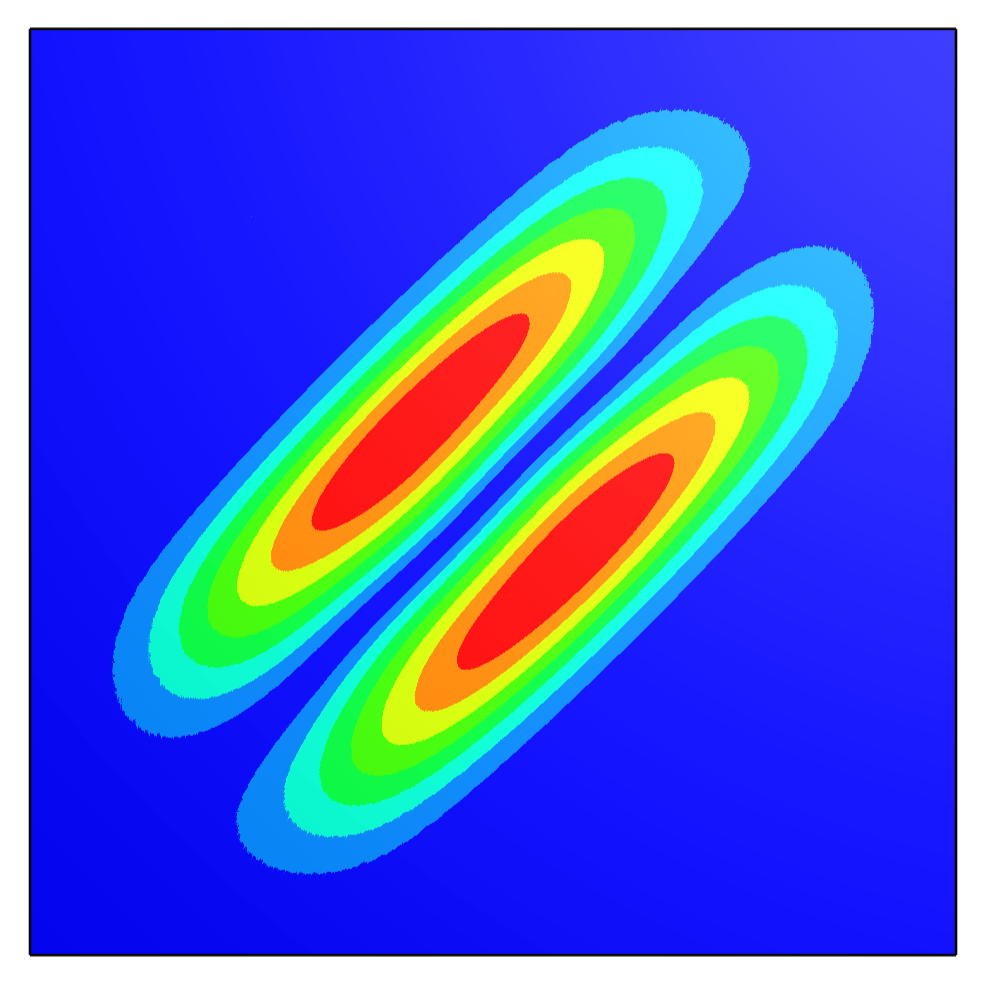}
\includegraphics[width=.16\textwidth]{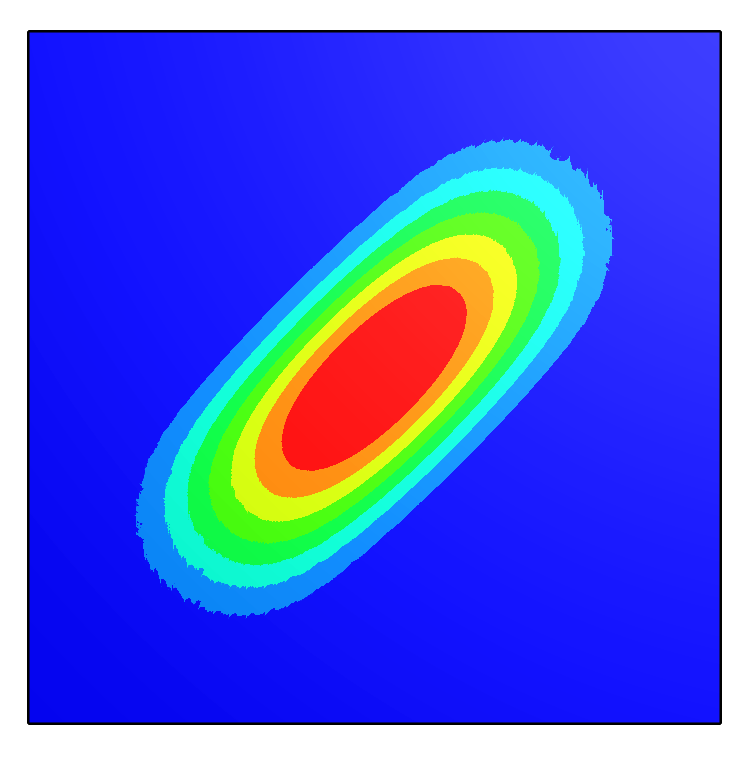}
\includegraphics[width=.16\textwidth]{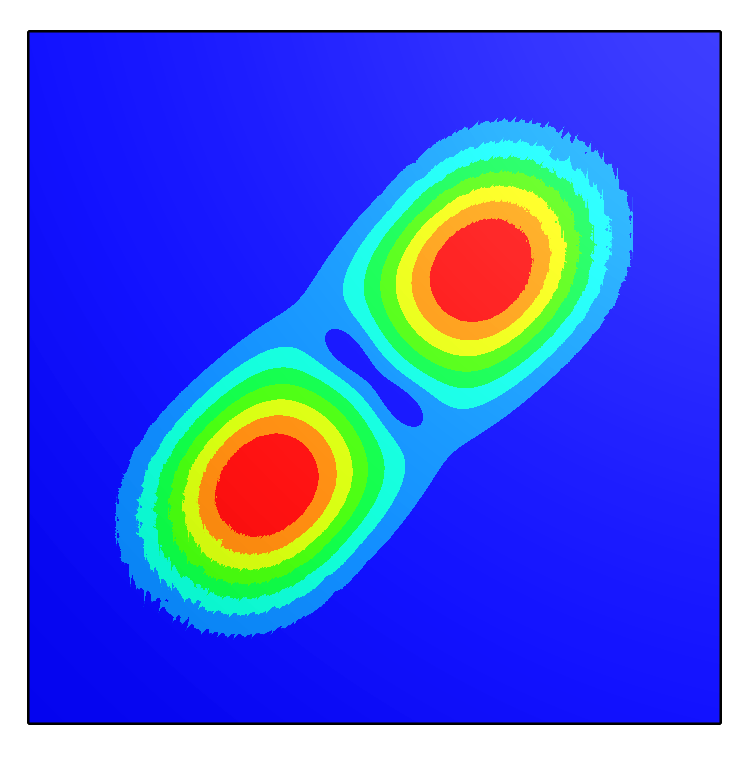}
\includegraphics[width=.16\textwidth]{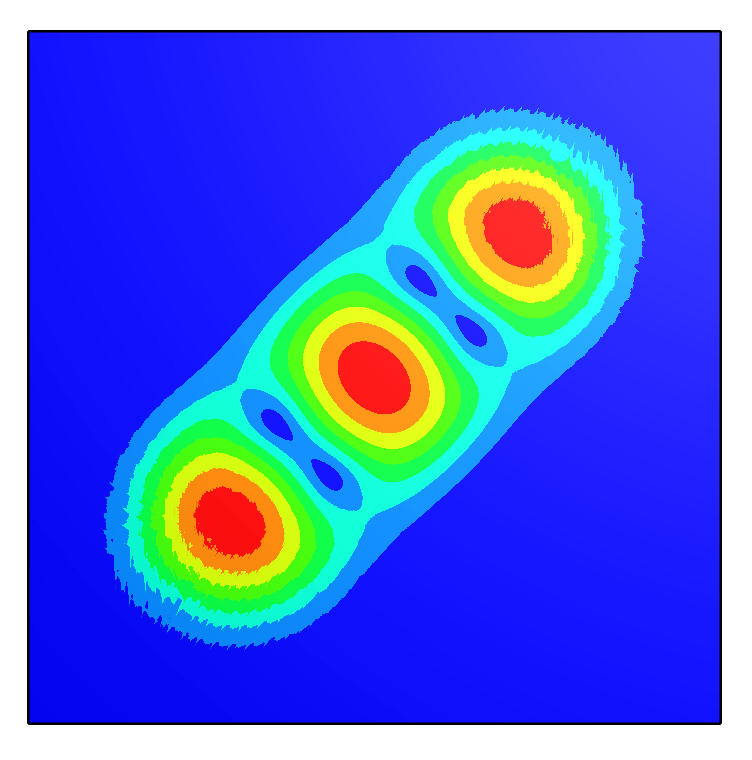}
\includegraphics[width=.16\textwidth]{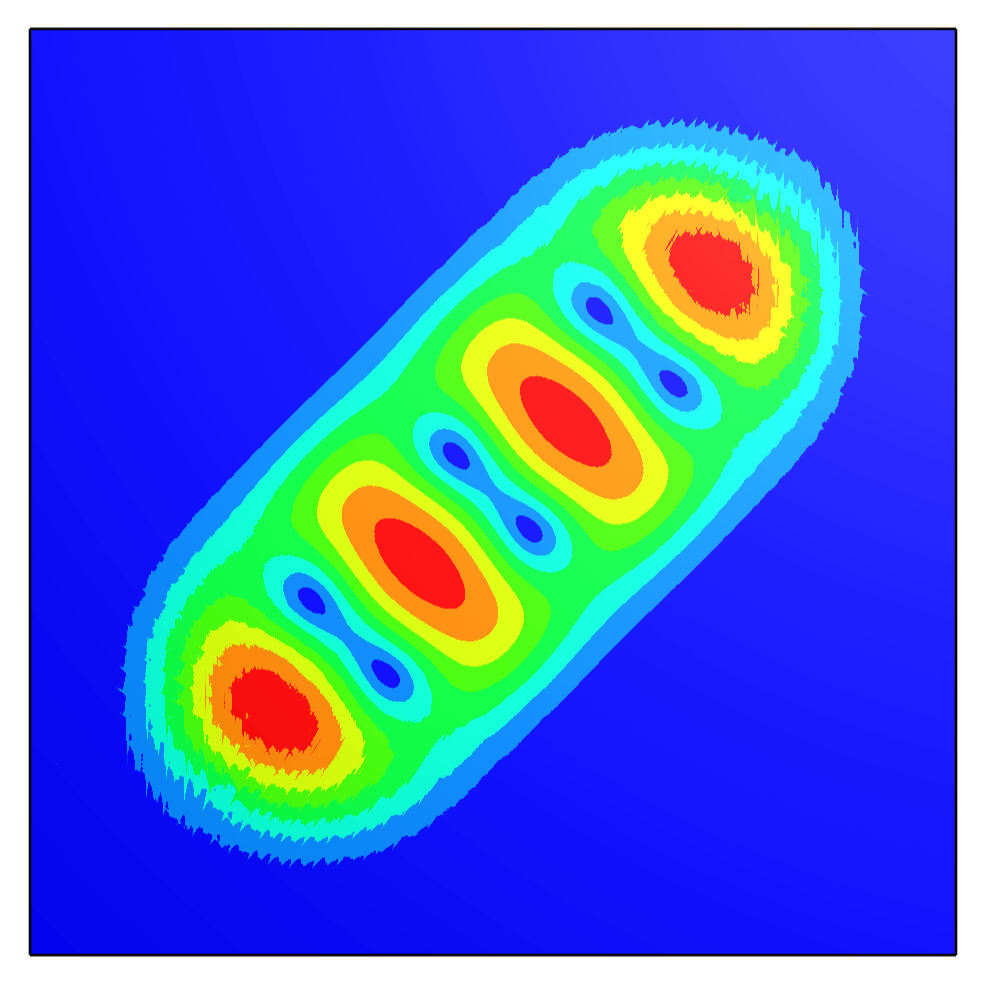}
\includegraphics[width=.16\textwidth]{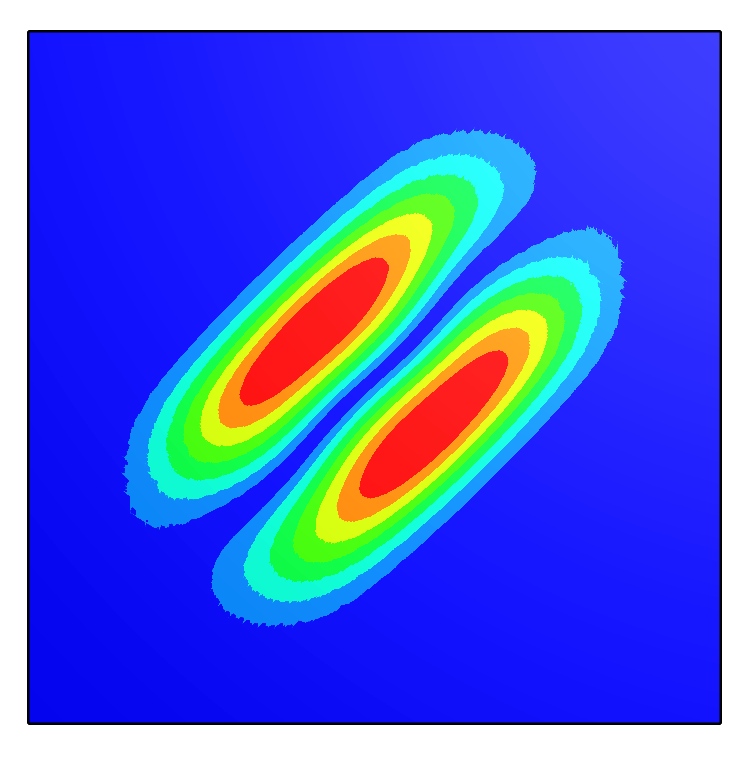}
\includegraphics[width=.16\textwidth]{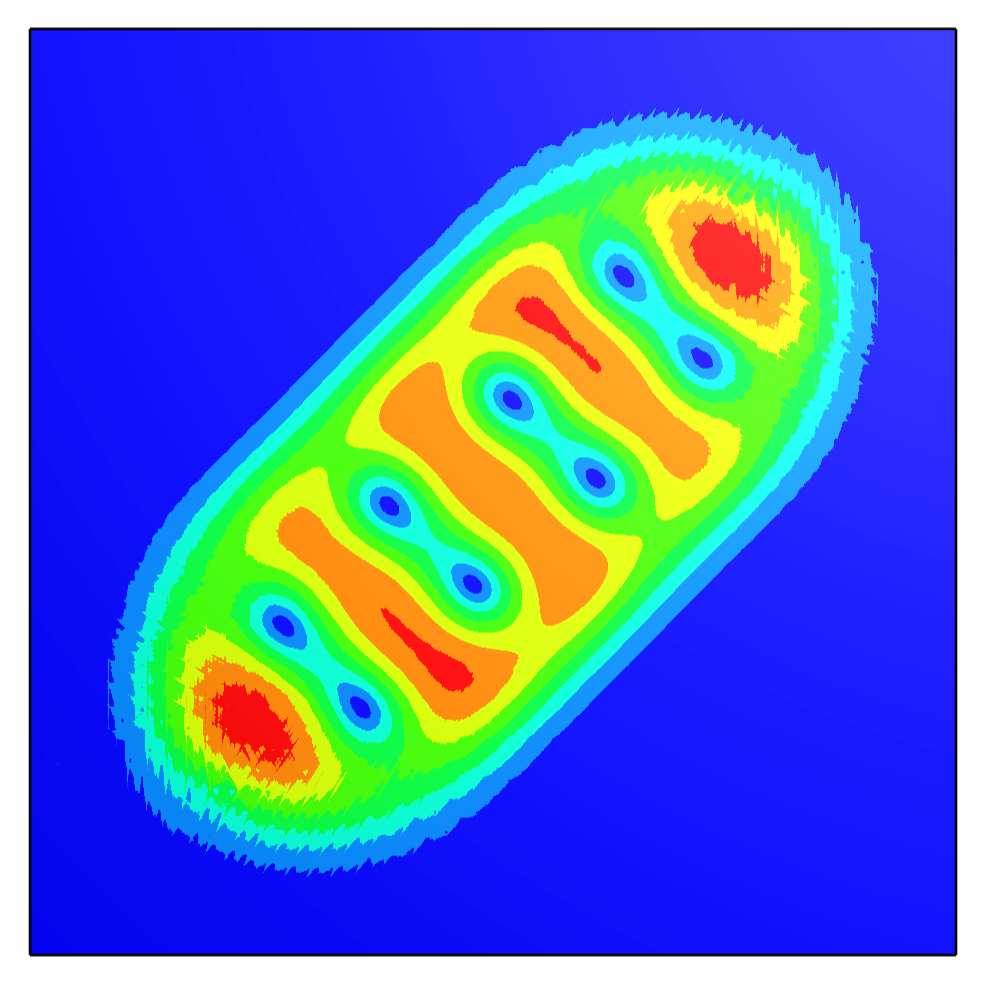}
\caption{The computed eigenvectors $|\psi_j|$ of $H(\mb{A})$, $1\leq j\leq 6$, when
  $h=0.01$ (top row), $h=0.03$ (middle row), and $h=0.05$, for Example 1.\label{Ex1Fig2}}
\end{figure}

\subsection{Example 2}\label{Example2}
We take the vector field
\begin{align}\label{Ex2A}
\bfA = -100(\cos( f_1)\sin(f_2), \sin(f_1)\cos(f_2)))\;,\; f_1 = 5\pi \sin(x^2 + y^2) \;,\; f_2 =5\pi \cos(x^2 + y^2)~,
\end{align}
with homogeneous Dirichlet conditions on the domain $\Omega=(-1,1)\times(-1,1)$, see
Figure~\ref{VectorFields}.
In this case, we compute
\begin{align*}
  \|A\|_{L^2(\Omega)}=130.6436\quad,\quad \|F\|_{L^2(\Omega)}=89.8614~,
\end{align*}
when $h=0.01$.
We will consider variants of the problem without a scalar potential
($V=0$) and with a scalar potential.

\subsubsection{Example 2a: Eigenpairs for $H(\mb{A})$, $H(\mb{F})$.}\label{Example2a}
 Table~\ref{Ex2aTable} provides the computed
eigenvalues and timing information on each of three meshes, as well as
the validation of ~\change{\eqref{Heuristic2} and~\eqref{Heuristic3}}.  We first observe
that the computed $\lambda_5$ and $\lambda_6$ almost certainly
represent an eigenvalue of multiplicity two.  As with Example 1, we
see that the computed eigenvalues for $H(\mb{F})$ vary little has $h$
increases, while there is a marked change in what is computed for
$H(\mb{A})$, particularly on the coarsest mesh.  As before, the
quality of the computed eigenpairs for $H(\mb{F})$ when $h=0.03$ is as
good as that for $H(\mb{A})$ when $h=0.01$, so we again see that
computing with $H(\mb{F})$ is more efficient.  We look at the computed
eigenvectors of $H(\mb{A})$ in Figure~\ref{Ex2aFig2}, showing those
for $h=0.01$ and $h=0.05$.  We omit those for $h=0.03$ because there
is a negligible difference between it and $h=0.01$, apart from a
change in order between the final two eigenvectors.  This is
immaterial because they are approximating the basis of a
two-dimensional eigenspace.  When $h=0.05$, we see a marked
deterioration in the quality of the approximations for $H(\mb{A})$, as
well as a misordering among the final three eigenvectors.  In
contrast, the computed eigenvectors for $H(\mb{F})$ (not shown) remain
essentially the same across the three meshes, up to ordering of the
final two vectors.

\begin{table}
\caption{Computed eigenvalues and timing information, and validation
  of ~\change{\eqref{Heuristic2} and~\eqref{Heuristic3}} when $h=0.01$, for Example 2a.\label{Ex2aTable}}
\begin{center}
\begin{tabular}{|c|cll|cccccc|}\hline
&$h$ & Total & FEAST& $\lambda_1$ &  $\lambda_2$  & $\lambda_3$ &  $\lambda_4$ & $\lambda_5$ &  $\lambda_6$\\\hline
\multirow{3}{*}{$H(\bfA)$}
&0.01  &310.55s& 270.39s    &104.069&   111.630 &154.607&   177.494 &196.589&   196.590\\
&0.03   &77.08s&  73.96s   &109.963&   119.903 &159.500&   184.112 &200.066&   200.089\\
&0.05  &25.93s&24.94s&134.774&   176.245 &219.781&   237.908 &238.637&   238.910\\
\hline
\multirow{3}{*}{$H(\bfF)$}
&0.01  &354.08s& 276.56s    &104.057&   111.613 &154.598&   177.481 &196.583&   196.583\\
&0.03   &25.07s&22.54s&104.075&   111.642&154.638&   177.554 &196.607&   196.608\\
&0.05  &7.46s&6.58s&104.444&   112.211 &155.502&   179.126 &197.103&   197.141\\
\hline
\end{tabular}

\vspace*{2mm}
 \begin{tabular}{ |c|cccccc|}
   \hline
   &$j=1$&$j=2$&$j=3$&$j=4$&$j=5$&$j=6$\\\hline
   $\|\nabla \psi_j\|_{L^2(\Omega)}$  & 71.0007 & 79.5820 & 60.4883  & 75.4745 & 51.9968 & 51.9967 \\
   $\|\mb{A} \psi_j\|_{L^2(\Omega)}$ & 70.7587 & 79.2955 & 59.8389  & 74.8731 & 51.4847 & 51.4847 \\  
   \hline
   $\|\nabla \phi_j\|_{L^2(\Omega)}$  & 31.0322 & 35.5500 & 33.6866  & 40.9010 & 22.0141 & 22.0141 \\ 
   $\|\mb{F} \phi_j\|_{L^2(\Omega)}$  &  30.4731 & 34.9024 & 32.5051  & 39.7792 & 20.7749 & 20.7749 \\ \hline
 \end{tabular}
\end{center}
\end{table}

\begin{figure}
\includegraphics[width=.16\textwidth]{Ex2/H_A/01Normeig_1.png}
\includegraphics[width=.16\textwidth]{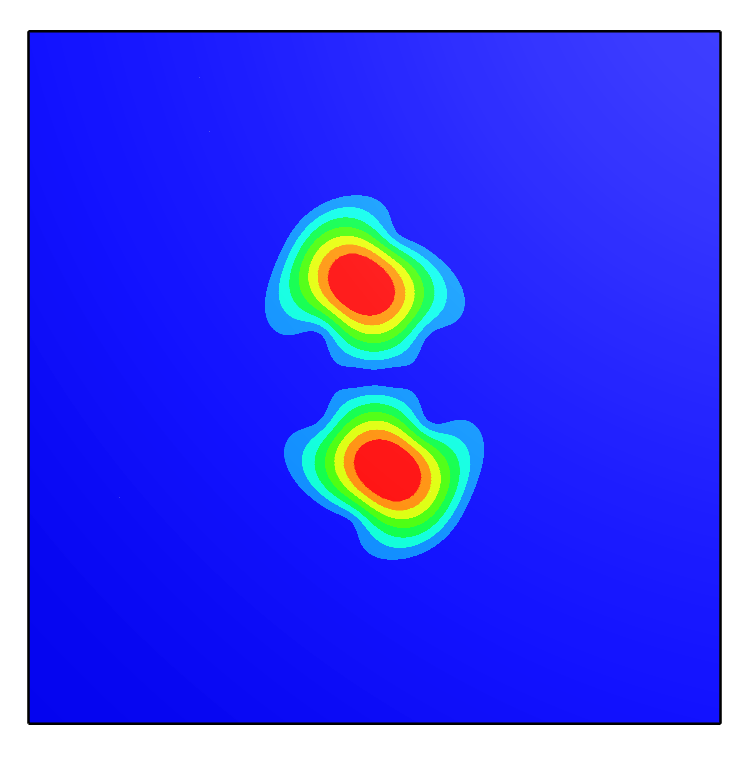}
\includegraphics[width=.16\textwidth]{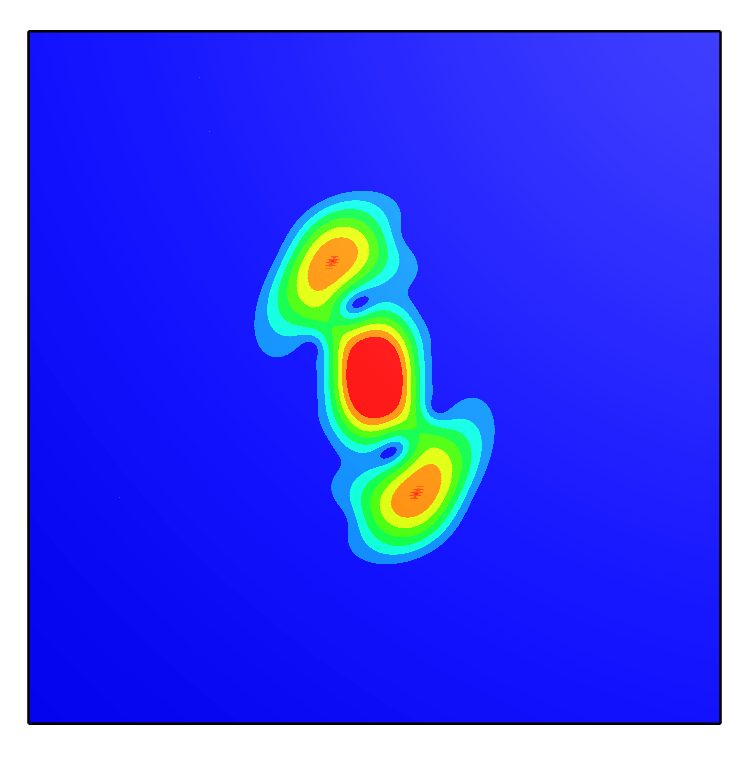}
\includegraphics[width=.16\textwidth]{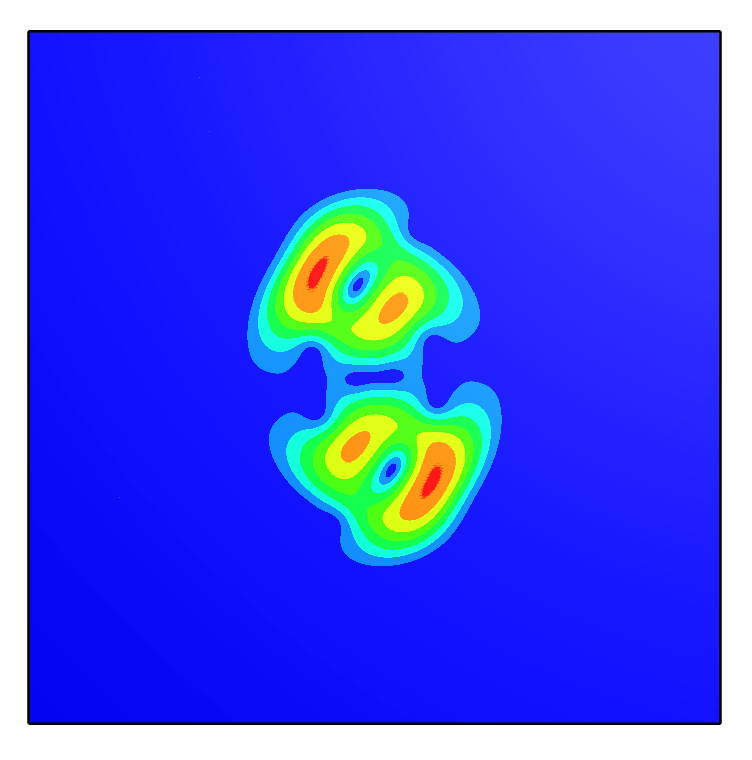}
\includegraphics[width=.16\textwidth]{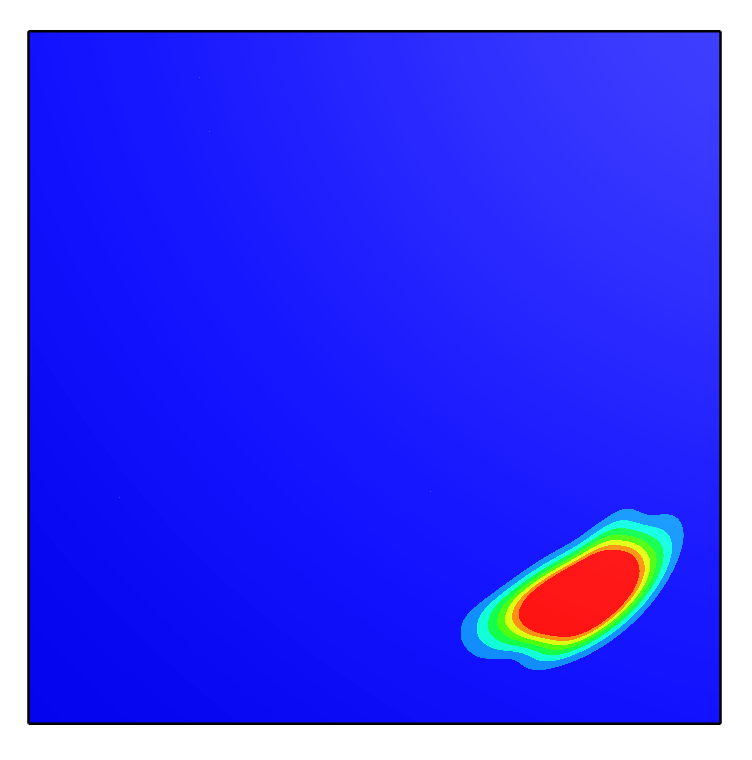}
\includegraphics[width=.16\textwidth]{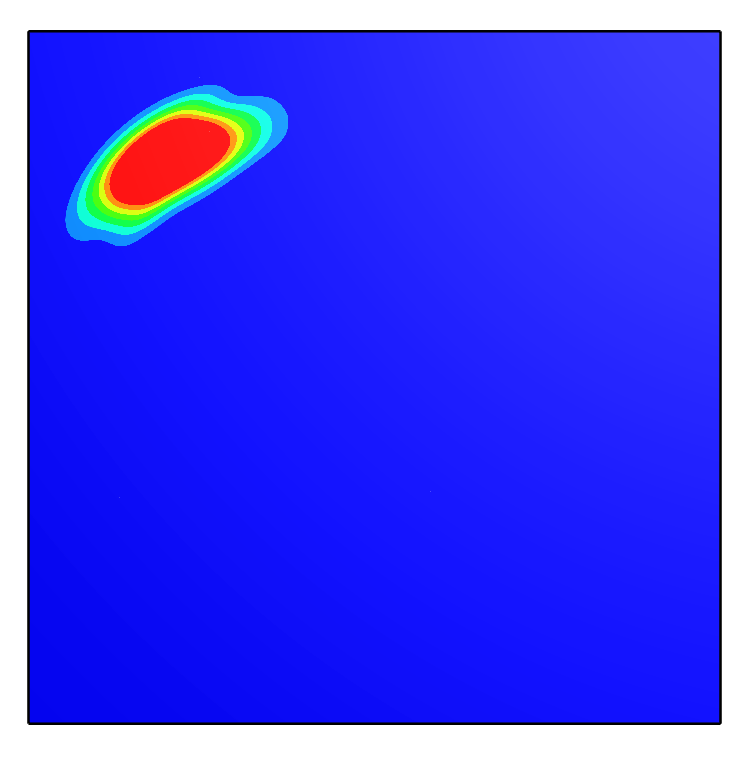}
\includegraphics[width=.16\textwidth]{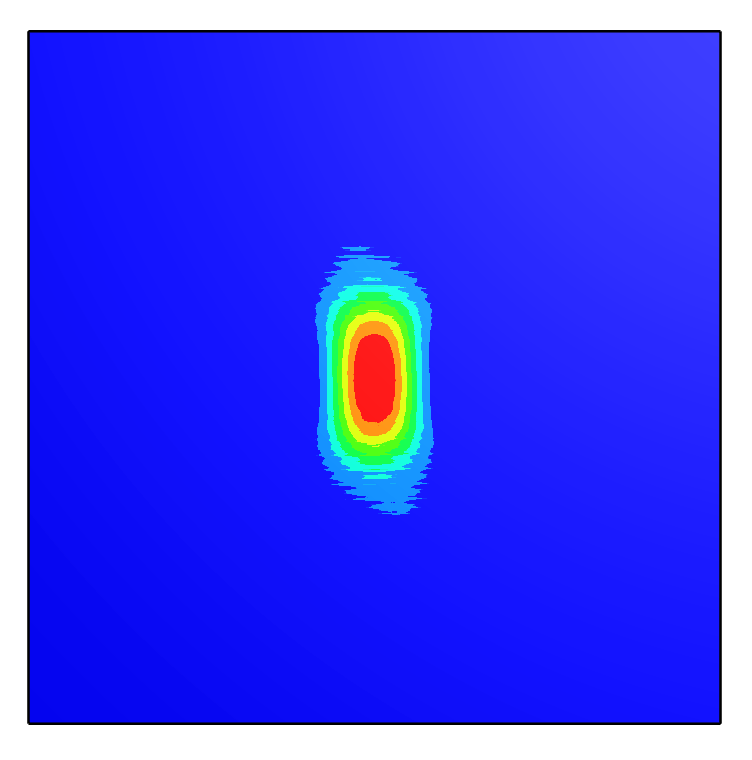}
\includegraphics[width=.16\textwidth]{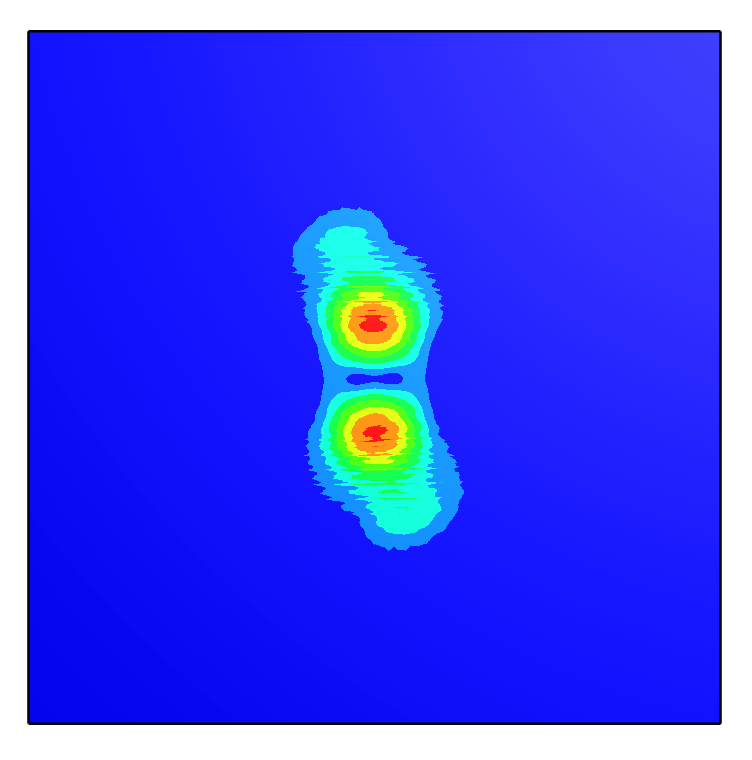}
\includegraphics[width=.16\textwidth]{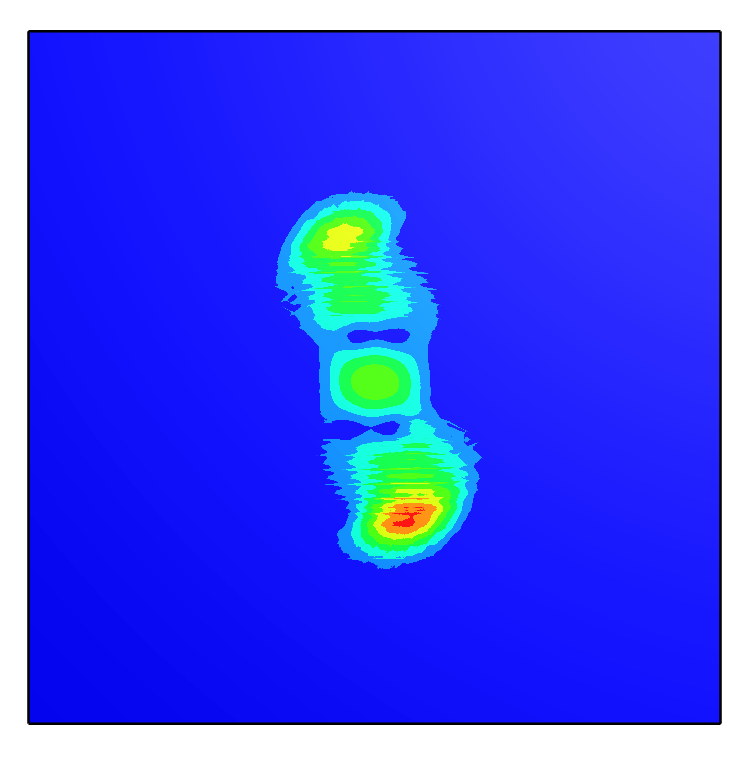}
\includegraphics[width=.16\textwidth]{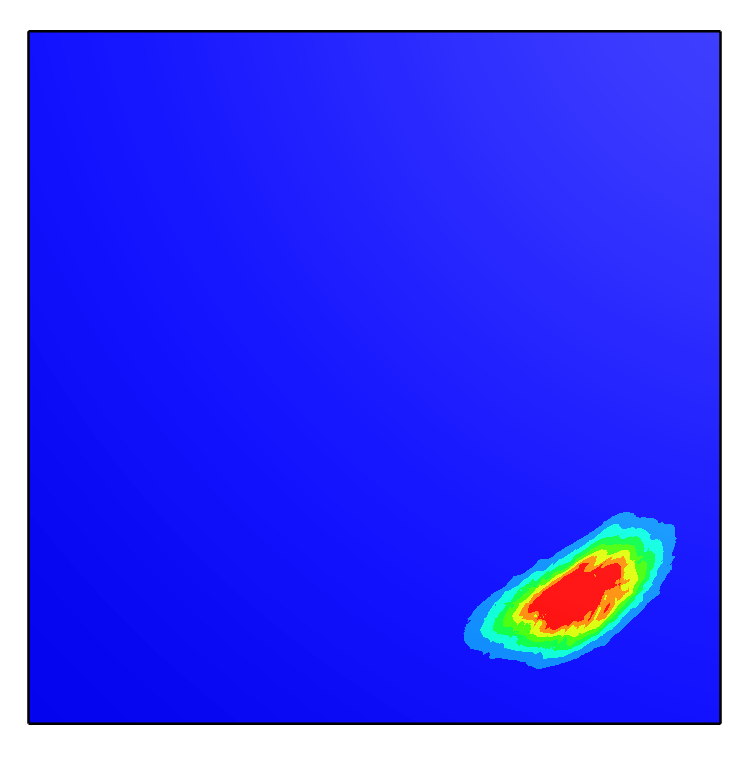}
\includegraphics[width=.16\textwidth]{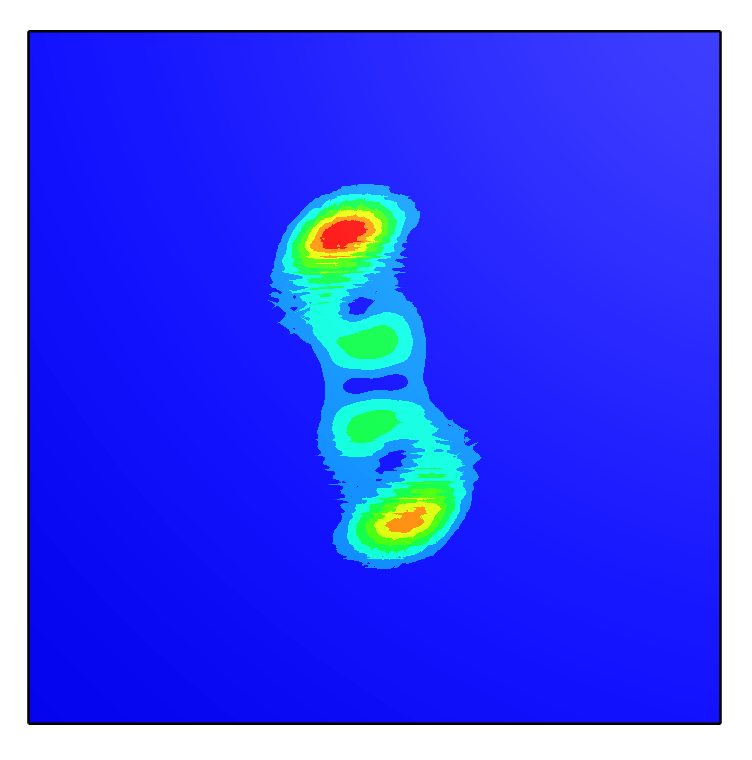}
\includegraphics[width=.16\textwidth]{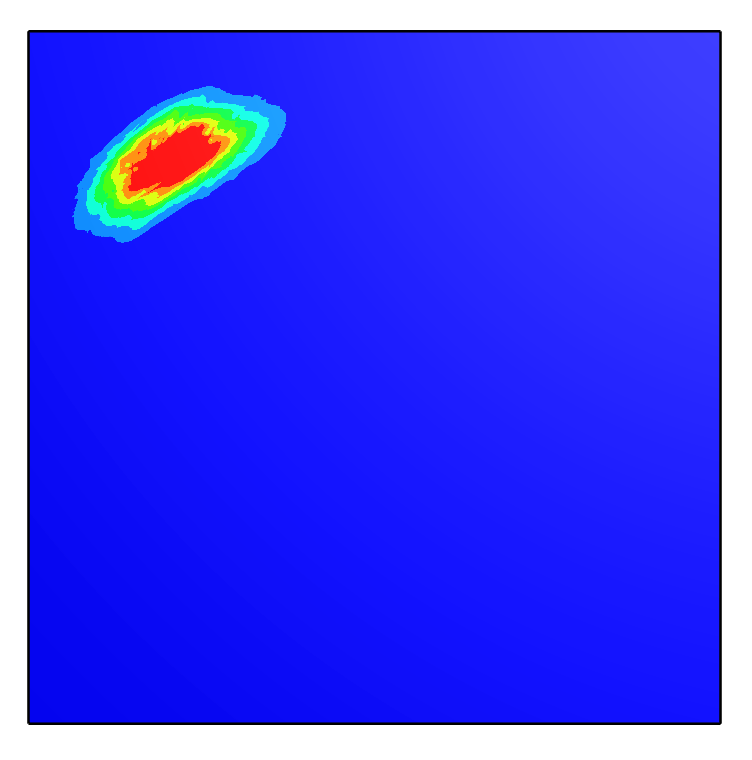}
\caption{The computed eigenvectors $|\psi_j|$ of $H(\mb{A})$, $1\leq j\leq 6$, when
  $h=0.01$ (top row), 
  and $h=0.05$, for Example 2a.\label{Ex2aFig2}}
\end{figure}
 
 \subsubsection{Example 2b: Eigenpairs for $H(\mb{A},V)$ and $H(\mb{F},V)$.}
We take $V$ to be piecewise constant and rapidly varying on $\Omega$,
see Figure~\ref{Ex2bV}.  This scalar potential was constructed by
partitioning $\Omega$ into $16\times 16$ congruent squares, and
randomly assigning a number between $0$ and \change{$V^*$} to each square.
Keeping the same seed for the random number generator, we chose
$\change{V^*}=100,500,1000$, as shown in the figure, so the three
different instances of $V$ are just scalings of each other.  We note that the
actual values of $V$ are strictly between $0$ and $V^*$.

It is known, empirically at least, that the presence of a scalar
potential such as this can dominate the influence of $\mb{A}$ on
how/where eigenvectors localize (cf.\cite{Hoskins2024}), and we
observe this in Figure~\ref{Ex2bFig2}. On obvious change, even when
$V^*=100$ is that we no longer have a repeated eigenvalue among the
first six.  The grid used to define $V$ is overlaid on the plots in
Figure~\ref{Ex2bFig2} for greater ease in comparison with the
potential shown in Figure~\ref{Ex2bFig2}.
\begin{figure}
  \centering
  \begin{minipage}[c]{0.43\textwidth}
    \includegraphics[width=\textwidth]{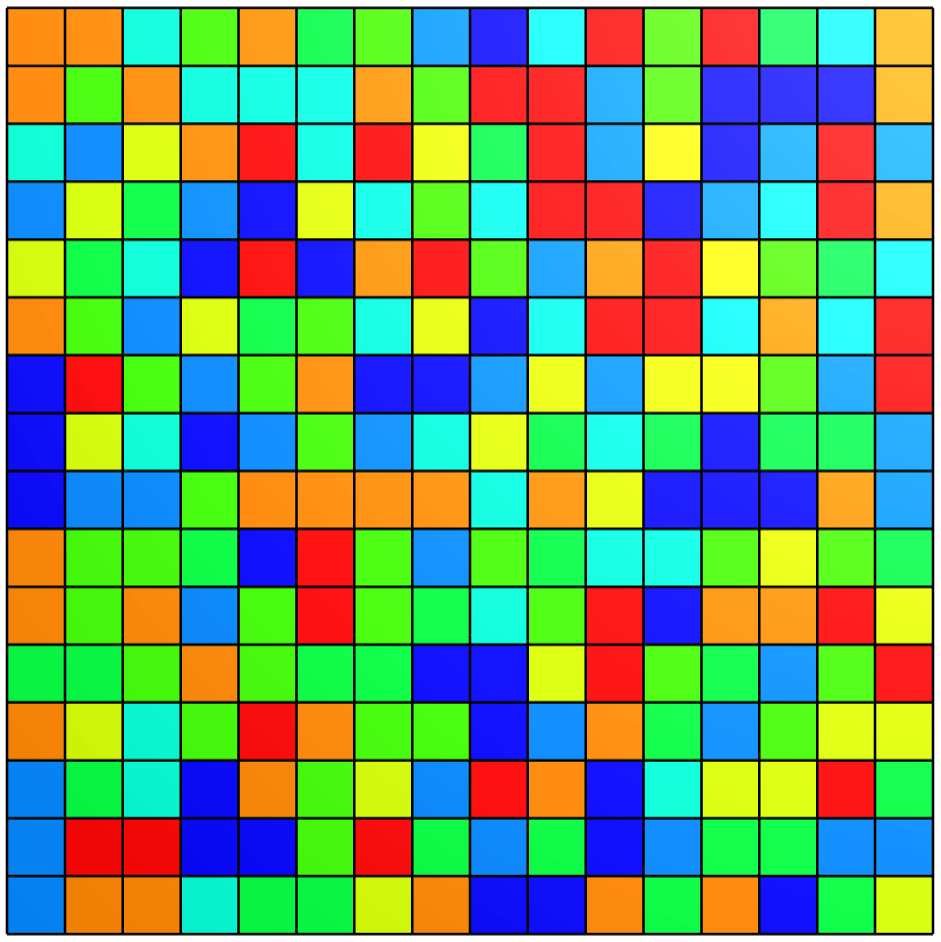}
  \end{minipage}
  \begin{minipage}[c]{0.55\textwidth}
    \centering
    $V^*=100$\\[3pt]
     \includegraphics[width=\textwidth]{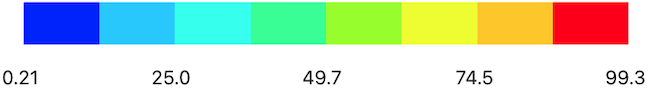}

     \vspace*{4mm}
     $V^*=500$\\[3pt]
    \includegraphics[width=\textwidth]{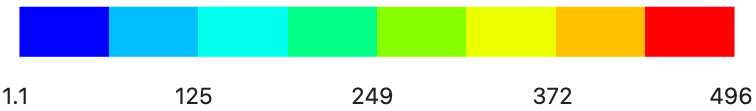}

     \vspace*{4mm}
     $V^*=1000$\\[3pt]  
     \includegraphics[width=\textwidth]{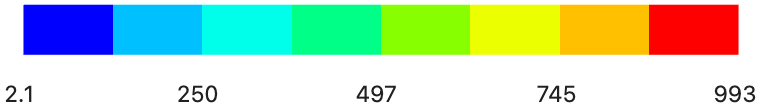}
  \end{minipage} 
 \caption{The piecewise constant scalar field V, Example 2b. \label{Ex2bV}}
\end{figure}

\begin{figure}
\includegraphics[width=.16\textwidth]{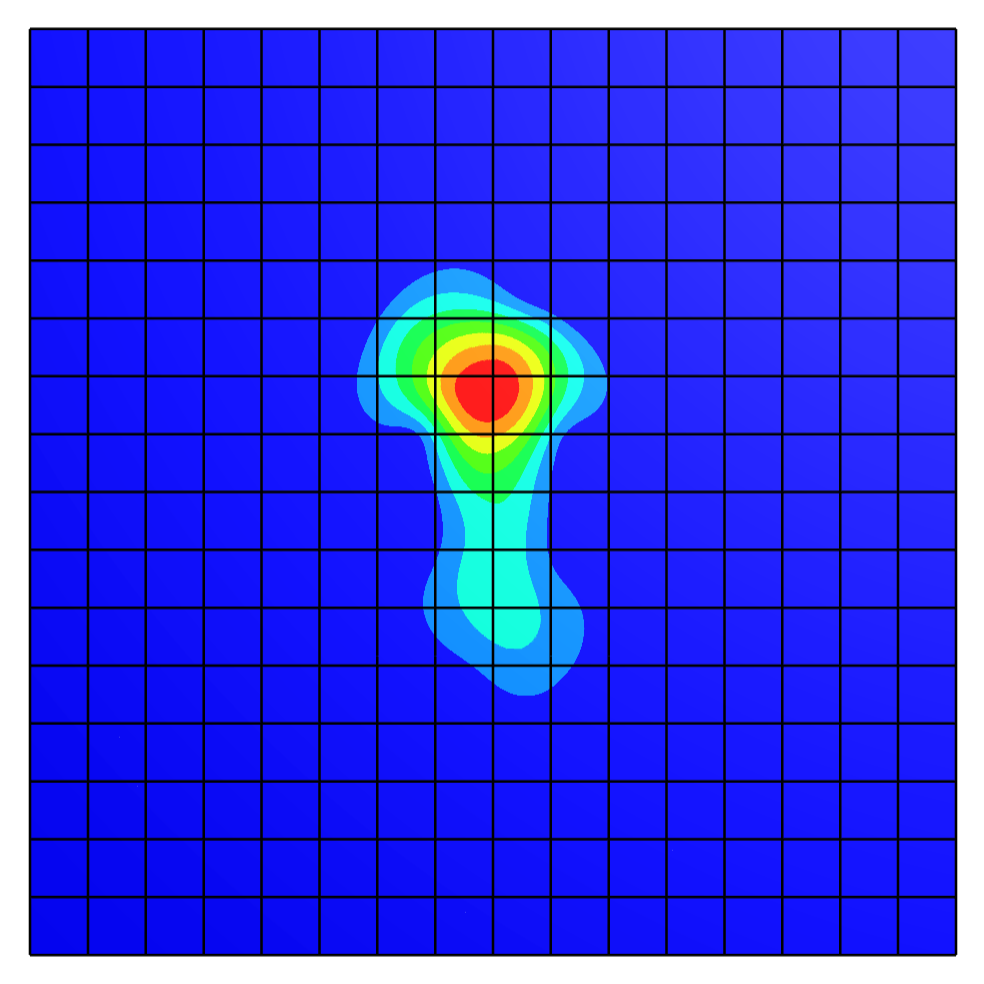}
\includegraphics[width=.16\textwidth]{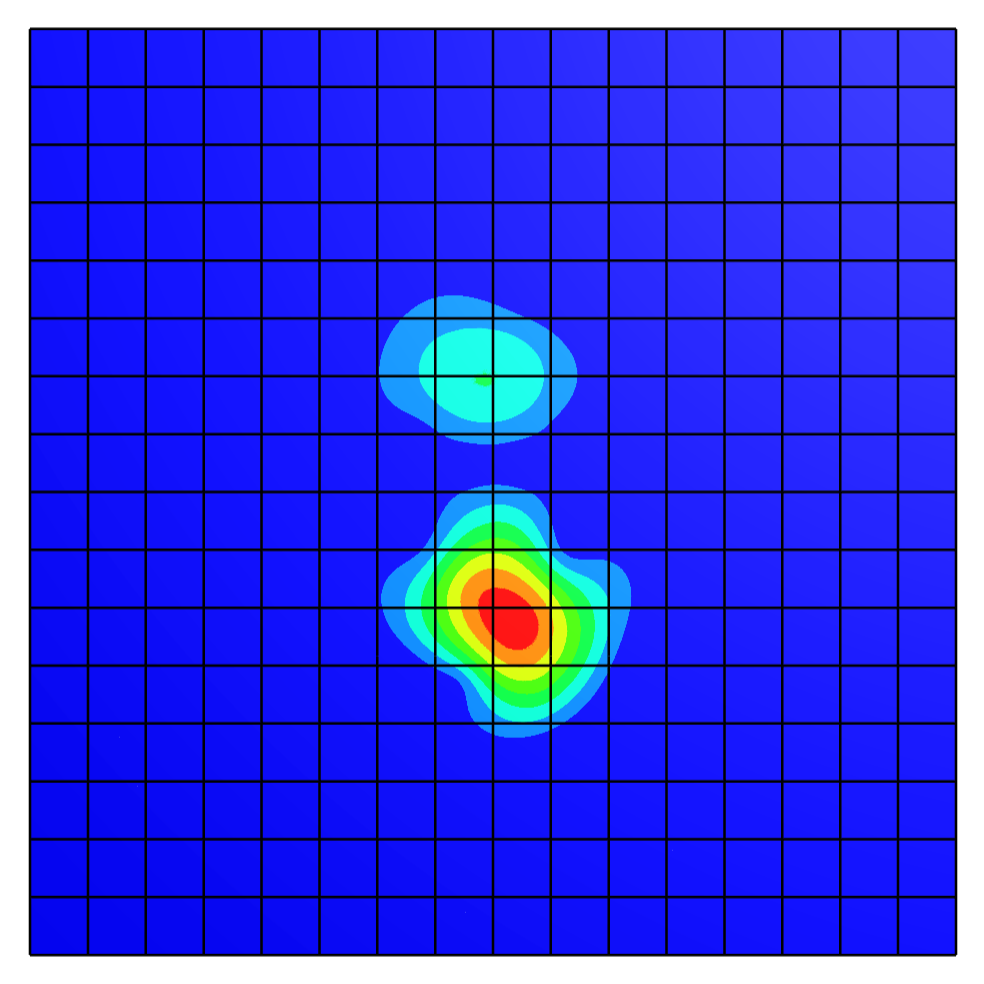}
\includegraphics[width=.16\textwidth]{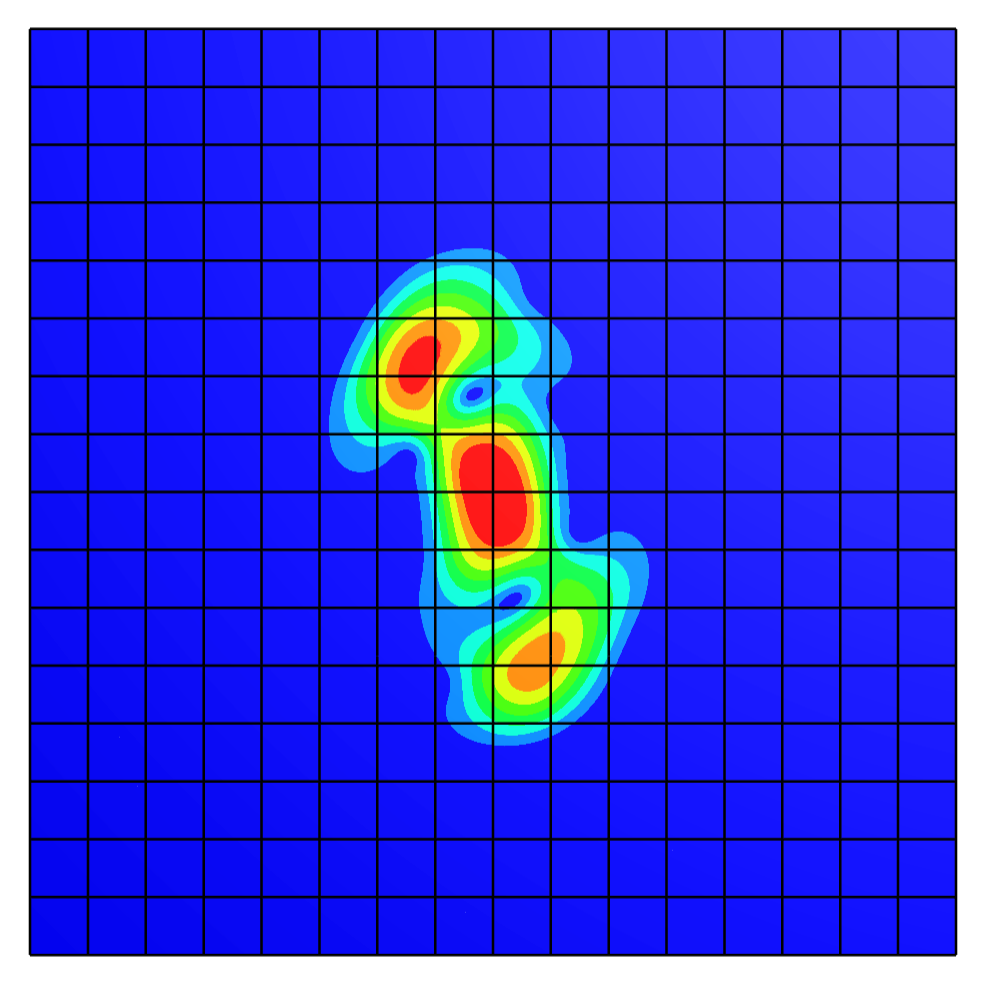}
\includegraphics[width=.16\textwidth]{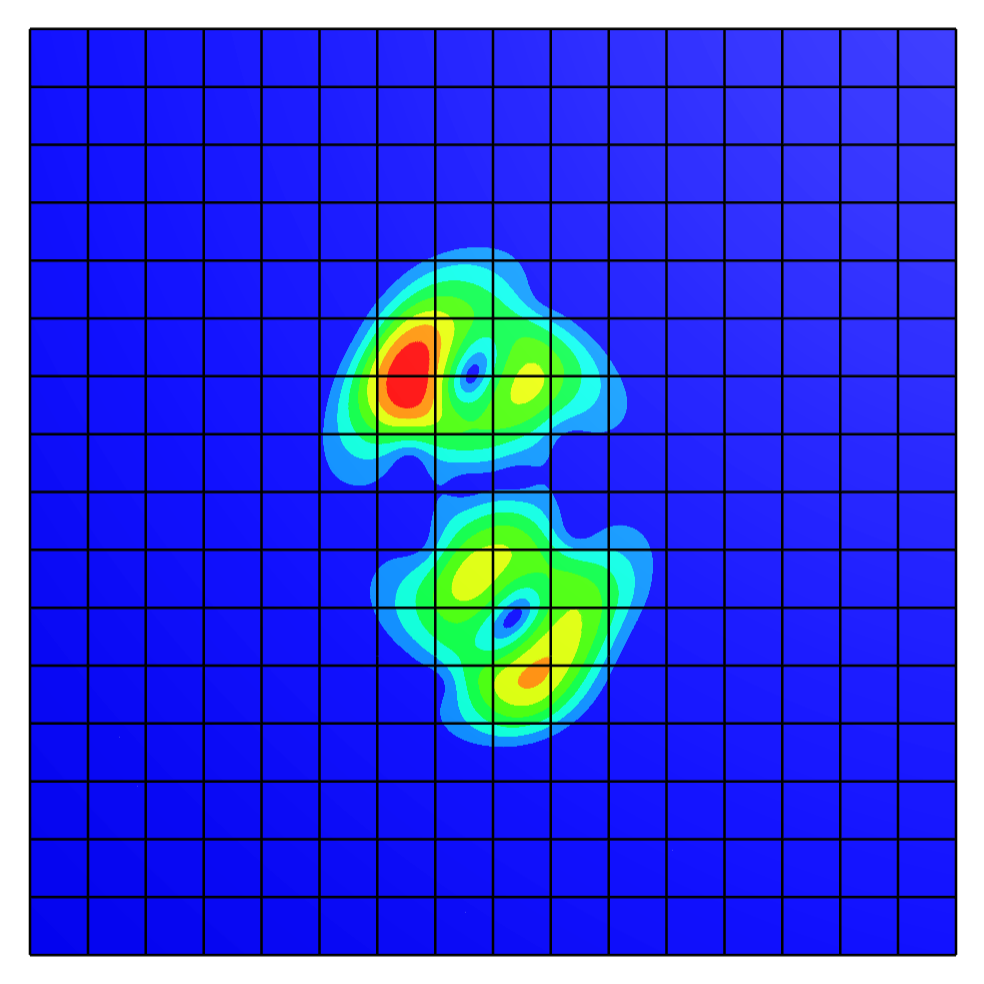}
\includegraphics[width=.16\textwidth]{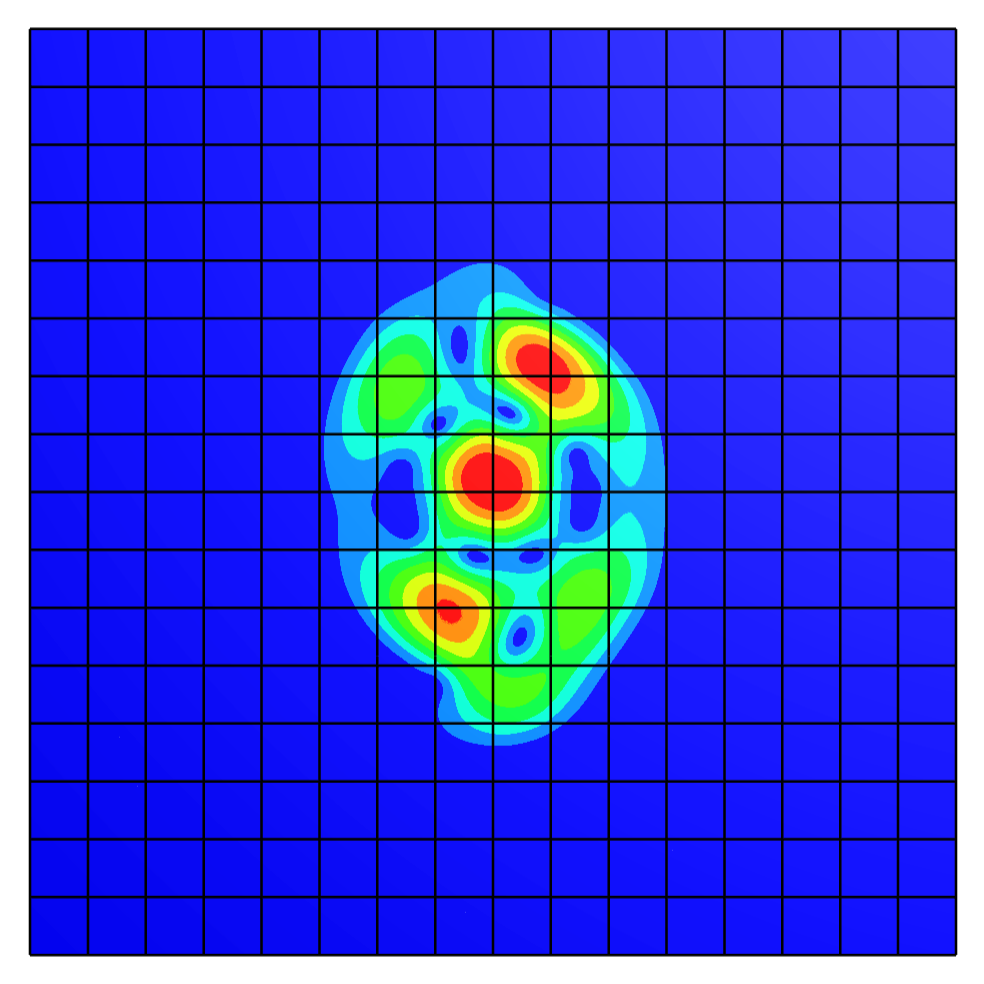}
\includegraphics[width=.16\textwidth]{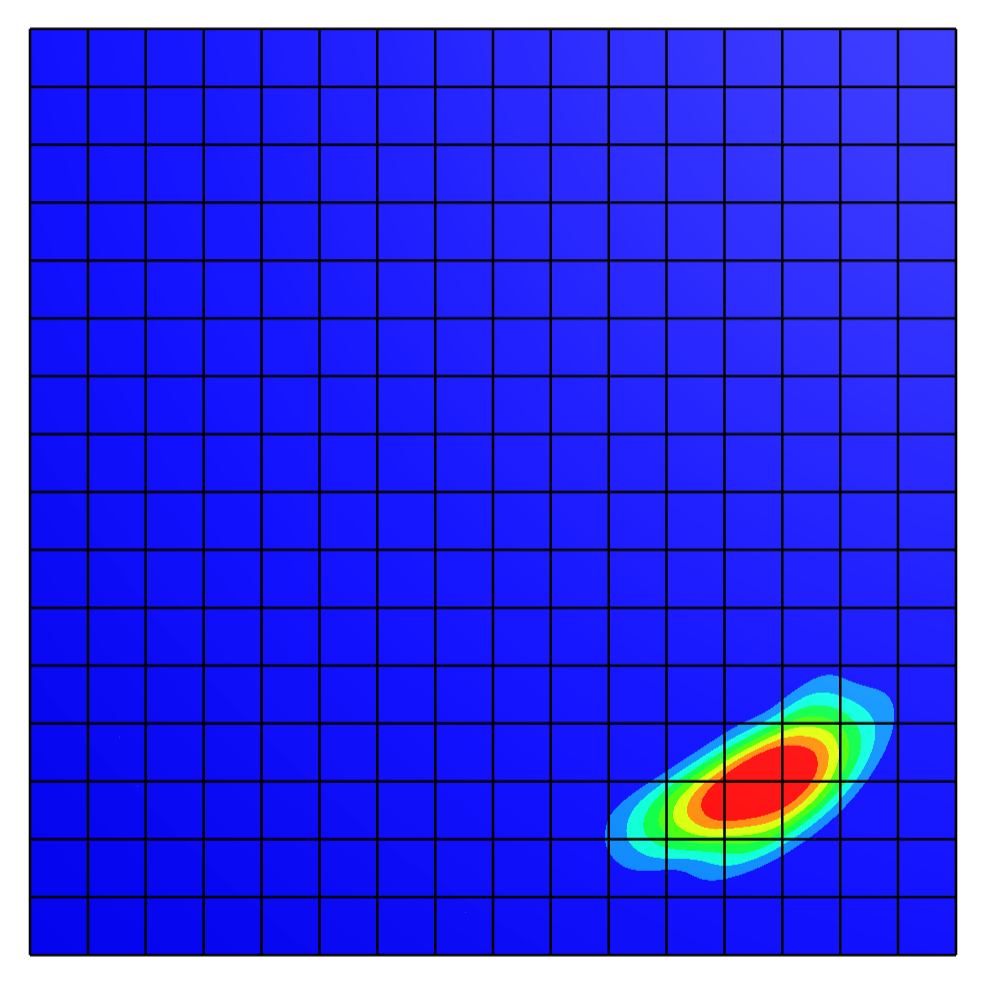}
\includegraphics[width=.16\textwidth]{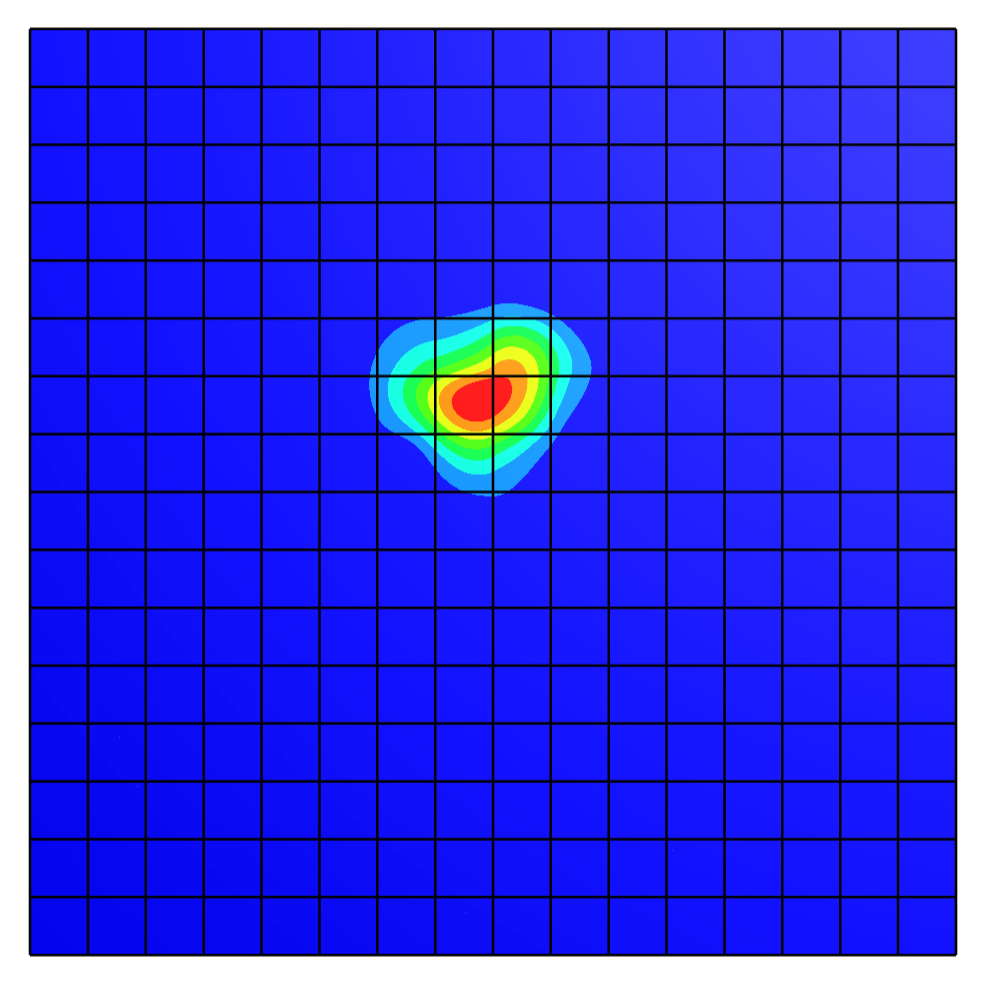}
\includegraphics[width=.16\textwidth]{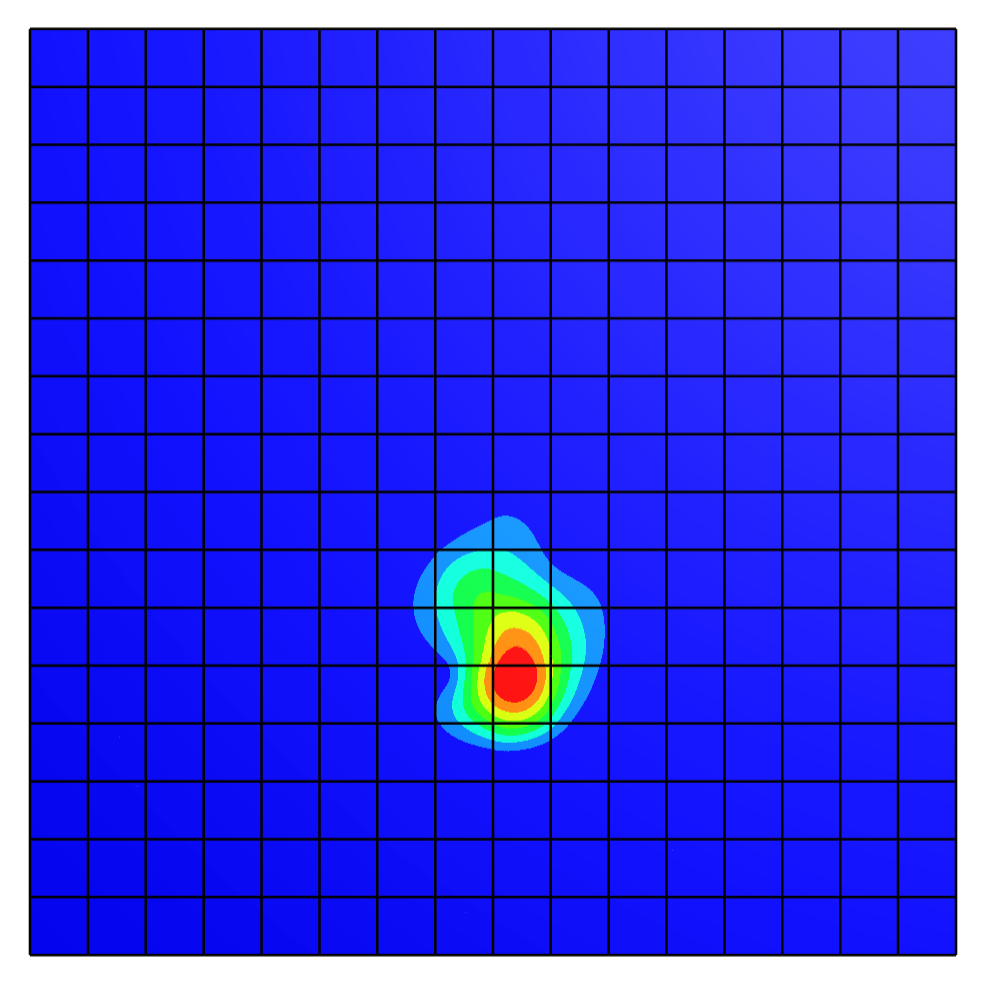}
\includegraphics[width=.16\textwidth]{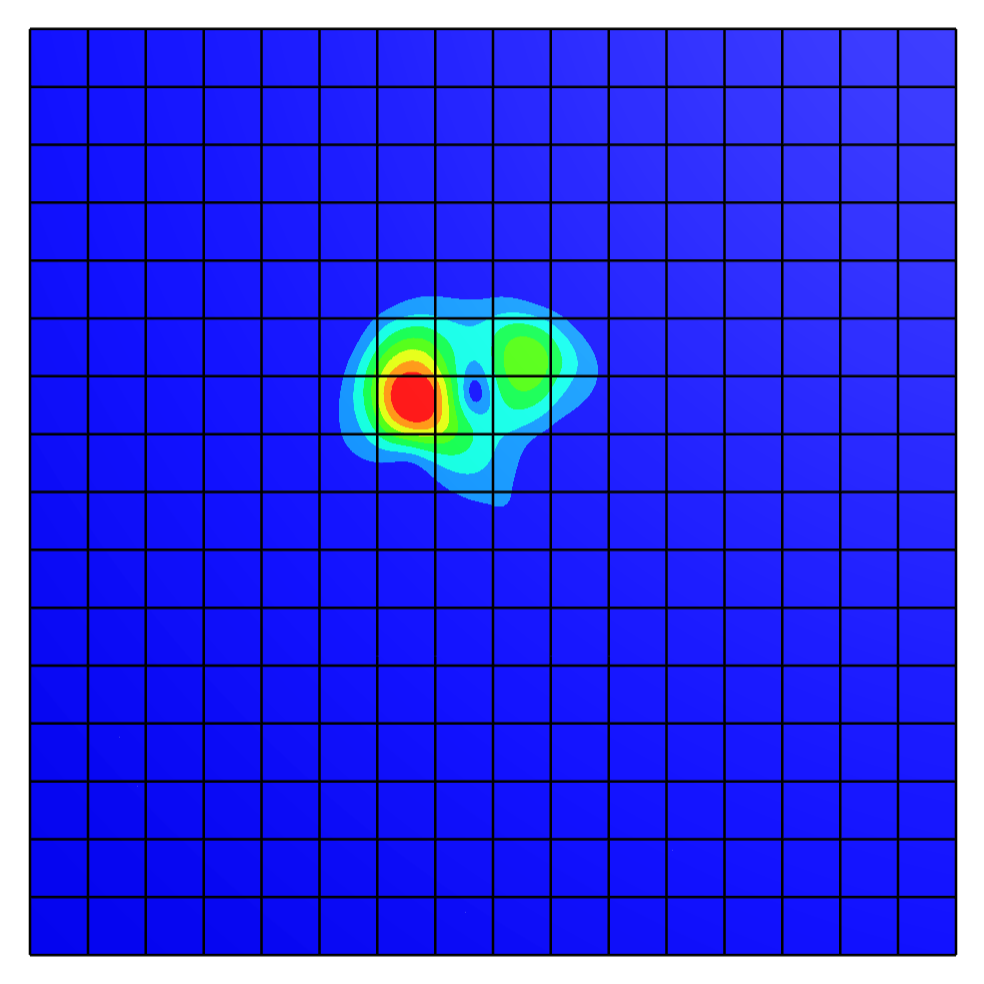}
\includegraphics[width=.16\textwidth]{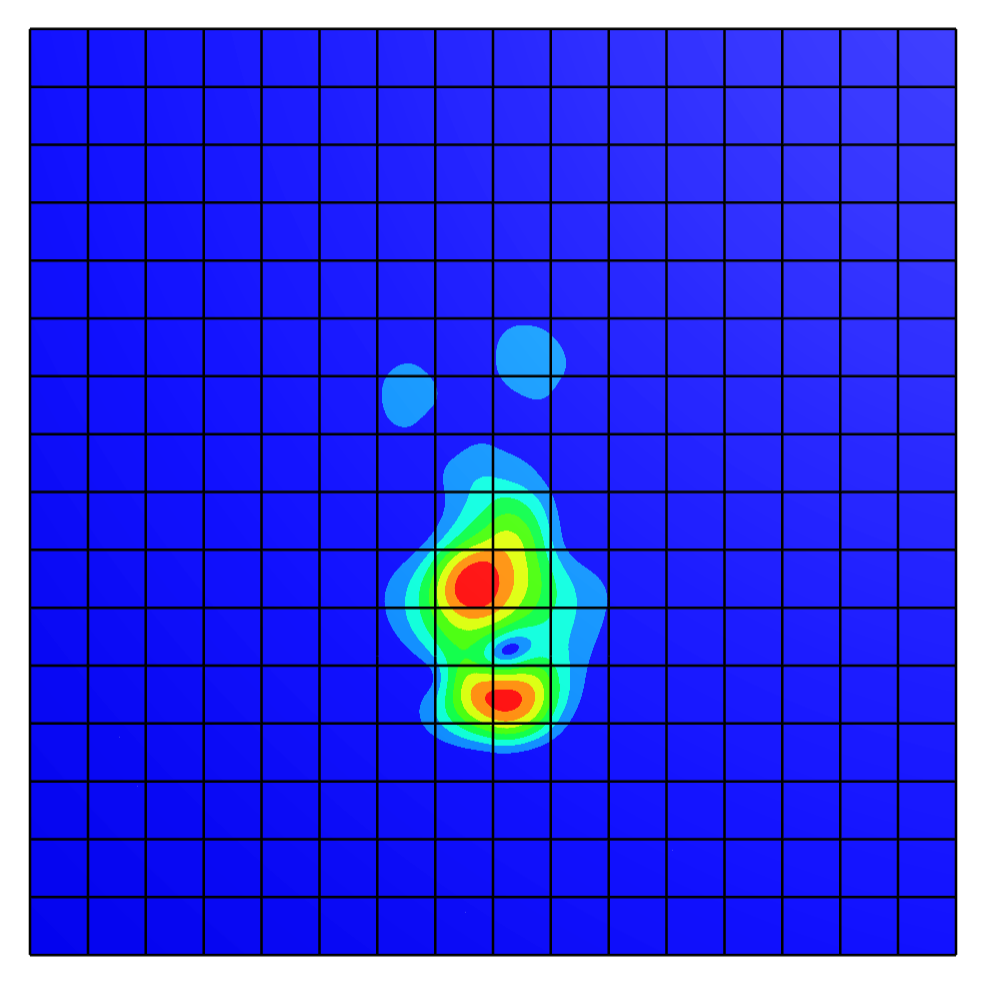}
\includegraphics[width=.16\textwidth]{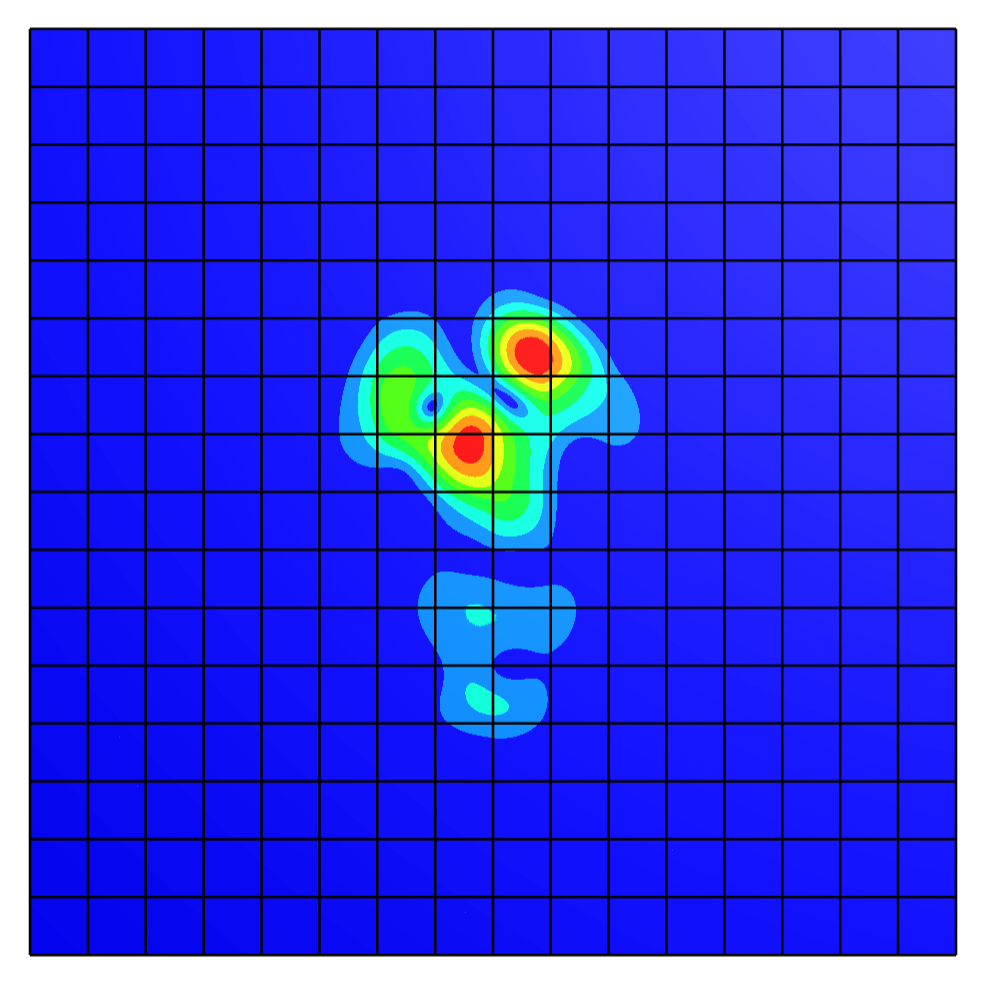}
\includegraphics[width=.16\textwidth]{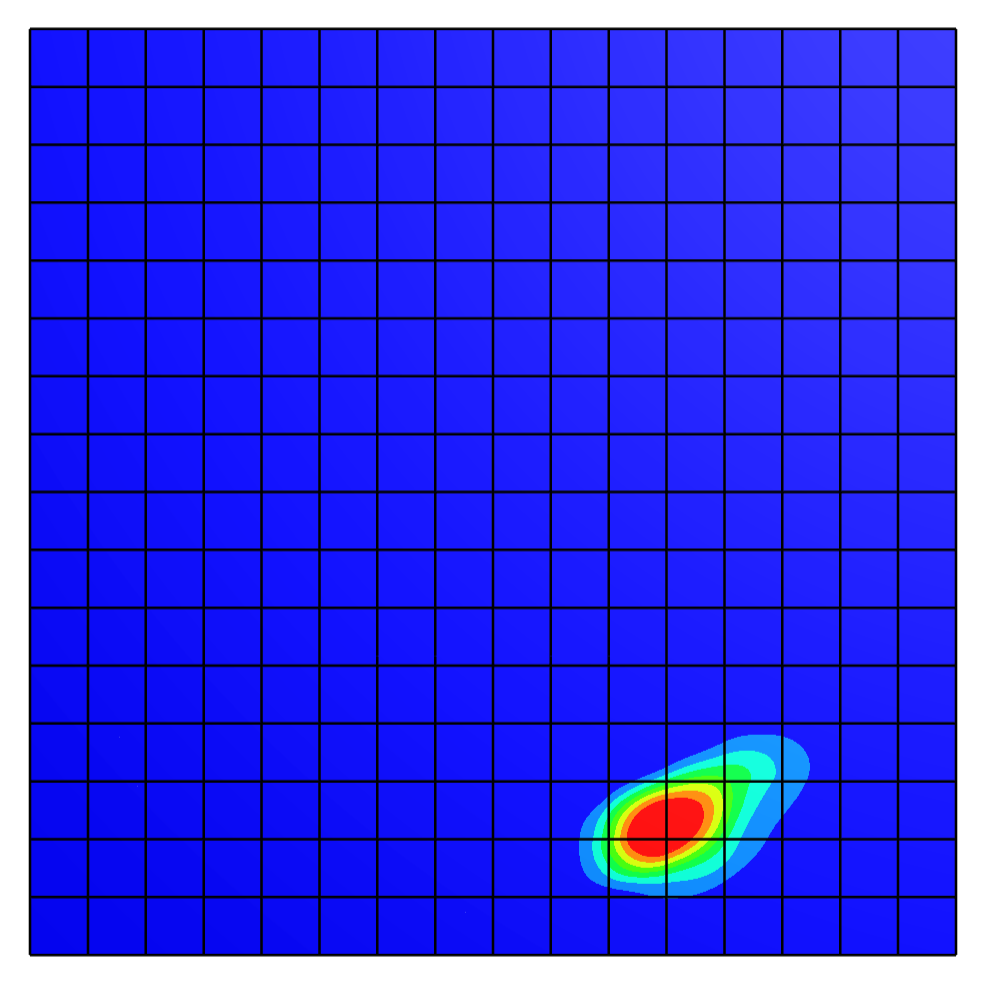}
\includegraphics[width=.16\textwidth]{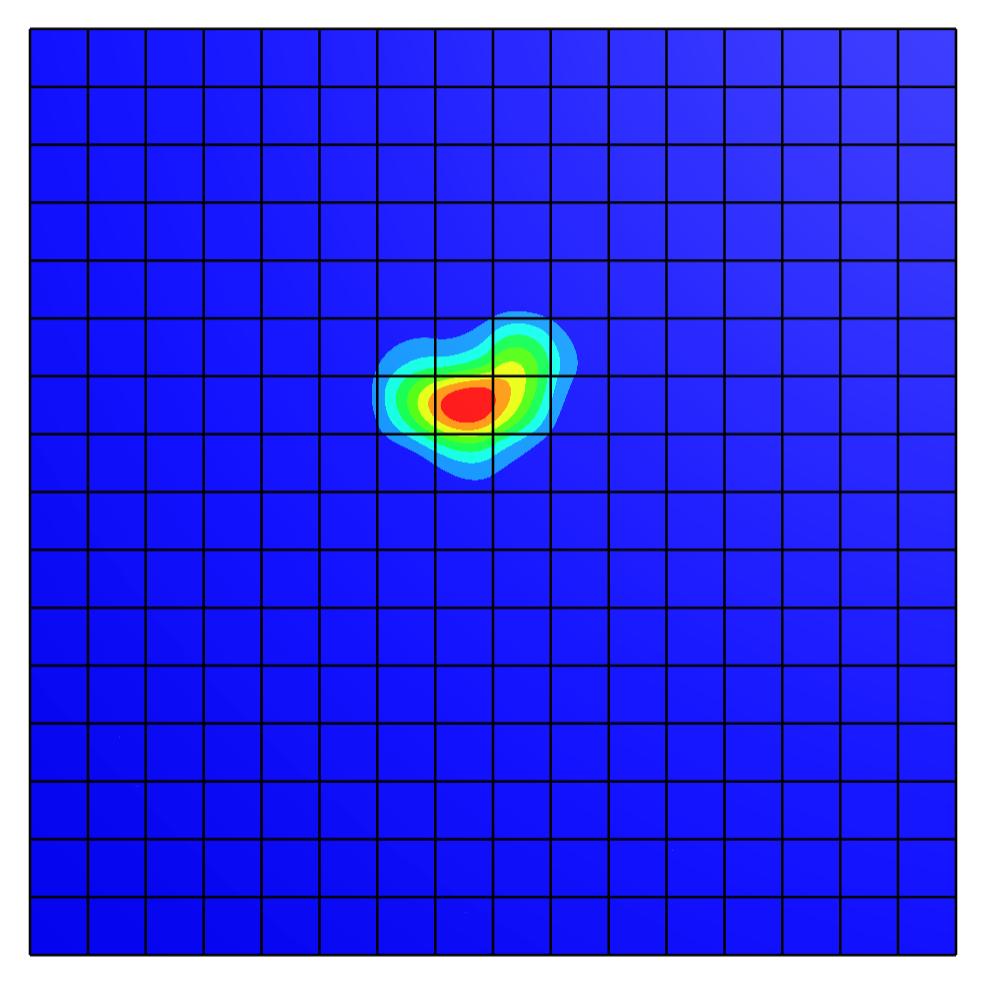}
\includegraphics[width=.16\textwidth]{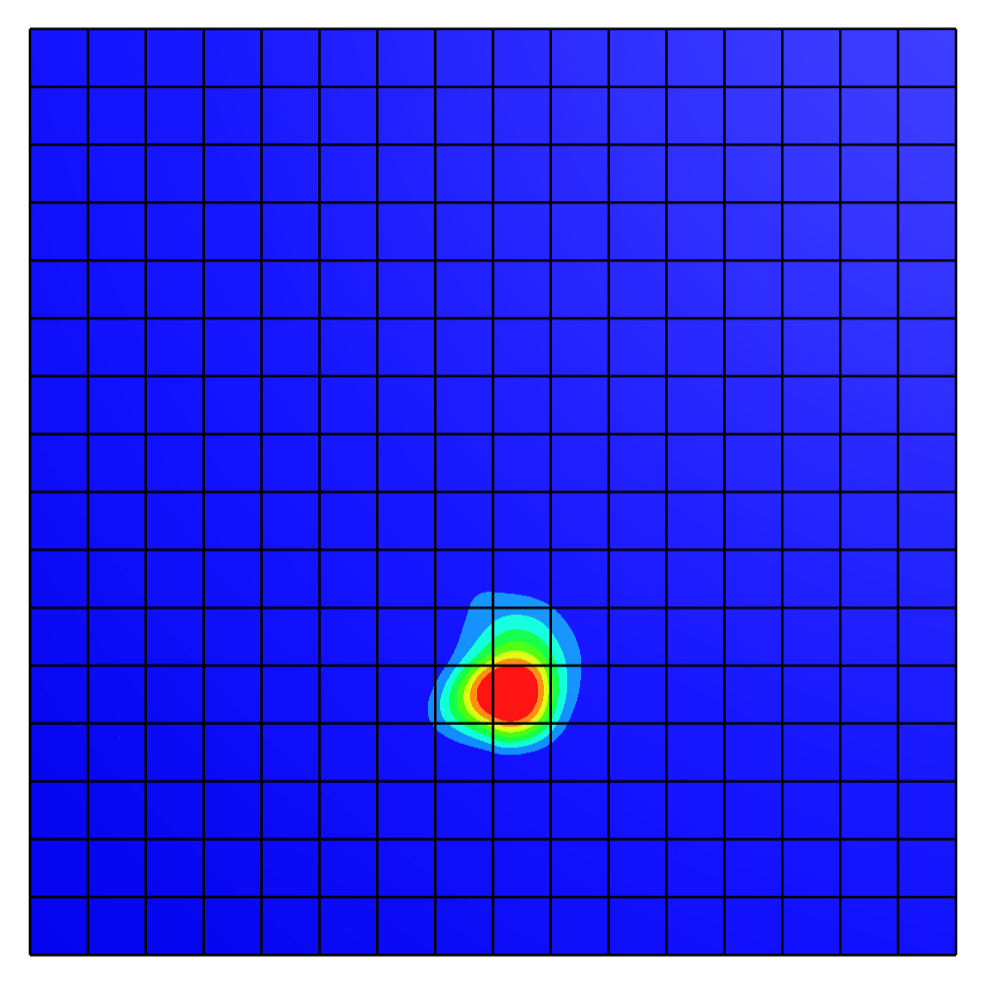}
\includegraphics[width=.16\textwidth]{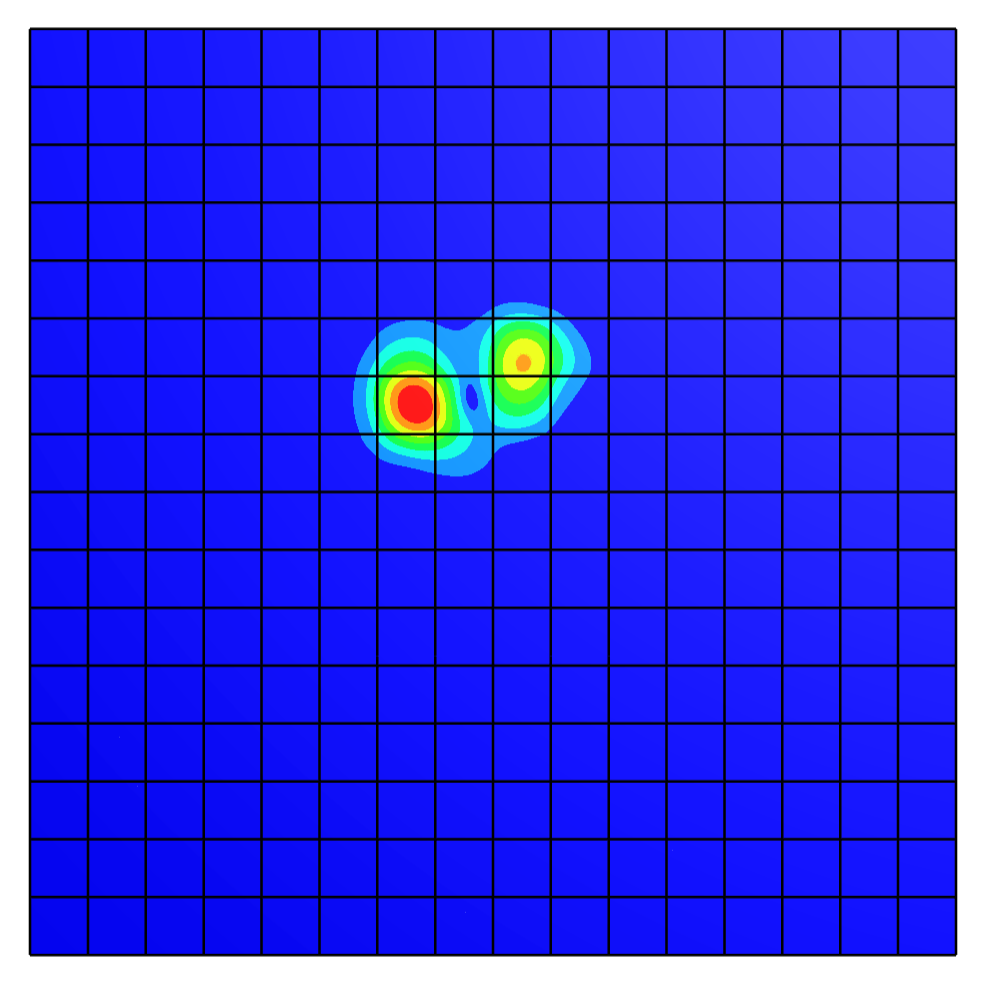}
\includegraphics[width=.16\textwidth]{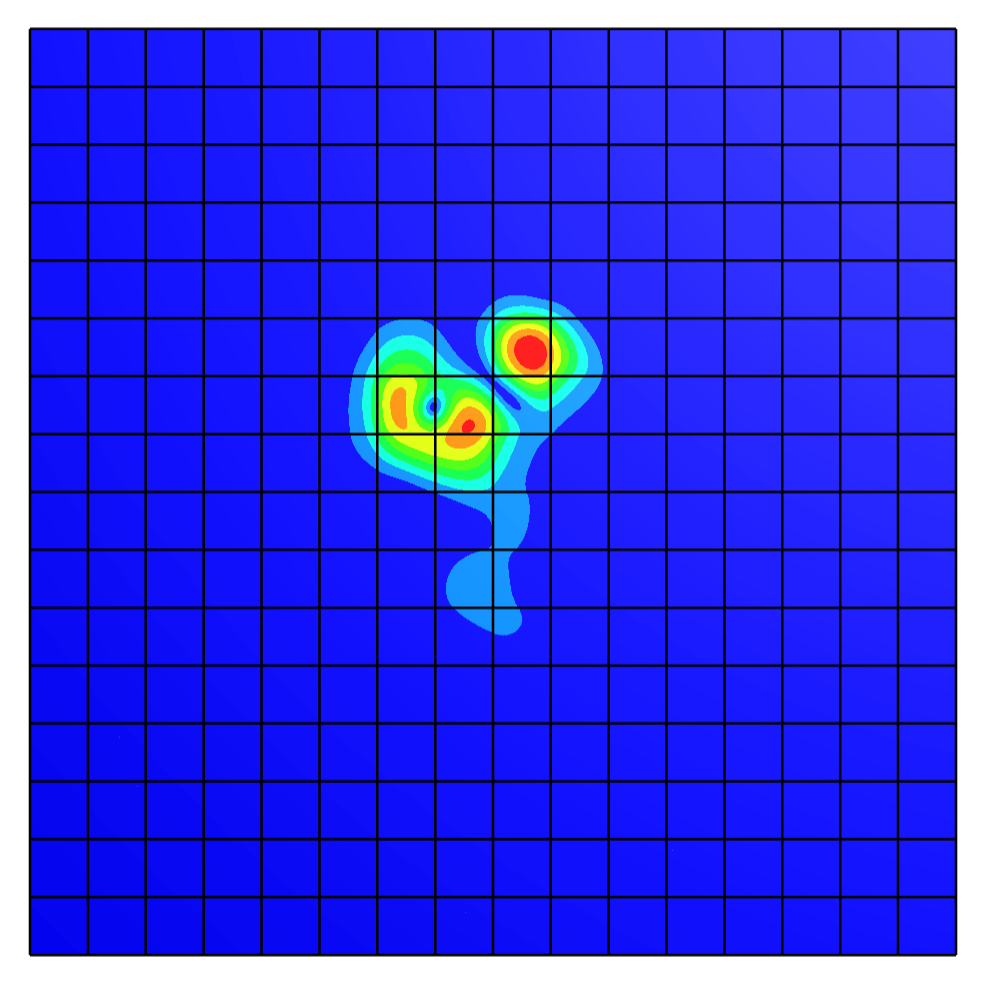}
\includegraphics[width=.16\textwidth]{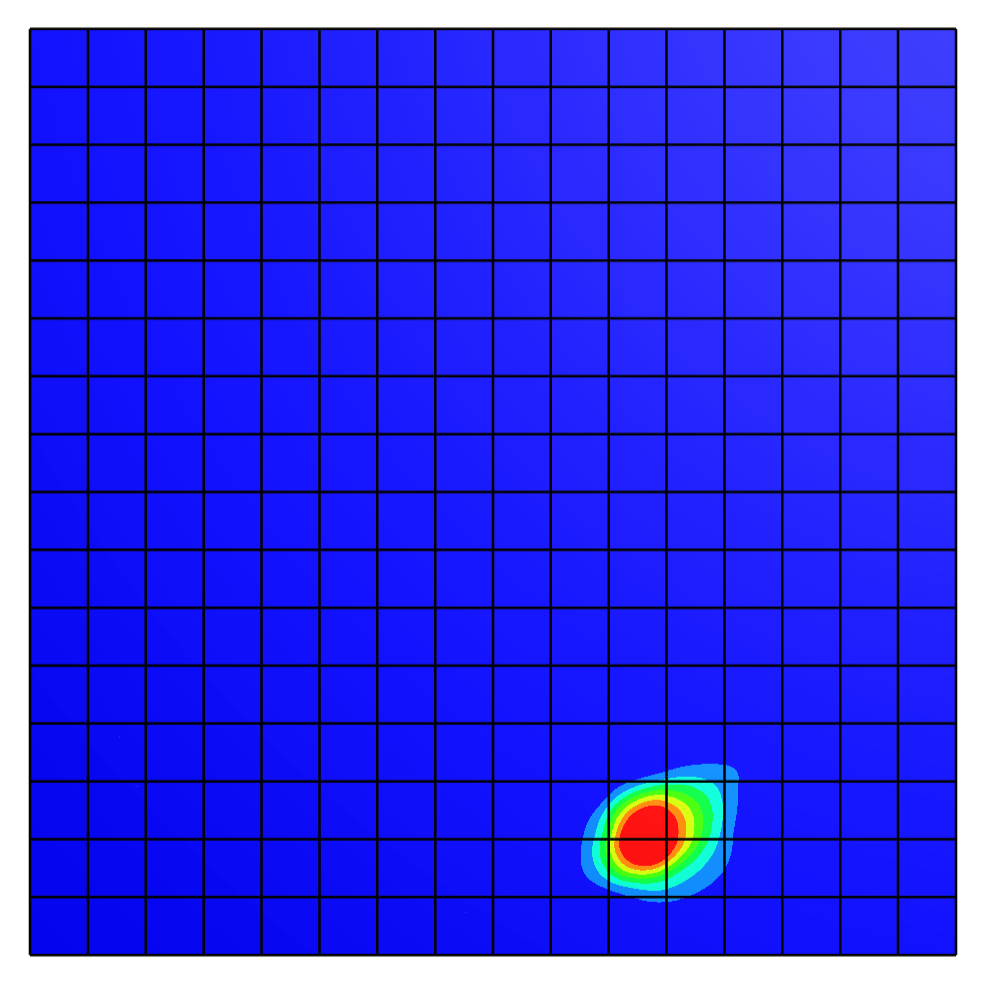}
\includegraphics[width=.16\textwidth]{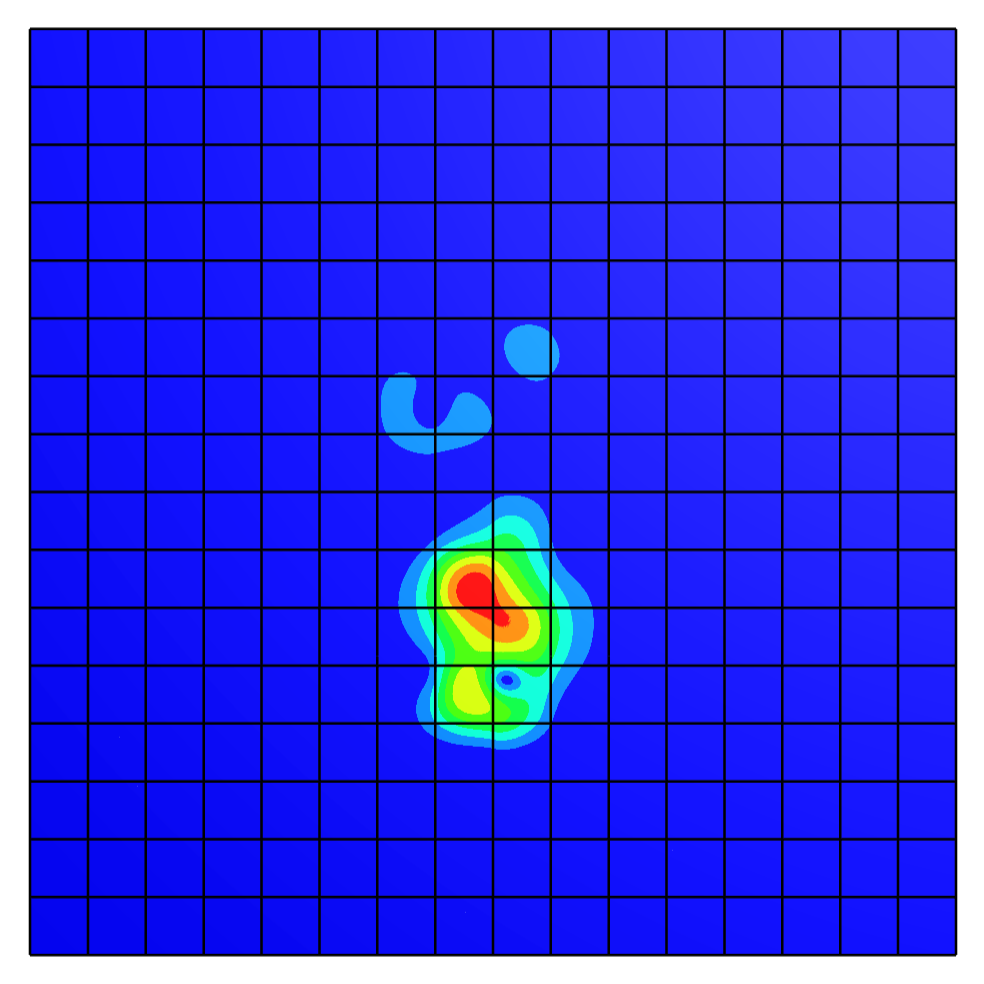}
\caption{The computed eigenvectors $|\phi_j|$ of $H(\mb{A},V)$, $1\leq j\leq 6$, when
     $h=0.01$, for Example 2b: $V^*=100$ (top row),
     $V^*=500$ (middle row), and $V^*=1000$. Compare with
     the top row of Figure~\ref{Ex2aFig2}.\label{Ex2bFig2}}
\end{figure}

The data given in Table~\ref{Ex2bTable1} shows that, even with the
presence of the scalar potential $V$, computing the eigenpairs of
$H(\mb{F},V)$ remains the better option in the sense that it provide
more accurate and more stable eigenvalue approximations on coarser
meshes.  \change{We observe the $L^2(\Omega)$-norm behavior indicated
  in~\eqref{Heuristic1}-\eqref{Heuristic3}  in Table~\ref{Ex2bTable2},
  for each of the three choices of $V^*$.}

\begingroup  
\renewcommand{\arraystretch}{1.5}
\begin{table}
\caption{Computed eigenvalues and timing information for Example
  2b.\label{Ex2bTable1}}
\begin{small}
\begin{center}
  \begin{tabular}{|cc|cll|cccccc|}\hline   
&&$h$ & Total & FEAST& $\lambda_1$ &  $\lambda_2$  & $\lambda_3$ &  $\lambda_4$ & $\lambda_5$ &  $\lambda_6$\\\hline
\parbox[t]{2mm}{\multirow{3}{*}{\rotatebox[origin=c]{90}{$H(\bfA,V)$}}}&\parbox[t]{2mm}{\multirow{3}{*}{\rotatebox[origin=c]{90}{$V^*=100$}}}
&0.01  &  835.32s & 791.97s    &137.200&   147.545 &197.729&   214.967 &240.362&   245.838\\
&&0.03   & 11.44s &9.45s &155.169&   173.436 &215.799&   239.026 &252.172&   257.267\\
&&0.05  & 5.81s & 4.59s & 168.803 & 198.294 &241.373&   265.962 &272.172&   278.125\\
\hline
\parbox[t]{2mm}{\multirow{3}{*}{\rotatebox[origin=c]{90}{$H(\bfF,V)$}}}&\parbox[t]{2mm}{\multirow{3}{*}{\rotatebox[origin=c]{90}{$V^*=100$}}}
&0.01  & 1130.55s& 1056.65s    &137.181&   147.521 &197.714&   214.945 &240.338&   245.832\\
&&0.03   & 7.60s& 6.12s&137.285&   147.677 &198.020&   215.442 &240.790&   245.942\\
&&0.05  & 3.67s & 3.14s&137.502&   148.008 &198.682&   216.5285 &241.700&   246.013\\
\hline\hline
\parbox[t]{2mm}{\multirow{3}{*}{\rotatebox[origin=c]{90}{$H(\bfA,V)$}}}&\parbox[t]{2mm}{\multirow{3}{*}{\rotatebox[origin=c]{90}{$V^*=500$}}}
&0.01  & 401.57s& 358.92s    &224.1576&   273.235 &313.723&   334.855 &374.314&   387.587\\
&&0.03   & 11.66s& 9.39s&242.230&   294.788 &344.036&   351.915 &397.155&   400.898\\
&&0.05  & 6.27s & 5.12s&260.100&   320.014 &361.198&   383.091 &408.173&   432.155\\
\hline
\parbox[t]{2mm}{\multirow{3}{*}{\rotatebox[origin=c]{90}{$H(\bfF,V)$}}}&\parbox[t]{2mm}{\multirow{3}{*}{\rotatebox[origin=c]{90}{$V^*=500$}}}
&0.01  & 431.87s& 357.98s    &224.140&   273.215 &313.6941&   334.843 &374.290&   387.581\\
&&0.03   & 11.97s& 9.33s&224.237&   273.563 &314.400&   335.186 &374.712&   387.637\\
&&0.05  & 5.75s & 4.37s&224.433&   274.276 &315.945&   335.992 &375.607&   387.792\\
\hline\hline
\parbox[t]{2mm}{\multirow{3}{*}{\rotatebox[origin=c]{90}{$H(\bfA,V)$}}}&\parbox[t]{2mm}{\multirow{3}{*}{\rotatebox[origin=c]{90}{$V^*=1000$}}}
&0.01  & 292.54s& 249.45s    &296.412&   354.952 &394.305&   478.797 &483.241&   488.778\\
&&0.03   & 9.56s & 7.50s&313.336&   367.788 &428.568&   494.140 &501.865&   509.388\\
&&0.05  & 5.70s &4.49s&329.695&   385.353 &464.333&   502.269 &531.462&   534.313\\
\hline
\parbox[t]{2mm}{\multirow{3}{*}{\rotatebox[origin=c]{90}{$H(\bfF,V)$}}}&\parbox[t]{2mm}{\multirow{3}{*}{\rotatebox[origin=c]{90}{$V^*=1000$}}}
&0.01  & 328.96s& 256.17s&296.398&   354.941 &394.275&   478.767 &483.235&   488.757\\
&&0.03   & 10.13s& 7.49s&296.537&   355.667 &394.991&   479.497 &483.283&   489.185\\
&&0.05  & 5.74s & 4.31s&296.828&   357.223 &396.527&   481.027 &483.488&   490.216\\
\hline
\end{tabular}
\end{center}
\end{small}
\end{table}
\endgroup

\begin{table}
\caption{Validating ~\change{\eqref{Heuristic2} and~\eqref{Heuristic3}} for Example 2b when $V^{*}\neq 0$.\label{Ex2bTable2}}
\begin{center} 
 \vspace*{2mm}
  \begin{tabular}{|c|c|cccccc|}
    \hline
    $V^*$&&$j=1$&$j=2$&$j=3$&$j=4$&$j=5$&$j=6$\\\hline
    \multirow{2}{*}{100}
    &$\|\nabla \psi_j\|_{L^2(\Omega)}$  &   72.7210 & 78.4483 & 63.3522  & 76.2744 & 74.3651& 51.9391 \\  
    &$\|\mb{A} \psi_j\|_{L^2(\Omega)}$ &   72.4420 & 78.1542 & 62.7169 & 75.6253 & 73.5758 & 51.4715 \\  
    \hline
    \multirow{2}{*}{100}
    &$\|\nabla \phi_j\|_{L^2(\Omega)}$  &   30.8666 & 34.8089 & 35.1378  & 41.5322 & 38.5640 & 21.4864 \\ 
    &$\|\mb{F} \phi_j\|_{L^2(\Omega)}$  &   30.2012 & 34.1386 & 33.9774  & 40.3259 & 37.0173 & 20.3293 \\ \hline
    \hline
    \multirow{2}{*}{500}
    &$\|\nabla \psi_j\|_{L^2(\Omega)}$  &   72.8499 & 73.1349 & 83.2860  & 64.7322& 71.1297 & 52.4223  \\  
    &$\|\mb{A} \psi_j\|_{L^2(\Omega)}$ &    72.3454& 72.6661& 82.5599  & 63.8395 & 69.9942 & 51.8223 \\  
    \hline
    \multirow{2}{*}{500}
    &$\|\nabla \phi_j\|_{L^2(\Omega)}$  &   30.3262 & 39.2781 & 43.9434  & 34.4332 & 35.8518 & 18.3745 \\ 
    &$\|\mb{F} \phi_j\|_{L^2(\Omega)}$  &   29.0915 & 38.3966 & 42.5488  & 32.7229 & 33.5400 & 16.5846\\ \hline
    \hline
    \multirow{2}{*}{1000}
    &$\|\nabla \psi_j\|_{L^2(\Omega)}$  &     71.3297& 61.4968 & 84.4442  & 77.5003 & 52.7697 & 75.3932  \\  
    &$\|\mb{A} \psi_j\|_{L^2(\Omega)}$ &    70.5915 & 60.4635 & 83.4330  & 75.9612& 51.9162 & 74.4759 \\  
    \hline
    \multirow{2}{*}{1000}
    &$\|\nabla \phi_j\|_{L^2(\Omega)}$  &  31.3115 & 43.0073 & 42.9838 & 39.9066 & 19.0242 & 35.7392\\ 
    &$\|\mb{F} \phi_j\|_{L^2(\Omega)}$  &     29.5897 & 41.5157 & 40.9590 & 36.8256 & 16.5088 & 33.7592 \\ \hline
  \end{tabular}
  \end{center}
\end{table}

\subsection{Example 3}
We take the constant-norm vector field 
\begin{align}\label{Ex3A}
  \bfA =-100(\cos( f), \sin(f))\quad,\quad f = \pi\sin(\pi x)\cos(\pi y)~,
\end{align}
on the domain $\Omega=(-1,1)\times(-1,1)$, see
Figure~\ref{VectorFields}.   In fact, we consider a few variations of
this problem, in which we allow holes in the domain, and impose
different boundary conditions on these holes.  A summary of the
different configurations is given below, using $B(\mb{c},r)$ to denote
the disk of radius $r$ centered at $\mb{c}$:
\begin{itemize}
\item $\Omega=(-1,1)\times(-1,1)$ with Dirichlet boundary conditions.
\item $\Omega_1^D = \Omega \setminus B((0,1/2),0.1)$ with Dirichlet boundary conditions.
\item  $\Omega_1^M = \Omega \setminus B((0,1/2),0.1)$ with Dirichlet boundary conditions on $\partial\Omega$ and
  Neumann conditions on $\partial B((0,1/2),0.1)$.
\item $\Omega_2=\Omega \setminus \left(B((-0.8,-0.5),0.1)\cup
    B((0.8,-0.5),0.1)\right)$ with
  Dirichlet conditions.
\item $\Omega_3=\Omega\setminus \left(B((-0.8,-0.5),0.1)\cup
    B((0.8,-0.5),0.1)\cup [-0.05, 0.15] \times [0.4, 0.6]\right)$ with
  Dirichlet conditions.
\end{itemize}
The norms of $\mb{A}$ and $\mb{F}$ are given in Table~\ref{Ex3Table1},
where we see the these variations have a negligible affect on them. 
\begin{table}
  \caption{The norms of $\mb{A}$ and $\mb{F}$ on the various domains
    in Example 3, computed with $h=0.01$.\label{Ex3Table1}}
  \centering
\begin{tabular}{|c|cccc|}\hline
  &$\Omega$&$\Omega_1$&$\Omega_2$&$\Omega_3$\\\hline
  $\|\mb{A}\|_{L^2(\cdot)}$&200&199.2214&198.4256&197.4151\\
  $\|\mb{F}\|_{L^2(\cdot)}$&118.9613&118.6121&117.7474&117.5523\\\hline
\end{tabular}
\end{table}

The decision to include holes in the domain reflects the fact that not
all curl-free vector fields are gradients in multiply-connected
domains, so vector fields having the same curl may differ by more than
a gradient.  Lemma~\ref{lemma1} established an equivalence class of
gauges of $\mb{A}$ that differ from each other only by gradients, and
Theorem~\ref{CanonicalGauge} asserts that our choice of gauge has
minimal norm in this class.  If we were to expand the equivalence
class to all include vector fields having the same curl as $\mb{A}$,
the one having minimal norm for $\Omega_j$ would likely be different.
Whether or not broadening the definition of the equivalence class
would allow for further improvements in computational efficiency is a
topic we do not take up here.  The locations of the holes are chosen
to interfere with where eigenvectors localize in comparison to the
``base case'' $\Omega$.  We show the first six eigenvectors, in
modulus, for each of the configurations, in Figure~\ref{Ex3Fig1}.  In
$\Omega_3$, the square is a little off center, in order to break
symmetry in eigenvectors that would have localized there in the base
case, and we see that it does so for its $\phi_3$.

\begin{figure}
  \centering
 \includegraphics[width=.16\textwidth]{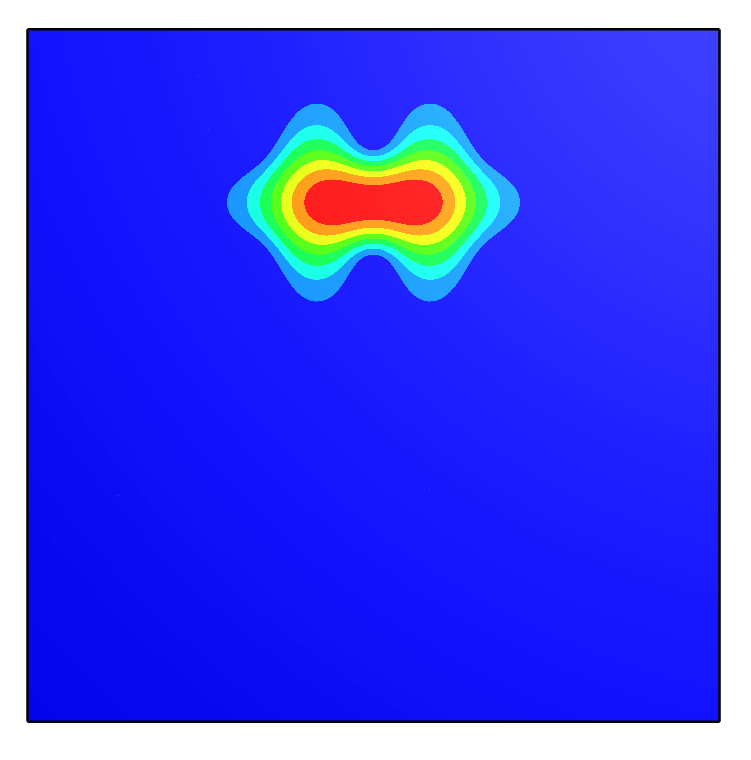}
  \includegraphics[width=.16\textwidth]{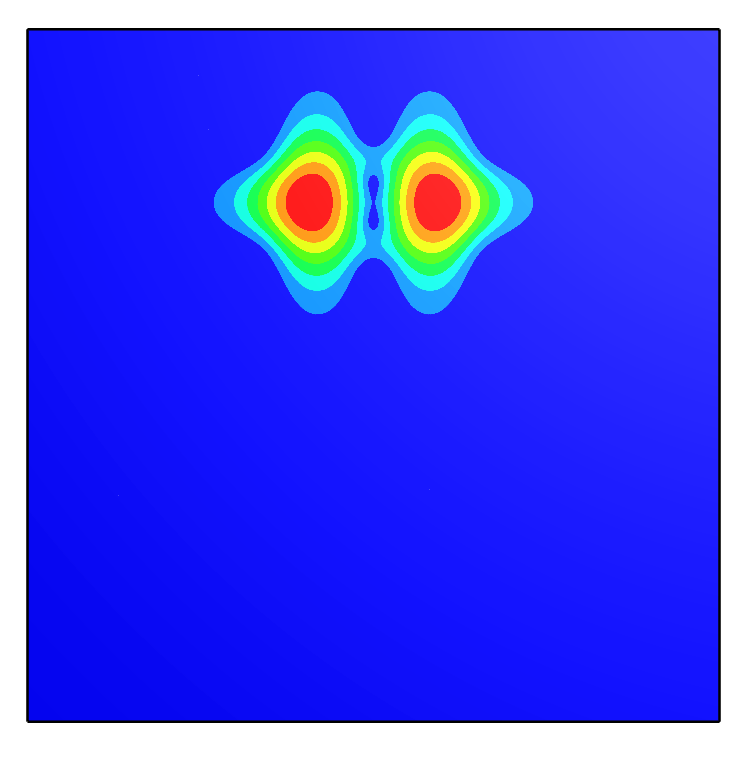}
  \includegraphics[width=.16\textwidth]{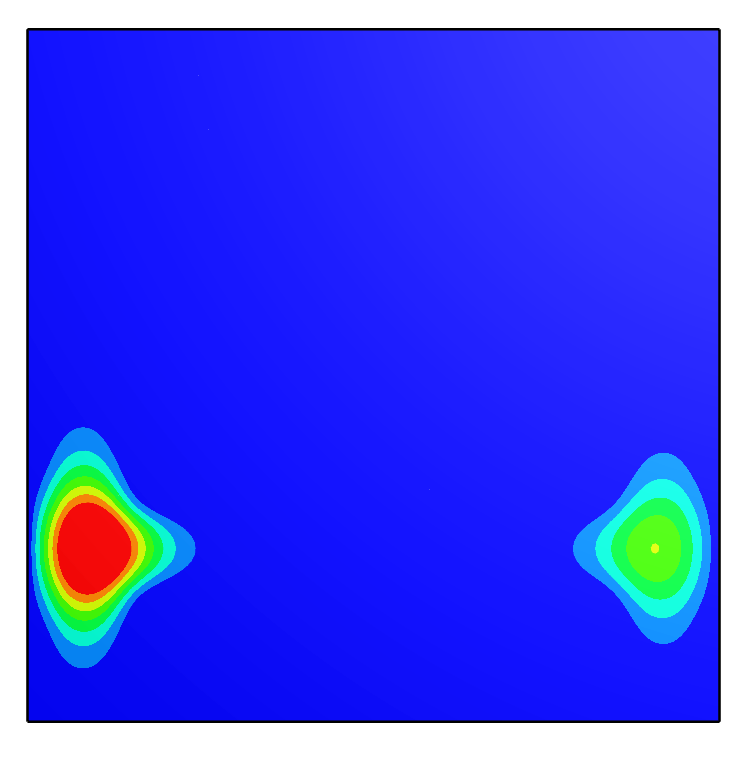}
  \includegraphics[width=.16\textwidth]{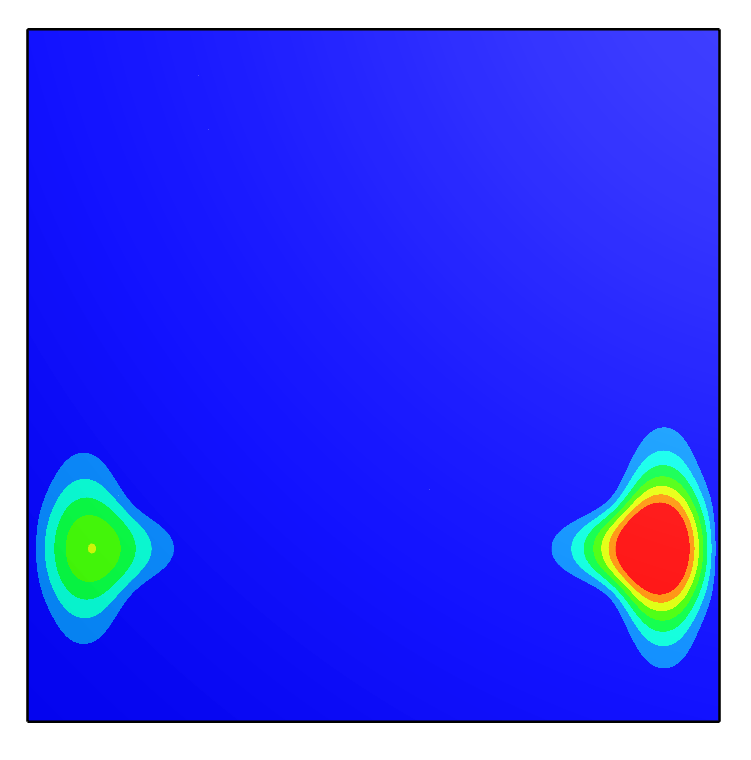}
  \includegraphics[width=.16\textwidth]{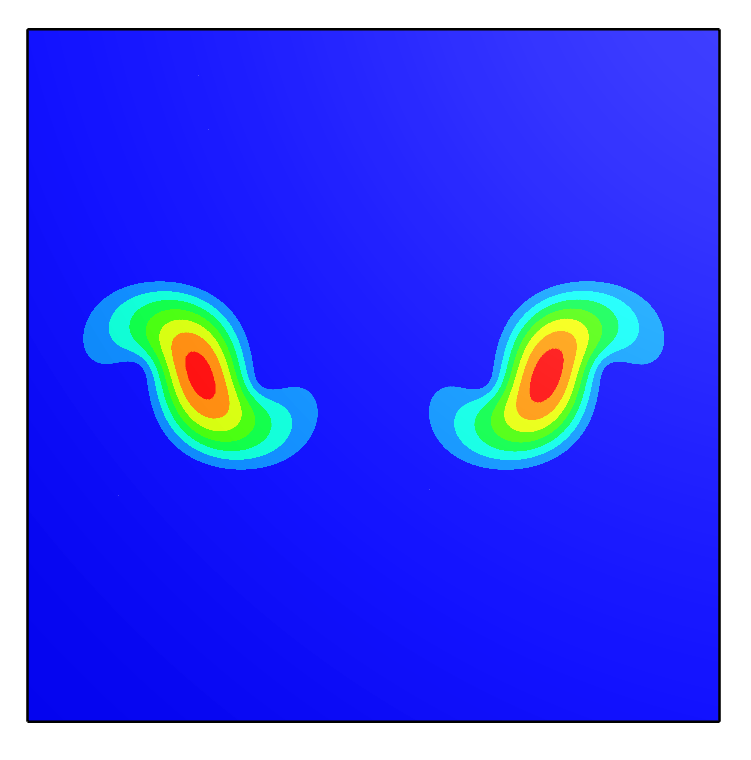}
  \includegraphics[width=.16\textwidth]{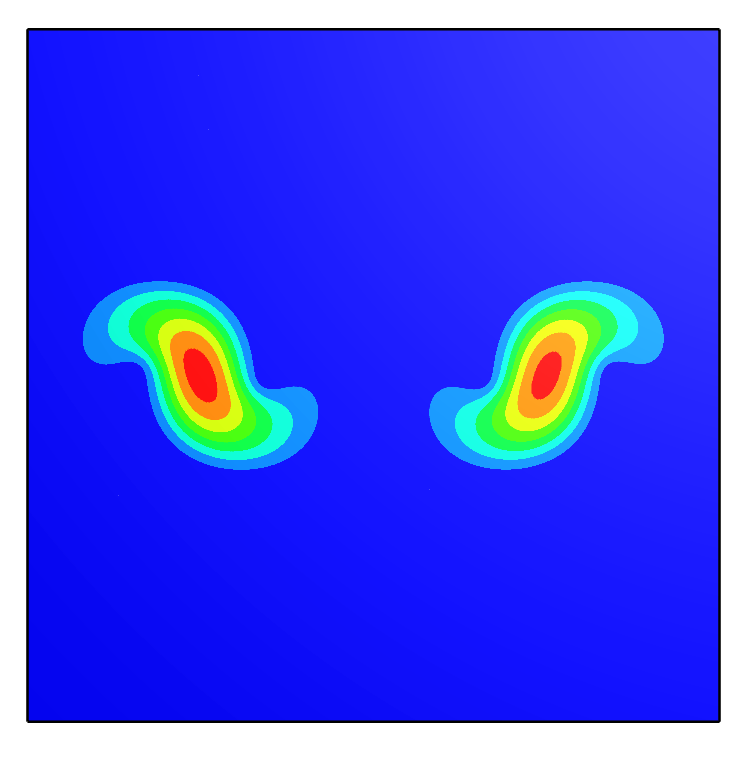}
\includegraphics[width=.16\textwidth]{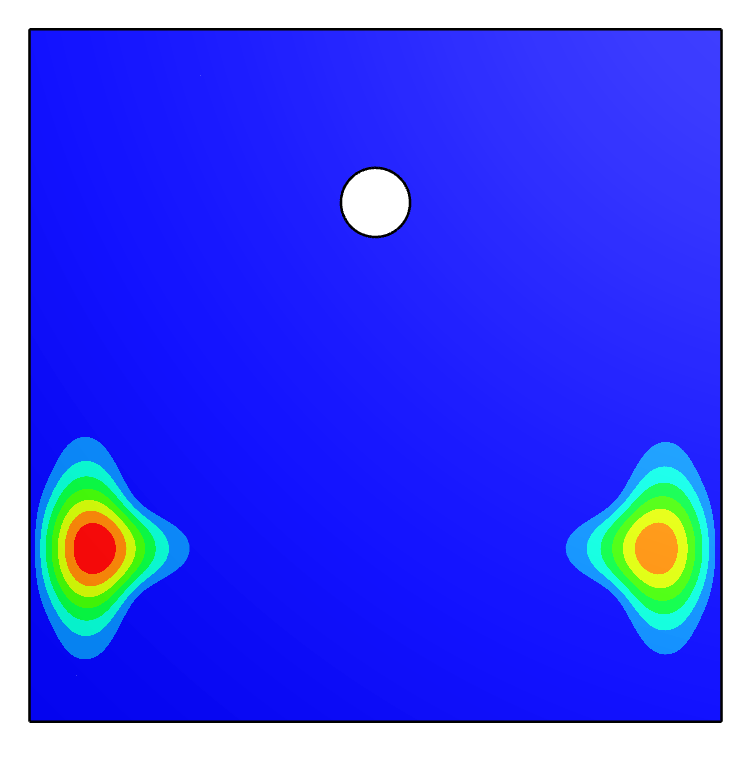}
\includegraphics[width=.16\textwidth]{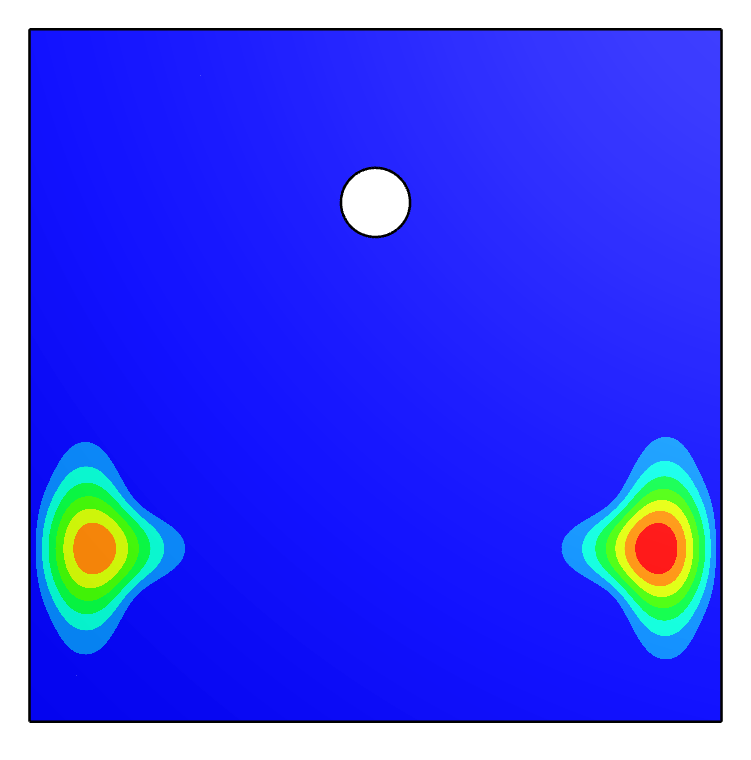}
\includegraphics[width=.16\textwidth]{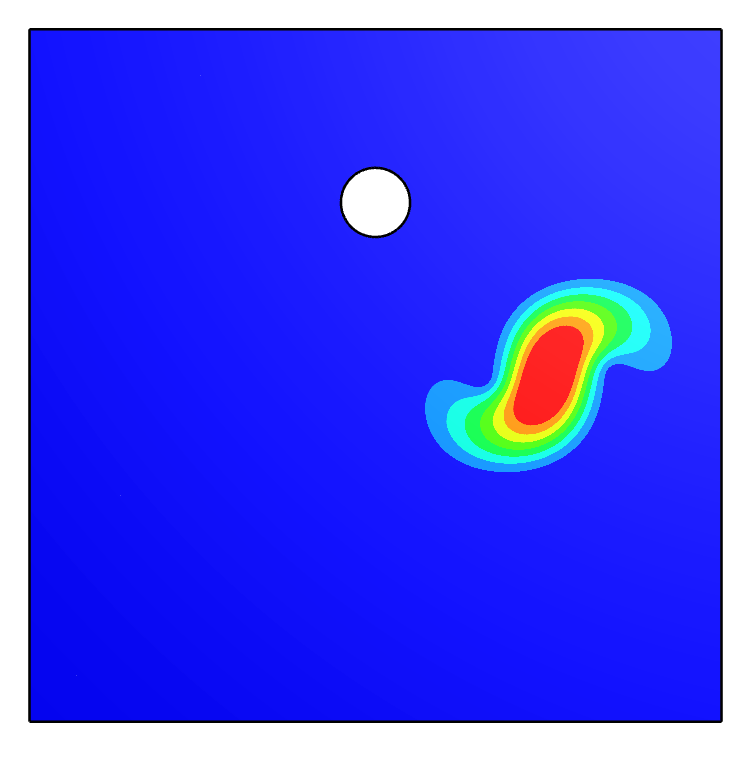}
\includegraphics[width=.16\textwidth]{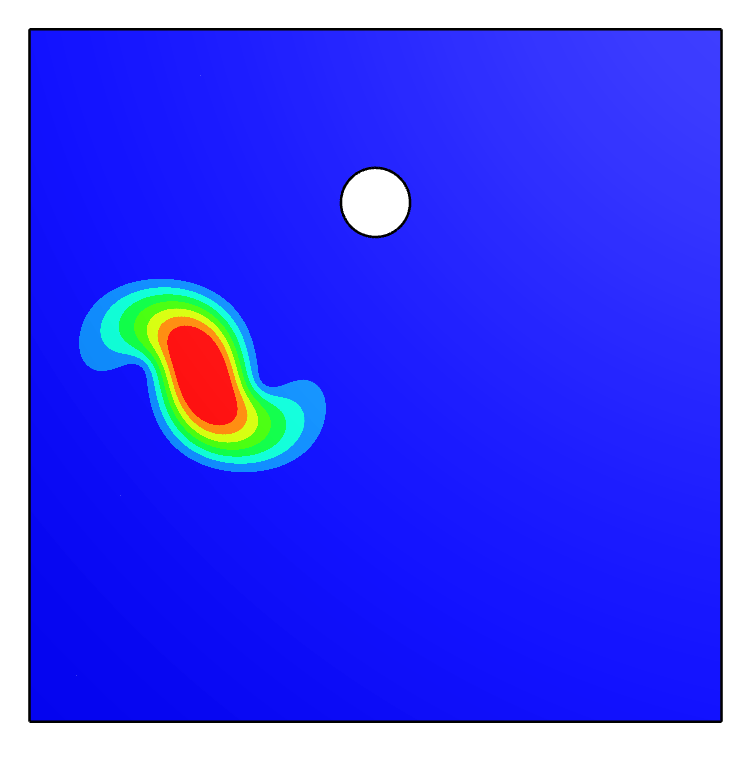}
\includegraphics[width=.16\textwidth]{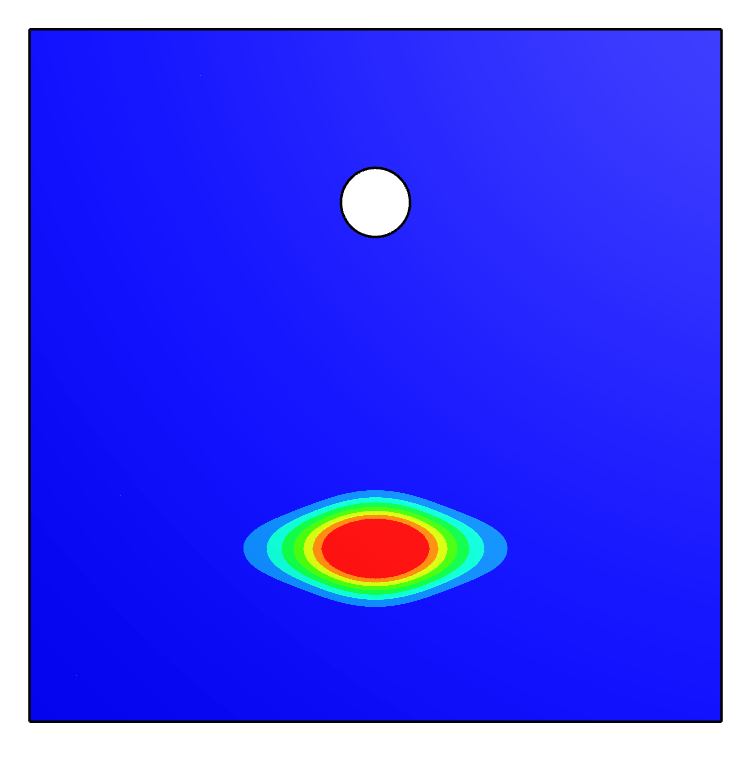}
\includegraphics[width=.16\textwidth]{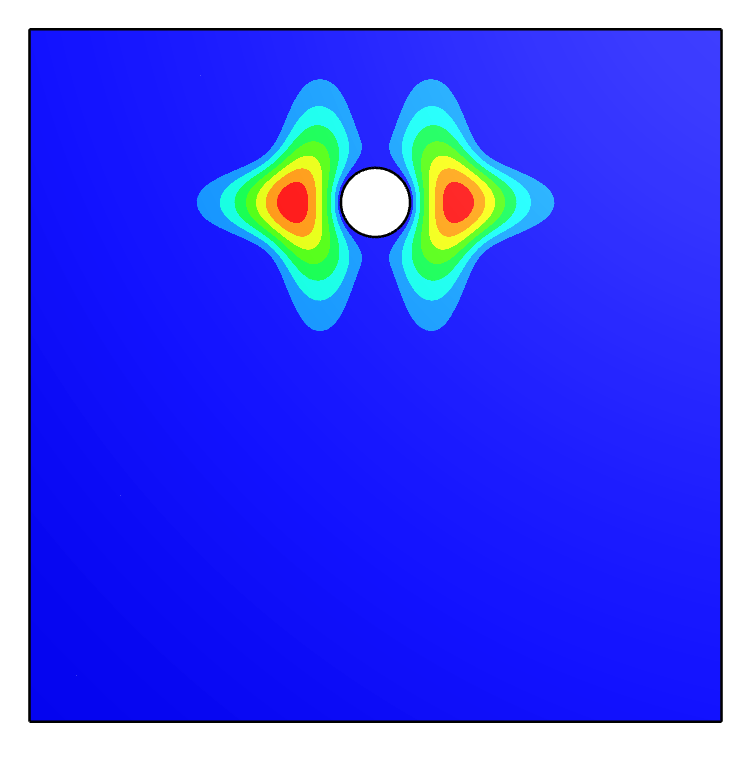}
\includegraphics[width=.16\textwidth]{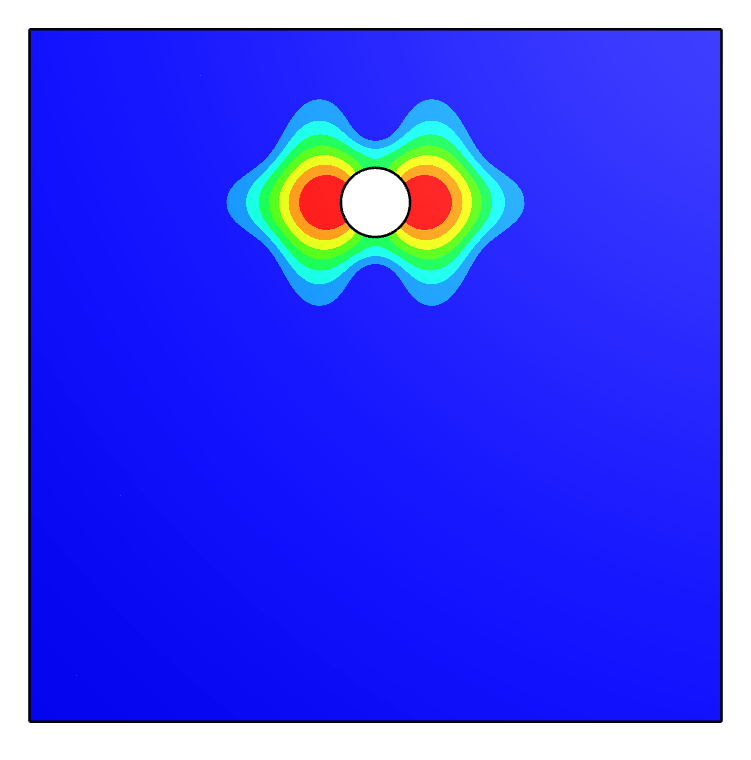}
\includegraphics[width=.16\textwidth]{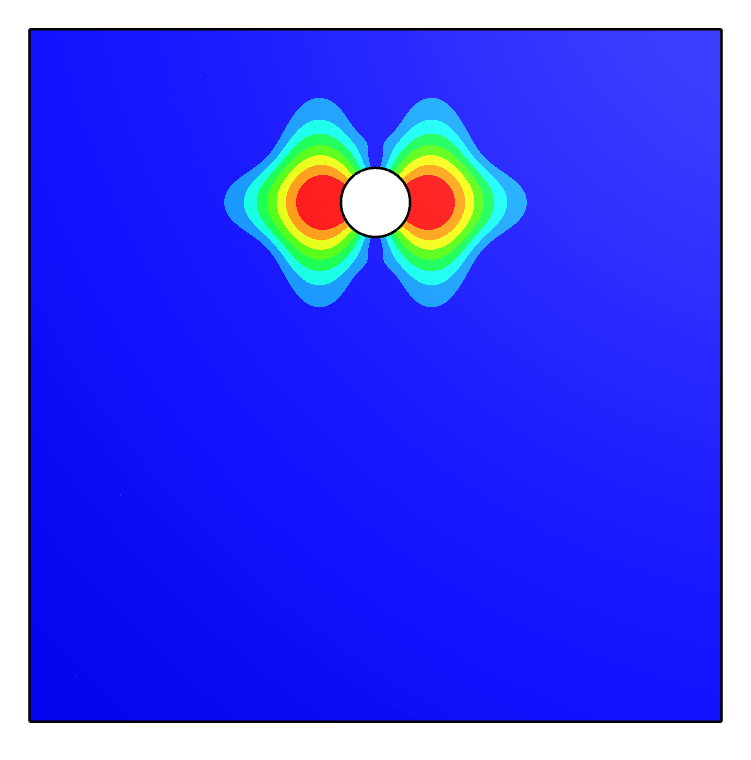}
\includegraphics[width=.16\textwidth]{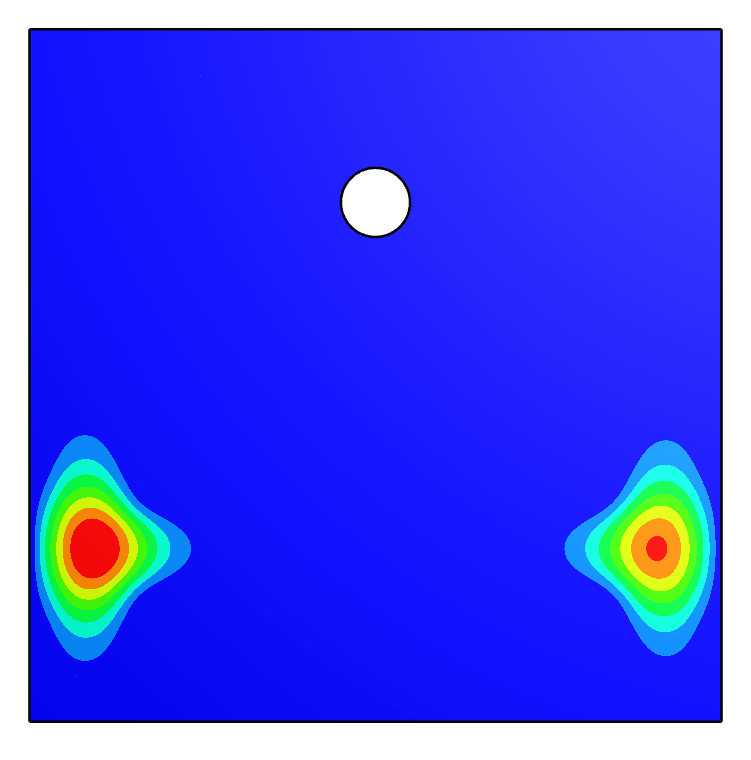}
\includegraphics[width=.16\textwidth]{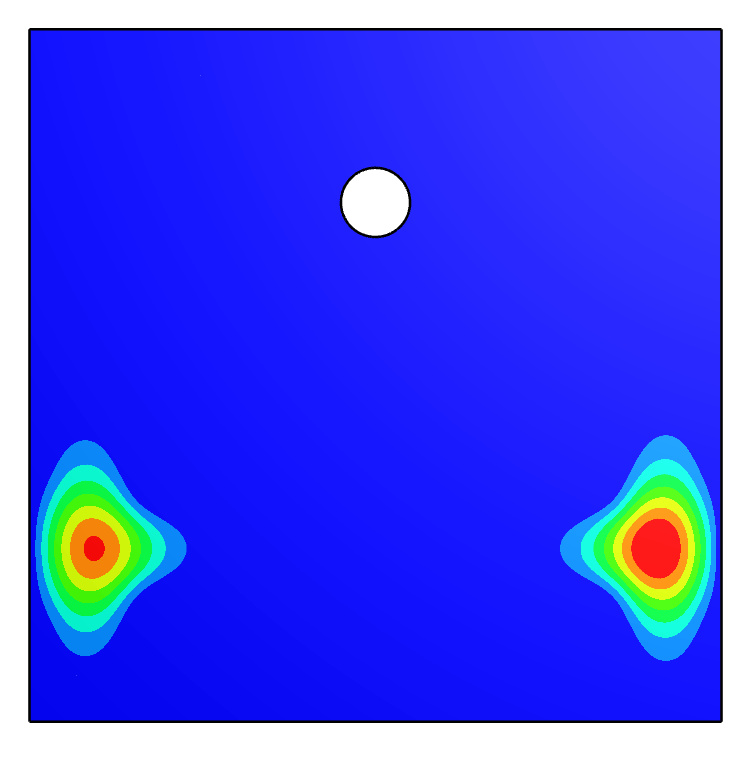}
\includegraphics[width=.16\textwidth]{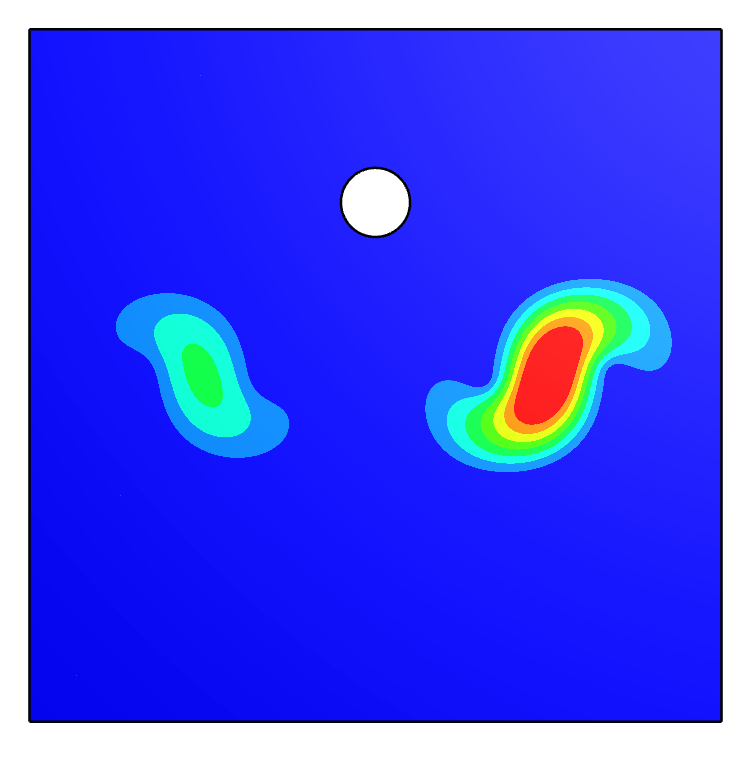}
\includegraphics[width=.16\textwidth]{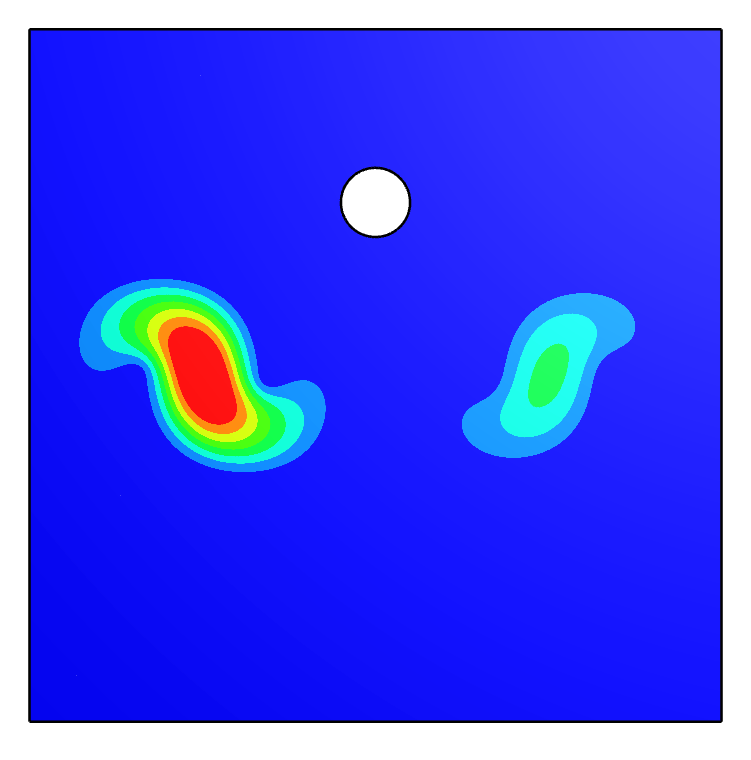}
\includegraphics[width=.16\textwidth]{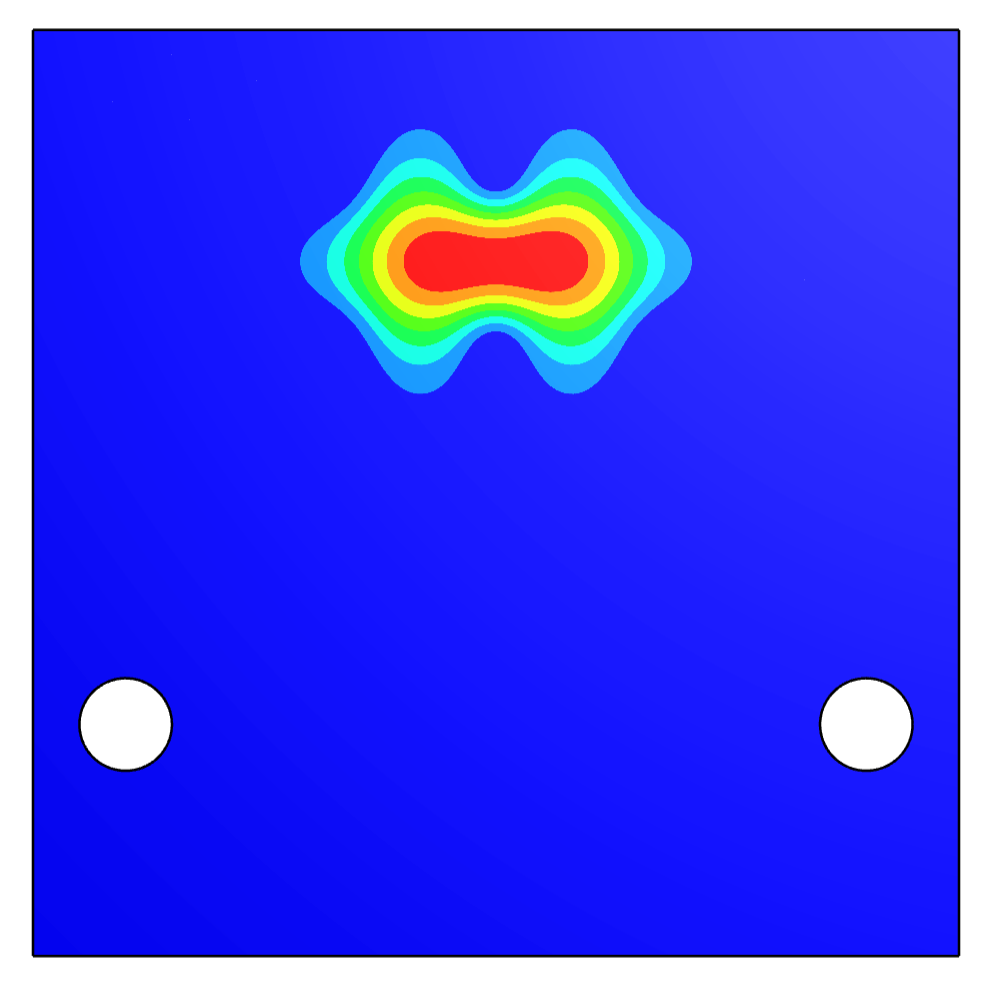}
\includegraphics[width=.16\textwidth]{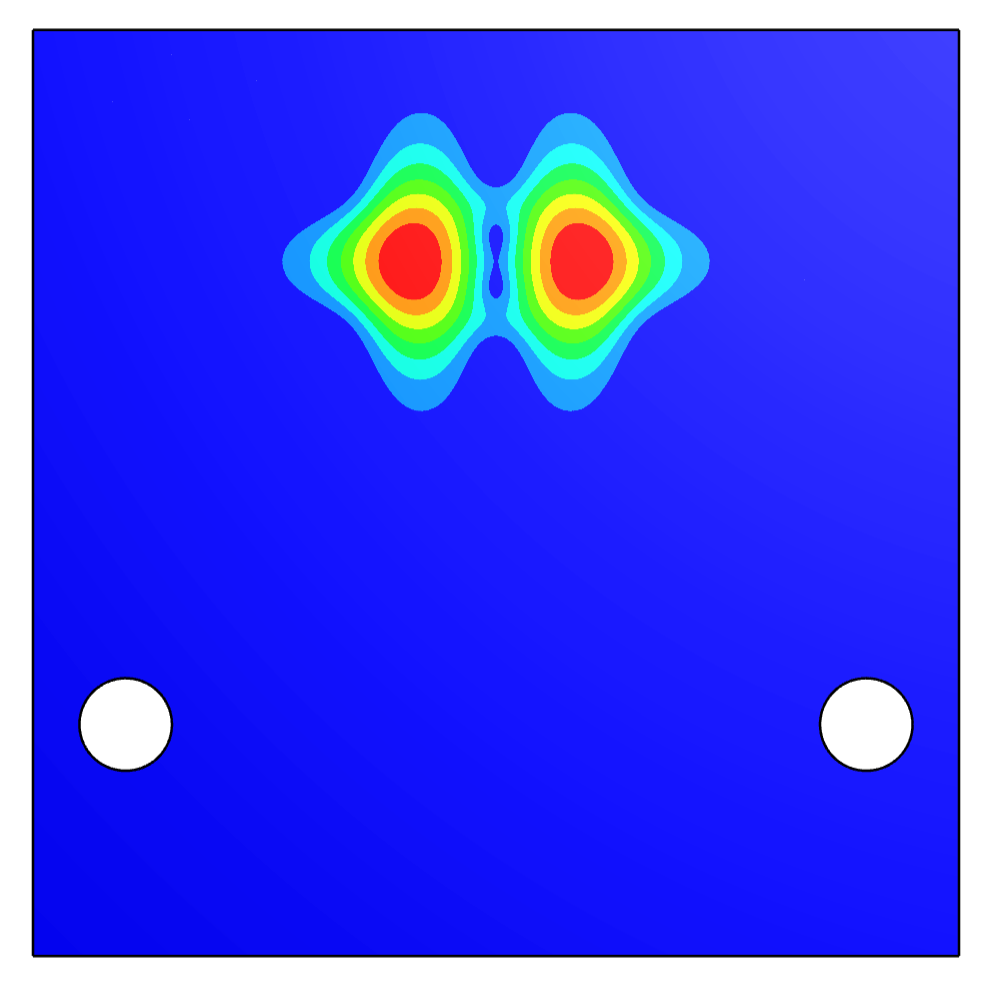}
\includegraphics[width=.16\textwidth]{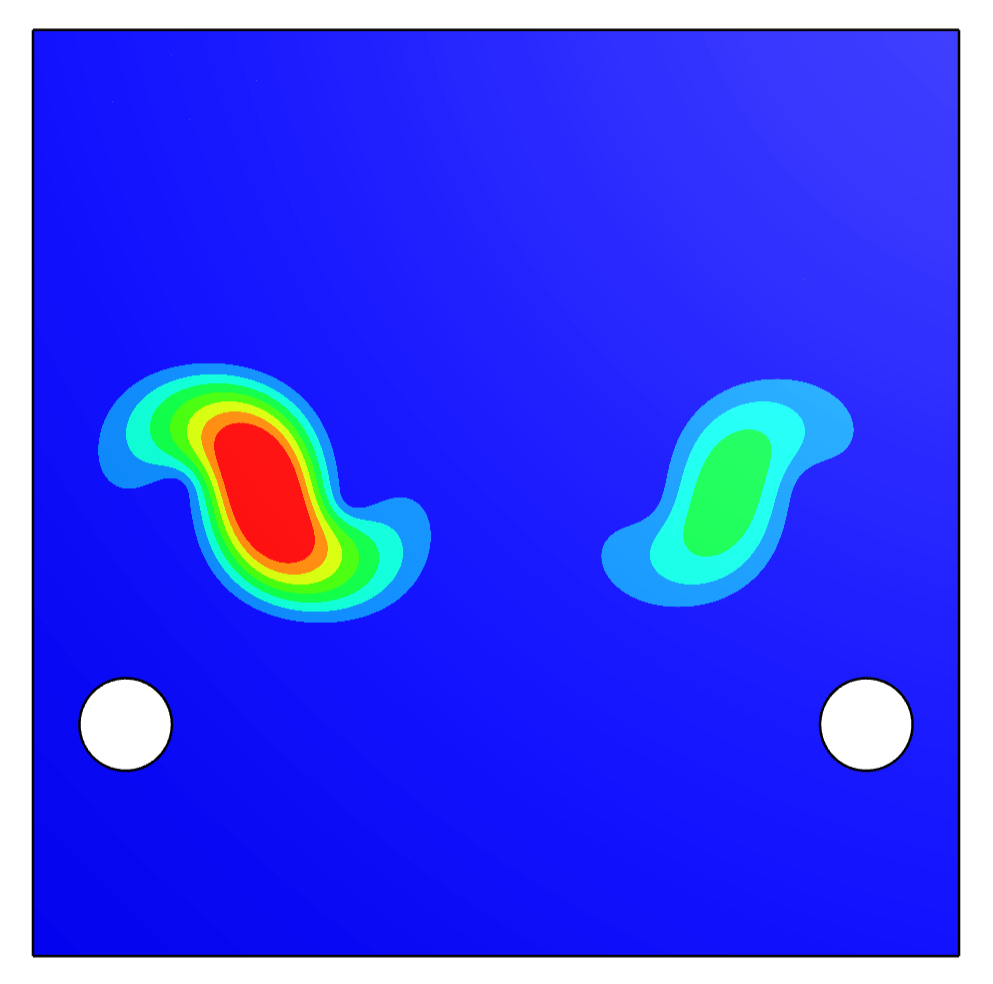}
\includegraphics[width=.16\textwidth]{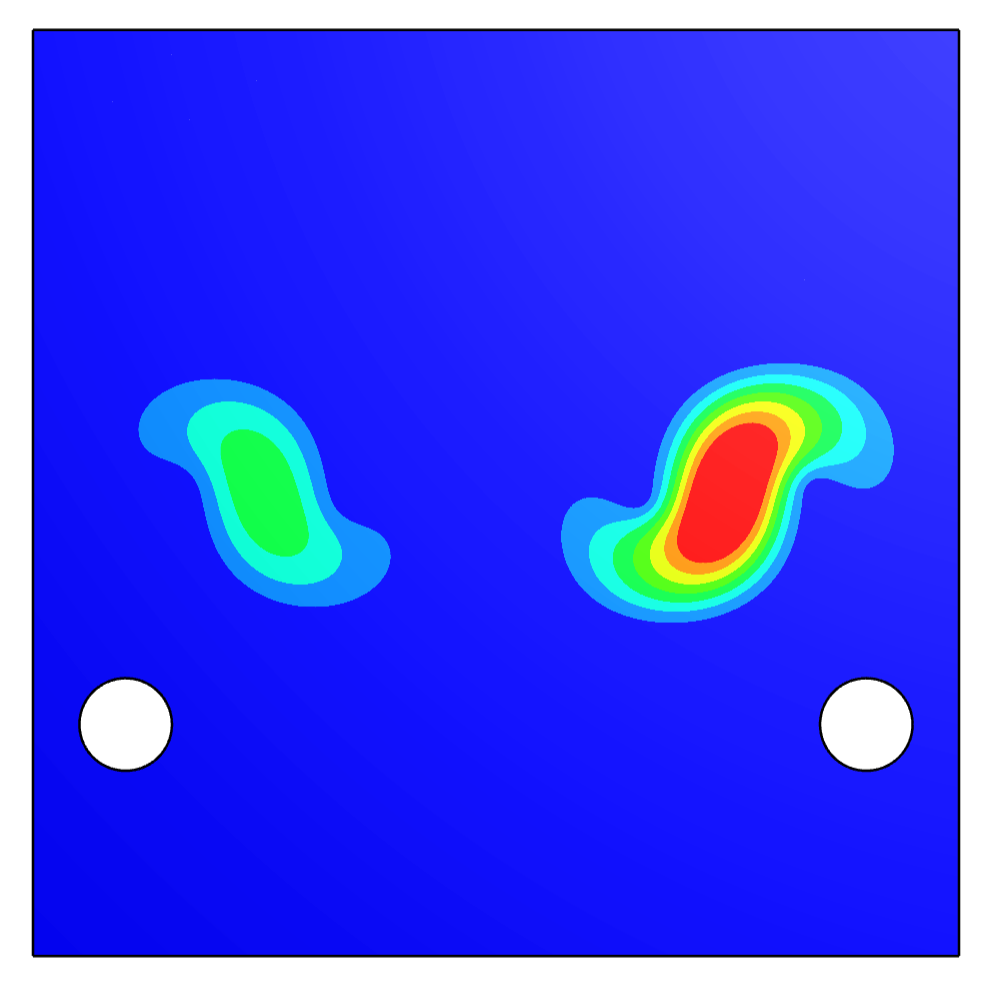}
\includegraphics[width=.16\textwidth]{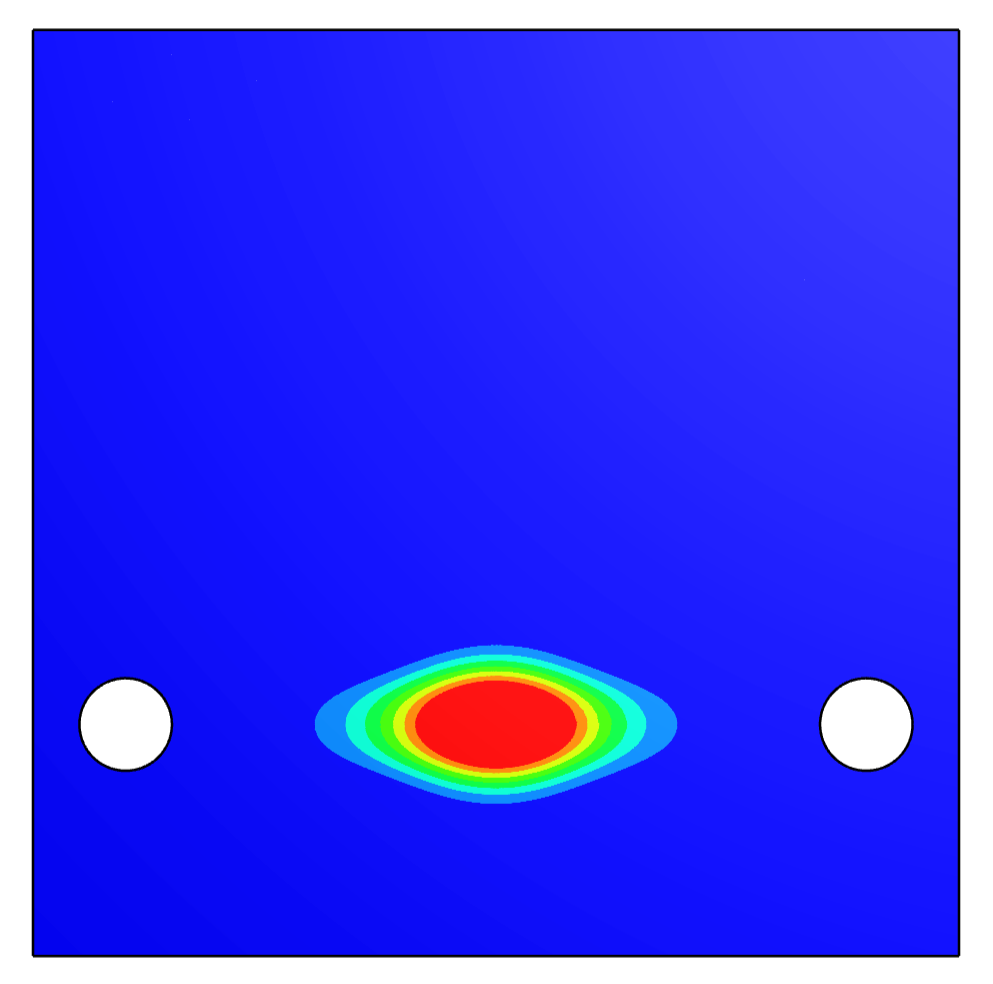}
\includegraphics[width=.16\textwidth]{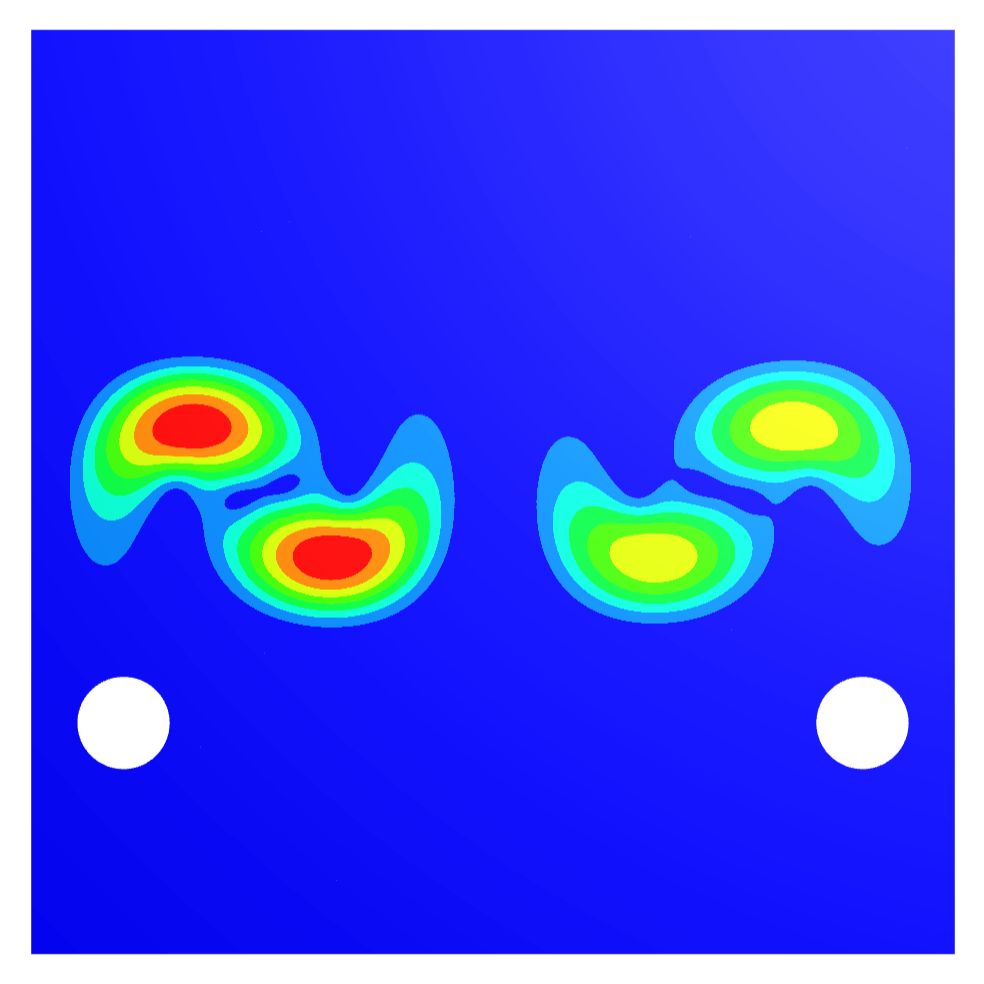}
\includegraphics[width=.16\textwidth]{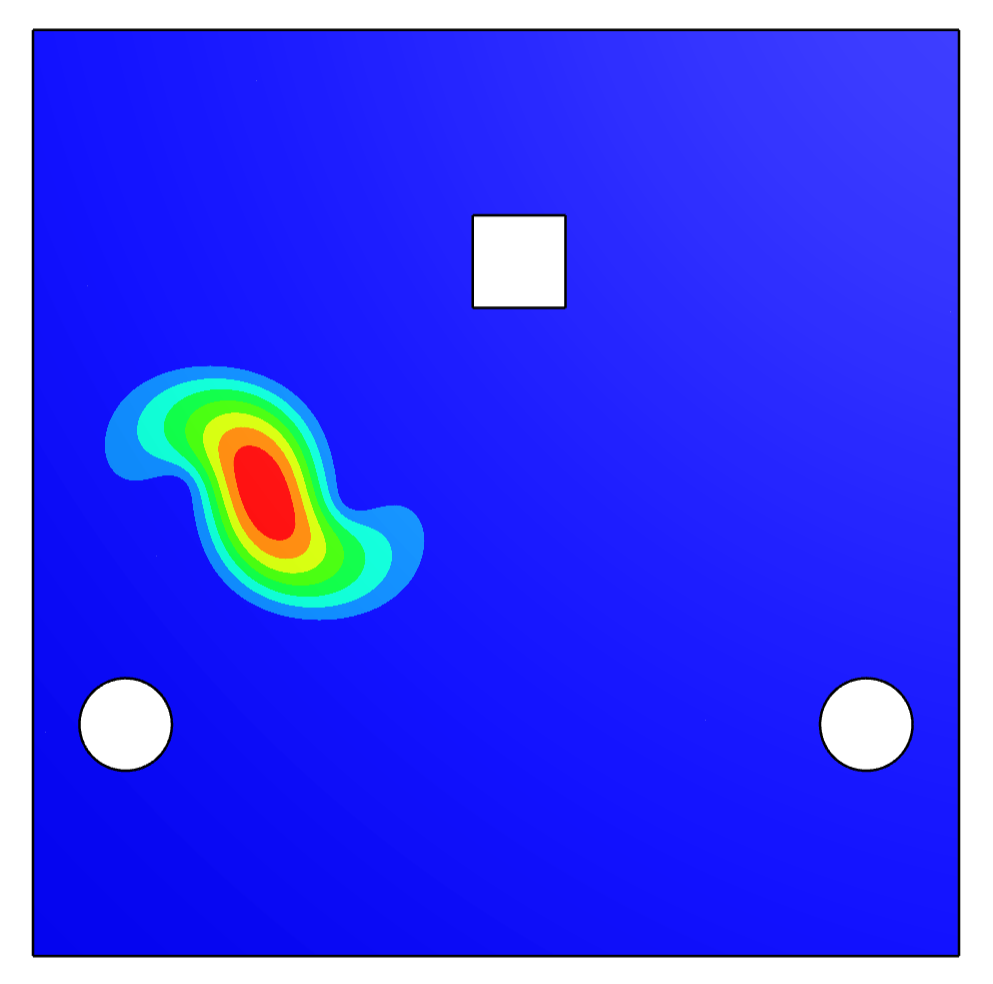}
\includegraphics[width=.16\textwidth]{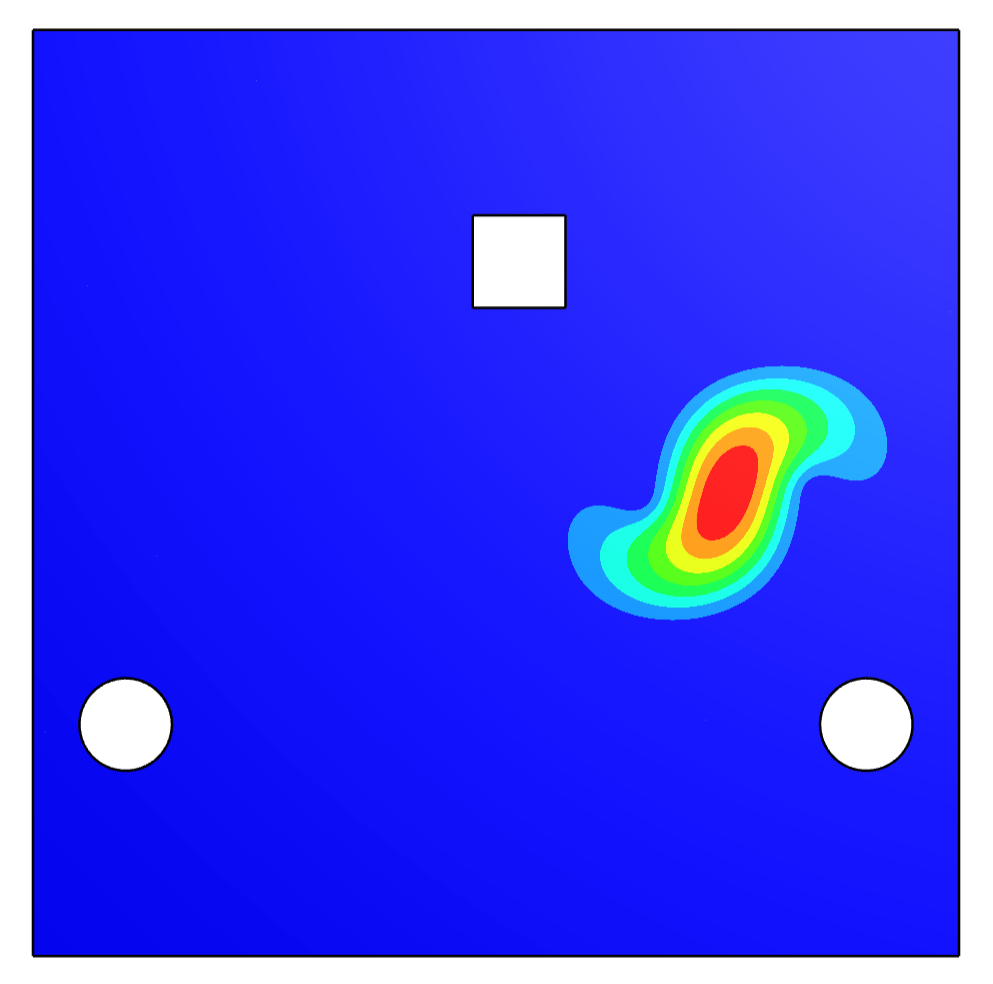}
\includegraphics[width=.16\textwidth]{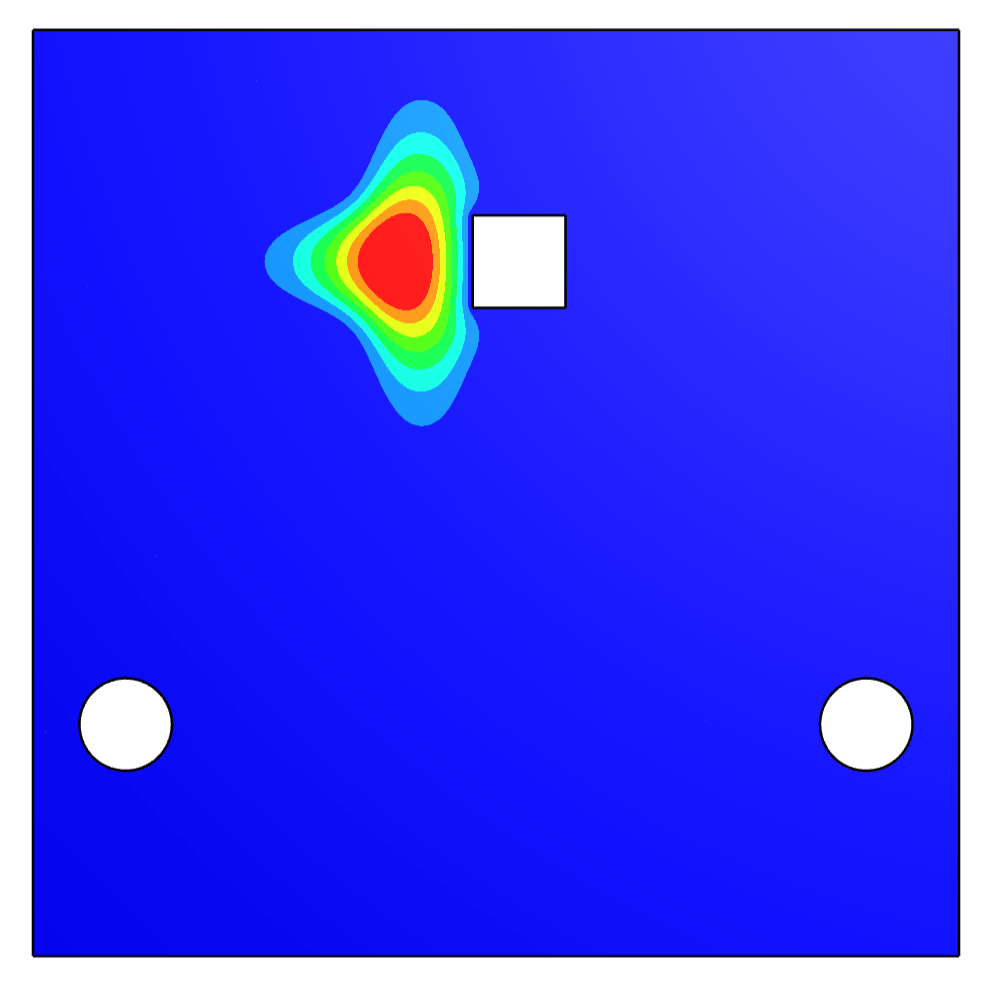}
\includegraphics[width=.16\textwidth]{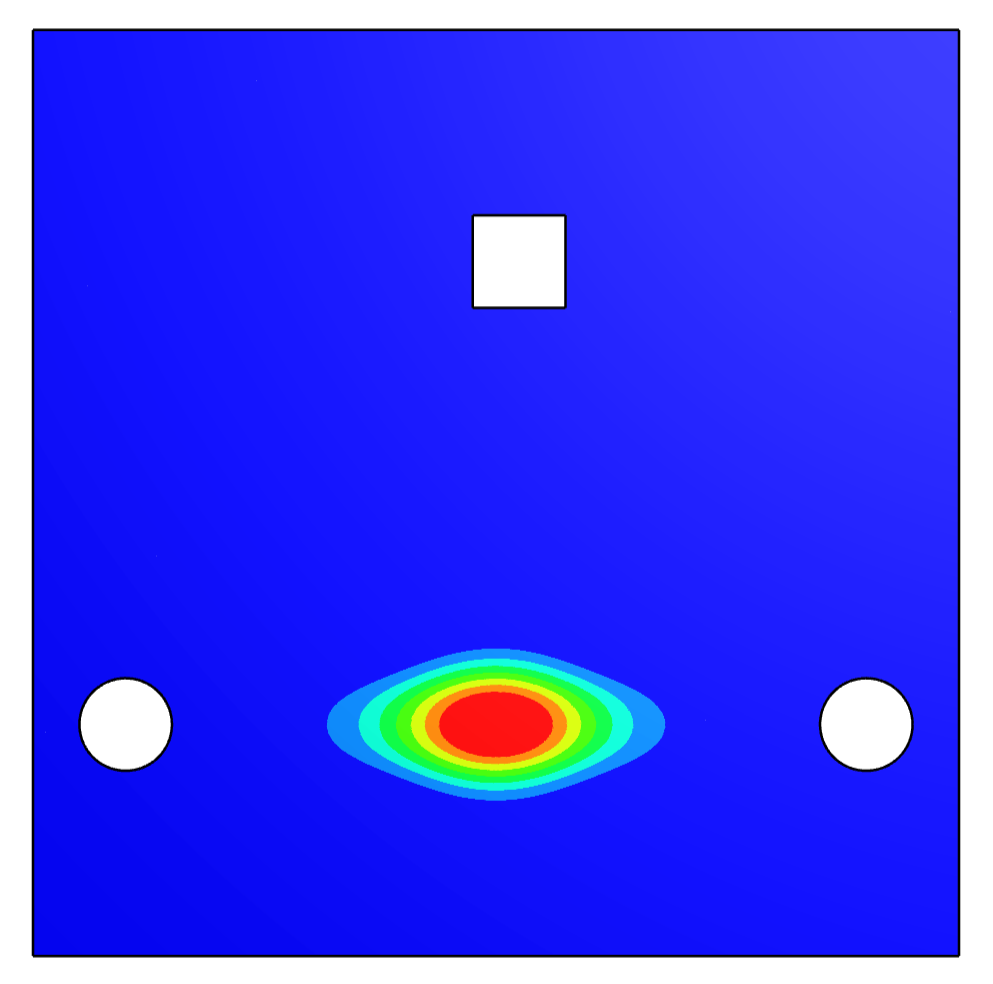}
\includegraphics[width=.16\textwidth]{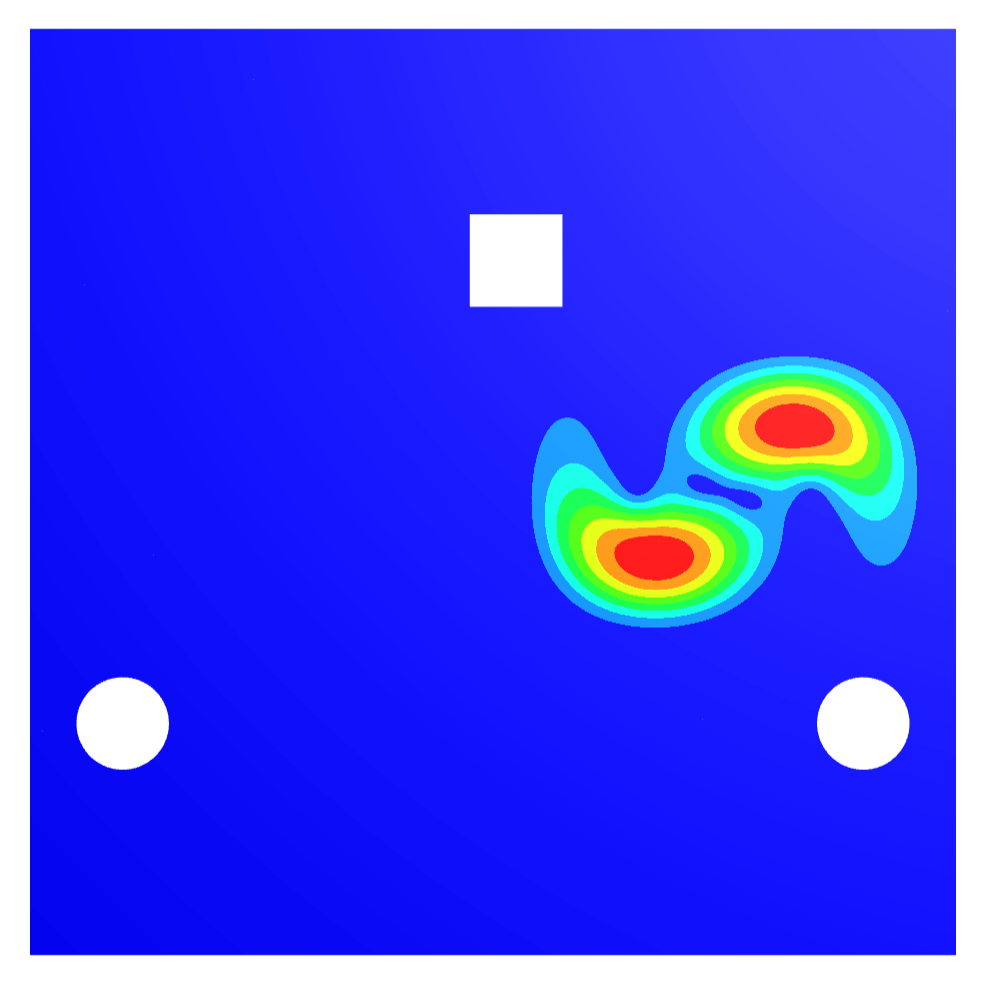}
\includegraphics[width=.16\textwidth]{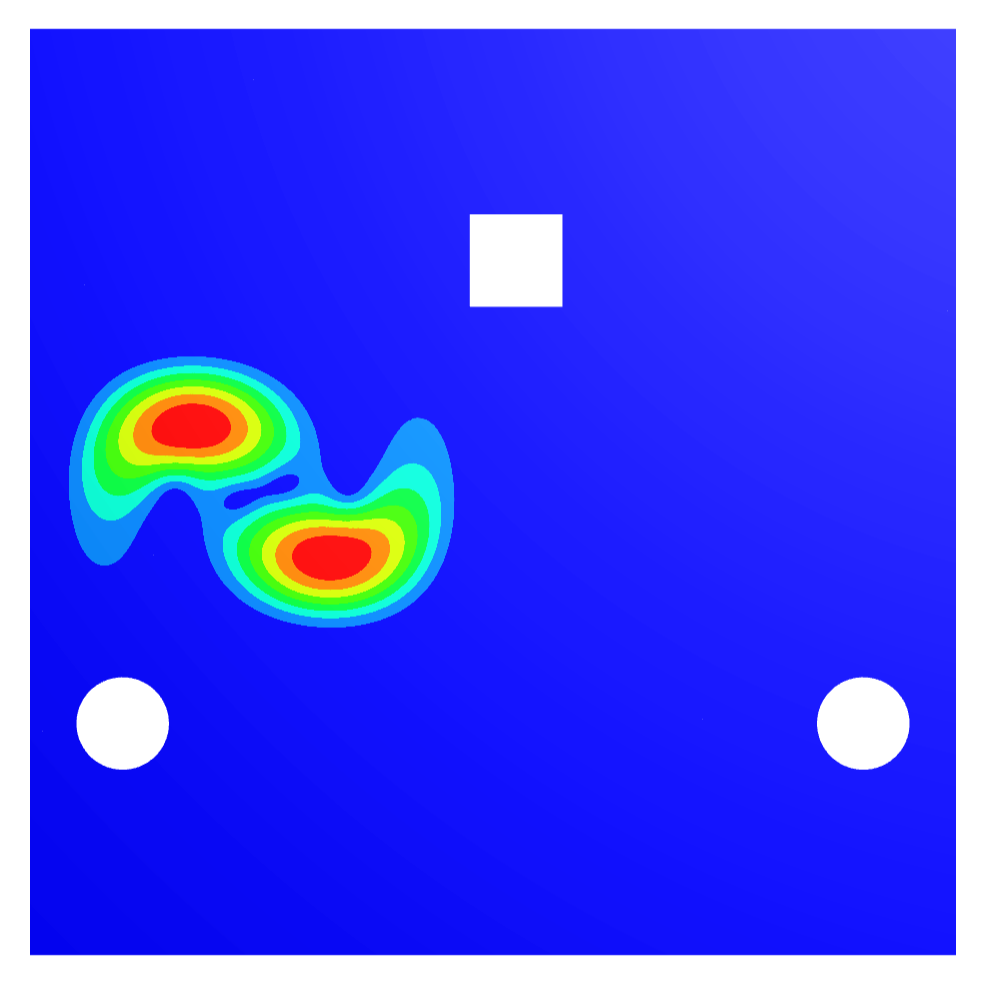}
  \caption{\label{Ex3Fig1}The computed eigenvectors $|\phi_j|$ of
    $H(\mb{F})$, $1\leq j\leq 6$, when $h=0.01$.  The rows from top to
    bottom correspond to $\Omega$, $\Omega_1^D$, $\Omega_1^M$,
    $\Omega_2$ and $\Omega_3$, respectively.}
\end{figure}

As we have seen in previous examples, the computation of eigenvalues
of $H(\mb{F})$ remains (much!) more stable under mesh coarsening than
is the case for $H(\mb{F})$, across all configurations---see
Table~\ref{Ex3Table2}.  We note that this problem does seem more
challenging in the sense that there is more deviation between what is
computed for $H(\mb{F})$ across the three mesh sizes, and both
$H(\mb{A})$ and $H(\mb{F})$ have some approximate eigenpairs change
their order in the line-up even when $h=0.03$, e.g. an apparent
repeated eigenvalue shifting positions.  In prior examples, we
have seen variations on the order of $1$ for eigenvalues on the order
of $10^2$.  Here we see variations as large as $19$ for eigenvalues on
the order of $10^2$ for $H(\mb{F})$---compare the computed values
$\lambda_6$ for $h=0.01$ and $h=0.03$ on $\Omega_3$.  We note,
however, that the computed counterparts for $H(\mb{A})$ differ by over
245!  Again, $H(\mb{F})$ is the clear winner.
\begin{table}
\caption{Computed eigenvalues and timing information for Example 3.
 \label{Ex3Table2}}
\begin{center}
\begin{tabular}{|c|cll|cccccc|}\hline
&$h$ & Total & FEAST& $\lambda_1$ &  $\lambda_2$  & $\lambda_3$ &  $\lambda_4$ & $\lambda_5$ &  $\lambda_6$\\ \hline
$H(\bfA)$&0.01  &330.74s& 282.86s    &94.278&   117.901 &120.607&   120.608 &135.035&   135.035\\
$\Omega$&0.03   &152.81s&  150.05s   &114.358&   139.424 &142.114&   142.475 &157.316&   157.408\\
&0.05  &13.01s&12.27s&306.600&   322.561 &326.090&   344.096 &352.115&   357.653\\
\hline
$H(\bfF)$&0.01  &364.94s& 282.73s    &94.240&   117.860&120.568&   120.568 &134.993&   134.993\\
$\Omega$&0.03   &15.27s&12.19s&94.415&   118.188 &121.615&   121.701 &135.636&   135.639\\
&0.05  &4.94s&4.03s&96.688&   122.132 &129.168&   129.214 &144.567&   144.724\\\hline
\hline
$H(\bfA)$&0.01  &370.12s& 324.13s &120.607& 120.608 &135.053& 135.054 &153.411& 157.023\\
$\Omega_1^D$&0.03   &17.20s&  14.20s &142.393& 142.544 &162.323& 163.678 &173.655& 198.691\\
&0.05  &27.10s&16.30s&330.648& 342.127 &381.4465& 388.992 &399.021& 410.414\\
\hline
$H(\bfF)$&0.01  &362.76s& 285.73s    &120.568& 120.568 &134.994&   134.994 &153.456&   156.896\\
$\Omega_1^D$&0.03   &16.87s&13.27s&121.572& 121.639 &135.939& 136.236 &158.068&   158.328\\
&0.05  &3.56s&2.60s&129.368&   130.461 &147.520&   149.396 &166.672&   166.931\\
\hline
\hline
$H(\bfA)$&0.01  &313.72s& 273.84s &96.042&   100.647 &120.607&   120.608 &135.053&135.054\\
$\Omega_1^M$&0.03   &14.91s&  12.36s   &142.393& 142.544&143.149& 148.959 &162.323& 163.678\\
&0.05  &5.43s&4.60s&330.648&   342.127 &381.446&   388.992 &399.021&   410.414\\
\hline
$H(\bfF)$&0.01  &379.67s& 312.14s    &95.911&   100.515 &120.568&   120.568 &134.994&   134.994\\
$\Omega_1^M$&0.03   &16.43s&13.29s&96.486&   101.212 &121.572&   121.639 &135.939&   136.237\\
&0.05  &4.77s&3.84s&100.836&   105.872 &129.368&   130.461 &147.520&   149.396\\
  \hline
\hline
$H(\bfA)$&0.01  &301.39s& 263.20s  &94.278& 117.901 &135.061& 135.071 &153.418& 170.632\\
$\Omega_2$&0.03   &20.36s&  17.85s &114.158& 139.372 &171.988& 174.003 &174.442& 206.290\\
&0.05  &12.37s&11.63s&305.712&   339.046 &356.581& 397.561 &401.280& 417.617\\
\hline
$H(\bfF)$&0.01  &325.71s& 296.15s    &94.240&   117.860 &134.993&   134.994 &153.471&   170.575\\
$\Omega_2$&0.03   &14.17s&12.52s&94.451&   118.267&136.060&   136.171 &172.465&   172.715\\
&0.05  &4.78s&3.87s&96.561&   121.892 &146.710&   148.115 &184.990&   186.549\\
  \hline\hline
$H(\bfA)$&0.01  &308.72s& 271.89s &135.073&   135.079 &137.885&   153.421 &170.661&   170.664\\
$\Omega_3$&0.03   &15.64s&  13.30s &166.767& 169.243 &172.275& 177.050 &203.154&204.759\\
&0.05  &3.90s&3.18s&350.344& 395.788 &407.278& 415.606 &415.788& 416.278\\
\hline
$H(\bfF)$&0.01  &345.02s& 284.18s&134.992&   134.994 &137.810&   153.468 &170.586&   170.587\\
$\Omega_3$&0.03   &15.63s&12.82s&135.928&   136.115 &138.628&   172.526 &172.569&   183.853\\
&0.05  &4.31s&3.49s&146.329&151.990 &152.063&   186.594 &186.660&   189.789\\
\hline  
\end{tabular}
\end{center}
\end{table}

In Table~\ref{Ex3Table3} we see the validation of
\change{~\eqref{Heuristic2} and~\eqref{Heuristic3}} in most cases.  There are three cases for
which the vector fields corresponding to $\mb{F}$ are about 10\%
larger than their counterparts for $\mb{A}$, but they are about
half as large (or less) in most cases. 
\begin{table}
  \caption{Validating ~\change{\eqref{Heuristic2} and~\eqref{Heuristic3}} for Example 3 in
    most cases.\label{Ex3Table3}}
  \centering
\begin{tabular}{ |cc|cccccc|}
    \hline
    &&$j=1$&$j=2$&$j=3$&$j=4$&$j=5$&$j=6$\\\hline
    \multirow{2}{*}{$\Omega$}&$\|\nabla \psi_j\|_{L^2(\Omega)}$  &   100.1883 & 100.3179 & 100.3407  & 100.3407 & 100.2370 & 100.2370 \\ 
    &$\|\mb{A} \psi_j\|_{L^2(\Omega)}$ &   100 & 100 & 100  & 100 & 100 & 100 \\  
    \hline
    \multirow{2}{*}{$\Omega$}&$\|\nabla \phi_j\|_{L^2(\Omega)}$  &
                                                                   38.0083&
                                                                            41.6704 & 50.0658 & 50.0658& 59.5955 & 59.5947 \\
  &$\|\mb{F} \phi_j\|_{L^2(\Omega)}$  &   37.5056 & 40.8958& 49.3768  & 49.3768 & 59.1937& 59.1929\\  \hline
  \hline
   \multirow{2}{*}{$\Omega_1^D$}&$\|\nabla \psi_j\|_{L^2(\Omega)}$  &   100.3407& 100.3407& 100.2364  & 100.2363& 100.2894 & 100.4052 \\  
   &$\|\mb{A} \psi_j\|_{L^2(\Omega)}$ &   100 & 100 & 100  & 100 & 100 & 100 \\  
   \hline
  \multirow{2}{*}{$\Omega_1^D$} &$\|\nabla \phi_j\|_{L^2(\Omega)}$  &     49.8756 & 49.8748 & 59.8136  & 59.8136 & 112.2854 & 48.4734 \\ 
   &$\|\mb{F} \phi_j\|_{L^2(\Omega)}$  &    49.1839 & 49.1831 & 59.4133  & 59.4133 & 112.0284 & 47.6194 \\  \hline
  \hline
 \multirow{2}{*}{$\Omega_1^M$}  &$\|\nabla \psi_j\|_{L^2(\Omega)}$  &   100.2230 & 100.2740 & 100.3407  & 100.3407 & 100.2364 & 100.2363 \\  
   &$\|\mb{A} \psi_j\|_{L^2(\Omega)}$ &   100 & 100 & 100  & 100 & 100 & 100 \\  
   \hline
 \multirow{2}{*}{$\Omega_1^M$}  &$\|\nabla \phi_j\|_{L^2(\Omega)}$  &    36.9995 & 36.8660 & 49.8756  & 49.8748 & 59.8130 & 59.8129 \\ 
   &$\|\mb{F} \phi_j\|_{L^2(\Omega)}$  &    36.3791 & 36.1019 &
                                                                49.1839  & 49.1831 & 59.4127& 59.4126 \\ \hline
   \hline
 \multirow{2}{*}{$\Omega_2$}   &$\|\nabla \psi_j\|_{L^2(\Omega)}$  &     100.1883 & 100.3179 & 100.2361  & 100.2358 & 100.2892 & 100.3336\\
   &$\|\mb{A} \psi_j\|_{L^2(\Omega)}$ &   100 & 100 & 100  & 100 & 100 & 100 \\  
   \hline
\multirow{2}{*}{$\Omega_2$}    &$\|\nabla \phi_j\|_{L^2(\Omega)}$  &       37.8805 & 41.7133 & 58.4620  & 58.4614 & 112.0237 & 61.3137 \\ 
   &$\|\mb{F} \phi_j\|_{L^2(\Omega)}$  &   37.3761 & 40.9396 & 58.0524
                             & 58.0518 & 111.7666 & 60.7634 \\
  \hline
  \hline
\multirow{2}{*}{$\Omega_3$}   & $\|\nabla \psi_j\|_{L^2(\Omega)}$  &      100.2357 & 100.2355 & 100.3873  & 100.2891 & 100.3332 & 100.3329 \\  
   &$\|\mb{A} \psi_j\|_{L^2(\Omega)}$ &   100 & 100 & 100  & 100 & 100 & 100 \\  
   \hline
 \multirow{2}{*}{$\Omega_3$}  &$\|\nabla \phi_j\|_{L^2(\Omega)}$  &       58.7244 & 58.6738 & 43.7127 & 111.3903 & 61.2991 & 61.3292 \\ 
   &$\|\mb{F} \phi_j\|_{L^2(\Omega)}$  &    58.3166& 58.2657 & 42.8100  & 111.1316 & 60.7483 & 60.7785 \\  \hline  
\end{tabular}
\end{table}

\subsection{Example 4}
As a final example, we take the vector field 
\begin{align}\label{Ex4A}
  \mb{A} = 25(-y, x)~,
\end{align}
with Neumann boundary conditions on the L shape domain $\Omega=
(0,3) \times (0,3) \setminus [2, 3) \times [2,3)$, see
Figure~\ref{VectorFields}.  This may seem the least interesting of the
fields, as it has a constant curl, but it is the most studied of the
vector fields in the literature---though not with computational
efficiency considerations in mind.  For a qualitative comparison, we
refer to \cite[Section 5.3]{Bonnaillie-Noel2007}.  Although the domain
used there is $(0,2) \times (0,2) \setminus [1, 2) \times [1,2)$, the
qualitative behavior of the first five eigenpairs in their case and
ours is very similar: $\lambda_1\neq\lambda_2$ are very close to
each other, and $\lambda_4\neq \lambda_5$ are very close to each other.
Their Figure 13 shows eigenvector behavior quite similar to what we
observe in the top row of Figure~\ref{Ex4Fig2}. 

We highlight that both $\mb{A}$ and $\mb{F}$ are divergence-free in
this case.  However, there is a significant difference in their norms.
We compute 
 \begin{align*}
  \|\change{\mb{A}}\|_{L^2(\Omega)}=109.9242\quad,\quad \|\change{\mb{F}}\|_{L^2(\Omega)}=30.9111~,
\end{align*}
when $h=0.01$.  The computed eigenvalues and timing information for
$H(\mb{A})$ and $H(\mb{F})$ are provided in Table~\ref{Ex4Table},
together with the validation of ~\eqref{Heuristic2}.  This is
the only example for which the computational cost for $H(\mb{A})$ is
consistently less than that for $H(\mb{F})$, if only slightly.
Looking at the computed eigenvalues begins to illustrate why
$H(\mb{F})$ remains the superior choice.  The change in the computed
eigenvalues for $H(\mb{F})$ as $h$ increases is extremely small in
this case---smaller than in previous examples.  The change in the computed
eigenvalues for $H(\mb{A})$ as $h$ increases is also fairly small, but
the computed eigenvectors show some significant changes---see
Figure~\ref{Ex4Fig1}.  Comparing the case $h=0.03$ to $h=0.01$, the
first third eigenvector moves into the first position, shifting the
first two eigenvectors into positions two and three.  The behavior
when $h=0.05$ is much worse.  Of the first six eigenvectors computed
when $h=0.05$, only two correspond to eigenvectors computed when
$h=0.01$: the first and second on the coarse mesh correspond to the
third and sixth on the fine mesh.  The remaining four computed
eigenvectors on the coarse mesh have no counterparts on the fine mesh!
The behavior for $H(\mb{F})$, seen in Figure~\ref{Ex4Fig2} is far superior.

\begin{table}
\caption{Computed eigenvalues and timing information, and validation
  of ~\change{\eqref{Heuristic2} and~\eqref{Heuristic3}} when $h=0.01$, for Example 4.
 \label{Ex4Table}}
\begin{center}
\begin{tabular}{|c|cll|cccccc|}\hline
&$h$ & Total & FEAST& $\lambda_1$ &  $\lambda_2$  & $\lambda_3$ &  $\lambda_4$ & $\lambda_5$ &  $\lambda_6$\\\hline
\multirow{3}{*}{$H(\bfA)$}
&0.01  &937.97s& 884.07s    &24.6272&   24.6274 &25.4954 & 26.8344 &26.8352&30.3674\\
&0.03   &22.18s &18.83s   &25.4954&   26.5064 &26.6700 &28.5015 &28.8521 &30.4596\\
&0.05  &4.90s&3.87s&25.4963&   30.7517 &31.2396  &32.1223 &33.4258   &33.8494\\
\hline
\multirow{3}{*}{$H(\bfF)$}
&0.01  &1167.48s& 1067.06s    &24.6244&   24.6245&25.4954&   26.8318 &26.8326&   30.3673\\
&0.03   &24.27s&20.28s&24.6244&   24.6245 &25.4954&   26.8318 &26.8326&   30.3675\\
&0.05  &7.24s&6.00s&24.6246&   24.6246 &25.4959&   26.8320 &26.8328&   30.3693\\
\hline
\end{tabular}

\vspace*{2mm}
 \begin{tabular}{ |c|cccccc|}
   \hline
   &$j=1$&$j=2$&$j=3$&$j=4$&$j=5$&$j=6$\\\hline
   $\|\nabla \psi_j\|_{L^2(\Omega)}$  & 72.4348 & 72.4355 & 9.259  & 70.8129& 70.8300 & 38.1424 \\
   $\|\mb{A} \psi_j\|_{L^2(\Omega)}$ &72.4424 & 72.4431 & 9.2589  & 70.7989 & 70.8156 & 38.1343 \\  
   \hline
   $\|\nabla \phi_j\|_{L^2(\Omega)}$  &  11.8373 & 11.8362 & 14.4127  & 12.3038 & 12.2969 & 20.0244 \\ 
   $\|\mb{F} \phi_j\|_{L^2(\Omega)}$  &  11.8825 & 11.8815 & 14.4127  & 12.2223 & 12.2133 & 20.0090  \\ \hline
 \end{tabular}
\end{center}
\end{table}

\begin{figure}
\includegraphics[width=.16\textwidth]{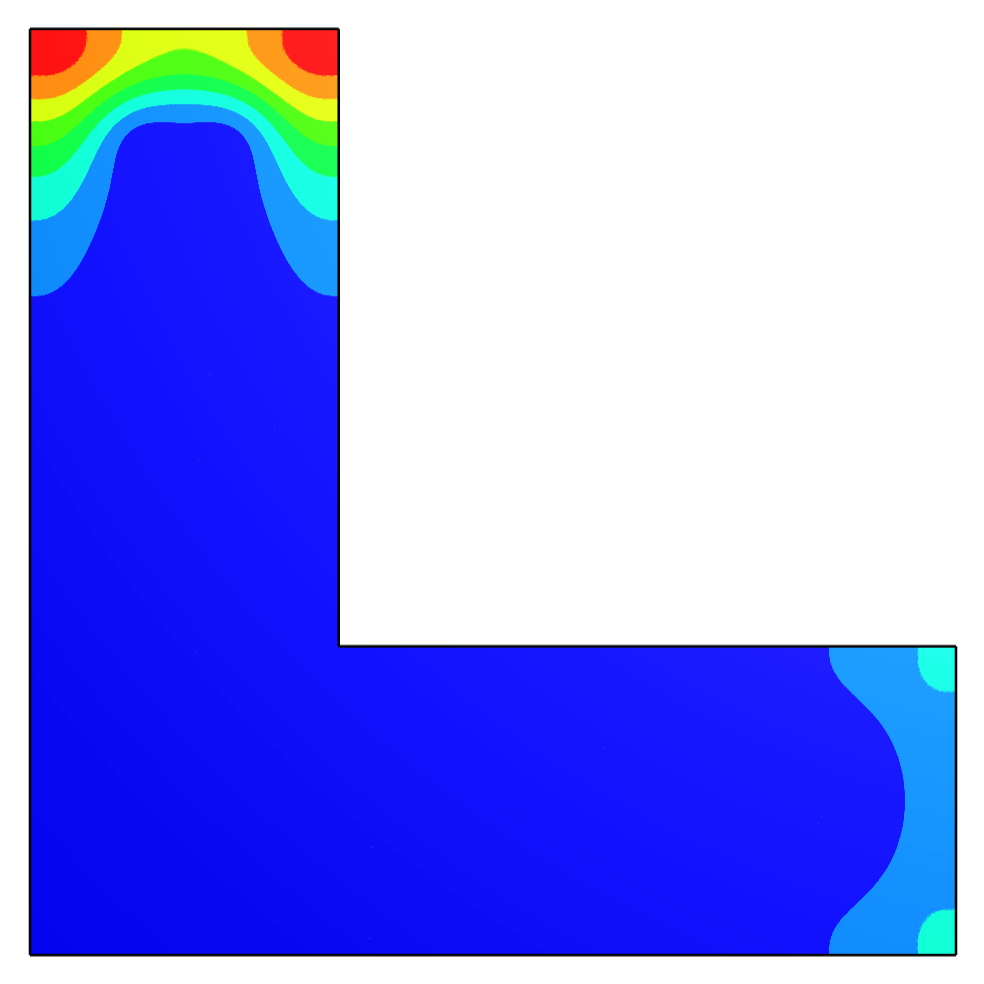}
\includegraphics[width=.16\textwidth]{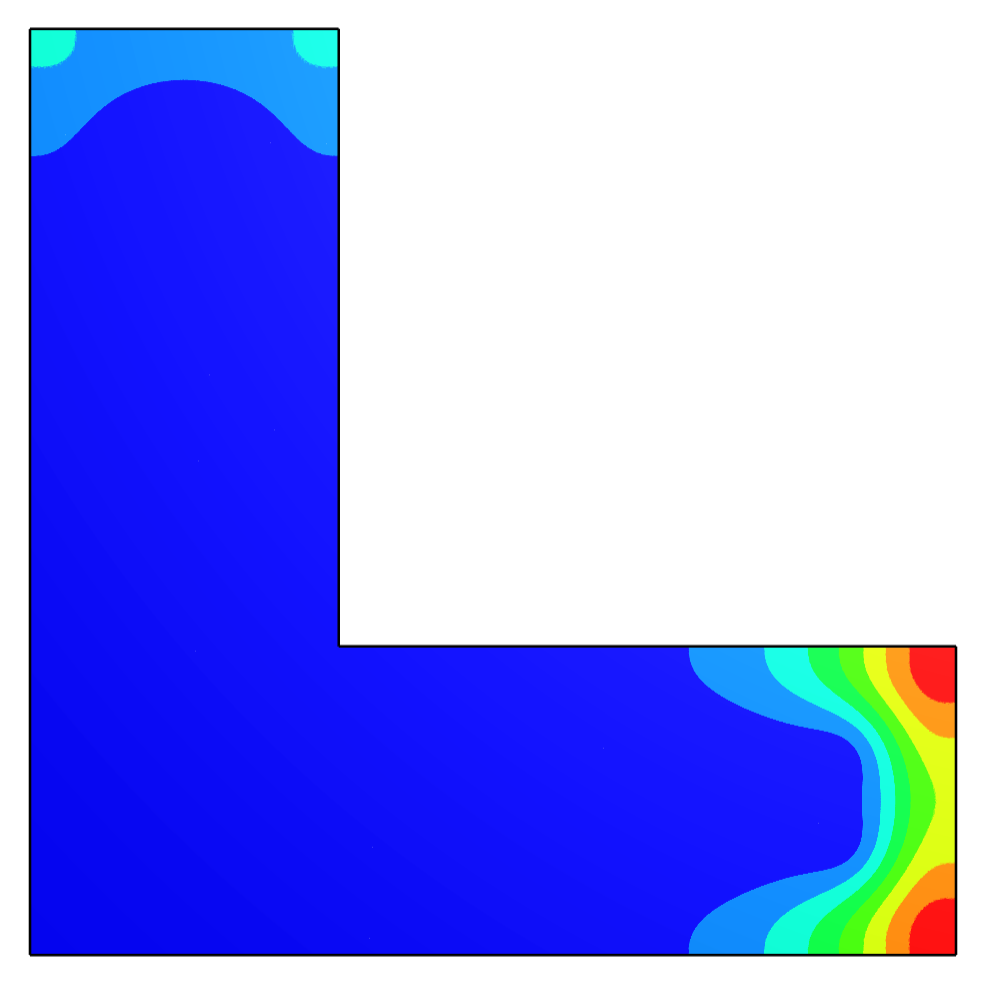}
\includegraphics[width=.16\textwidth]{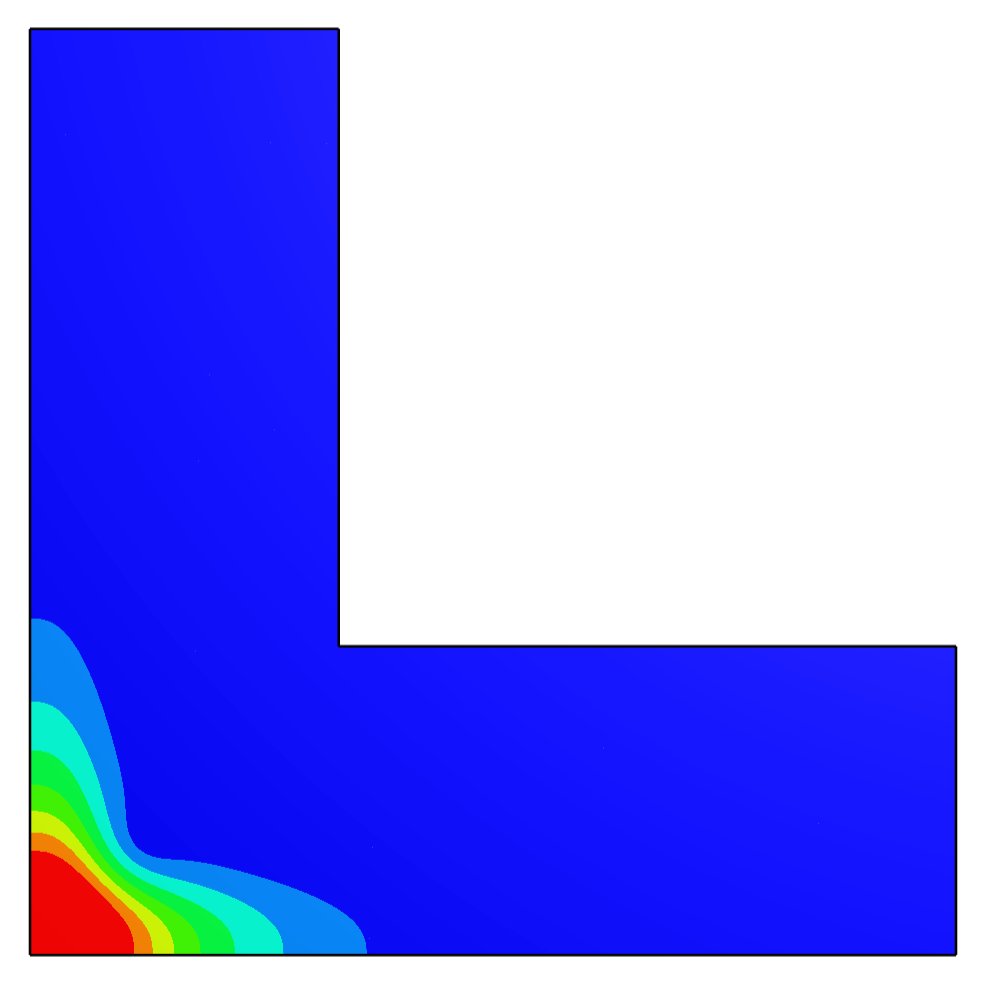}
\includegraphics[width=.16\textwidth]{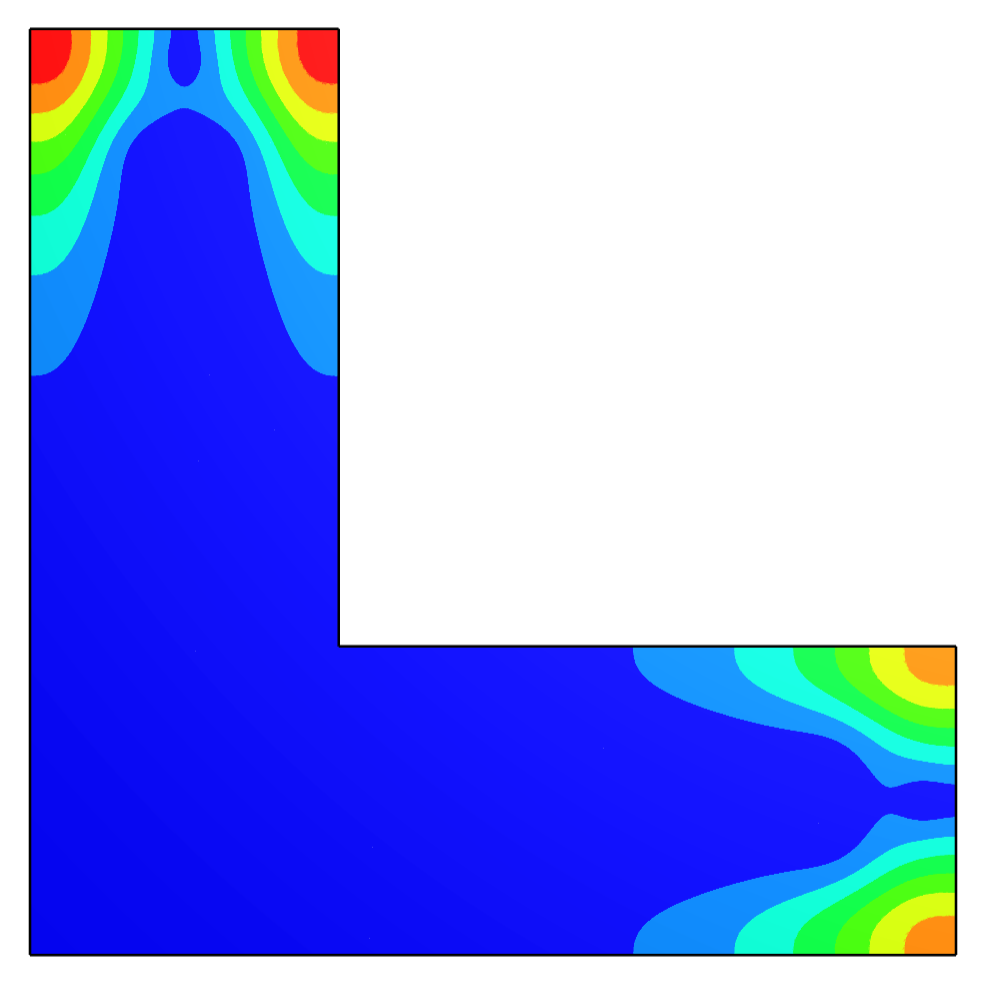}
\includegraphics[width=.16\textwidth]{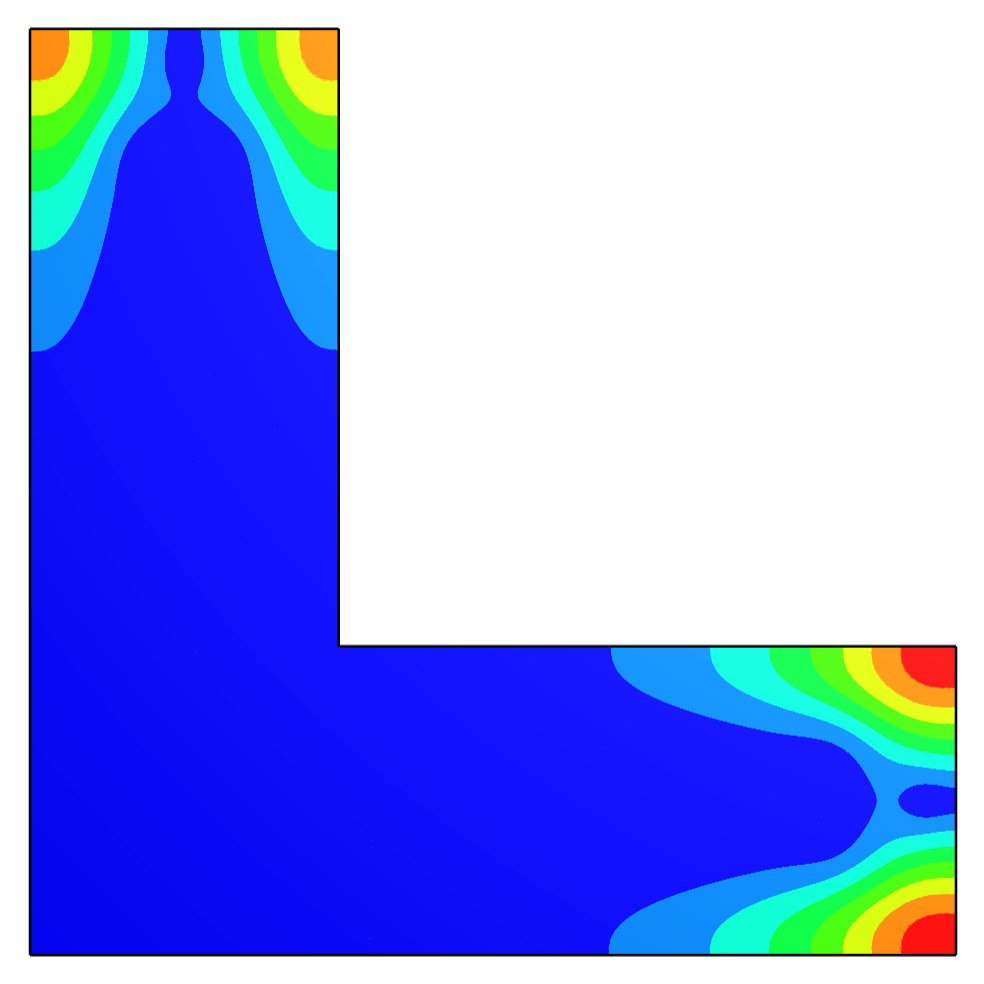}
\includegraphics[width=.16\textwidth]{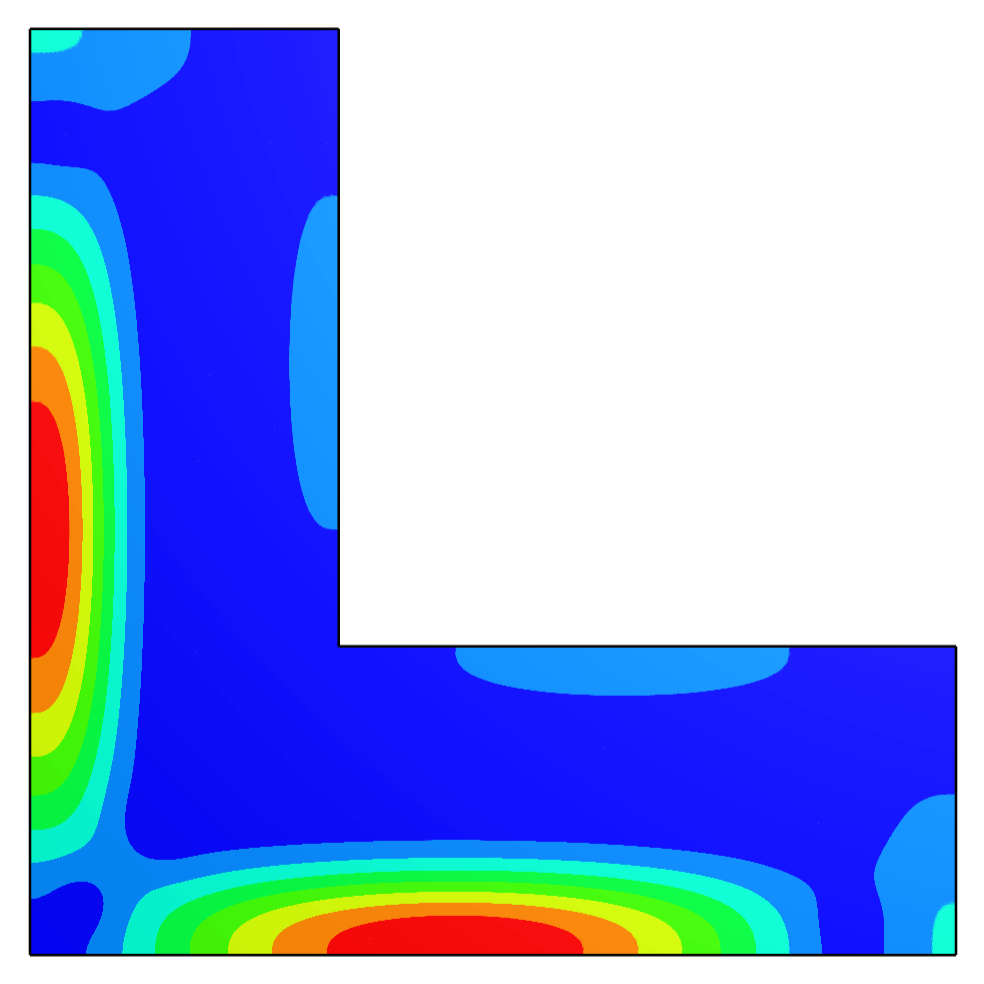}
\includegraphics[width=.16\textwidth]{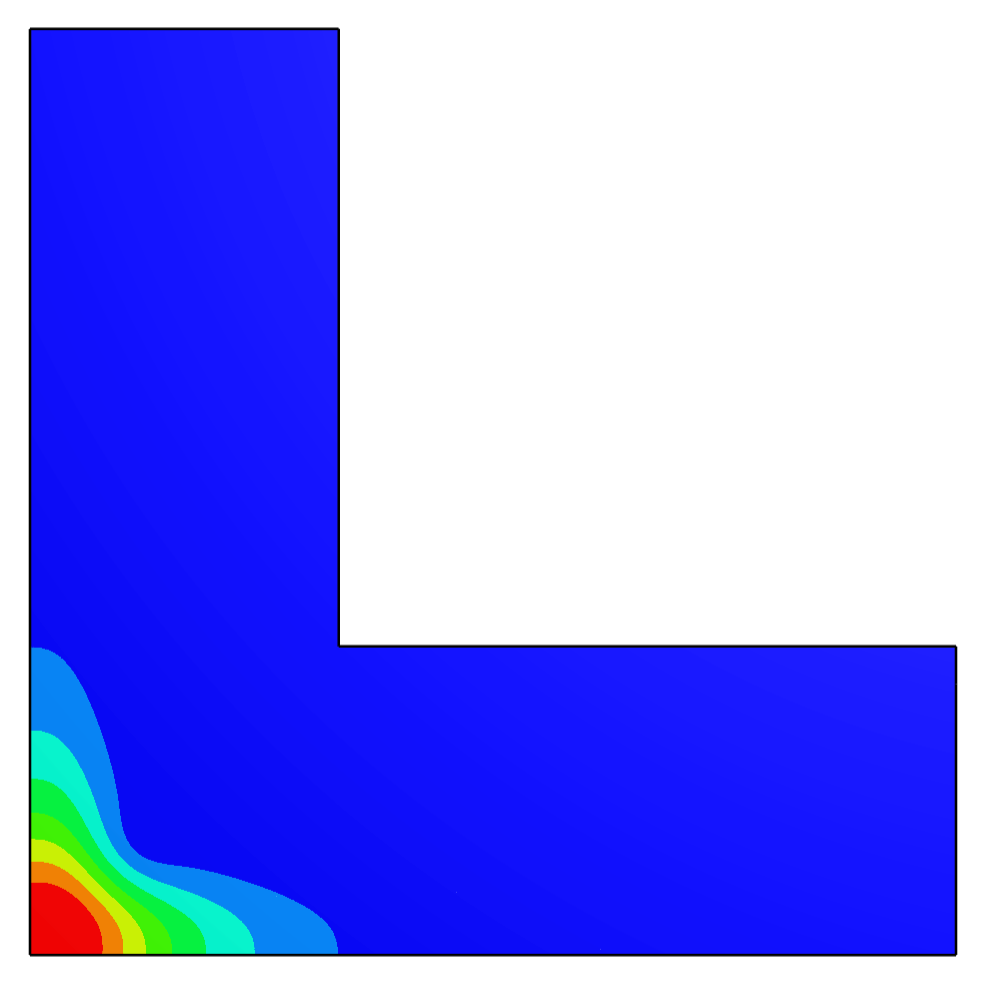}
\includegraphics[width=.16\textwidth]{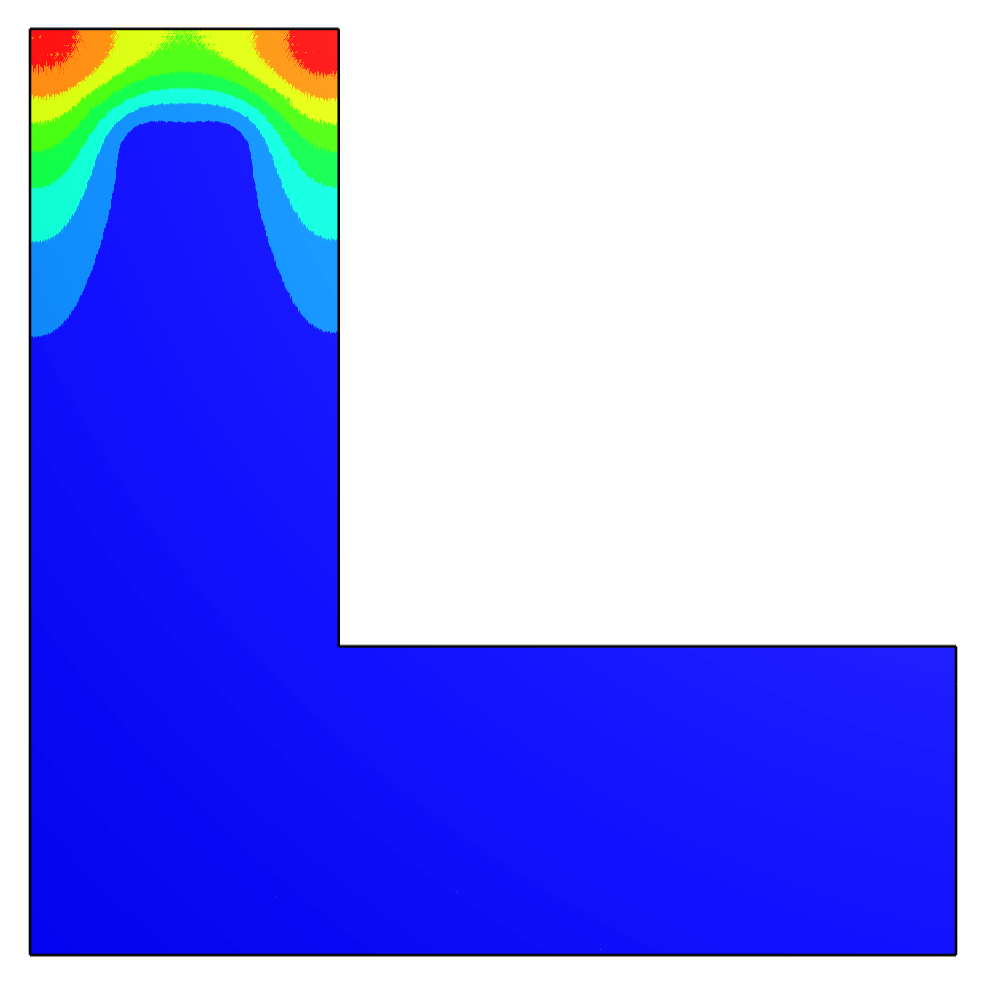}
\includegraphics[width=.16\textwidth]{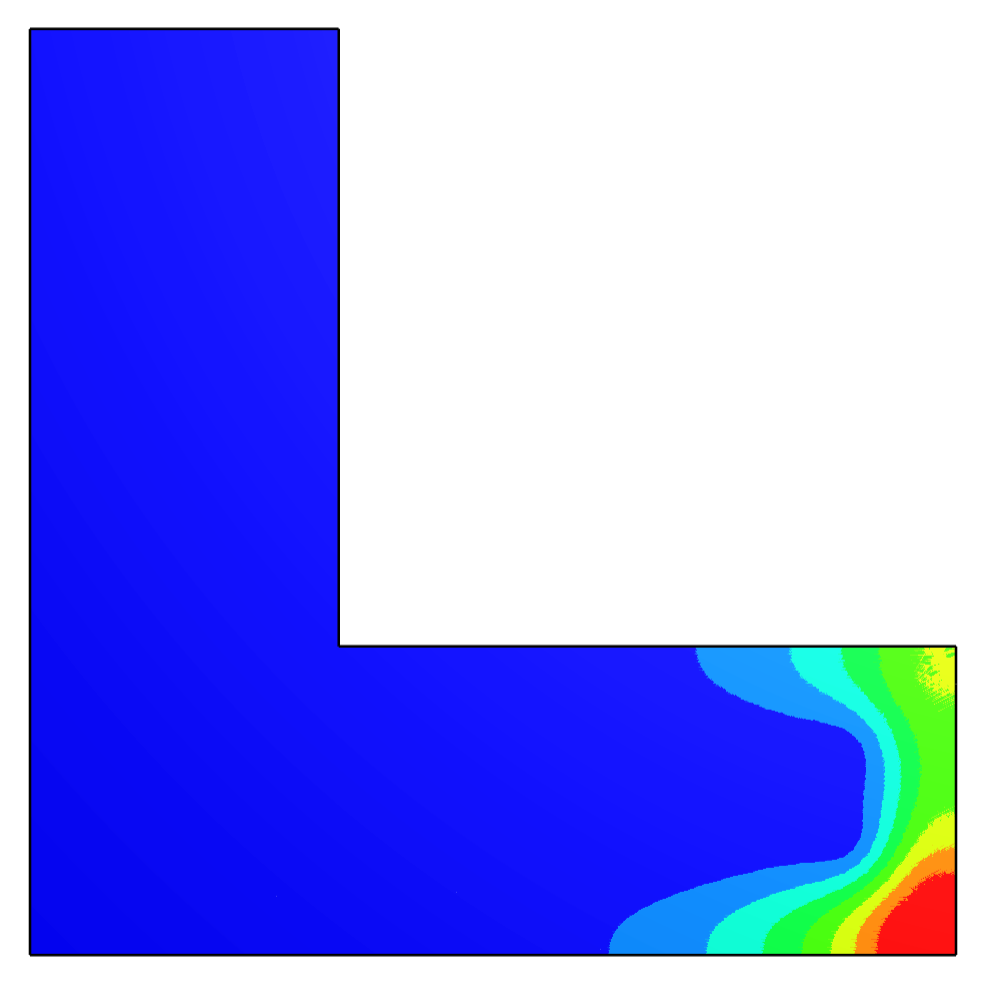}
\includegraphics[width=.16\textwidth]{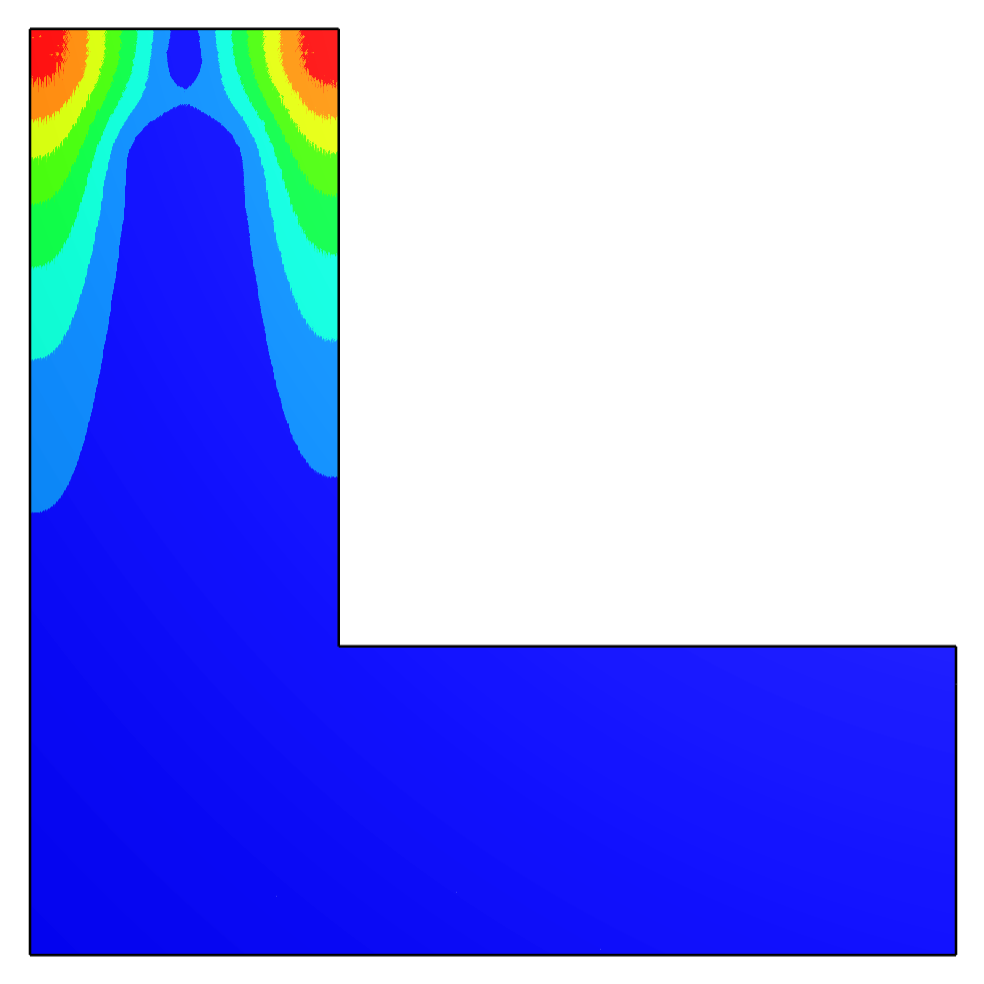}
\includegraphics[width=.16\textwidth]{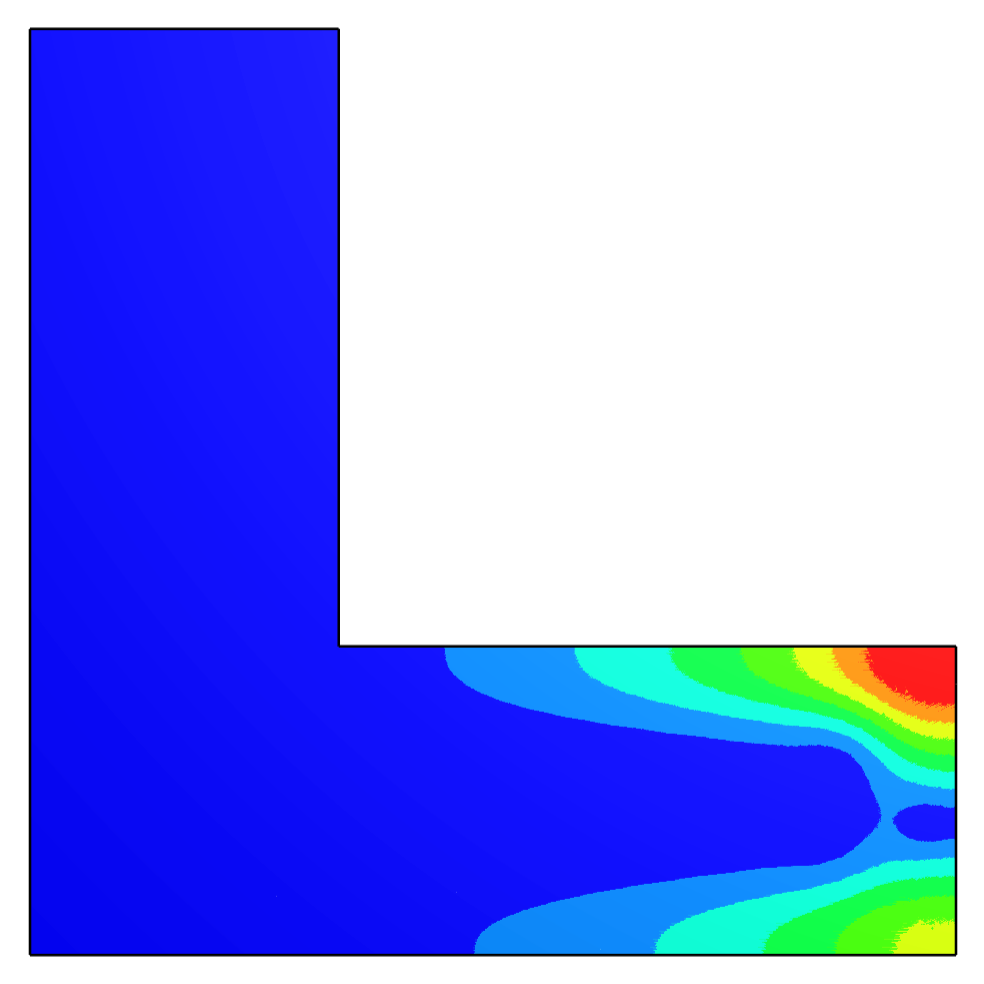}
\includegraphics[width=.16\textwidth]{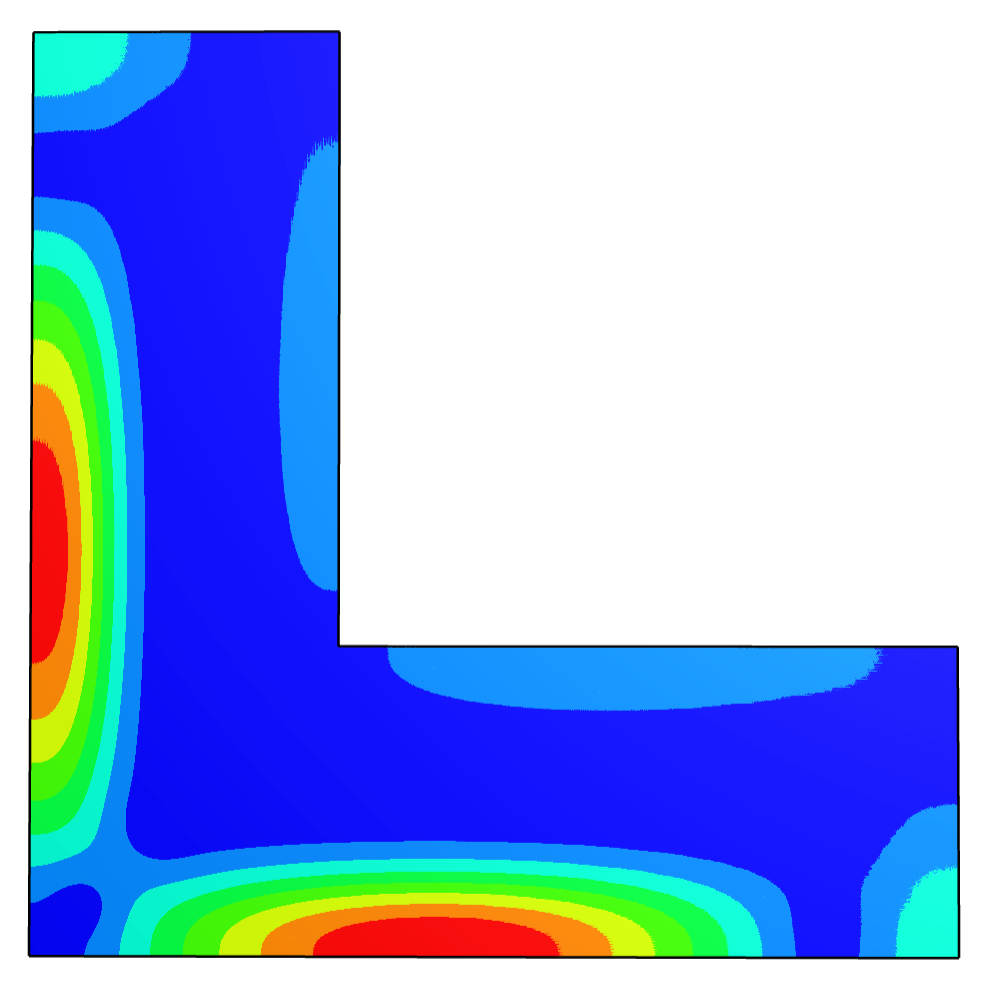}
\includegraphics[width=.16\textwidth]{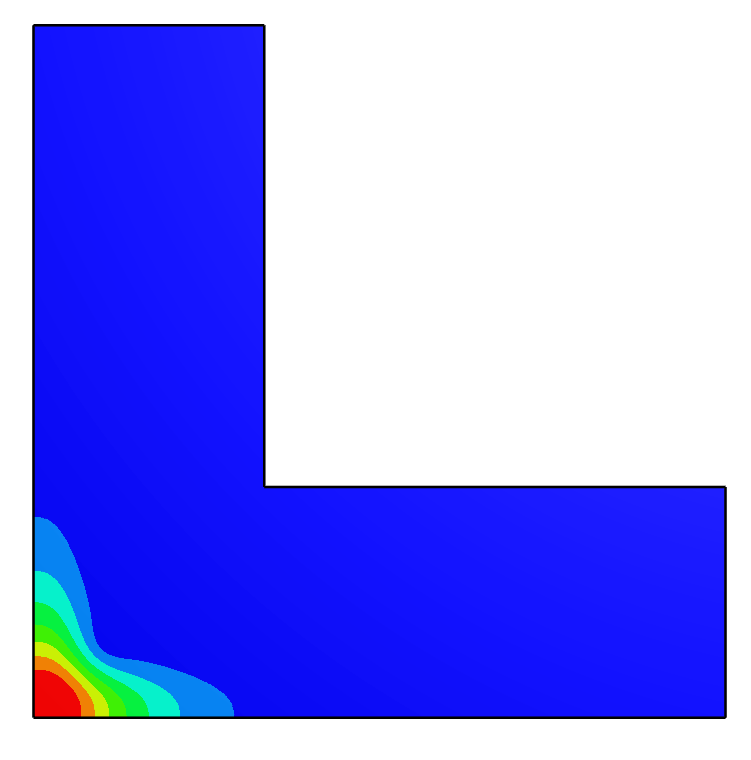}
\includegraphics[width=.16\textwidth]{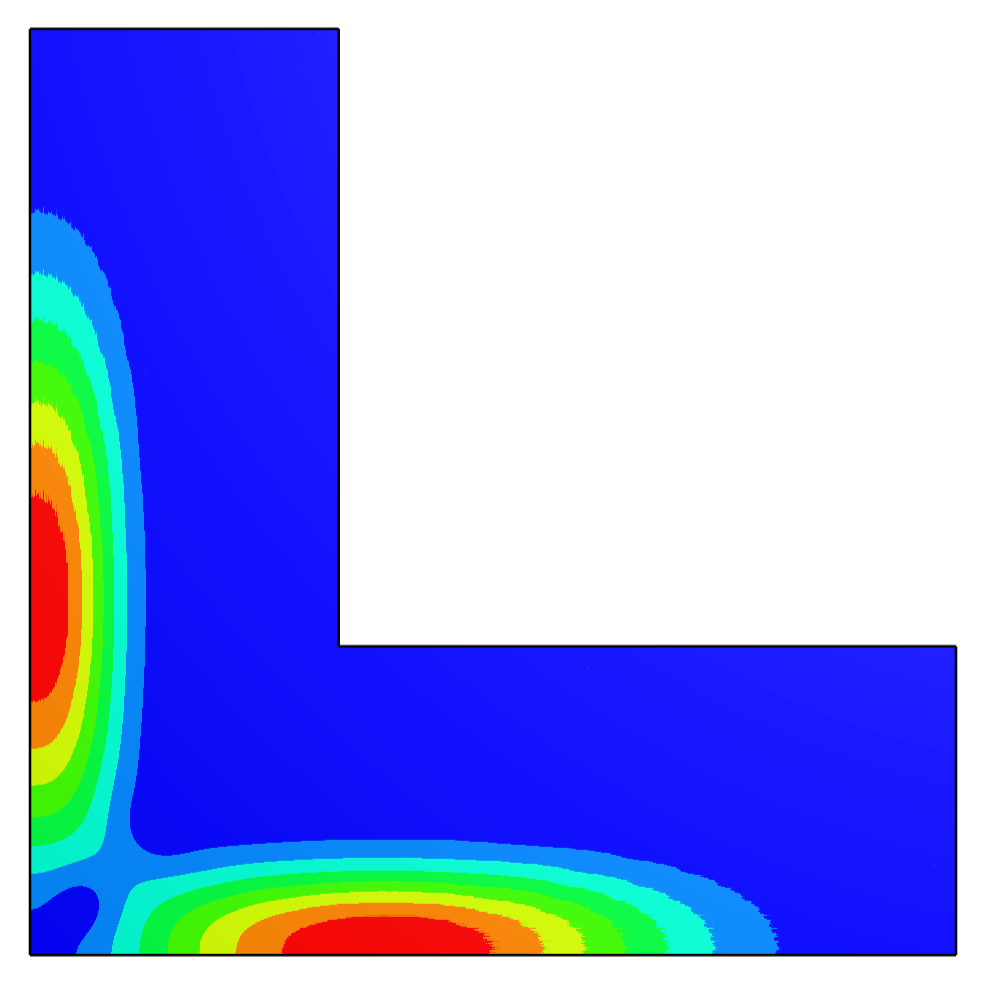}
\includegraphics[width=.16\textwidth]{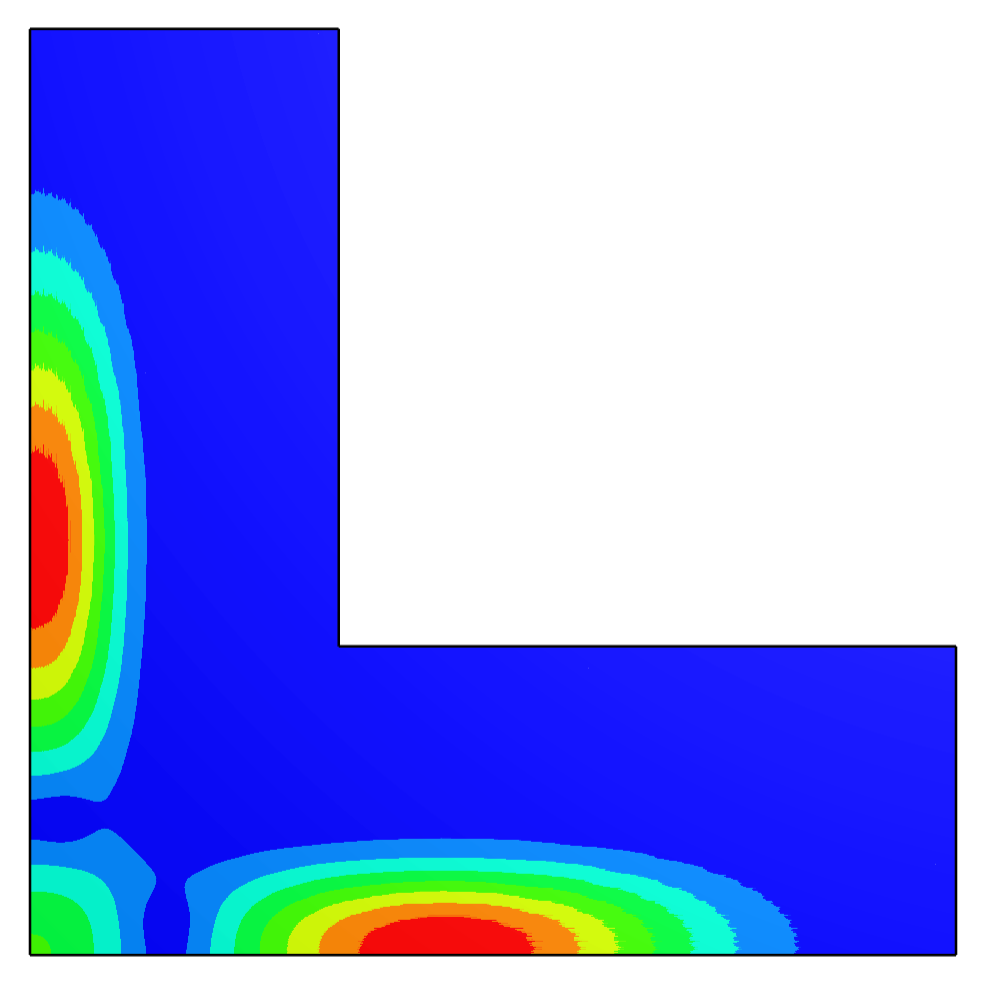}
\includegraphics[width=.16\textwidth]{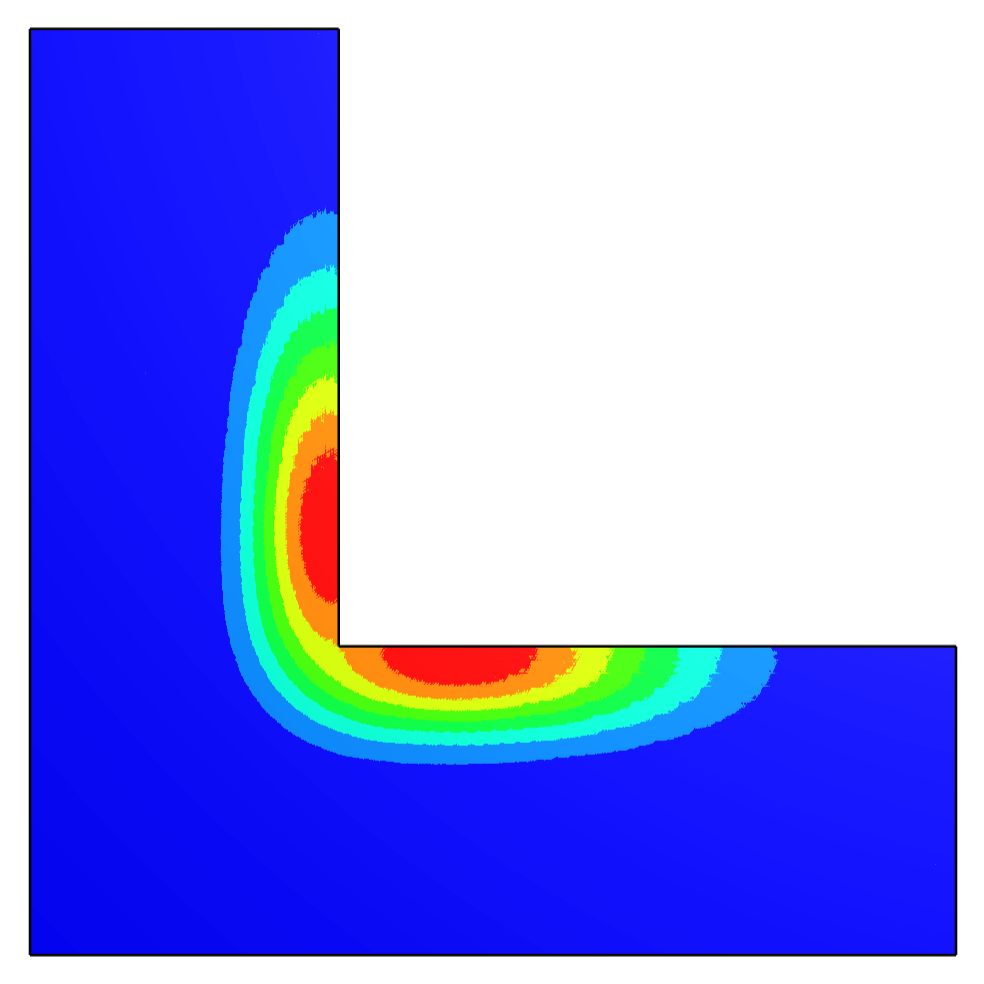}
\includegraphics[width=.16\textwidth]{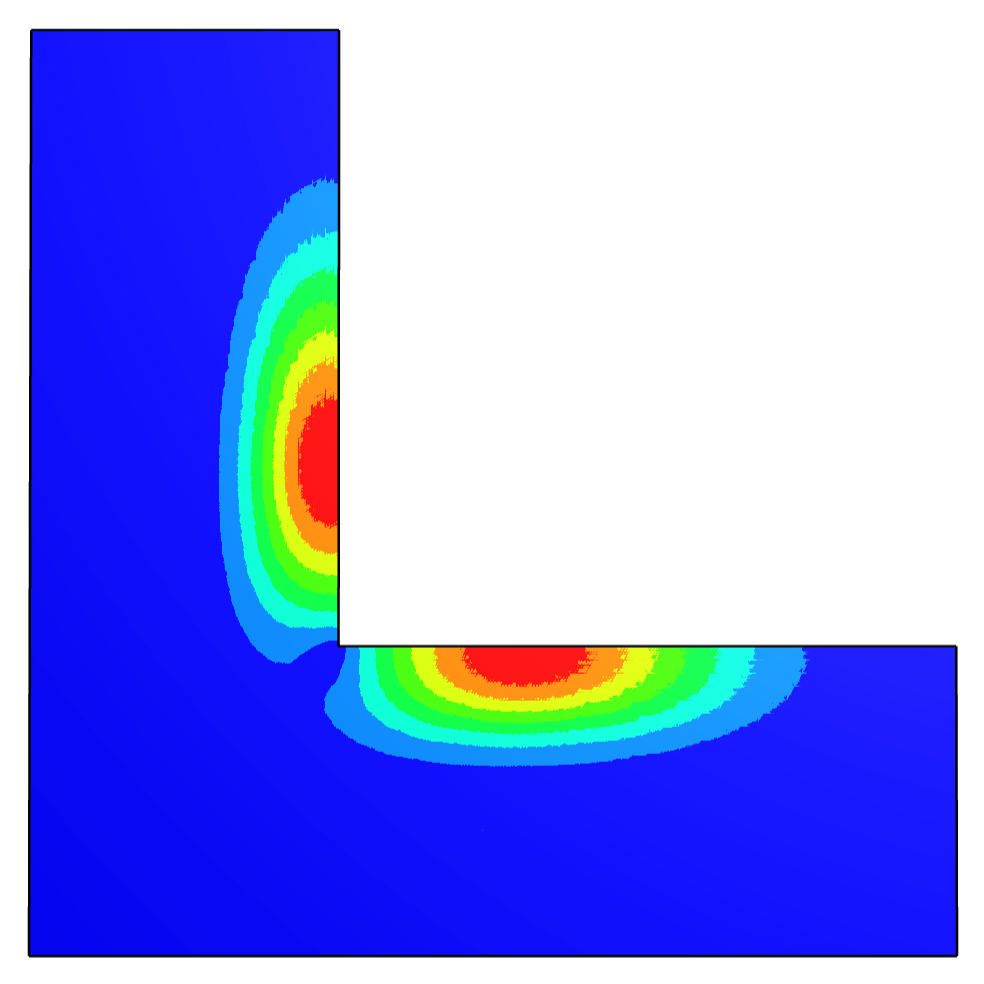}
\includegraphics[width=.16\textwidth]{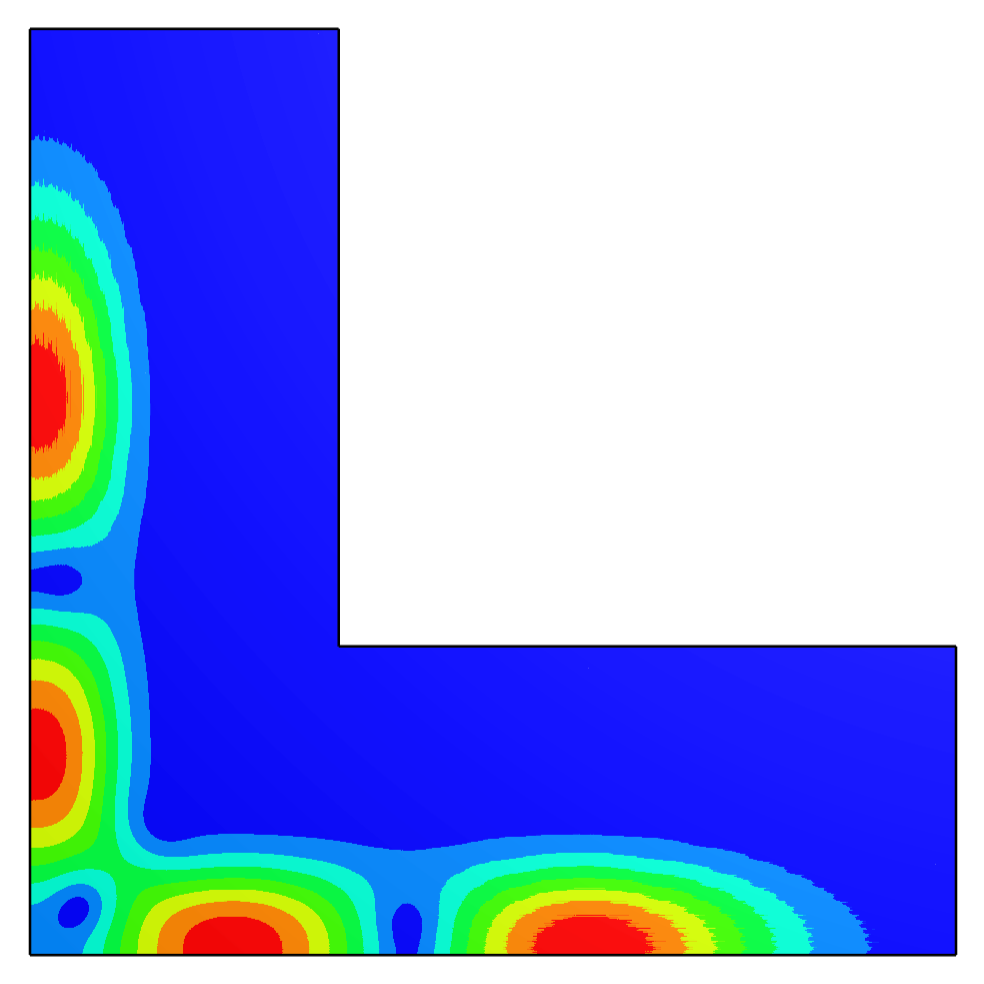}
\caption{The computed eigenvectors $|\psi_j|$ of $H(\mb{A})$, $1\leq j\leq 6$, when
  $h=0.01$ (top row), $h=0.03$ (middle row), and $h=0.05$, for Example 4. 
 \label{Ex4Fig1}}
\end{figure}

 \begin{figure}
\includegraphics[width=.16\textwidth]{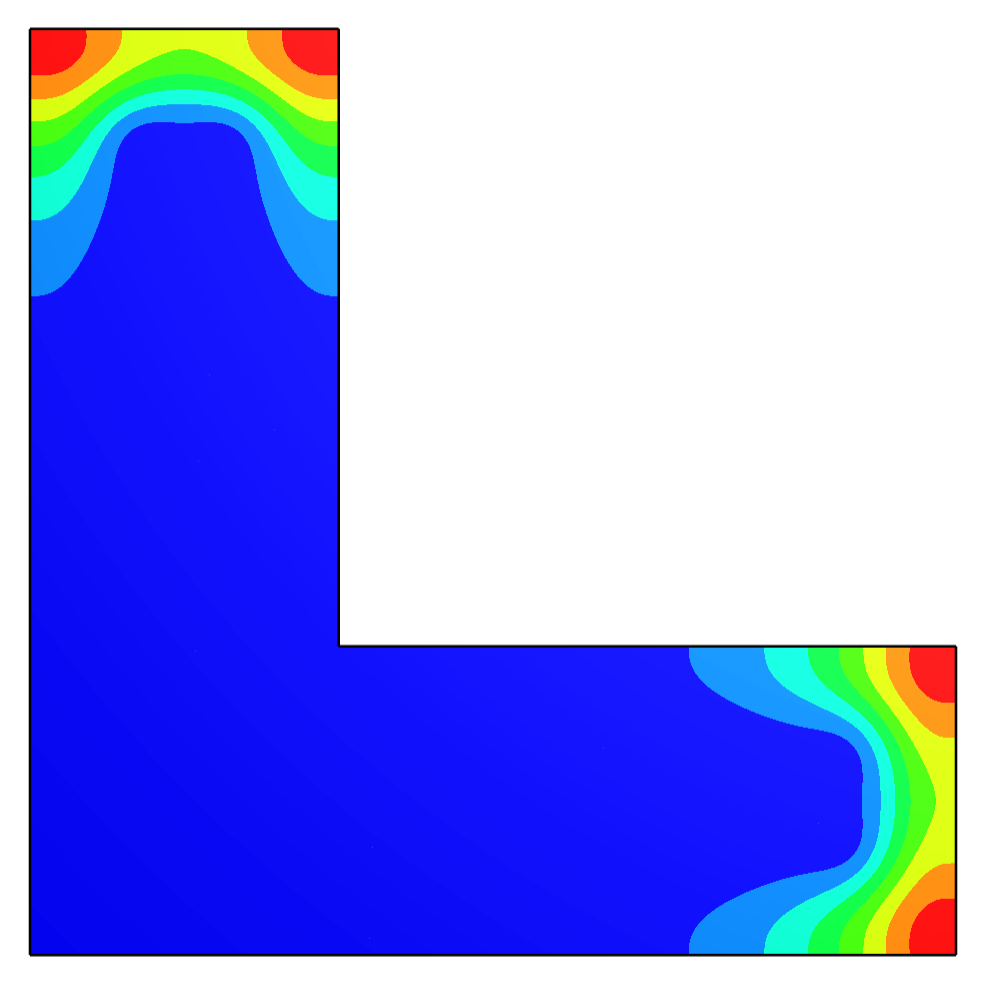}
\includegraphics[width=.16\textwidth]{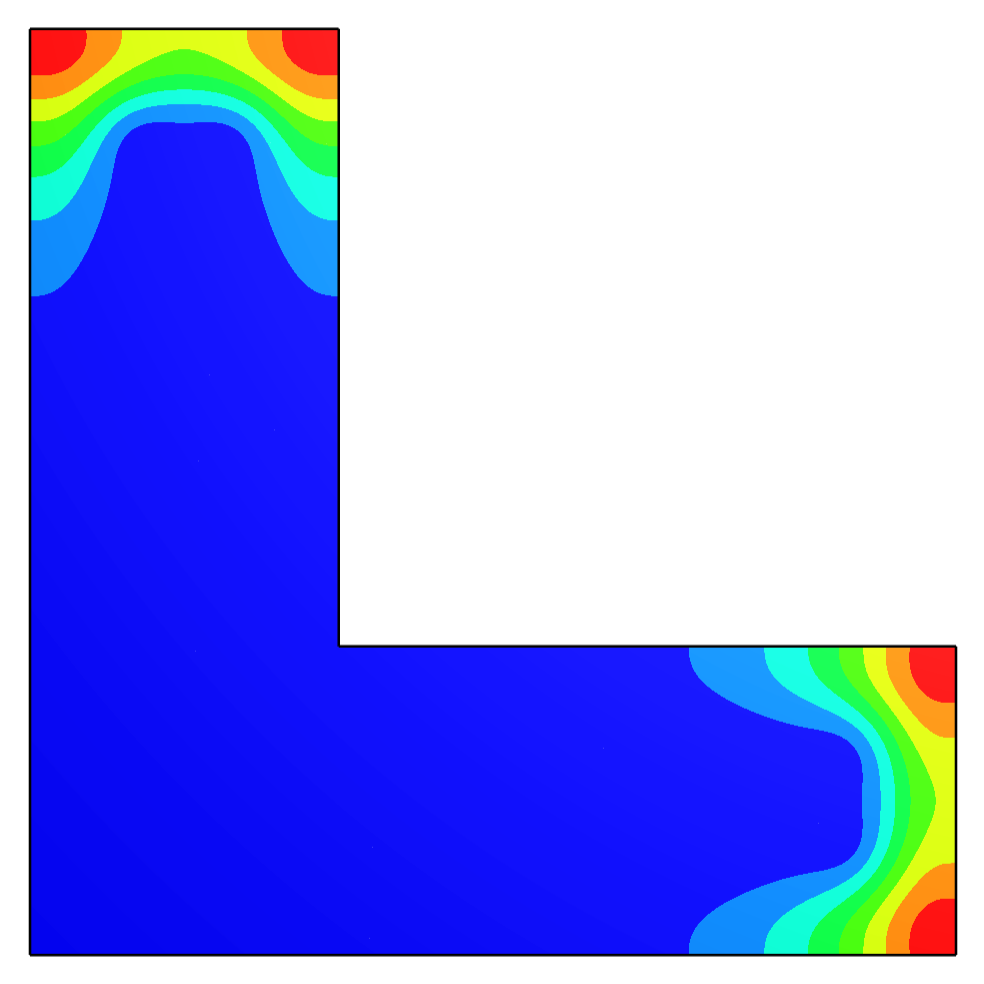}
\includegraphics[width=.16\textwidth]{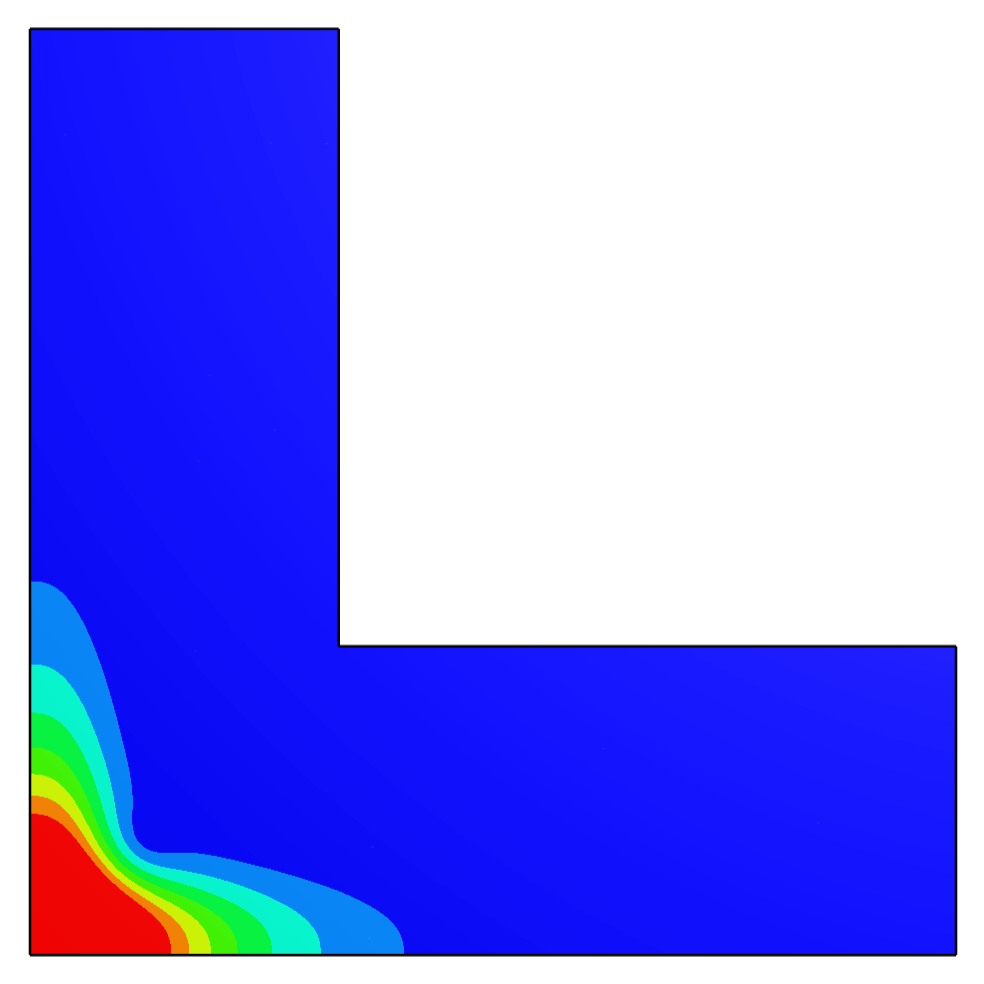}
\includegraphics[width=.16\textwidth]{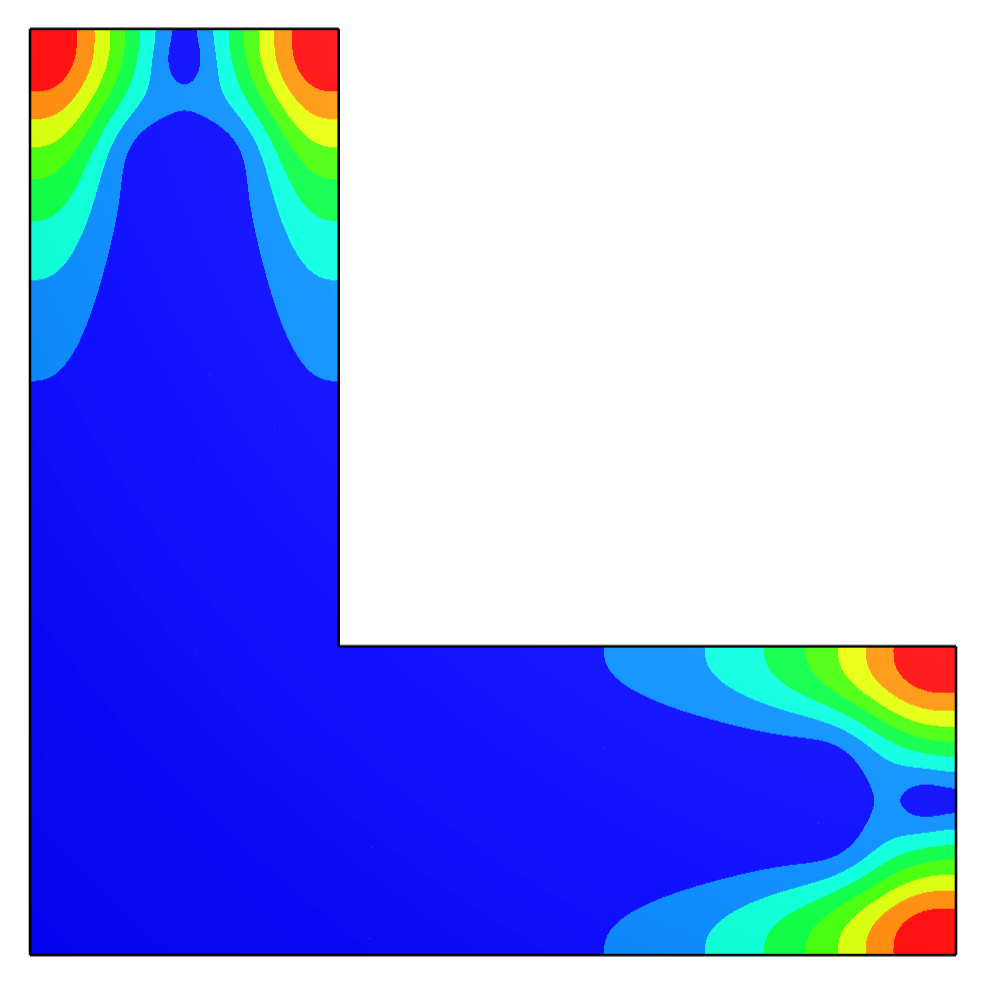}
\includegraphics[width=.16\textwidth]{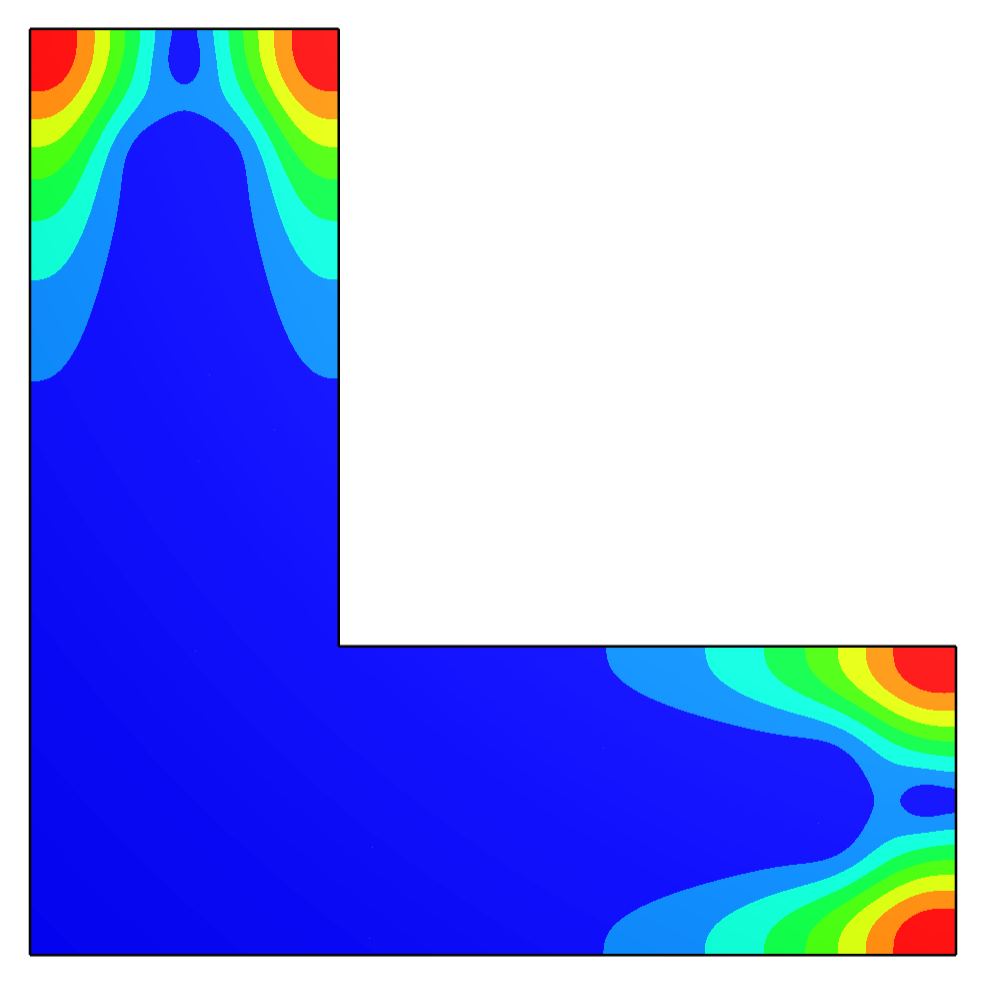}
\includegraphics[width=.16\textwidth]{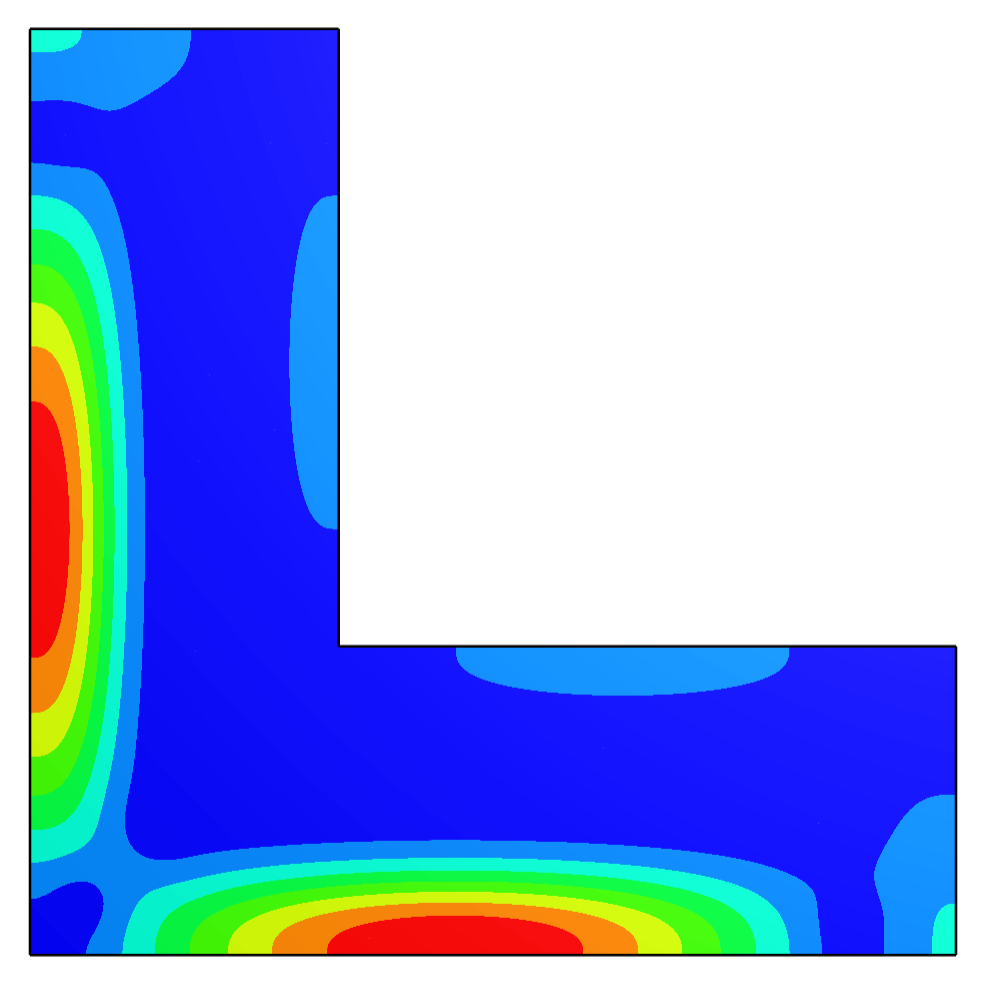}
\includegraphics[width=.16\textwidth]{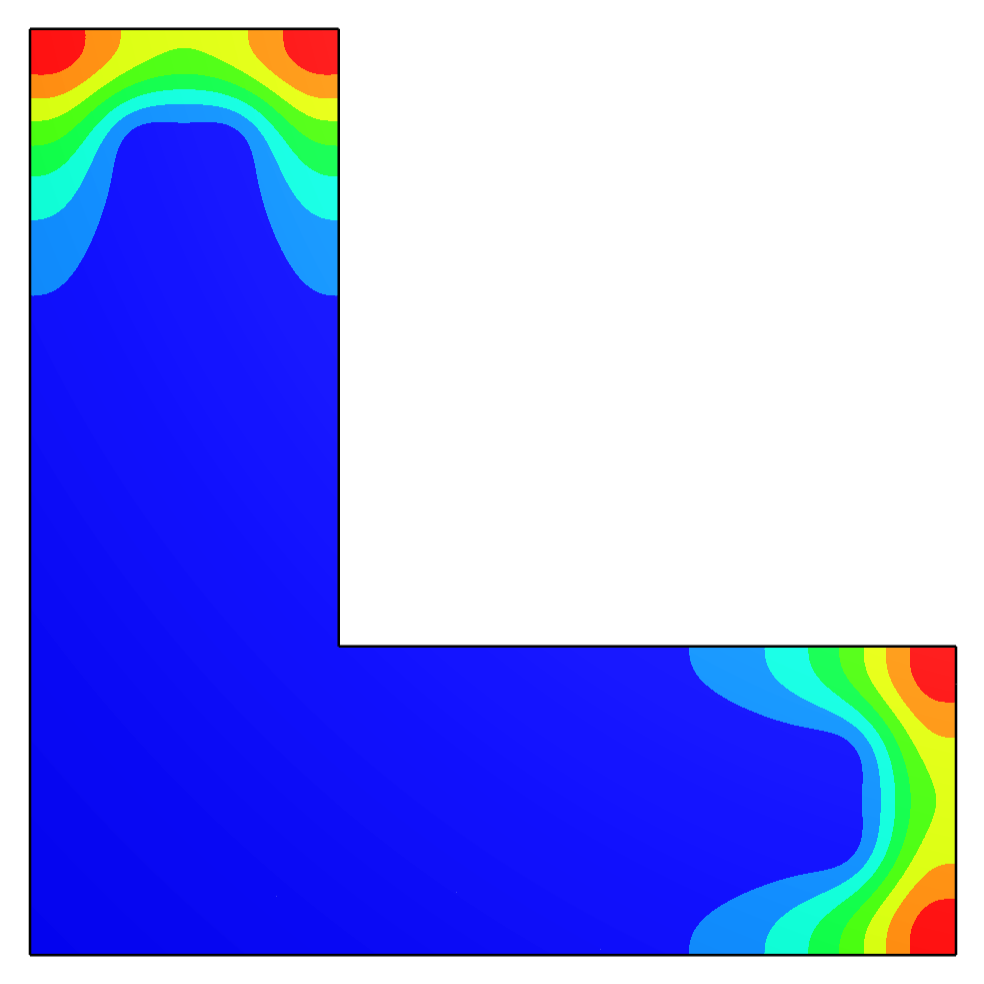}
\includegraphics[width=.16\textwidth]{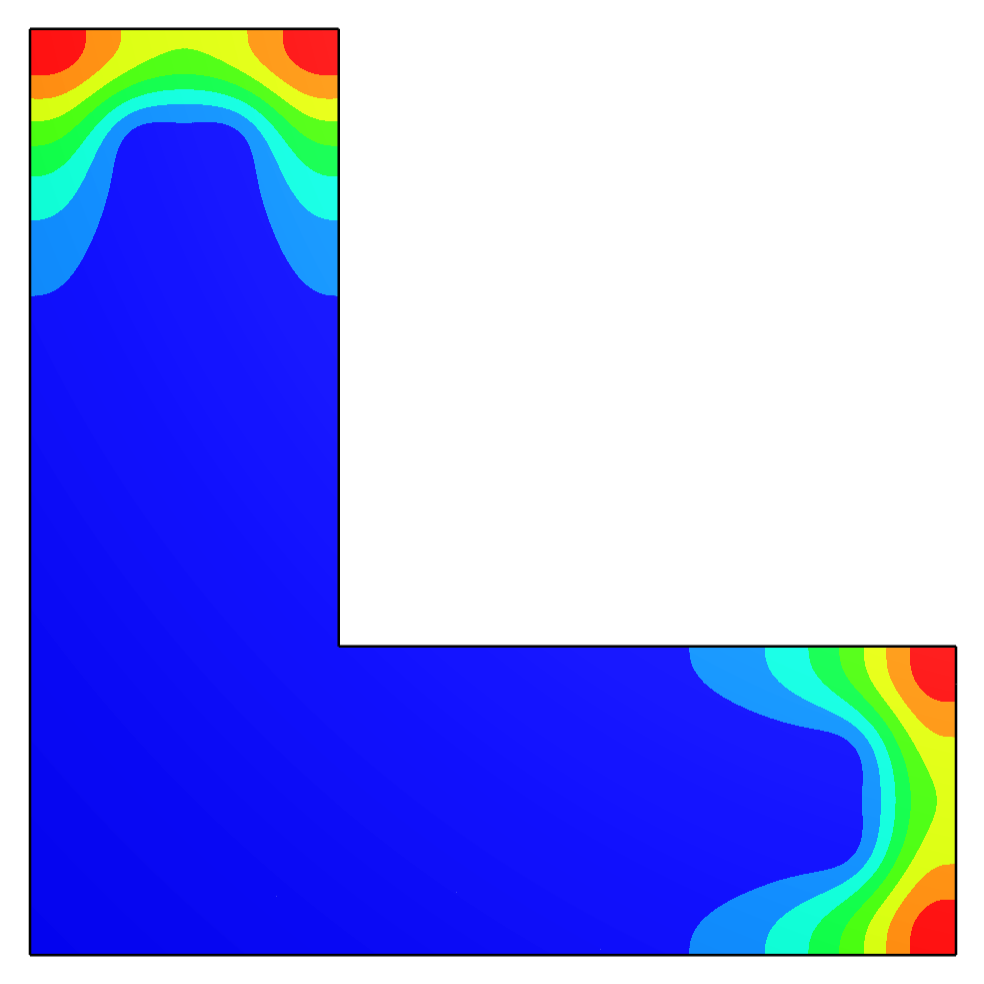}
\includegraphics[width=.16\textwidth]{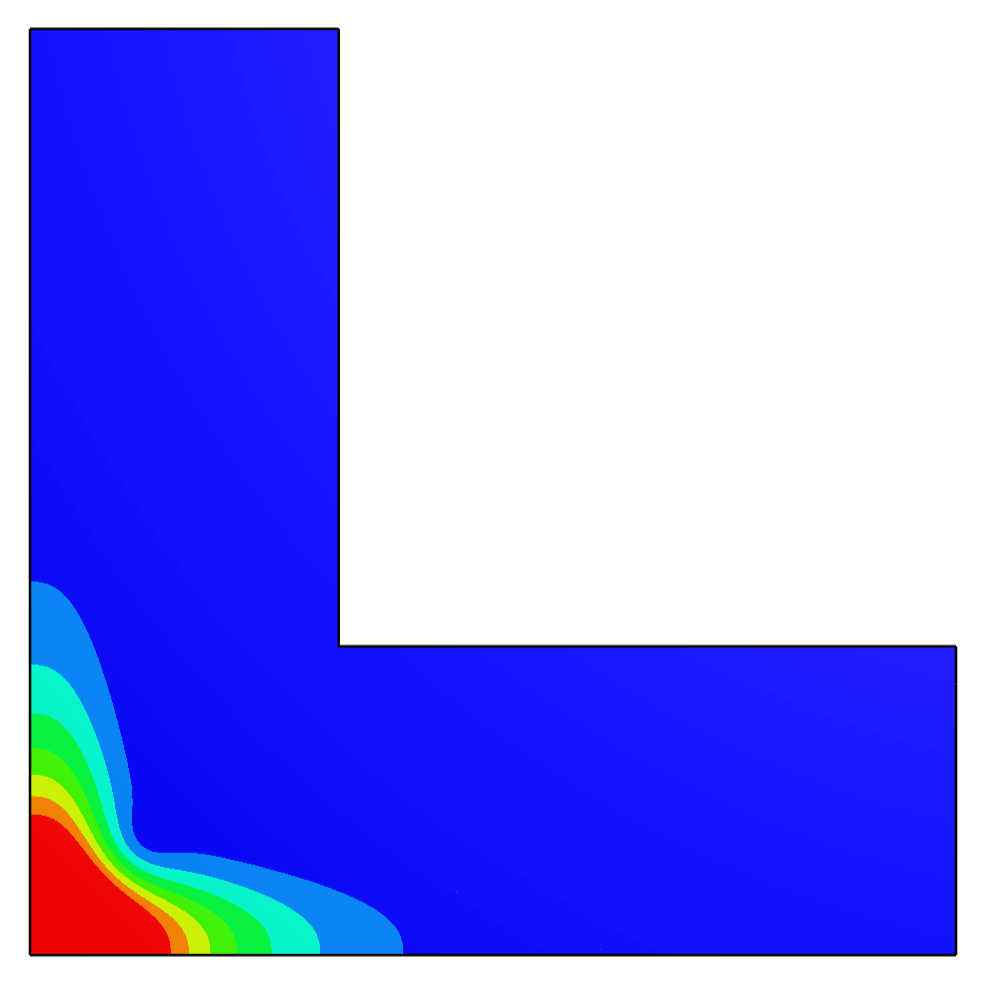}
\includegraphics[width=.16\textwidth]{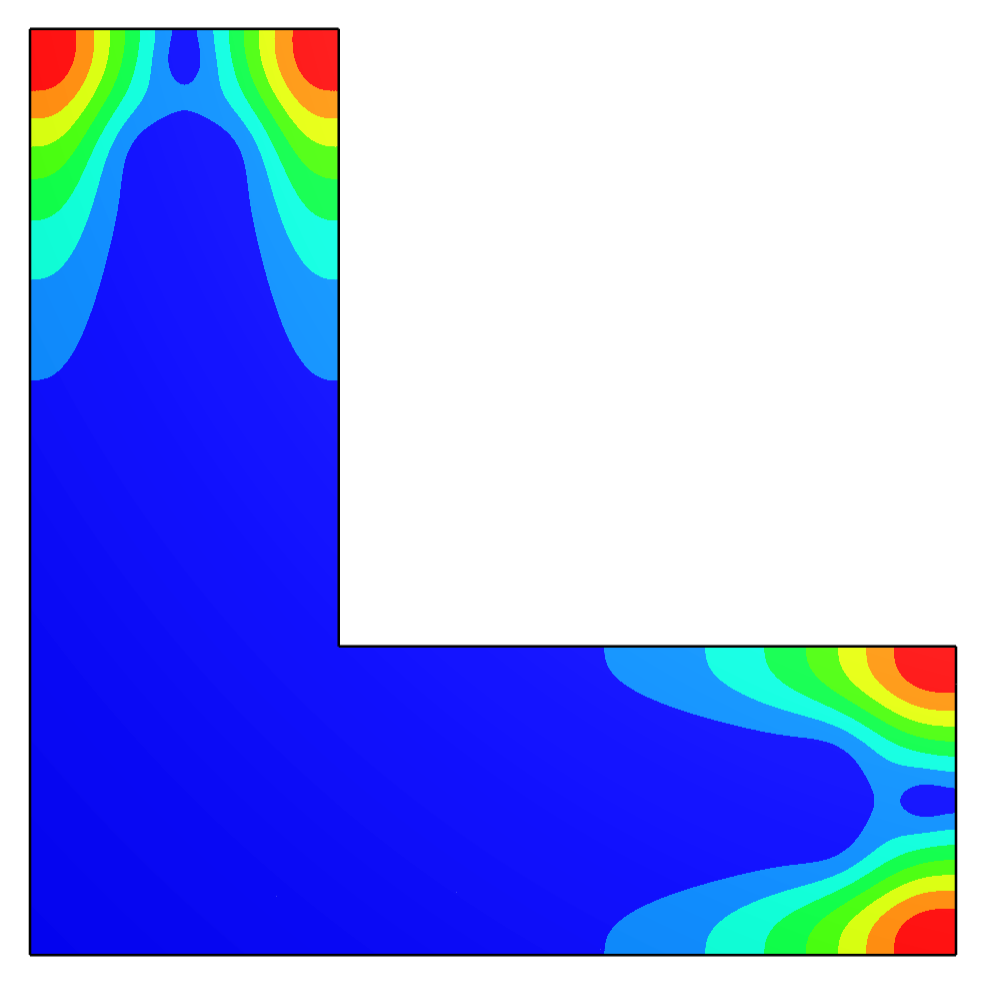}
\includegraphics[width=.16\textwidth]{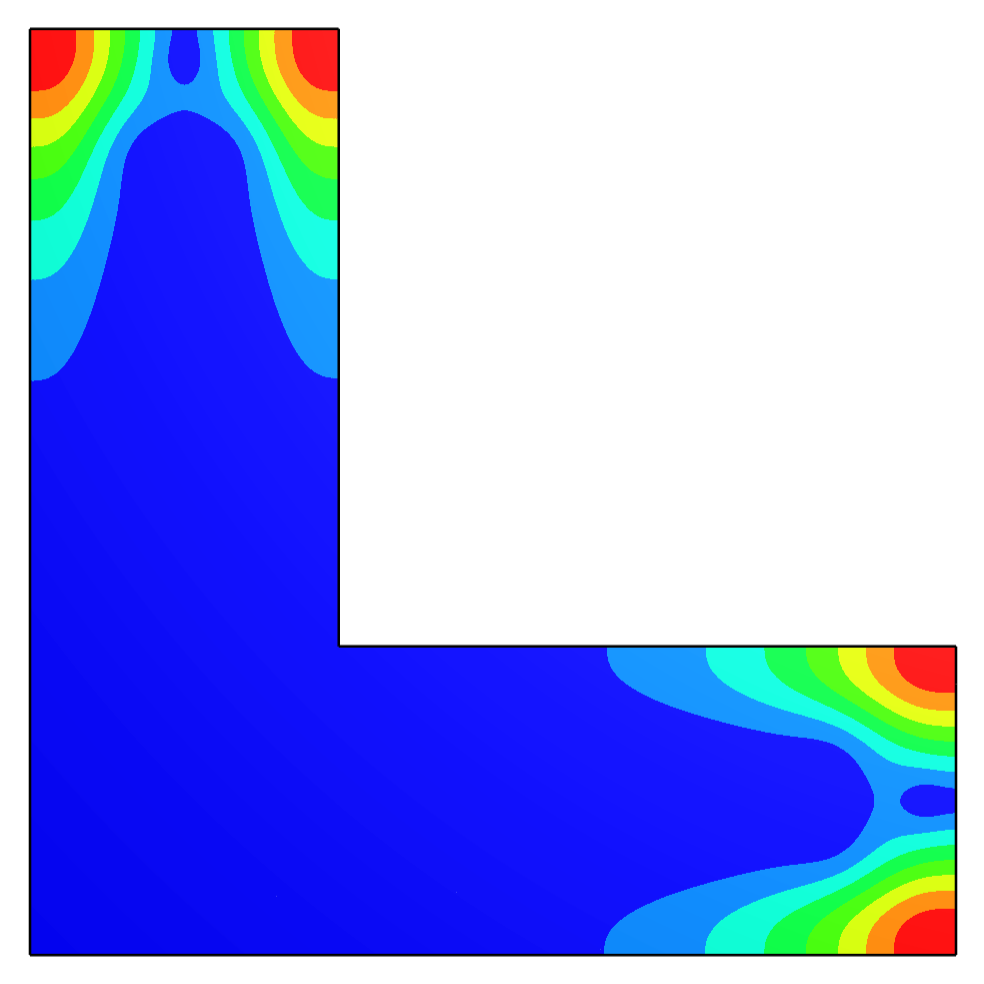}
\includegraphics[width=.16\textwidth]{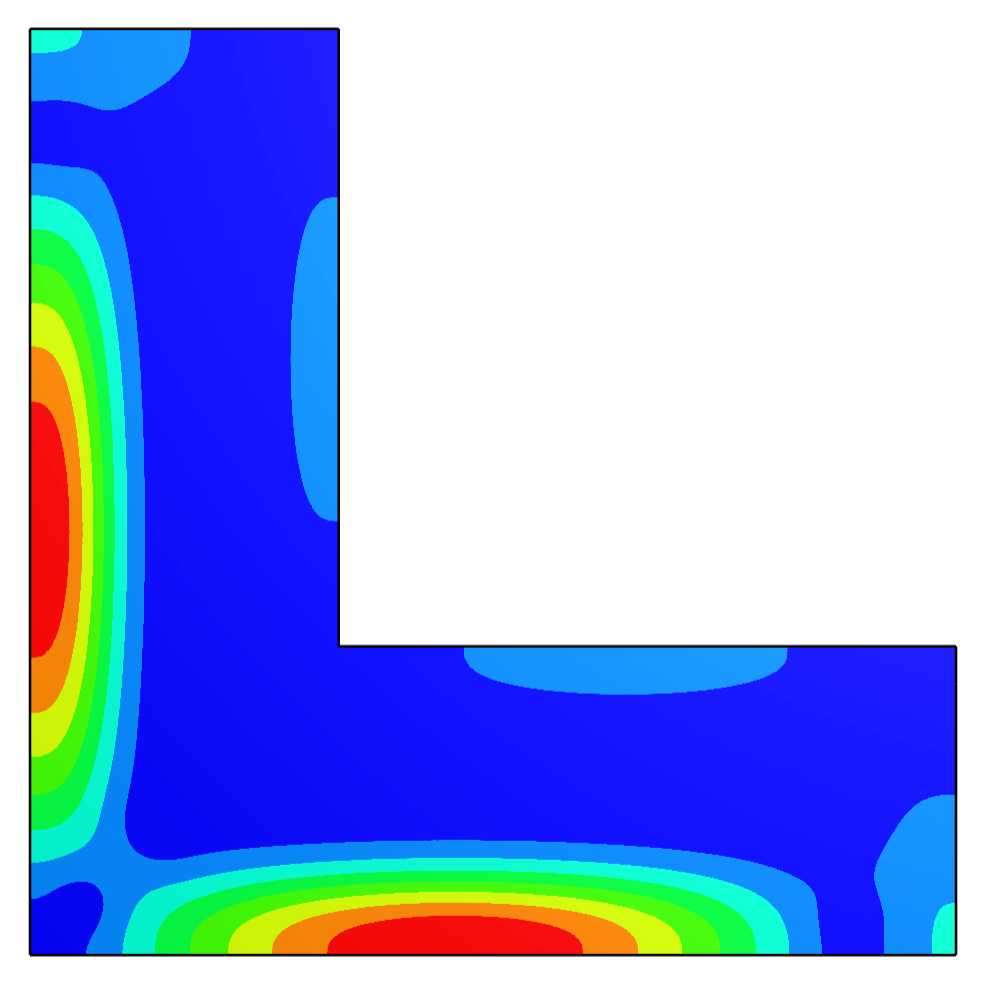}
\includegraphics[width=.16\textwidth]{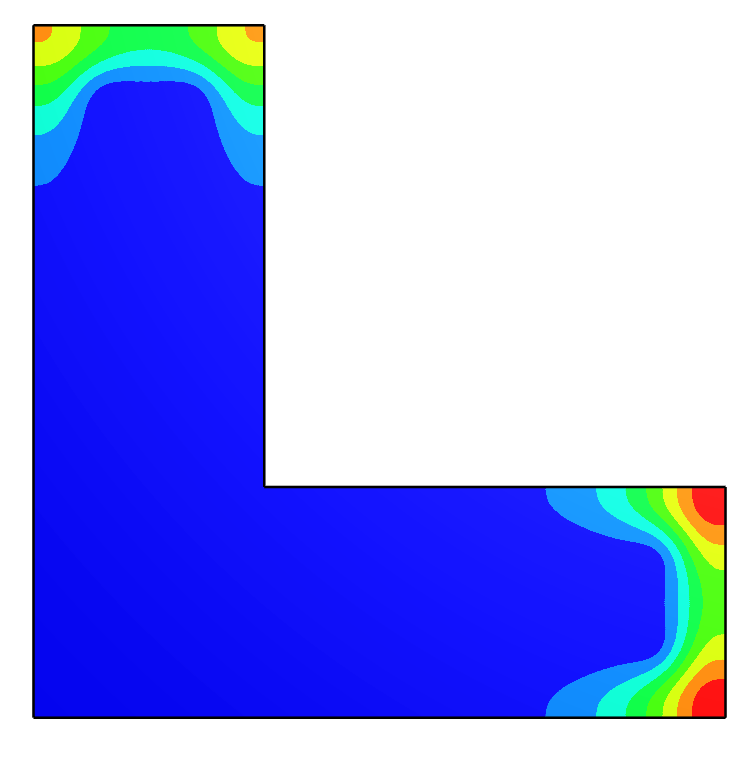}
\includegraphics[width=.16\textwidth]{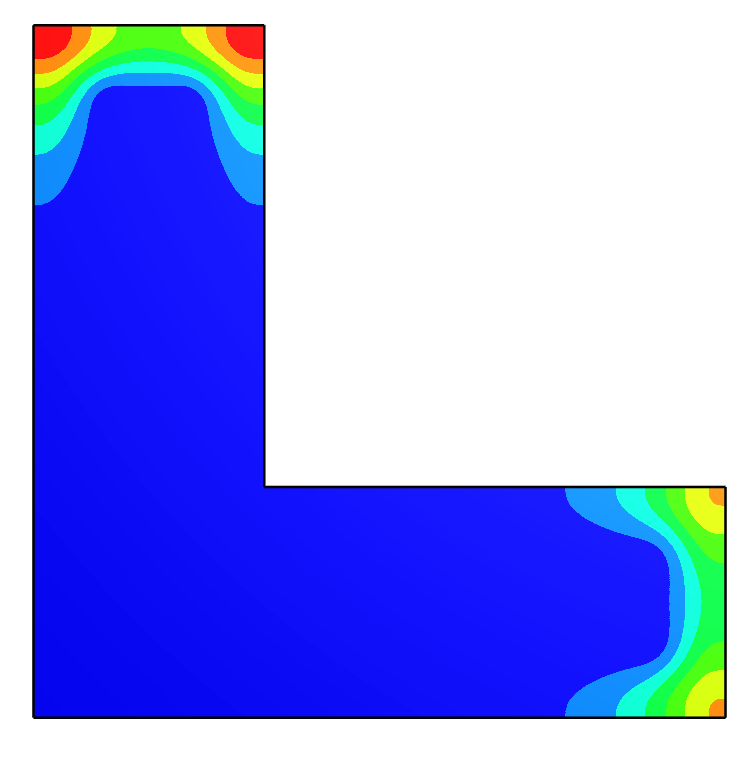}
\includegraphics[width=.16\textwidth]{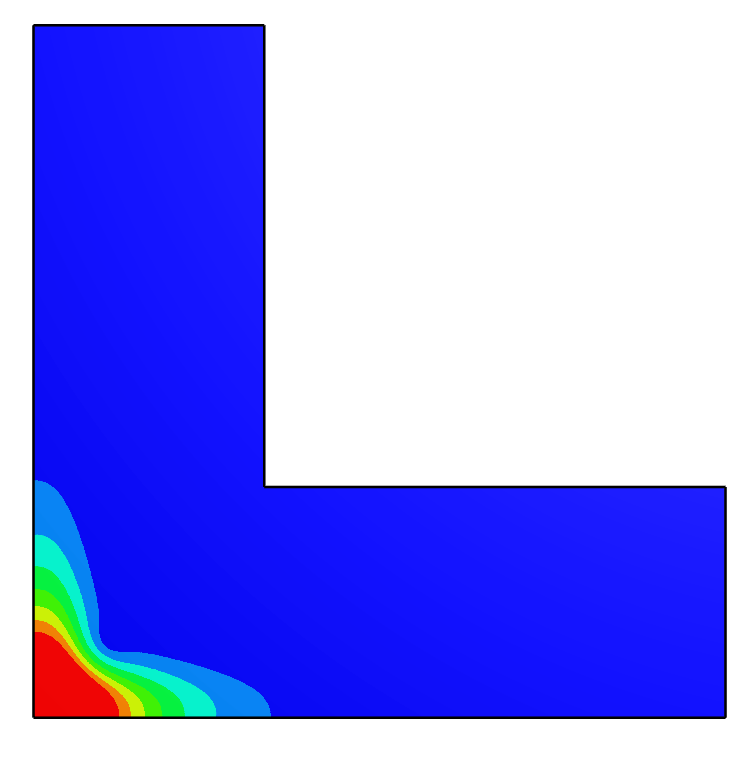}
\includegraphics[width=.16\textwidth]{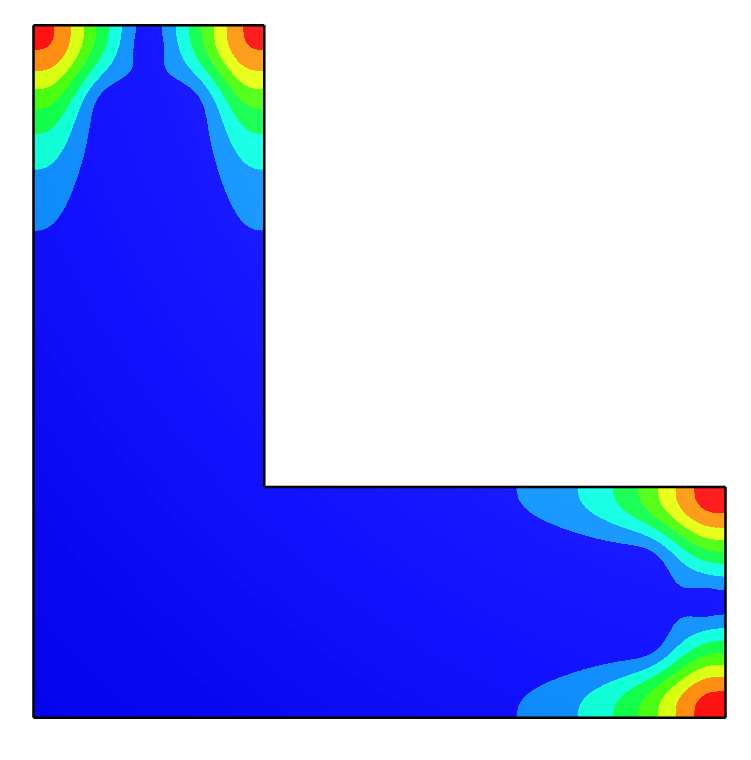}
\includegraphics[width=.16\textwidth]{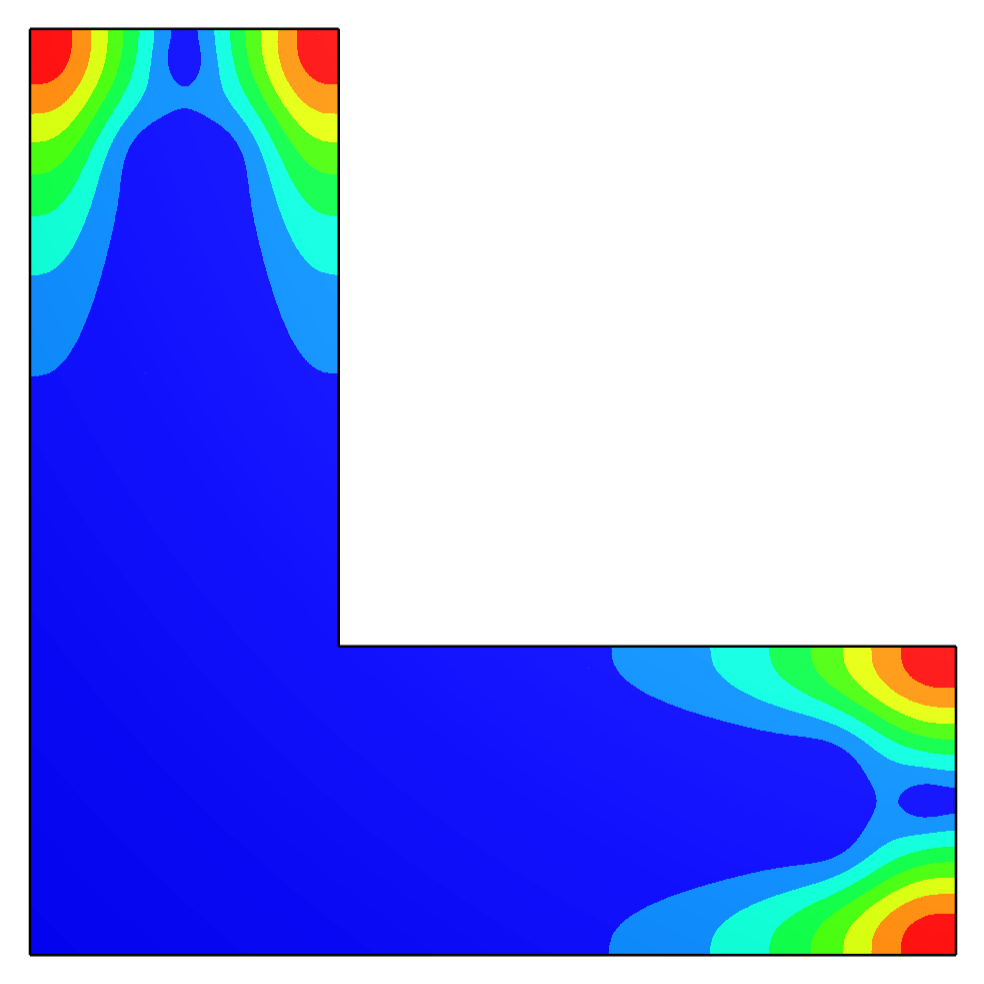}
\includegraphics[width=.16\textwidth]{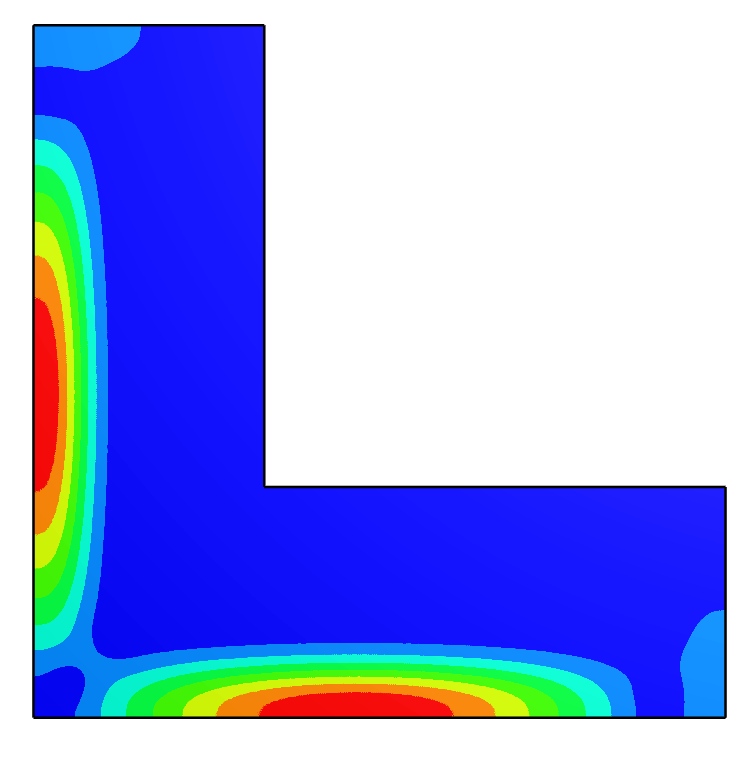}
 \caption{The computed eigenvectors $|\phi_j|$ of $H(\mb{F})$, $1\leq j\leq 6$, when
   $h=0.01$ (top row), $h=0.03$ (middle row), and $h=0.05$, for Example 4.\label{Ex4Fig2}}
 \end{figure}

\section{Conclusion}\label{Conclusions}
We have provided strong heuristic and empirical evidence that the
accurate computation of eigenpairs of the magnetic Schr\"odinger
operator can be done more efficiently and/or stably if a canonical
magnetic gauge is chosen.  This gauge $\mb{F}$ is part of a natural
Helmholtz decomposition of the given magnetic gauge $\mb{A}$, having
minimal norm among all gauges that differ from $\mb{A}$ by a gradient,
and is computed by solving a Poisson problem.  In all examples, we
observe greater stability in the computed eigenpairs with respect to
the mesh parameter $h$ for the canonical gauge---sometimes by a
significant margin.  We also observe a significant reduction in
computational cost to achieve similar accuracy---eigenvalues computed
with the canonical gauge when $h=0.03$ are typically just as accurate
as those computed with the original gauge when $h=0.01$, at a very
small fraction of the cost.


\end{document}